\documentclass[11pt]{article}
\usepackage[a4paper, width=16cm]{geometry}
\usepackage{authblk}
\usepackage{datetime}
\usepackage{amsmath,amsfonts,amssymb,amsthm}
\usepackage{mathtools}  
\usepackage{amsthm}         
\usepackage{breqn}
\usepackage{tcolorbox}
\usepackage{multirow}
\usepackage{hyperref}
\usepackage{fullpage}
\usepackage{enumerate}
\usepackage{subfigure}
\usepackage[title]{appendix}
\usepackage{lineno}
\usepackage[normalem]{ulem}
\usepackage{lmodern}
\usepackage{cite}
\usepackage[normalem]{ulem}
\usepackage[textsize=tiny]{todonotes}

\usepackage{float}

\hypersetup{
	colorlinks=true, 
	linkcolor=blue, 
	urlcolor=red, 
	citecolor=magenta,
	linktoc=all 
}

\usepackage{xcolor}  
\definecolor{lightyellow}{HTML}{fffcf1}

\newcommand{\reva}[1]{\color{black}#1 \color{black}}

\newtheorem{defn}{Definition}[section]
\newtheorem{remark}{Remark}[section]
\newtheorem{lemma}{Lemma}[section]
\newtheorem{thm}{Theorem}[section]
\newtheorem{prop}[thm]{Proposition}

\providecommand{\keywords}[1]
{
	\small	
	\textbf{{Keywords: }} #1
}
\title{High-order finite-difference entropy stable schemes for two-fluid relativistic plasma flow equations}
\date{}
\author[$\dagger$]{Deepak Bhoriya}
\author[$\dagger$]{Harish Kumar}
\author[$\ddagger$]{Praveen Chandrashekar}
\affil[$\dagger$]{Department of Mathematics, Indian Institute of Technology Delhi, New Delhi, India}
\affil[$\ddagger$]{Centre for Applicable Mathematics, Tata Institute of Fundamental Research, Bangalore, India}
\date{}
\begin{document}
	
	\date{}
	\maketitle
	\tableofcontents
	\begin{abstract}
		In this article, we propose high-order finite-difference entropy stable schemes for the two-fluid relativistic plasma flow equations. This is achieved by exploiting the structure of the equations, which consists of three independent flux components. The first two components describe the ion and electron flows, which are modeled using the relativistic hydrodynamics equation and the third component is Maxwell's equations. The coupling of the ion and electron flows and electromagnetic fields is via source terms only, but the source terms do not affect the entropy evolution. To design semi-discrete entropy stable schemes, we extend the entropy stable schemes for relativistic hydrodynamics in \cite{Bhoriya2020} to three dimensions. This is then coupled with entropy stable discretization of the Maxwell's equations. Finally, we use SSP-RK schemes to discretize in time. We also propose ARK-IMEX schemes to treat the stiff source terms; the resulting nonlinear set of algebraic equations is local (at each discretization point) and hence can be solved cheaply using the Newton's Method. The proposed schemes are then tested using various test problems to demonstrate their stability, accuracy and efficiency.
	\end{abstract}
	\keywords{Finite-difference entropy stable schemes, two-fluid relativistic plasma flows, balance laws, IMEX-schemes}
	
	\linenumbers
	\section{Introduction}
	
	Relativistic plasma flows play a central role in astrophysics. In these flows, fluids moving with the speed comparable to the speed of light interact with electric and magnetic fields. Some examples are pulsar winds, gamma-ray bursts, relativistic jets from active galactic nuclei, and quasars \cite{Gallant1994},\cite{Mochkovitch1995},\cite{Landau1987},\cite{Wardle1998}. To model relativistic plasma flows, often, equations of relativistic Magneto-hydrodynamics (RMHD) are used \cite{Komissarov1999},\cite{Balsara2001},\cite{DelZanna2003},\cite{Mignone2006},\cite{Komissarov2007},\cite{Balsara2016aderweno}. However, RMHD has limitations when modeling several astrophysical systems, like pulsar, gamma-ray burst, etc., see \cite{Amano2016} for a detailed discussion. Following the development of non-relativistic two-fluid plasma flow equations \cite{Shumlak2003},\cite{Hakim2006},\cite{Kumar2012},\cite{Abgrall2014},\cite{bond2016plasma},\cite{Li2020},\cite{Meena2019}, several authors have considered an analogous relativistic two-fluid plasma flow equations  \cite{Zenitani2009b},\cite{Zenitani2010},\cite{Amano2013},\cite{Barkov2014},\cite{Barkov2016}. 
	
	Two-fluid relativistic plasma flow equations describe each fluid (ion and electron) in plasma using the equations of special relativistic hydrodynamics (RHD). The fluid components are coupled via electromagnetic quantities using Lorentz force terms. Finally, the electric and magnetic fields are evolved using Maxwell's equations, where current and charges are described using fluid variables. The resulting set of PDEs is a Hyperbolic system of Balance laws with nonlinear flux and stiff source. Due to nonlinearity in the flux, the solutions will exhibit discontinuities \cite{LeVeque2002}. Hence, we need to consider the weak solutions, which can be characterized using the Rankine-Hugoniot condition across the discontinuities. As the weak solutions are non-unique, entropy inequality is imposed to avoid nonphysical solutions. More recently, even the entropy solutions for the systems have been shown to be non-unique \cite{Chiodaroli2015}, still, entropy stability is one of the few nonlinear stability estimates for the solutions, and hence it is desirable to have a numerical scheme which replicates this stability at the discrete level.
	
	As the flux and source terms are nonlinear, in general, it is not possible to find an analytical solution for the two-fluid relativistic plasma flow equations. Hence, the development of stable, efficient and accurate numerical methods is highly desirable. However, in addition to the well-known difficulties in designing stable numerical schemes for hyperbolic balance laws \cite{Godlewski1996},\cite{LeVeque2002}, there are several additional challenges in designing such a scheme for two-fluid relativistic plasma flow equations. First, we need to compute the primitive variables from the conservative variables as the analytical expressions are not available. Second, we must satisfy additional constraints for the magnetic and electric fields. Third, we have nonlinear stiff source terms, and a suitable time implicit treatment of such source is highly desirable. Furthermore, as discussed above, it is also desirable to have entropy stability of the scheme.
	
	Given these difficulties, there are few numerical methods for the two-fluid relativistic plasma flow model. In \cite{Zenitani2009a}, authors have simulated the relativistic magnetic reconnection problem by assuming the symmetric motions of the fluid species, which was further improved in \cite{Zenitani2009b} to include the independent motion of the two-fluid species under a guide field. In \cite{Amano2013}, authors used one-dimensional two-fluid relativistic plasma flow system to investigate the importance of the electric-field-dominated regime in the dynamics of a strongly magnetized plasma by simulating a standing shock front. In \cite{Barkov2014}, authors presented a third-order accurate explicit multidimensional numerical scheme for special relativistic two-fluid plasma flow equations. Here, the authors used the method of generalized Lagrange multiplier to keep the divergence error small. Divergence-free and constraint preserving multidimensional scheme has been designed in \cite{Amano2016,Balsara2016}. 
	
	In this article, we propose arbitrarily high-order finite-difference entropy stable schemes for the relativistic two-fluid plasma flow model. High order finite-difference entropy stable schemes for general hyperbolic systems were first proposed in \cite{Fjordholm2012, Fjordholm2013}. Recently, several authors have developed entropy stable finite-difference scheme for non-relativistic fluids and plasma flows \cite{Kumar2012},\cite{Chandrashekar2013},\cite{Sen2018},\cite{Duan2021}. For the relativistic fluids, Duan {\it et.al}~\cite{Duan2020}, Biswas \cite{Biswas2022}, and Bhoriya {\it et.al}~\cite{Bhoriya2020} have proposed entropy stable schemes.
	
	We proceed as follows.
	
	\begin{itemize}
		\item We first present the entropy framework for the two-fluid relativistic plasma flow equations. We also prove that the source terms do not affect entropy evolution. 
		
		\item To design entropy stable scheme, we exploit the structure of the flux. Similar to the case of \cite{Kumar2012}, the flux consists of three independent blocks. For the fluid parts, we extend the entropy conservative schemes of \cite{Bhoriya2020} to include a three-dimensional velocity field for each fluid component. This is then coupled with a Rusanov's solver based on higher-order Maxwell's discretization. The resultant scheme is then shown to be entropy conservative.
		
		\item To make the scheme entropy stable, following \cite{Fjordholm2012}, \cite{Kumar2012}, \cite{Bhoriya2020}, we design high-order diffusion operator, for each component. As the source does not affect entropy evolution, the semi-discrete scheme is then shown to be entropy stable. 
		
		\item As the source terms are stiff, in addition to the explicit Runge-Kutta method, we also propose IMEX schemes for the model, where flux contributions are treated explicitly, and the source is treated implicitly. The resulting set of nonlinear algebraic equations is local in each cell and are solved using Newton's method for nonlinear systems.
	\end{itemize}
	
	The rest of the article is organized as follows. In Section~\eqref{sec:equations}, we describe the two-fluid relativistic plasma flow equations. In Section~\eqref{sec:hyp_ent}, we discuss hyperbolicity and the entropy framework for the model. In Section~\eqref{sec:semi_disc}, we propose semi-discrete entropy stable numerical schemes for the model. This includes the derivation of the fluid entropy conservative numerical flux and the construction of the high-order entropy diffusion operators. We also describe the discretization of the Maxwell's equations. In Section~\eqref{sec:full_disc}, we present the fully discrete scheme using time explicit and IMEX schemes. In Section ~\eqref{sec:num_test}, we provide detailed numerical tests in one and two dimensions.
	
	\section{Two-fluid relativistic plasma flow equations} \label{sec:equations}
	
	The two-fluid relativistic plasma flows consist of electrons and ions, whose flows are governed by the equations of special relativistic flows. The macroscopic variables for the fluid parts are density, fluid velocity and pressure, denoted by $\rho_{\alpha}, \ \mathbf{u}_{\alpha}$, and $\ p_{\alpha}$, respectively with $\alpha \in \{i, \ e\}$. Here, the subscripts $\{i, \ e \}$ denote the variables corresponding to the ion and electron species. These are coupled to Maxwell's equations via source terms. The complete system is given as follows:
	\begin{subequations}\label{eq:TFRHD_sys}
		\begin{align}
			%
			\frac{\partial (\rho_i \Gamma_i )}{\partial t} + \nabla \cdot (\rho_i \Gamma_i \mathbf{u}_i) &= 0, \label{ion_density} \\ 
			%
			\frac{\partial ( \rho_i h_i \Gamma_i^2 \mathbf{u}_i)}{\partial t} + \nabla \cdot (\rho_i h_i \Gamma_i^2 \textbf{u}_i \mathbf{u}_i^\top + p_i \mathbf{I}) &= r_i \Gamma_i \rho_i (\mathbf{E}+\mathbf{u}_i \times \mathbf{B}), \label{ion_momentum}	\\ 
			%
			\frac{\partial \mathcal{E}_i}{\partial t} + \nabla \cdot ((\mathcal{E}_i+p_i)\mathbf{u}_i) &= r_i \Gamma_i \rho_i (\mathbf{u}_i \cdot \mathbf{E}) ,	\label{ion_energy} \\ 
			%
			\frac{\partial ( \rho_e \Gamma_e)}{\partial t} + \nabla \cdot (\rho_e \Gamma_e \mathbf{u}_e) &= 0, \label{elec_density} \\ 
			%
			\ \frac{\partial ( \rho_e h_e \Gamma_e^2 \mathbf{u}_e)}{\partial t} + \nabla \cdot (\rho_e h_e \Gamma_e^2 \textbf{u}_e \mathbf{u}_e^\top + p_e \mathbf{I}) &= r_e \Gamma_e \rho_e (\mathbf{E}+\mathbf{u}_e \times \mathbf{B}), \label{elec_momentum} \\ 
			%
			\frac{\partial \mathcal{E}_e}{\partial t} + \nabla \cdot ((\mathcal{E}_e+p_e) \mathbf{u}_e) &= r_e \Gamma_e \rho_e (\mathbf{u}_e \cdot \mathbf{E}) , \label{elec_energy} \\ 
			%
			\reva{\dfrac{\partial \mathbf{B}}{\partial t} +  \nabla \times \mathbf{E} +\kappa \nabla \psi &= 0, \label{max1}} \\
			%
			\reva{	\dfrac{\partial \mathbf{E}}{\partial t} - \nabla \times \mathbf{B} +\chi \nabla \phi &= -\mathbf{j} , \label{max2}}\\
			\reva{\dfrac{\partial \phi }{\partial t}+\chi \nabla \cdot \mathbf{E} &= \chi\rho_c,} \label{max3}\\
			\reva{\dfrac{\partial \psi }{\partial t}+\kappa \nabla \cdot \mathbf{B} &= 0\label{max4},}
		\end{align}
	\end{subequations}
	where, $\mathcal{E}_\alpha=\rho_\alpha h_\alpha \Gamma_\alpha^2 - p_\alpha$ are the energy densities, $h_\alpha$ are the specific enthalpy, $\Gamma_\alpha$ are the Lorentz factors, given by $\Gamma_\alpha = \dfrac{1}{\sqrt{1 - \mathbf{u}_\alpha^2}}$, $\mathbf{B}=(B_x,B_y,B_z)$ is the magnetic field vector, and $\mathbf{E}=(E_x,E_y,E_z)$ is the electric field vector. We have assumed the speed of light to be unity. We use notations, $q_\alpha$ for the particle charge and $m_\alpha$ for the particle mass, while the charge to mass ratios are denoted by $r_{\alpha}=\dfrac{q_\alpha}{m_\alpha}$, $\alpha \in \{i,\ e\}.$ The total charge density $\rho_c$ and the current density vector $\mathbf{j}=(j_x,j_y,j_z)$ are given by
	\[
	\rho_c = r_i \rho_i \Gamma_i  + r_e \rho_e \Gamma_e, 
	\qquad
	\mathbf{j} = r_i \rho_i \Gamma_i \mathbf {u}_i + r_e \rho_e \Gamma_e \mathbf{u}_e.
	\]
	The above system of equations is closed using the equation of state, 
	$ h_\alpha=h_\alpha(p_\alpha,\rho_\alpha),$
	where we use the ideal equation of state for $h_\alpha$, which is given by
	\begin{equation*}
		h_\alpha = 1+ \frac{\gamma_\alpha}{\gamma_\alpha - 1} \frac{p_\alpha}{\rho_\alpha} \label{h},
	\end{equation*}
	where, $\gamma_\alpha=c_{p_\alpha}/c_{V_\alpha}$ is the ratio of specific heats. Accordingly, the polytropic index $n_\alpha$, and the sound speed $c_\alpha = c_{s_\alpha}$ can be written as
	\begin{equation*}
		n_\alpha=k_\alpha-1 \qquad \text{  and  } \qquad c_\alpha^2=\frac{k_\alpha p_\alpha }{n_\alpha \rho_\alpha h_\alpha}, 
	\end{equation*}
	where $k_\alpha = \dfrac{\gamma_\alpha}{\gamma_\alpha-1}$ are constants. 
	
	Equations \eqref{ion_density} and \eqref{elec_density} are mass conservation laws for ion and electrons. Similarly, Equations \eqref{ion_momentum} and \eqref{elec_momentum} represent momentum balance, where source terms are due to Lorentz force acting on the fluid due to the electromagnetic variables. The energy conservation equations for ion and electron fluids are \eqref{ion_energy} and \eqref{elec_energy}, respectively. Here, source terms are kinetic energy contributions due to electric and magnetic fields. 
	
	\reva{To deal with the electromagnetic constraints $\nabla\cdot\mathbf B=0$ and $\nabla\cdot\mathbf E = \rho_c$, in \eqref{max1}-\eqref{max4}, we consider the perfectly hyperbolic formulation of the Maxwell's equations \cite{Munz2000}. Here, \eqref{max1} and \eqref{max2} are the evolution equations for the magnetic and electric fields, respectively. Equations \eqref{max3} and \eqref{max4} are equations for the correction potentials $\phi$ and $\chi$, and $\chi$ and $\kappa$ are penalizing speeds.} 
	\section{Analysis of continuous problem} \label{sec:hyp_ent}
	%
	Let us introduce the notations, $D_\alpha = \Gamma_\alpha \rho_\alpha$ and $\mathbf{M}_\alpha=(M_{x_\alpha},M_{y_\alpha},M_{z_\alpha})=\rho_\alpha h_\alpha \Gamma^2_\alpha \mathbf{u}_\alpha$ for $\alpha\in\{i, e\}$ . The vector of conservative variables $\mathbf{U}$ can be written as, 
	\[
	\mathbf{U}=(\mathbf{U}_i^\top, \mathbf{U}_e^\top, \mathbf{U}_m^\top)^\top, 	
	\]
	where $\mathbf{U}_i=(D_i, \mathbf{M}_i, \mathcal{E}_i)^\top$ are ion fluid variables, $\mathbf{U}_e=(D_e, \mathbf{M}_e, \mathcal{E}_e)^\top$ are electron fluid variables, and $\mathbf{U}_m = (\mathbf{B}, \ \mathbf{E}, \ \phi, \ \psi)^\top$ are Maxwell's variables. In two dimensions, let us also denote $x$- and $y$-directional fluxes as $\mathbf{f}^x$ and $\mathbf{f}^y$, respectively. From \eqref{eq:TFRHD_sys}, expressions for the fluxes can be written as,
	\[
	\mathbf{f}^x =\begin{pmatrix}
		\mathbf{f}^x_i(\mathbf{U}_i)\\
		\mathbf{f}^x_e(\mathbf{U}_e)\\
		\mathbf{f}^x_m(\mathbf{U}_m)	
	\end{pmatrix}\qquad
	\text{   and   }\qquad
	\mathbf{f}^y =\begin{pmatrix}
		\mathbf{f}^y_i(\mathbf{U}_i)\\
		\mathbf{f}^y_e(\mathbf{U}_e)\\
		\mathbf{f}^y_m(\mathbf{U}_m)	
	\end{pmatrix},
	\]
	where
	\begin{align}
		\mathbf{f}_\alpha^x = 
		\begin{pmatrix}
			D_\alpha u_{x_\alpha}\\
			M_{x_\alpha} u_{x_\alpha} + p_\alpha\\
			M_{y_\alpha} u_{x_\alpha}\\
			M_{z_\alpha} u_{x_\alpha}\\
			M_{x_\alpha}
		\end{pmatrix},
		\qquad
		\mathbf{f}_\alpha^y = 
		\begin{pmatrix}
			D_\alpha u_{y_\alpha}\\
			M_{x_\alpha} u_{y_\alpha}\\
			M_{y_\alpha} u_{y_\alpha}+p_\alpha\\
			M_{z_\alpha} u_{y_\alpha}\\
			M_{y_\alpha}  
		\end{pmatrix}, \qquad
		\mathbf{f}_m^x = 
			\begin{pmatrix}
				\kappa \psi\\
				- E_z\\
				E_y\\
				\chi \phi\\
				B_z\\
				- B_y\\
				\chi E_x\\
				\kappa B_x
			\end{pmatrix},
			\qquad
			\mathbf{f}_m^y = 
			\begin{pmatrix}
				E_z\\
				\kappa \psi\\
				- E_x\\
				- B_z\\
				\chi \phi\\  
				B_x\\
				\chi E_y\\
				\kappa B_y
		\end{pmatrix}
	\end{align}
	We note that the flux for the whole system contains three independent parts, two fluid parts modeled using the equations of special relativistic flows, and the third part is linear Maxwell's flux. These three parts are coupled via source terms, which are given as, 
	\[
	\mathbf{s}= \begin{pmatrix}
		\mathbf{s}_i(\mathbf{U}_i,\mathbf{U}_m)\\
		\mathbf{s}_e (\mathbf{U}_e,\mathbf{U}_m)\\
		\mathbf{s}_m(\mathbf{U}_i,\mathbf{U}_e)	
	\end{pmatrix},
	\]
	where,
	\begin{align*}
		\mathbf{s}_\alpha =
		\begin{pmatrix}
			0\\
			r_\alpha D_\alpha(E_x + u_{y_\alpha} B_z - u_{z_\alpha}B_y)\\
			r_\alpha D_\alpha(E_y + u_{z_\alpha} B_x - u_{x_\alpha}B_z)\\
			r_\alpha D_\alpha(E_z + u_{x_\alpha} B_y - u_{y_\alpha}B_x) \\
			\ r_\alpha D_\alpha(u_{x_\alpha}E_x + u_{y_\alpha}E_y + u_{z_\alpha}E_z)
		\end{pmatrix}\qquad
		\text{   and   }\qquad
			\mathbf{s}_m = 
			\begin{pmatrix}
				0\\
				0\\ 
				0\\
				-j_x\\
				-j_y\\ 
				-j_z\\
				\chi \rho_c\\
				0
			\end{pmatrix},
	\end{align*}
	for $\alpha\in\{i,e\}$. Using the above notations, the system~\eqref{eq:TFRHD_sys}, for the two-dimensional case can be written in the conservation form as follows:
	\begin{equation}
		\frac{\partial \mathbf{U}}{\partial t}+\frac{\partial \mathbf{f}^x}{\partial x} +\frac{\partial \mathbf{f}^y}{\partial y}= \mathbf{s} \label{conservedform}.
	\end{equation}
	Let us also introduce the vector of primitive variables  $\mathbf{W}=(\mathbf{W}_i^\top, \mathbf{W}_e^\top, \mathbf{W}_m^\top)^\top$, where $\mathbf{W}_\alpha = (\rho_\alpha, \mathbf{u}_\alpha, {p}_\alpha)^\top$ and $\mathbf{W}_m=(\mathbf{B,E},\phi,\psi)^\top$. We follow the procedure given in~\cite{Schneider1993,Bhoriya2020} to extract the primitive variables from the conservative variables. We define the set of admissible solution space 
	\begin{equation*}
		\Omega = \{ \mathbf{U} \in\mathbb{R}^{18}: \ \rho_i>0, \rho_e>0, \ p_i>0,p_e>0, \ |\mathbf{u}_i|<1, |\mathbf{u}_e| < 1 \}.
	\end{equation*}
	The eigenvalues of the system in $x$-direction are,
	\begin{align*} 
		\Lambda^x 
		= 
		\biggl\{
		& 	\frac{(1-c_i^2)u_{x_i}-(c_i/\Gamma_i) \sqrt{Q_i^x}}{1-c_i^2 |\mathbf{u}_i|^2},\
		u_{x_i}, \ u_{x_i}, \ u_{x_i},
		\frac{(1-c_i^2)u_{x_i}+(c_i/\Gamma_i) \sqrt{Q_i^x}}{1-c_i^2 |\mathbf{u}_i|^2}, 
		\nonumber
		\\  
		& 	\frac{(1-c_e^2)u_{x_e}-(c_e/\Gamma_e) \sqrt{Q_e^x}}{1-c_e^2 |\mathbf{u}_e|^2},\
		u_{x_e}, \ u_{x_e}, \ u_{x_e},
		\frac{(1-c_e^2)u_{x_e}+(c_e/\Gamma_e) \sqrt{Q_e^x}}{1-c_e^2 |\mathbf{u}_e|^2}, 
		\nonumber
		\\
		&	-\chi, -\kappa, -1, \ -1, \	1, \ 1, \ \kappa, \ \chi
			\biggr\},
	\end{align*}
	where, $Q_\alpha^x=1-u_{x_{\alpha}}^2-c_\alpha^2 (u_{y_\alpha}^2+u_{z_\alpha}^2), \ \alpha \in \{ i, \ e \}$.  For $\mathbf{U} \in \Omega$ and $\gamma_\alpha \in (1,2]$ we have $c_\alpha<1$ which implies that $Q_\alpha^x > 0$, hence all the eigenvalues are real. The complete set of right eigenvectors for the Jacobian matrix $\dfrac{\partial \mathbf{f}^x}{\partial \mathbf{U}}$ is given in Appendix~\eqref{ap:right_eig}. The Appendix also contains the eigenvalues and right eigenvectors for the jacobian matrix $\dfrac{\partial \mathbf{f}^y}{\partial \mathbf{U}}$. Consequently, we can state the following Lemma.
	\begin{lemma}
		The system~\eqref{conservedform} is hyperbolic for the states $\mathbf{U} \in \Omega$, with real eigenvalues and a complete set of eigenvectors. \label{hyperboliclemma}
	\end{lemma}
	We now introduce entropy functions for the two-fluid relativistic plasma flow equations. The entropy functions $\mathcal{U}_\alpha$ and associated entropy fluxes $\mathcal{F}^d_\alpha$ for the fluid part of the system~\eqref{conservedform} are given by,
	\begin{equation*}
		\mathcal{U}_\alpha=-\frac{\rho_\alpha \Gamma_\alpha s_\alpha}{\gamma_\alpha -1} \qquad \text{ and } \qquad \mathcal{F}^d_\alpha=-\frac{\rho_\alpha \Gamma_\alpha s_\alpha u_{d_\alpha}}{\gamma_\alpha -1}, \qquad \alpha \in \{i,e\}, \quad d=x,y. 
	\end{equation*} 
	Here, $s_\alpha = \ln(p_\alpha \rho_\alpha^{-\gamma_\alpha})$. The pair $(\mathcal{U}_\alpha,\mathcal{F}^d_\alpha)$  is called an {\em entropy-entropy flux pair}. For simplicity, we only consider the one-dimensional case, i.e., 
	\begin{equation}
		\frac{\partial \mathbf{U}}{\partial t}+\frac{\partial \mathbf{f}^x}{\partial x}= \mathbf{s} \label{conservedform_1d}.
	\end{equation}
	The extension of entropy framework to higher dimensions is straightforward. We will now prove the following result:
	\begin{prop} \label{prop:entropy}
		The smooth solutions of \eqref{conservedform_1d} satisfy the entropy equality, 
		\begin{equation*}
			\partial_t s_\alpha+u_{x_\alpha} \partial_x s_\alpha=0, \ \ \alpha \in \{i,e\}. 
		\end{equation*}
		As a consequence, every smooth function $H(s_\alpha)$ of $s_\alpha$ satisfies
		\begin{equation}
			\partial_t (\rho_\alpha \Gamma_\alpha H(s_\alpha))+ \partial_x (\rho_\alpha \Gamma_\alpha u_{x_\alpha} H(s_\alpha))=0. \label{h_s}
		\end{equation} 
		In particular, we have the following entropy equality for the smooth solutions,
		\begin{equation}
			\partial_t \mathcal{U}_\alpha+ \partial_x \mathcal{F}^x_\alpha=0. \label{entropy_pair_equality}
		\end{equation}
	\end{prop}
	The proof is given in the Appendix~\eqref{ap:entropy}.
	\begin{remark}
		The entropy equality~\eqref{entropy_pair_equality} in Proposition~\eqref{prop:entropy} is replaced by the entropy inequality
		\begin{equation}
			\partial_t \mathcal{U}_\alpha + \partial_x \mathcal{F}^x_\alpha \le 0, \label{ent_inq}
		\end{equation}   
		in the sense of distributions for non-smooth solutions.
	\end{remark}
	In the above analysis, we also note that the Maxwell's flux does not affect the fluid entropies. This fact will be exploited to design semi-discrete high-order numerical schemes which will satisfy the entropy inequality~\eqref{ent_inq}.
	
	\section{Semi-discrete entropy stable schemes} \label{sec:semi_disc}
	
	In this section, we will design semi-discrete entropy stable numerical schemes for the two-dimensional two-fluid relativistic plasma flow Eqns. \eqref{conservedform}. The extension to three dimensions is fairly easy. Let us consider the domain $D=I_x \times I_y$, where $I_x=(x_{min},x_{max})$ and $I_y=(y_{min}, y_{max})$ which is discretized using a uniform mesh with cells of size $\Delta x \times \Delta y$. We define $x_i = x_{min} + i \Delta x$ for $0 \le i \le N_x$ and $y_j = y_{min} + j \Delta y$ for $0 \le j \le N_y$. The cell $(i,j)$ is given by $I_{i,j} = [x_{i-1/2}, x_{i+1/2}] \times [y_{j-1/2}, y_{j+1/2}]$, where $x_{i+1/2}= \frac{x_i+x_{i+1}}{2}$ and $y_{j+1/2}= \frac{y_j+y_{j+1}}{2}$.\\
	
	\par The semi-discrete finite difference scheme for the evolution of solution $\mathbf{U}_{i,j}$ at $(x_i,y_j)$ of system \eqref{conservedform} is given by,
	\begin{equation}
		\frac{d}{dt}\mathbf{U}_{i,j}(t)
		+
		\frac{1}{\Delta x}  \left(\mathbf{F}_{i+\frac{1}{2},j}^x(t)-\mathbf{F}_{i-\frac{1}{2},j}^x(t)\right)
		+
		\frac{1}{\Delta y}  \left(\mathbf{F}_{i,j+\frac{1}{2}}^y(t)-\mathbf{F}_{i,j-\frac{1}{2}}^y(t)\right) = \mathbf{s}(\mathbf{U}_{i,j}(t)), \label{scheme}
	\end{equation} 
	where $\mathbf{F}_{i+\frac{1}{2},j}^x$, $\mathbf{F}_{i,j+\frac{1}{2}}^y$ are the numerical fluxes consistent with the continuous fluxes $\mathbf{f}^x$, $\mathbf{f}^y$, respectively. 
	\subsection{Entropy stable numerical schemes for the fluid equations}
	To simplify the discussion, we will first present the discretization of the fluid part, i.e., we consider the following differential form:
	\begin{equation}
		\frac{\partial \mathbf{U}_\alpha}{\partial t}+\frac{\partial \mathbf{f}^x_\alpha}{\partial x} +\frac{\partial \mathbf{f}^y_\alpha}{\partial y}= \mathbf{s}_\alpha(\mathbf{U}_{\alpha},\mathbf{U}_m) \label{eq:fluidpart_2d},
	\end{equation}
	for $\alpha\in\{i, e\}$. Let us now define:
	\begin{defn}
		The numerical scheme
		\begin{eqnarray}
			\frac{d}{dt}\mathbf{U}_{\alpha,i,j}(t)
			+
			\frac{1}{\Delta x}  \left(\mathbf{F}_{\alpha,i+\frac{1}{2},j}^x(t)-\mathbf{F}_{\alpha,i-\frac{1}{2},j}^x(t)\right)
			+
			\frac{1}{\Delta y}  \left(\mathbf{F}_{\alpha,i,j+\frac{1}{2}}^y(t)-\mathbf{F}_{\alpha,i,j-\frac{1}{2}}^y(t)\right) \nonumber \\= \mathbf{s}_\alpha(\mathbf{U}_{\alpha,i,j}(t),\mathbf{U}_{m,i,j}(t)), \label{eq:scheme_fluid}
		\end{eqnarray} 
		is said to be entropy conservative if the following entropy equality is satisfied
		\begin{equation}
			\frac{d}{dt}  \mathcal{U}_\alpha(\mathbf{U}_{ij})  +\frac{1}{\Delta x} \left( \hat{\mathcal{F}}_{\alpha,i+\frac{1}{2},j}^x - \hat{\mathcal{F}}_{\alpha,i-\frac{1}{2},j}^x\right)+\frac{1}{\Delta y}\left( \hat{\mathcal{F}}_{\alpha,i,j+\frac{1}{2}}^y - \hat{\mathcal{F}}_{\alpha,i,j-\frac{1}{2}}^y\right) = 0, \ \ \ \alpha \in \{i,e\}, \label{eq:semi_dis_fluid_entropy_equality}
		\end{equation}
		where $\hat{\mathcal{F}}_{\alpha,i+\frac{1}{2},j}^x$,  $\hat{\mathcal{F}}_{\alpha,i,j+\frac{1}{2}}^y$ are some numerical entropy flux functions consistent with the entropy fluxes ${\mathcal{F}}_\alpha^x$, ${\mathcal{F}}_\alpha^y$, respectively. Here $\mathbf{F}_{\alpha,i+\frac{1}{2},j}^x$,  $\mathbf{F}_{\alpha,i,j+\frac{1}{2}}^y$ are the numerical fluxes consistent with $\mathbf{f}^x$, $\mathbf{f}^y$, respectively.
		
		The scheme is said to be entropy stable if the following entropy inequality is satisfied:
		\begin{equation}
			\frac{d}{dt}  \mathcal{U}_\alpha(\mathbf{U}_{\alpha,i,j})  +\frac{1}{\Delta x} \left( \hat{\mathcal{F}}_{\alpha,i+\frac{1}{2},j}^x - \hat{\mathcal{F}}_{\alpha,i-\frac{1}{2},j}^x\right)+\frac{1}{\Delta y}\left( \hat{\mathcal{F}}_{\alpha,i,j+\frac{1}{2}}^y - \hat{\mathcal{F}}_{\alpha,i,j-\frac{1}{2}}^y\right) \le 0 \ \ \ \alpha \in \{i,e\}, \label{eq:semi_dis_fluid_entropy_inequality}
		\end{equation}
	\end{defn}
	
	To obtain the entropy stable schemes, we will first construct entropy conservative schemes and then add dissipative fluxes to obtain entropy inequality.
	\subsubsection{High-order entropy conservative schemes for fluid equations}
	Let us introduce {\em entropy variables}, 
	$$
	\mathbf{V}_\alpha(\mathbf{U}_\alpha)=\frac{\partial  \mathcal{U}_\alpha}{\partial \mathbf{U}_\alpha}$$
	and
	{\em entropy potentials}, 
	$$ \psi_\alpha^x(\mathbf{U}_\alpha)=\mathbf{V}_\alpha^\top(\mathbf{U}_\alpha) \cdot \mathbf{f}_\alpha^x(\mathbf{U}_\alpha)-\mathcal{F}_\alpha^x(\mathbf{U}_\alpha),\qquad  \psi_\alpha^y(\mathbf{U}_\alpha)=\mathbf{V}_\alpha^\top(\mathbf{U}_\alpha) \cdot \mathbf{f}_\alpha^y(\mathbf{U}_\alpha)-\mathcal{F}_\alpha^y(\mathbf{U}_\alpha)
	$$
	for $\alpha \in \{i,e\}$. A simple calculation results in the following expressions
	\begin{equation}
		\mathbf{V}_\alpha=\begin{pmatrix}
			\dfrac{\gamma_\alpha- s_\alpha}{\gamma_\alpha -1} +{\beta_\alpha} \\
			{u_{x_\alpha} \Gamma_\alpha \beta_\alpha } \\ 
			{u_{y_\alpha} \Gamma_\alpha \beta_\alpha } \\
			{u_{z_\alpha} \Gamma_\alpha \beta_\alpha } \\
			-{\Gamma_\alpha \beta_\alpha}
		\end{pmatrix}, \qquad
		\text{with } \qquad
		\beta_\alpha = \dfrac{\rho_\alpha}{p_\alpha},
		\label{eq:ent_var} 
	\end{equation}
	and
	\begin{equation}
		{\psi}_\alpha^x=\rho_\alpha \Gamma_\alpha u_{x_\alpha},\qquad {\psi}_\alpha^y=\rho_\alpha \Gamma_\alpha u_{y_\alpha}.
		\label{eq:ent_pot}
	\end{equation}
	We also introduce the following notations for jumps and averaging operations over the cell interfaces
	\begin{equation*}
		[\![a]\!]_{i+\frac{1}{2},j}=a_{i+1,j}-a_{i,j},  \qquad \bar{a}_{i+\frac{1}{2},j}=\frac{1}{2}(a_{i+1,j}+a_{i,j}),
	\end{equation*}
	\begin{equation*}
		[\![a]\!]_{i,j+\frac{1}{2}}=a_{i,j+1}-a_{i,j}, \qquad \bar{a}_{i,j+\frac{1}{2}}=\frac{1}{2}(a_{i,j+1}+a_{i,j}).
	\end{equation*}
	Let us first consider the homogeneous fluid part, i.e., we consider
	\begin{equation}
		\frac{\partial \mathbf{U}_\alpha}{\partial t}+\frac{\partial \mathbf{f}^x_\alpha}{\partial x} +\frac{\partial \mathbf{f}^y_\alpha}{\partial y}= 0. \label{eq:fluidpart_2d_homo}
	\end{equation}
	and the corresponding semi-discrete scheme
	\begin{eqnarray}
		\frac{d}{dt}\mathbf{U}_{\alpha,i,j}(t)
		+
		\frac{1}{\Delta x}  \left(\mathbf{F}_{\alpha,i+\frac{1}{2},j}^x(t)-\mathbf{F}_{\alpha,i-\frac{1}{2},j}^x(t)\right)
		+
		\frac{1}{\Delta y}  \left(\mathbf{F}_{\alpha,i,j+\frac{1}{2}}^y(t)-\mathbf{F}_{\alpha,i,j-\frac{1}{2}}^y(t)\right) =0 \label{eq:scheme_fluid_homo}
	\end{eqnarray} 
	Then we have the following result from  \cite{Tadmor1987}.
	
	\begin{thm}[{{\em Tadmor}} \cite{Tadmor1987}] 
		Let $\tilde{\mathbf{F}}^x_\alpha$ and $\tilde{\mathbf{F}}^y_\alpha$ be the consistent numerical fluxes which satisfy
		\begin{equation}
			[\![\mathbf{V}_\alpha ]\!]^{\top}_{i+\frac{1}{2},j}\,\tilde{\mathbf{F}}_{\alpha,i+\frac{1}{2},j}^x=[\![\psi^x_\alpha]\!]_{i+\frac{1}{2},j}, \ \quad \ 
			[\![\mathbf{V}_\alpha ]\!]^{\top}_{i,j+\frac{1}{2}}\,\tilde{\mathbf{F}}_{\alpha,i,j+\frac{1}{2}}^y=[\![\psi^y_\alpha]\!]_{i,j+\frac{1}{2}},
			\label{tadmor_thm}
		\end{equation}
		then the scheme~\eqref{eq:scheme_fluid_homo} with the numerical fluxes $\tilde{\mathbf{F}}^x_\alpha$ and $\tilde{\mathbf{F}}^y_\alpha$ is second order accurate and entropy conservative, i.e., the entropy equality
		\begin{equation*}
			\frac{d}{dt}  \mathcal{U}_\alpha(\mathbf{U}_{ij})  +\frac{1}{\Delta x} \left( \tilde{\mathcal{F}}_{\alpha,i+\frac{1}{2},j}^x - \tilde{\mathcal{F}}_{\alpha,i-\frac{1}{2},j}^x\right)+\frac{1}{\Delta y}\left( \tilde{\mathcal{F}}_{\alpha,i,j+\frac{1}{2}}^y - \tilde{\mathcal{F}}_{\alpha,i,j-\frac{1}{2}}^y\right) = 0,
		\end{equation*}
		is satisfied with the consistent entropy numerical fluxes,
		\begin{equation*}
			\tilde{\mathcal{F}}_{\alpha,i+\frac{1}{2},j}^x=\bar{\mathbf{V}}_{\alpha,i+\frac{1}{2},j}^{\top}\tilde{\mathbf{F}}_{\alpha,i+\frac{1}{2},j}^x-\bar{\psi}_{\alpha,i+\frac{1}{2},j}^x, 
			\qquad \text{ and } \qquad 
			\tilde{\mathcal{F}}_{\alpha,i,j+\frac{1}{2}}^y=\bar{\mathbf{V}}_{\alpha,i,j+\frac{1}{2}}^{\top}\tilde{\mathbf{F}}_{\alpha,i,j+\frac{1}{2}}^y-\bar{\psi}_{\alpha,i,j+\frac{1}{2}}^y.
		\end{equation*}
	\end{thm}

	We consider the $x$-directional identity of Eqn.~\eqref{tadmor_thm} and observe that the equation~\eqref{tadmor_thm} contains five unknowns, $\tilde{\mathbf{F}}^x_\alpha=(F^x_{\alpha,1} ,\ F^x_{\alpha,2},\ F^x_{\alpha,3},\ F^x_{\alpha,4},\ F^x_{\alpha,5})^{\top}$. Hence, in general, we will have non-uniqueness for the solutions of this algebraic equation. Several authors~\cite{Ismail2009,Chandrashekar2013} have presented different techniques to find an affordable entropy conservative flux. We will follow the procedure of~\cite{Chandrashekar2013,Bhoriya2020} to derive the expression for the fluxes.  
	
	Consider the flux  at $(i+\frac{1}{2},j)$. To simplify the notation, we suppress the cell indices $i,j$, and let
	\[
	[\![ \cdot ]\!] = (\cdot)_{i+1,j} - (\cdot)_{i,j}, \qquad (\bar\cdot) = \frac{(\cdot)_{i,j} + (\cdot)_{i+1,j}}{2}
	\]
	Define $a^{\ln}=\frac{[\![a]\!]}{[\![\log a]\!]}$ as the logarithmic average for a strictly positive scalar $a$. 
	
	Applying the jump condition $[\![ab]\!]=\bar{a}[\![b]\!]+\bar{b}[\![a]\!]$ on the equation $\Gamma_\alpha= \dfrac{1}{\sqrt{1 - \mathbf{u}^2_\alpha}}$ written as $\Gamma_\alpha^2 = 1 + m_{x_\alpha}^2 + m_{y_\alpha}^2 + m_{z_\alpha}^2$, we obtain
	\begin{equation*}
		[\![\Gamma_\alpha]\!] = \frac{1}{\overline{\Gamma_\alpha}} (\overline{m_{x_\alpha}} [\![m_{x_\alpha}]\!] + \overline{m_{y_\alpha}} [\![m_{y_\alpha}]\!]+ \overline{m_{z_\alpha}} [\![m_{z_\alpha}]\!]) ,
	\end{equation*}
	where, $m_{x_\alpha}=\Gamma_\alpha u_{x_\alpha}$, $m_{y_\alpha}=\Gamma_\alpha u_{y_\alpha}$ and  $m_{z_\alpha}=\Gamma_\alpha u_{z_\alpha}$. Using this, we can write the jump in $\mathbf{V}_\alpha$ in terms of the jump in $\rho_\alpha$, $\beta_\alpha$, $m_{x_\alpha}$, $m_{y_\alpha}$ and $m_{z_\alpha}$ as follows:
	\begin{equation*}
		{[\![\mathbf{V}_\alpha]\!]}=
		\begin{pmatrix}
			\frac{[\![\rho_\alpha]\!]}{{\rho_\alpha}^{\ln}} +
			k_{\alpha} [\![\beta_\alpha]\!] \\
			\overline{m_{x_\alpha}}[\![ \beta_\alpha]\!] + \bar{\beta_\alpha}[\![ m_{x_\alpha}]\!] \\
			\overline{m_{y_\alpha}}[\![ \beta_\alpha]\!] + \bar{\beta_\alpha}[\![ m_{y_\alpha}]\!] \\
			\overline{m_{z_\alpha}}[\![ \beta_\alpha]\!] + \bar{\beta_\alpha}[\![ m_{z_\alpha}]\!] \\
			-\bar{\Gamma_\alpha}[\![\beta_\alpha]\!]-
			\frac{\bar{\beta_\alpha} \overline{m_{x_\alpha}}}{\bar{\Gamma_\alpha}}  [\![m_{x_\alpha}]\!] -
			\frac{\bar{\beta_\alpha} \overline{m_{y_\alpha}}}{\bar{\Gamma_\alpha}}  [\![m_{y_\alpha}]\!] -
			\frac{\bar{\beta_\alpha} \overline{m_{z_\alpha}}}{\bar{\Gamma_\alpha}}  [\![m_{z_\alpha}]\!]
		\end{pmatrix},
	\end{equation*}
	where $\beta_\alpha=\frac{\rho_\alpha}{p_\alpha}$,  and $k_{\alpha}=\left(\dfrac{1}{\gamma-1} \dfrac{1}{{\beta^{\ln}_\alpha}}+1  \right)$. Furthermore, $[\![\psi^x_\alpha]\!]$ can be written as
	\begin{equation*}
		[\![\psi^x_\alpha]\!]=\overline{\rho_\alpha}[\![  m_{x_\alpha}]\!] +\overline{m_{x_\alpha}} [\![ \rho_\alpha ]\!].  
	\end{equation*}
	and Eqn.~\eqref{tadmor_thm} becomes
	\begin{equation*}
		F_{\alpha,1}^x[\![ V_{\alpha,1} ]\!] + F_{\alpha,2}^x[\![ V_{\alpha,2} ]\!] + F_{\alpha,3}^x[\![ V_{\alpha,3} ]\!] + F_{\alpha,4}^x[\![ V_{\alpha,4} ]\!] + F_{\alpha,5}^x[\![ V_{\alpha,5} ]\!]= [\![ \psi^x_\alpha ]\!], 
	\end{equation*}
	which simplifies to,
	\begin{gather*} 
		\left(\frac{F_{\alpha,1}^x}{{\rho_\alpha}^{\ln}} \right) [\![\rho_\alpha]\!]
		+\left(\overline{\beta_\alpha} F_{\alpha,2}^x - \frac{\bar{\beta_\alpha} \overline{m_{x_\alpha}}}{\bar{\Gamma_\alpha}}F_{\alpha,5}^x  \right) [\![m_{x_\alpha}]\!]
		+\left(\overline{\beta_\alpha} F_{\alpha,3}^x - \frac{\bar{\beta_\alpha} \overline{m_{y_\alpha}}}{\bar{\Gamma_\alpha}}F_{\alpha,5}^x  \right) [\![m_{y_\alpha}]\!]
		\\
		+\left(\overline{\beta_\alpha} F_{\alpha,4}^x - \frac{\bar{\beta_\alpha} \overline{m_{z_\alpha}}}{\bar{\Gamma_\alpha}}F_{\alpha,5}^x  \right) [\![m_{z_\alpha}]\!]
		+ \bigg( k_{\alpha} F_{\alpha,1}^x+ \overline{m_{x_\alpha}} F_{\alpha,2}^x+ \overline{m_{y_\alpha}} F_{\alpha,3}^x+ \overline{m_{z_\alpha}} F_{\alpha,4}^x-F_{\alpha,5}^x \bar{\Gamma_\alpha} \bigg)[\![\beta_\alpha]\!]
		\\
		= \bar{\rho_\alpha}[\![  m_{x_\alpha}]\!] +\overline{m_{x_\alpha}} [\![ \rho_\alpha ]\!].
	\end{gather*}
	We want this equation to hold for all possible values of the jumps. This is possible if the coefficients of $[\![  \rho_\alpha]\!]$, $[\![  \beta_\alpha]\!]$, $[\![  m_{x_\alpha}]\!]$, $[\![  m_{y_\alpha}]\!]$, and $[\![ m_{z_\alpha}]\!]$, agree on both sides, which yields the following fluxes,
	\begin{gather}
		\renewcommand{\arraystretch}{1.5}
		\mathbf{\tilde F}^x_{\alpha,i+\frac{1}{2},j} = \mathbf{\tilde{F}}^x_\alpha(\mathbf U_{\alpha,i,j}, \mathbf U_{\alpha,i+1,j}) = \begin{pmatrix}
			{\rho_\alpha^{\ln}} \overline{m_{x_\alpha} } \vspace{0.2cm} \\
			\frac{1}{\overline{\beta_\alpha}} \big( \frac{\overline{\beta_\alpha} \overline{m_{x_\alpha}}}{\bar{\Gamma_\alpha}}F_{\alpha,5}^x + \overline{\rho_\alpha}
			\big)	\\
			\frac{\overline{m_{y_\alpha}}}{\overline{\Gamma_\alpha}}F_{\alpha,5}^x \\ 
			\frac{\overline{m_{z_\alpha}}}{\overline{\Gamma_\alpha}}F_{\alpha,5}^x  
			\vspace{0.2cm} \\ 
			\frac{-\overline{\Gamma_\alpha}{\big( k_{\alpha} {\rho_\alpha^{\ln}}\overline{m_{x_\alpha}}
					+\frac{\overline{m_{x_\alpha}} \ \overline{\rho_\alpha}}{\overline{\beta_\alpha}}}
				\big)} {\big(    
				\overline{m_{x_\alpha}}^2 
				+ \overline{m_{y_\alpha}}^2 
				+ \overline{m_{z_\alpha}}^2 
				-{(\overline{\Gamma_\alpha})}^2  \big)}
		\end{pmatrix}_{i + \frac{1}{2},j} 	\label{eq:num_flux_x_ent_conser}
	\end{gather}
	It is easy to verify that $\tilde{{\mathbf F}}_\alpha^x$ is consistent with the flux $\mathbf{f}_{\alpha}^x$. 	 Similarly, we can derive the expression for $\mathbf{\tilde F}_\alpha^y$ to get 
	\begin{gather}
		\renewcommand{\arraystretch}{1.5}
		\mathbf{\tilde F}^y_{\alpha,i,j+\frac{1}{2}} = \mathbf{\tilde{F}}^y_\alpha(\mathbf U_{\alpha,i,j}, \mathbf U_{\alpha,i,j+1}) = \begin{pmatrix}
			{\rho_\alpha^{\ln}} \overline{m_{y_\alpha} } \vspace{0.2cm} \\
			%
			%
			\frac{\overline{m_{x_\alpha}}}{\overline{\Gamma_\alpha}}F_{{\alpha,5}}^y	\\
			\frac{1}{\overline{\beta_\alpha}} \bigg( \frac{\bar{\beta_\alpha} \overline{m_{y_\alpha}}}{\overline{\Gamma_\alpha}}F_{{\alpha,5}}^y + \overline{\rho_\alpha}
			\bigg)\\ 
			\frac{\overline{m_{z_\alpha}}}{\overline{\Gamma_\alpha}}F_{{\alpha,5}}^y
			\vspace{0.2cm} \\ 
			\frac{-\overline{\Gamma_\alpha}{\left( k_{\alpha} {\rho_\alpha^{\ln}} \overline{m_{y_\alpha} } +
					\frac{\overline{m_{y_\alpha}} \ \overline{\rho_\alpha}}{\bar{\beta_\alpha}}\right)}
			} {\left(    \overline{m_{x_\alpha}}^2 + \overline{m_{y_\alpha}}^2 + \overline{m_{z_\alpha}}^2 -(\overline{\Gamma_\alpha})^2  \right)}
		\end{pmatrix}_{i, j+\frac{1}{2}}, 	\label{eq:num_flux_y_ent_conser}%
	\end{gather}
	which is consistent with ${\mathbf{f}}^y_\alpha$.
	
	\begin{thm} The numerical scheme \eqref{eq:scheme_fluid}, with the numerical fluxes \eqref{eq:num_flux_x_ent_conser} and \eqref{eq:num_flux_y_ent_conser} is second-order accurate and entropy conservative, i.e., it satisfies \eqref{eq:semi_dis_fluid_entropy_equality}.
	\end{thm}
	\begin{proof}
		Following \cite{Tadmor1987} and \cite{Fjordholm2012}, multiplying \eqref{eq:scheme_fluid} with $\mathbf{V}_{\alpha,i,j}^\top$, we get,
		\begin{eqnarray*}
			\frac{d}{dt}  \mathcal{U}_\alpha(\mathbf{U}_{\alpha,i,j}) +\frac{1}{\Delta x} \left( \tilde{\mathcal{F}}_{\alpha,i+\frac{1}{2},j}^x - \tilde{\mathcal{F}}_{\alpha,i-\frac{1}{2},j}^x\right)+\frac{1}{\Delta y}\left( \tilde{\mathcal{F}}_{\alpha,i,j+\frac{1}{2}}^y - \tilde{\mathcal{F}}_{\alpha,i,j-\frac{1}{2}}^y\right) \\ =\mathbf{V}_{\alpha,i,j}^\top\cdot \mathbf{s}_\alpha(\mathbf{U}_{\alpha,i,j}(t),\mathbf{U}_{m,i,j}(t))
		\end{eqnarray*}
		Now, we note that
		$$
		\mathbf{V}_{\alpha,i,j}^\top\cdot \mathbf{s}_\alpha(\mathbf{U}_{\alpha,i,j}(t),\mathbf{U}_{m,i,j}(t))=0
		$$
		to arrive at the equality \eqref{eq:semi_dis_fluid_entropy_equality}.
	\end{proof}

	The entropy conservative fluxes derived above are only second-order accurate. Following \cite{Lefloch2002}, we can use second-order fluxes to construct $2p^{th}$-order accurate fluxes for any positive integer $p.$ In particular, the $4^{th}$-order ($p=2$) $x$-directional entropy conservative flux 
	$\tilde{\mathbf{{F}}}^{x,4}_{\alpha,i+\frac{1}{2},j}$ is given by,
	\begin{equation*}
		\tilde{\mathbf{{F}}}^{x,4}_{\alpha,i+\frac{1}{2},j}=\frac{4}{3}\tilde{\mathbf{{F}}}^x_{\alpha}(\mathbf{U}_{\alpha,i,j},\mathbf{U}_{\alpha,i+1,j})-\frac{1}{6} \bigg( \tilde{\mathbf{{F}}}^x_{\alpha} (\mathbf{U}_{\alpha,i-1,j},\mathbf{U}_{\alpha,i+1,j})+
		\tilde{\mathbf{{F}}}^x_{\alpha}(\mathbf{U}_{\alpha,i,j},\mathbf{U}_{\alpha,i+2,j}) \bigg). 
	\end{equation*}
	Similarly, we can obtain the fourth-order entropy conservative flux $\tilde{\mathbf{{F}}}^{y,4}_{\alpha,i+\frac{1}{2},j}$ in $y$-direction. Combining this we have the following remarks.
	\begin{remark}
		Replacing second-order fluxes with the higher-order fluxes in \eqref{eq:scheme_fluid}, we arrive at high-order, entropy conservative schemes.
	\end{remark}
	\begin{remark}
		The result from \cite{Lefloch2002} gives only even order accurate fluxes.  So, for $q^{th}$-order (q is an odd integer) accurate scheme, we use a $(q+1)^{th}$-order (an even number) accurate flux.
	\end{remark}
	
	\subsubsection{High-order entropy stable schemes for fluid equations}
	The entropy conservative schemes presented above will produce high-frequency oscillations near the shocks as they are central fluxes which do not dissipate entropy. Following \cite{Tadmor1987}, we introduce modified fluxes which will ensure entropy dissipation at the shocks. We consider a modified numerical flux of the form,
	\begin{equation}
		\begin{aligned}
			\mathbf{F}_{\alpha,i+\frac{1}{2},j}^x =\tilde{\mathbf{{F}}}_{\alpha,i+\frac{1}{2},j}^x - \frac{1}{2} \textbf{D}_{\alpha,i+\frac{1}{2},j}^x[\![ \mathbf{V}_\alpha]\!]_{i+\frac{1}{2},j},
			\\
			\mathbf{F}_{\alpha,i,j+\frac{1}{2}}^y = \tilde{\mathbf{{F}}}_{\alpha,i,j+\frac{1}{2}}^y - \frac{1}{2} \textbf{D}_{\alpha,i,j+\frac{1}{2}}^y[\![ \mathbf{V}_\alpha]\!]_{i,j+\frac{1}{2}},
			\label{es_numflux}
		\end{aligned}
	\end{equation}
	where $\textbf{D}^x_{\alpha,i+\frac{1}{2},j}$ and $\textbf{D}^y_{\alpha,i,j+\frac{1}{2}}$ are symmetric positive definite matrices. Then we have the following Lemma: 
	
	\begin{lemma}[{\em Tadmor}\cite{Tadmor1987}] The numerical scheme~\eqref{scheme} with the numerical fluxes~\eqref{es_numflux} is entropy stable, i.e., the entropy inequality
		\begin{equation*}
			\frac{d}{dt}  \mathcal{U}_\alpha(\mathbf{U}_{ij})  +\frac{1}{\Delta x} \left( \hat{\mathcal{F}}_{\alpha,i+\frac{1}{2},j}^x - \hat{\mathcal{F}}_{\alpha,i-\frac{1}{2},j}^x\right)+\frac{1}{\Delta y}\left( \hat{\mathcal{F}}_{\alpha,i,j+\frac{1}{2}}^y - \hat{\mathcal{F}}_{\alpha,i,j-\frac{1}{2}}^y\right) \le 0, \ \ \ \alpha \in \{i,e\},
		\end{equation*}
		is satisfied with the consistent numerical entropy flux functions,
		$$
		\hat{\mathcal{F}}^{x}_{\alpha,i+\frac{1}{2},j}=  \tilde{\mathcal{F}}^{x}_{\alpha,i+\frac{1}{2},j} + \frac{1}{2}(\bar{\mathbf{V}}_\alpha)^\top_{i+\frac{1}{2},j}  \mathbf{D}_{\alpha,i+\frac{1}{2},j}^{x}[\![ \mathbf{V}_\alpha]\!]_{i+\frac{1}{2},j}, 
		$$ 
		$$
		\hat{\mathcal{F}}^{y}_{\alpha,i,j+\frac{1}{2}}=    \tilde{\mathcal{F}}^{y}_{\alpha,i,j+\frac{1}{2}} + \frac{1}{2}(\bar{\mathbf{V}}_\alpha)^\top_{i,j+\frac{1}{2}}  \mathbf{D}_{\alpha,i,j+\frac{1}{2}}^{y}[\![ \mathbf{V}_\alpha]\!]_{i,j+\frac{1}{2}}. $$ 
	\end{lemma}
	
	\noindent We will use \textit{Rusanov type} diffusion operators for the matrix $\textbf{D}_\alpha$, which is given by,
	\begin{equation}  \label{diffusiontype}
		\mathbf{D}_{\alpha,i+\frac{1}{2},j}^x = {\mathbf{\tilde{R}}}_{\alpha,i+\frac{1}{2},j}^x \Lambda_{\alpha,i+\frac{1}{2},j}^x {\mathbf{\tilde{R}}}_{\alpha,i+\frac{1}{2},j}^{x \top}, \text{ \ \ and \ \ } \textbf{D}_{\alpha,i,j+\frac{1}{2}}^y = {\mathbf{\tilde{R}}}_{\alpha,i,j+\frac{1}{2}}^y \Lambda_{\alpha,i,j+\frac{1}{2}}^y {\mathbf{\tilde{R}}}_{\alpha,i,j+\frac{1}{2}}^{y \top}.
	\end{equation}
	where ${\mathbf{\tilde{R}}^d}_\alpha,\, d \in \{x,y\}$,  are matrices of the entropy scaled right eigenvectors and ${\Lambda^d_\alpha}$ are $5 \times 5$ diagonal matrices of the form
	\[
	{\Lambda^d_\alpha}=  
	\left( \max_{1 \leq k \leq 5} |\Lambda_{\alpha_k}^d|\right) \mathbf{I}_{5 \times 5}, 
	\]
	where $\{\Lambda_{\alpha_k}^d: 1 \leq k \leq 5 \},\, d \in \{x,y\}$, are the eigenvalues of the flux Jacobian. Following~\cite{Barth1999}, we have $\partial_{\mathbf{V}_\alpha} \mathbf{U}_\alpha = \mathbf{\tilde{R}}_\alpha^{d} (\mathbf{\tilde{R}}_\alpha^{d})^{\top}$ for $d \in\{x,y\}$, hence,
	$$
	\mathbf{\tilde{R}}^{d}_\alpha \Lambda_\alpha^{d} ( \mathbf{\tilde{R}}^{d}_\alpha )^{(-1)} [\![ \mathbf{U}_\alpha]\!] \approx \mathbf{\tilde{R}}^{d}_\alpha \Lambda^{d}_\alpha ( \mathbf{\tilde{R}}^{d}_\alpha )^{(-1)} \partial_{\mathbf{V}_\alpha} \mathbf{U}_\alpha [\![ \mathbf{V}_\alpha]\!]=  \mathbf{\tilde{R}}^{d}_\alpha \Lambda^{d}_\alpha ( \mathbf{\tilde{R}}^{d}_\alpha )^{\top} [\![ \mathbf{V}_\alpha]\!],
	$$
	i.e., the entropy diffusion operator used here is similar to the {\em Roe} diffusion operator when the entropy scaled right eigenvectors are used. The complete expressions for $\mathbf{\tilde{R}}^d_\alpha$ are derived in Appendix~\eqref{ap:scaled_eig}.

	With the choice of diffusion operator~\eqref{diffusiontype}, the numerical scheme~\eqref{scheme} with the numerical flux~\eqref{es_numflux} is entropy stable.  However, the scheme is only first-order accurate due to the presence of the first-order jump terms $[\![\mathbf{V}_\alpha]\!]_{i+\frac{1}{2},j}$ and $[\![\mathbf{V}_\alpha]\!]_{i,j+\frac{1}{2}}$. A straightforward way to increase the order of accuracy is to approximate the jumps in the diffusive term of~\eqref{es_numflux} using a higher-order reconstruction process. However, proving the entropy stability of the resulting scheme is not possible. Instead, we follow the process prescribed in \cite{Fjordholm2012} and introduce the {\em scaled entropy variables}. We illustrate the procedure for the $x$-direction only, as the $y$-directional case is similar. Define the change of variables
	\[
	\mathcal{V}_{\alpha,{m,j}}^{x,\pm}\,=\, (\mathbf{\tilde{R}}^{x}_{\alpha,{i\pm\frac{1}{2},j}} )^\top\mathbf{V}_{\alpha,{m,j}}, \qquad \textrm{$m$ are neighbours of cell $(i,j)$ along $x$}
	\]
	Using ENO procedure we choose a stencil of cells and construct the polynomials $P^{x,\pm}_{i,j}(x)$ of degree $k$ and evaluate it at the faces of cell $(i,j)$
	\[
	\mathcal{\tilde V}_{\alpha,i,j}^{x,\pm} = P_{i,j}^{x,\pm}(x_{i \pm \frac{1}{2}})
	\]
	Converting back to the entropy variables,
	$$\mathbf{\tilde{V}}_{\alpha,{i,j}}^{x,\pm}\,=\, \left\lbrace (\mathbf{\tilde{R}}^{x}_{\alpha,{i\pm\frac{1}{2},j}})^\top\right\rbrace ^{(-1)}\tilde{\mathcal{V}}_{\alpha,{i,j}}^{x,\pm}$$ 
	are the corresponding $k$-th order reconstructed values for $\mathbf{V}_{\alpha}$. Hence the high-order entropy stable numerical flux is given by,
	\begin{equation}
		\mathbf{F}_{\alpha,i+\frac{1}{2},j}^{x,k}\,=\,\tilde{\mathbf{F}}_{\alpha,i+\frac{1}{2},j}^{x,2p}\,-\,\frac{1}{2}\,\mathbf{D}_{\alpha,i+\frac{1}{2},j}^x[\![ \mathbf{\tilde{V}}^x_\alpha]\!]_{i+\frac{1}{2},j}
		\label{eq:entropy_stable_flux}
	\end{equation}
	where $[\![ \mathbf{\tilde{V}}^x_\alpha]\!]_{i+\frac{1}{2},j}$ stands for, $$[\![ \mathbf{\tilde{V}}^x_\alpha]\!]_{i+\frac{1}{2},j}\,=\,\mathbf{\tilde{V}}^{x,-}_{\alpha,{i+1,j}}\,-\,\mathbf{\tilde{V}}^{x,+}_{\alpha,{i,j}}.$$
	As only even order entropy conservative fluxes $\tilde{\mathbf{F}}_{\alpha,i+\frac{1}{2},j}^{x,2p}$ are available, we choose $p \in \mathbb{N}$ as
	\begin{itemize}
		\item $p=k/2$ if $k$ is even,
		\item $p=(k+1)/2$ if $k$ is odd,
	\end{itemize}
	where $k$ is the order of the scheme.
	Following  \cite{Fjordholm2012}, a sufficient condition for the numerical flux~\eqref{eq:entropy_stable_flux} to be entropy stable is that the reconstruction process for ${\mathcal{V}}$ must satisfy the {\em sign preserving property}. For the second-order reconstruction, we use the {\em min-mod}  reconstruction, which satisfies this property. Following \cite{Fjordholm2013}, we use ENO reconstruction as it satisfies the  {\em sign preserving property}. For third and fourth-order schemes, we use fourth-order entropy conservative flux along with the third and fourth-order ENO reconstruction of the scaled entropy variables.
	We now have the following result:
	\begin{thm}
		The numerical scheme \eqref{eq:scheme_fluid} with entropy stable flux \eqref{eq:entropy_stable_flux} and ENO reconstruction satisfies the entropy inequality \eqref{eq:semi_dis_fluid_entropy_inequality}.
	\end{thm}	
	\begin{proof}
		The proof follows from \cite{Fjordholm2012} using the fact that,
		$$
		\mathbf{V}_{\alpha,i,j}^\top\cdot \mathbf{s}_\alpha(\mathbf{U}_{\alpha,i,j}(t),\mathbf{U}_{m,i,j}(t))=0
		$$
		and the sign property of ENO reconstruction.
	\end{proof}
	
	\begin{remark}
		The above result holds for any discretization of Maxwell's equations. 
	\end{remark}
	
	\subsection{Discretization of Maxwell's equations}
	
	To spatially discretize the Maxwell's equations  \eqref{max1}-\eqref{max4}, we use the finite-difference scheme based on the Rusanov solver. The $x$- and $y$-directional numerical fluxes are given by
	\begin{subequations}
		\begin{align*}
			\mathbf{F}^x_{m,i+\frac{1}{2},j} = \dfrac{1}{2} ( \mathbf{f}^x_m(\mathbf{U}_{m,i,j}) + \mathbf{f}^x_m(\mathbf{U}_{m,i+1,j}) )
			-
			\dfrac{c}{2} (\mathbf{U}_{m,i+1,j}- \mathbf{U}_{m,i,j}),  \\
			\mathbf{F}^y_{m,i,j+\frac{1}{2}} = \dfrac{1}{2} ( \mathbf{f}^y_m(\mathbf{U}_{m,i,j}) + \mathbf{f}^y_m(\mathbf{U}_{m,i,j+1}) )
			-
			\dfrac{c}{2} (\mathbf{U}_{m,i,j+1}- \mathbf{U}_{m,i,j}). 
		\end{align*}
	\end{subequations}
	where $c=1$ is the speed of light. To obtain higher accuracy, we use \textit{min-mod} limiter \cite{LeVeque2002} to reconstruct the variables $\mathbf{U}_{m}$ at the cell faces to get the second-order scheme. Furthermore, we use third and fifth-order WENO reconstructions \cite{Jiang1996} to obtain the third- and fourth-order Rusanov's solver for the Maxwell's equations. 

\section{Fully discrete scheme} \label{sec:full_disc}

Let $\mathbf{U}^n$ be the discrete solution at time $t^n$, where the time step is $\Delta t = t^{n+1} - t^n$. The semi-discrete scheme~\eqref{scheme} can be expressed as 
\begin{equation}
	\frac{d }{dt}\mathbf{U}_{i,j}(t)
	=
	\mathcal{L}_{i,j}(\mathbf{U}(t)) + \mathbf{s}(\mathbf{U}_{i,j}(t)), \qquad \mathbf{U}_{i,j}(t^n)=\mathbf{U}^n_{i,j}
	\label{fullydiscrete}
\end{equation} 
where,
\begin{equation*}
	\mathcal{L}_{i,j}(\mathbf{U}(t))
	=
	- \frac{1}{\Delta x}  \left(\mathbf{F}_{i+\frac{1}{2},j}^x(t)-\mathbf{F}_{i-\frac{1}{2},j}^x(t)\right)
	-
	\frac{1}{\Delta y}  \left(\mathbf{F}_{i,j+\frac{1}{2}}^y(t)-\mathbf{F}_{i,j-\frac{1}{2}}^y(t)\right).
\end{equation*}
We propose fully explicit and IMEX schemes for the above system.
\subsection{Explicit schemes} \label{explicit}
For the explicit time discretizations, we use explicit strong stability preserving Runge Kutta (SSP-RK) methods \cite{Gottlieb2001}. The second and third-order accurate SSP-RK schemes are given as follows.
\begin{enumerate}
	\item Set $\mathbf{U}^{0} \ = \ \mathbf{U}^n$.
	\item For $m=1,\dots,k+1$, compute,
	\begin{equation*}
		\mathbf{U}_{i,j}^{(m)} \ 
		= \
		\sum_{l=0}^{m-1} \left[ \alpha_{ml}\mathbf{U}_{i,j}^{(l)}
		+
		\beta_{ml}\Delta t \big(\mathcal{L}_{i,j}(\mathbf{U}^{(l)}) + \mathbf{s}(\mathbf{U}_{i,j}^{(l)}) \big) \right]
	\end{equation*}
	where $\alpha_{ml}$ and $\beta_{ml}$ are given in Table~\eqref{table:ssp}.
	\item Finally, $\mathbf{U}_{i,j}^{n+1} \ =  \ \mathbf{U}_{i,j}^{(k+1)}$.
\end{enumerate}

\begin{table}[ht]
	\centering
	\begin{tabular}{ l|ccc|ccc }
		\hline
		Order & \multicolumn{3}{|c|}{$\alpha_{il}$} & \multicolumn{3}{|c}{$\beta_{il}$} \\
		\hline
		\multirow{2}{*}{2} & 1 & & & 1 & & \\
		& 1/2 & 1/2 & & 0 & 1/2 & \\
		\hline
		\multirow{3}{*}{3} & 1 & & & 1 & & \\
		& 3/4 & 1/4 & & 0 & 1/4 & \\
		& 1/3 & 0 & 2/3 & 0 & 0 & 2/3\\
		\hline
	\end{tabular}
	\caption[h]{Coefficients for SSP Runge-Kutta time stepping} 
	\label{table:ssp}
\end{table}
The fourth order RK-SSP is given by,
\begin{subequations}
	\begin{align*}
		\mathbf{U}^{(1)}_{i,j} &= \textbf{U}^n_{i,j} + 0.39175222700392 \;(\Delta t) \left( \mathcal{L}_{i,j}(\mathbf{U}^n)+\mathbf{s}(\mathbf{U}^n_{i,j}) \right), \\
		\mathbf{U}^{(2)}_{i,j} &= 0.44437049406734 \;\textbf{U}^n_{i,j} + 0.55562950593266\; \mathbf{U}^{(1)}_{i,j}\\&+0.36841059262959\;  (\Delta t) \left( \mathcal{L}_{i,j}(\mathbf{U}^{(1)})+\mathbf{s}(\mathbf{U}^{(1)}_{i,j}) \right), \\
		\mathbf{U}^{(3)}_{i,j} &= 0.62010185138540\; \textbf{U}^n_{i,j} + 0.37989814861460\; \mathbf{U}^{(2)}_{i,j} \\ &+0.25189177424738\;  (\Delta t) \left( \mathcal{L}_{i,j}(\mathbf{U}^{(2)})+\mathbf{s}(\mathbf{U}^{(2)}_{i,j}) \right), \\
		\mathbf{U}^{(4)}_{i,j} &= 0.17807995410773\; \textbf{U}^n_{i,j} + 0.82192004589227\; \mathbf{U}^{(3)}_{i,j} \\&+ 0.54497475021237\;  (\Delta t) \left( \mathcal{L}_{i,j}(\mathbf{U}^{(3)})+\mathbf{s}(\mathbf{U}^{(3)}_{i,j}) \right), \\
		\textbf{U}^{n+1}_{i,j} &= 0.00683325884039 \;\textbf{U}^n + 0.51723167208978\; \mathbf{U}^{(2)}_{i,j} + 0.12759831133288 \;\mathbf{U}^{(3)}_{i,j} \\
		&+ 0.34833675773694\; \mathbf{U}^{(4)}_{i,j}   + 0.08460416338212 \; (\Delta t) \left( \mathcal{L}_{i,j}(\mathbf{U}^{(3)})+\mathbf{s}(\mathbf{U}^{(3)}_{i,j}) \right)\\
		&+ 0.22600748319395 \; (\Delta t) \left( \mathcal{L}_{i,j}(\mathbf{U}^{(4)})+\mathbf{s}(\mathbf{U}^{(4)}_{i,j}) \right).
	\end{align*}
\end{subequations}
We use a second-order spatial accurate scheme with a second-order explicit time scheme, which is denoted by, {\bf O2-ES-Exp}. Similarly, the third-order spatial accurate scheme is time-discretized with the third-order RK method above and denoted by, {\bf O3-ES-Exp}. The fourth-order scheme consists of fourth-order spatial discretization with the fourth-order time update above and denoted by {\bf O4-ES-Exp}.

\subsection{ARK-IMEX schemes} 
\label{sec:imex_scheme}
The source terms are stiff when the charge to mass ratio $r_\alpha$ is large, which results in a time-step restriction when using explicit time integration schemes. To overcome this difficulty, we will use the Additive Runge Kutta Implicit Explicit~(ARK IMEX) time stepping scheme from\cite{Pareschi2005}. Here, the flux terms will be treated explicitly, and the source terms will be treated implicitly. We write the semi-discrete scheme~\eqref{scheme} in a simplified form as follows
\begin{equation*}
	\dfrac{\partial \mathbf{U}}{\partial t} = F_{NS}(\mathbf{U}(t))+ F_S(\mathbf{U}(t)),  
\end{equation*}
where $F_{NS}(\mathbf{U}(t)) = \mathcal{L}_{i,j}(\mathbf{U}(t))$ denotes the non-stiff flux terms, and $F_S(\mathbf{U}(t)) = \mathbf{s}(\mathbf{U}_{i,j}(t))$ denotes the stiff source terms. Note that the source terms contain only the solution value in the parent cell and not any spatial derivatives, and hence they do not depend on the solution in neighboring cells. 

\subsubsection{Second order time discretization.} \label{o2_lstable}
For the second-order discretization, we use L-Stable second-order accurate ARK IMEX scheme from \cite{Pareschi2005}. The method has two implicit stages and is given by
\begin{subequations}
	\begin{align*}
		\mathbf{U}^{(1)} &= \mathbf{U}^n + \Delta t \ \big( \beta F_S(\mathbf{U}^{(1)}(t^n)) \big), \\
		\mathbf{U}^{(2)} &= \mathbf{U}^n + \Delta t \ \big( F_{NS}(\mathbf{U}^{(1)}(t^n)) + (1-2\beta) F_S(\mathbf{U}^{(1)}(t^n))  + \beta F_S(\mathbf{U}^{(2)}(t^n)) \big), \\
		\mathbf{U}^{(n+1)} &= \mathbf{U}^n + \Delta t \ \bigg( \dfrac{1}{2}F_{NS}(\mathbf{U}^{(1)}(t^n)) + \dfrac{1}{2}F_{NS}(\mathbf{U}^{(2)}(t^n)) + \dfrac{1}{2} F_S(\mathbf{U}^{(1)}(t^n))
		+ \dfrac{1}{2} F_S(\mathbf{U}^{(2)}(t^n)) \bigg),
	\end{align*}
\end{subequations}
where $\beta = 1 - \dfrac{1}{\sqrt{2}}$. Note that, to evaluate $\mathbf{U}^{(1)}$ and $\mathbf{U}^{(2)}$, we need to solve a nonlinear system of algebraic equations. However, the resulting equations are only local, i.e., we get an independent set of equations for each cell. We solve these equations using the Newton's method, where the convergence and the next direction search is based on the backtracing line search framework. A detailed procedure is given in \cite{Dennis1996}. In the higher-order method, we will use the same process to update the internal IMEX steps.

\subsubsection{Third and fourth-order time discretization.} \label{o34_lstable}
For the higher-order discretization, we use L-Stable ARK IMEX methods from \cite{Kennedy2003}. We ignore the subscripts $\{i,j\}$ for simplicity. The ARK scheme from \cite{Kennedy2003} has the following steps.
\begin{enumerate}
	\item Set $\mathbf{U}^{0} \ = \ \mathbf{U}^n$.
	\item For $m=1,\dots,k$, compute,
	\begin{equation*}
		\mathbf{U}^{(m)}  
		= \mathbf{U}^n  +
		\Delta t \sum_{l=0}^{m-1}a_{ml}^{[NS]} F_{NS}( \mathbf{U}^{l})
		+
		\Delta t \sum_{l=0}^{m}a_{ml}^{[S]} F_{S}( \mathbf{U}^{l}),
	\end{equation*}
	where, $\mathbf{U}^{(m)}$ approximates $\mathbf{U}(t^n + c_m \Delta t).$
	\item Finally,
	\begin{equation*}
		\mathbf{U}^{n+1}  
		= \mathbf{U}^n  +
		\Delta t \sum_{l=0}^{k} b_{l}^{[NS]} F_{NS}( \mathbf{U}^{l})
		+
		\Delta t \sum_{l=0}^{k} b_{l}^{[S]}   F_{S}( \mathbf{U}^{l}),
	\end{equation*}
	where $k$ is the order of the desired time discretization~($k$ ranges from 3 to 5), and  coefficients $a_{ml}^{[NS]}$, $a_{ml}^{[S]}$, $b_{l}^{[NS]}$, $b_{l}^{[S]}$, $c_l$ are defined accordingly by the order of accuracy and stability considerations.  
\end{enumerate}
We present coefficients for the third-order scheme in Appendix~\eqref{ap:ark_coeff}. The coefficients for the fourth-order time update can be found in \cite{Kennedy2003}. We denote the second-order IMEX scheme with {\bf O2-ES-IMEX}, third-order IMEX scheme with {\bf O3-ES-IMEX} and fourth-order IMEX scheme with {\bf O4-ES-IMEX}.

\section{Numerical results} \label{sec:num_test}
In this section, we present numerical results for various test cases. We set specific heat ratios $\gamma_i=\gamma_e=4/3$, unless stated otherwise. In all the test cases, we take initial values of the potentials $\phi$ and $\psi$ to be zero and potential speeds $\kappa=\chi=1$. In some test cases, we need to consider the resistive effects; following \cite{Amano2016}, we modify the momentum \eqref{ion_momentum},~\eqref{elec_momentum} and energy equations \eqref{ion_energy},~\eqref{elec_energy} for the fluid parts by adding resistive terms as follows,

\begin{subequations}
	\begin{align}
		\frac{\partial ( \rho_\alpha h_\alpha \Gamma_\alpha^2 \mathbf{u}_\alpha)}{\partial t} + \nabla \cdot (\rho_\alpha h_\alpha \Gamma_\alpha^2 \textbf{u}_\alpha \mathbf{u}_\alpha^\top + p_\alpha \mathbf{I}) &= r_\alpha \Gamma_\alpha \rho_\alpha (\mathbf{E}+\mathbf{u}_\alpha \times \mathbf{B}) + \boldsymbol{R}_\alpha  , \label{resist_momentum}	\\ 
		%
		\frac{\partial (\mathcal{E}_\alpha)}{\partial t} + \nabla \cdot ((\mathcal{E}_\alpha+p_\alpha)\mathbf{u}_\alpha) &= r_\alpha \Gamma_\alpha \rho_\alpha (\mathbf{u}_\alpha \cdot \mathbf{E}) + R_\alpha^0,	\label{resist_energy} 
	\end{align}
\end{subequations}
for $\alpha \in \{i,e\}$. Following\cite{Amano2016}, an anti-symmetry relationship, $(\boldsymbol{R}_e, {R}_e^0)=(-\boldsymbol{R}_i, -R_i^0)$, has been assumed for the conservation of the total momentum and energy density. The expressions for the terms $\boldsymbol{R}_i$ and $R_i^0$ are given by 
\begin{align*}
	\boldsymbol{R}_i =  -\eta \dfrac{\omega_p^2}{r_i-r_e} (\mathbf{j}- \rho_0 \boldsymbol{\Phi}), \qquad
	{R}_i^0 =  -\eta \dfrac{\omega_p^2}{r_i-r_e} (\rho_c- \rho_0 {\Lambda}),
\end{align*}
where,
\begin{eqnarray} \label{source_var}
	\omega_p^2 &= r_i^2 \rho_i + r_e^2 \rho_e, \qquad
	\boldsymbol{\Phi} = \dfrac{r_i \rho_i \Gamma_i \mathbf {u}_i + r_e \rho_e \Gamma_e \mathbf{u}_e}{\omega_p^2},
	\\
	\Lambda &= \dfrac{r_i^2 \rho_i \Gamma_i  + r_e^2 \rho_e \Gamma_e }{\omega_p^2}, \qquad
	\rho_0 = \Lambda \rho_c - \mathbf{j} \cdot \boldsymbol{\Phi} \nonumber
\end{eqnarray}
Here $\eta$ is the resistivity constant and $\omega_p$ is the total plasma frequency. The total plasma skin depth $d_p$ is defined as $d_p = \frac{1}{\omega_p}$. For the time update using ARK IMEX scheme, we couple the resistive terms with the other source terms and treat them implicitly, i.e., we modify $F_S$ to include the resistive terms.
\par To compare the numerical results with published work consistently, for all the test cases taken from \cite{Balsara2016}, we multiply the source term for Maxwell's equations by 4$\pi$, i.e., we consider,
\begin{equation}
	\mathbf{s}_m = 4\pi
	\begin{pmatrix}
			0\\
			0\\ 
			0\\
			-j_x\\
			-j_y\\ 
			-j_z\\
			\chi \rho_c\\
			0
	\end{pmatrix}.
	\label{eq:modifed_max_source}
\end{equation}
Consequently, in the test cases from \cite{Balsara2016}, the total plasma skin depth is given by $d_p=\frac{\sqrt{2}}{\omega_p}$ with $\omega_p^2=4\pi(r_i^2 \rho_i + r_e^2 \rho_e).$

\subsection{One-dimensional test cases}

For the one-dimensional test cases, the time step is chosen using 	
\[
{\Delta t} = \text{CFL} \cdot \min \left\{ \dfrac{\Delta x}{\Lambda_{max}^x(\mathbf{U}_i)} : 1 \le i \le N_x\right\}, \textrm{ where } \Lambda_{max}^x(\mathbf{U}_i) = \max\{|\Lambda^x_k(\mathbf{U}_i)|: 1 \le k \le 18\}.
\]
We take CFL to be $0.8$, for  the explicit ({\bf O2-ES-Exp}, {\bf O3-ES-Exp} and {\bf O4-ES-Exp}), and ARK IMEX schemes~({\bf O2-ES-IMEX}, {\bf O3-ES-IMEX} and {\bf O4-ES-IMEX}), unless stated otherwise.

\subsubsection{Accuracy test} \label{test:1d_smooth} 
\begin{table}[ht]
	\centering
	\begin{tabular}{|c|c|c|c|c|c|c|}
		\hline
		Number of cells & \multicolumn{2}{|c}{{\bf O2-ES-Exp} } &
		\multicolumn{2}{|c}{{\bf O3-ES-Exp}} &
		\multicolumn{2}{|c|}{{\bf O4-ES-Exp}} \\
		\hline
		-- & $L^1$ error & Order & $L^1$ error & Order & $L^1$ error & Order \\
		\hline
		50 & 4.65050e-02 & -- & 1.45793e-03 & -- & 1.49240e-04 & -- \\
		100 & 1.46231e-02 & 1.6691 & 2.11487e-04 & 2.7853 & 1.13502e-05 & 3.7168 \\
		200 & 4.06705e-03 & 1.8462 & 2.68696e-05 & 2.9765 & 8.01398e-07 & 3.8241 \\
		400 & 1.10159e-03 & 1.8844 & 3.38714e-06 & 2.9878 & 5.55355e-08 & 3.8510 \\
		800 & 2.96489e-04 & 1.8935 & 4.12164e-07 & 3.0388 & 3.74592e-09 & 3.8900 \\
		1600 & 7.83571e-05 & 1.9198 & 4.74082e-08 & 3.1200 & 2.48382e-10 & 3.9147 \\
		3200 & 2.04466e-05 & 1.9382 & 5.12307e-09 & 3.2101 & 1.77044e-11 & 3.8104 \\
		\hline
	\end{tabular}
	\caption[h]{\nameref{test:1d_smooth}: $L^1$ errors and order of convergence for $\rho_i$ using the explicit entropy stable schemes {\bf O2-ES-Exp}, {\bf O3-ES-Exp} and {\bf O4-ES-Exp}.}
	\label{table:order_exp}
\end{table}

\begin{table}[ht]
	\centering
	\begin{tabular}{|c|c|c|c|c|c|c| }
		\hline
		Number of cells & \multicolumn{2}{|c}{{\bf O2-ES-IMEX} } &
		\multicolumn{2}{|c}{{\bf O3-ES-IMEX}} &
		\multicolumn{2}{|c|}{{\bf O4-ES-IMEX}} \\
		\hline
		-- & $L^1$ error & Order & $L^1$ error & Order & $L^1$ error & Order \\
		\hline
		50 & 4.65793e-02 & -- & 1.47651e-03 & -- & 1.49978e-04 & -- \\
		100 & 1.45983e-02 & 1.6739 & 2.12698e-04 & 2.7953 & 1.13298e-05 & 3.7266 \\
		200 & 4.05330e-03 & 1.8486 & 2.70403e-05 & 2.9756 & 7.99773e-07 & 3.8244 \\
		400 & 1.09574e-03 & 1.8872 & 3.40147e-06 & 2.9909 & 5.54390e-08 & 3.8506 \\
		800 & 2.94528e-04 & 1.8954 & 4.13433e-07 & 3.0404 & 3.73935e-09 & 3.8900 \\
		1600 & 7.77465e-05 & 1.9216 & 4.75597e-08 & 3.1198 & 2.47509e-10 & 3.9172 \\
		3200 & 2.02807e-05 & 1.9387 & 5.14011e-09 & 3.2099 & 1.63939e-11 & 3.9162 \\
		\hline
	\end{tabular}
	\caption[h]{\nameref{test:1d_smooth}:  $L^1$ errors and order of convergence for $\rho_i$ using the ARK IMEX entropy stable schemes {\bf O2-ES-IMEX}, {\bf O3-ES-IMEX} and {\bf O4-ES-IMEX}.}
	\label{table:order_imex}
\end{table}

To test the accuracy and order of convergence of the proposed scheme, we consider a test case with smooth solution. We follow a forced solution approach of \cite{Kumar2012} to modify the smooth test case from \cite{Bhoriya2020}. We add forcing term $\mathcal{S}(x,t)$ to write the one-dimensional modified system~\eqref{conservedform} as
\begin{equation*}
	\frac{\partial \mathbf{U}}{\partial t}+\frac{\partial \mathbf{f}^x}{\partial x} = \mathbf{s} + \mathcal{S}(x,t) 
\end{equation*}
with 
\begin{equation*}
	\mathcal{S}(x,t)=\left(\mathbf{0}_{13},-\dfrac{1}{\sqrt{3}}(2+\sin(2 \pi (x-0.5t))), 0, -3\pi\cos(2 \pi (x-0.5t)),\frac{2}{\sqrt{3}}(2+\sin(2 \pi (x-0.5t))),0 \right)^\top.	
\end{equation*}
The initial ion and electron densities are $\rho_i = \rho_e = 2+\sin(2 \pi x)$, with initial velocities $u_{x_i}=u_{x_e} = 0.5$ and initial pressures $p_i=p_e=1$. The $y-$magnetic component is $B_y=2 \sin(2 \pi (x))$ and the $z-$electric field component is $E_z = -\sin(2\pi (x)$. All other primitive variables are set to zero. The computational domain is $I=[0,1]$ with periodic boundary conditions. The charge to mass ratios are $r_i = 1$ and $r_e = -2$, consequently we have non-zero electric field source term $\mathbf{j}$.  The ion-electron adiabatic index is $\gamma=5/3$.  For this problem, it is easy to verify that the exact solution is $\rho_i = \rho_e = 2+\sin(2 \pi (x-0.5 t))$.

In Table~\eqref{table:order_exp} and Table~\eqref{table:order_imex}, we have presented the $L^1$-errors for $\rho_i$ at different resolutions for explicit ({\bf O2-ES-Exp}, {\bf O3-ES-Exp} and {\bf O4-ES-Exp}) and ARK-IMEX schemes ({\bf O2-ES-IMEX}, {\bf O3-ES-IMEX}, and {\bf O4-ES-IMEX}) at the final time $t=2.0$. For both explicit and IMEX schemes, we observe a consistent order of accuracy at various resolutions. Furthermore, we also note that the errors for both explicit and IMEX schemes of same order are comparable at a given resolution.

\subsubsection{Relativistic Brio--Wu test problem with finite plasma skin depth} \label{test:1d_brio} 
This is a shock tube problem with finite plasma skin depth from \cite{Balsara2016,Amano2016} and is a two-fluid extension of the relativistic analogue of the Brio-Wu shock tube problem given in~\cite{Balsara2001}. Maxwell's source  terms~\eqref{eq:modifed_max_source} are included, but no resistive effects are considered. The computational domain is $[-0.5,0.5]$ with Neumann boundary conditions at both boundaries. The initial discontinuity is assumed to be placed at $x=0.0$. Following \cite{Balsara2016}, the initial Riemann data is given by,
\begin{align*}
	\mathbf{W_L}=\begin{pmatrix}
		\rho_i \\ p_i \\ \rho_e \\ p_e \\ B_x \\ B_y
	\end{pmatrix}_{L}
	=
	\begin{pmatrix}
		0.5 \\ 0.5 \\ 0.5 \\ 0.5 \\ \sqrt{\pi} \\ \sqrt{4 \pi}
	\end{pmatrix},
	\qquad
	\mathbf{W_R}=\begin{pmatrix}
		\rho_i \\ p_i \\ \rho_e \\ p_e \\ B_x \\ B_y
	\end{pmatrix}_{R}
	=
	\begin{pmatrix}
		0.0625 \\ 0.05 \\ 0.0625 \\ 0.05 \\ \sqrt{\pi} \\ -\sqrt{4 \pi}
	\end{pmatrix}.
\end{align*}
All other variables are set to zero.  Following \cite{Balsara2016}, the specific heat ratios are set to $\gamma_i=\gamma_e=2.0$, and we compute the solution till final time $t=0.4$.

First, we consider charge to mass ratios as $r_i=-r_e=10^3/\sqrt{4 \pi}$, thus the plasma skin depth is $10^{-3}/\sqrt{\rho_i}$. Consequently, we expect solutions to be close to the RMHD results whenever the resolution of the cells is larger than the plasma skin depth.
\begin{figure}[!htbp]
	\begin{center}
		\subfigure[Plots of total density $\rho_i+\rho_e$ for second-order schemes {\bf O2-ES-Exp} and {\bf O2-ES-IMEX} using $400$, $1600$ and $6400$ cells.]{
			\includegraphics[width=3.0in, height=2.3in]{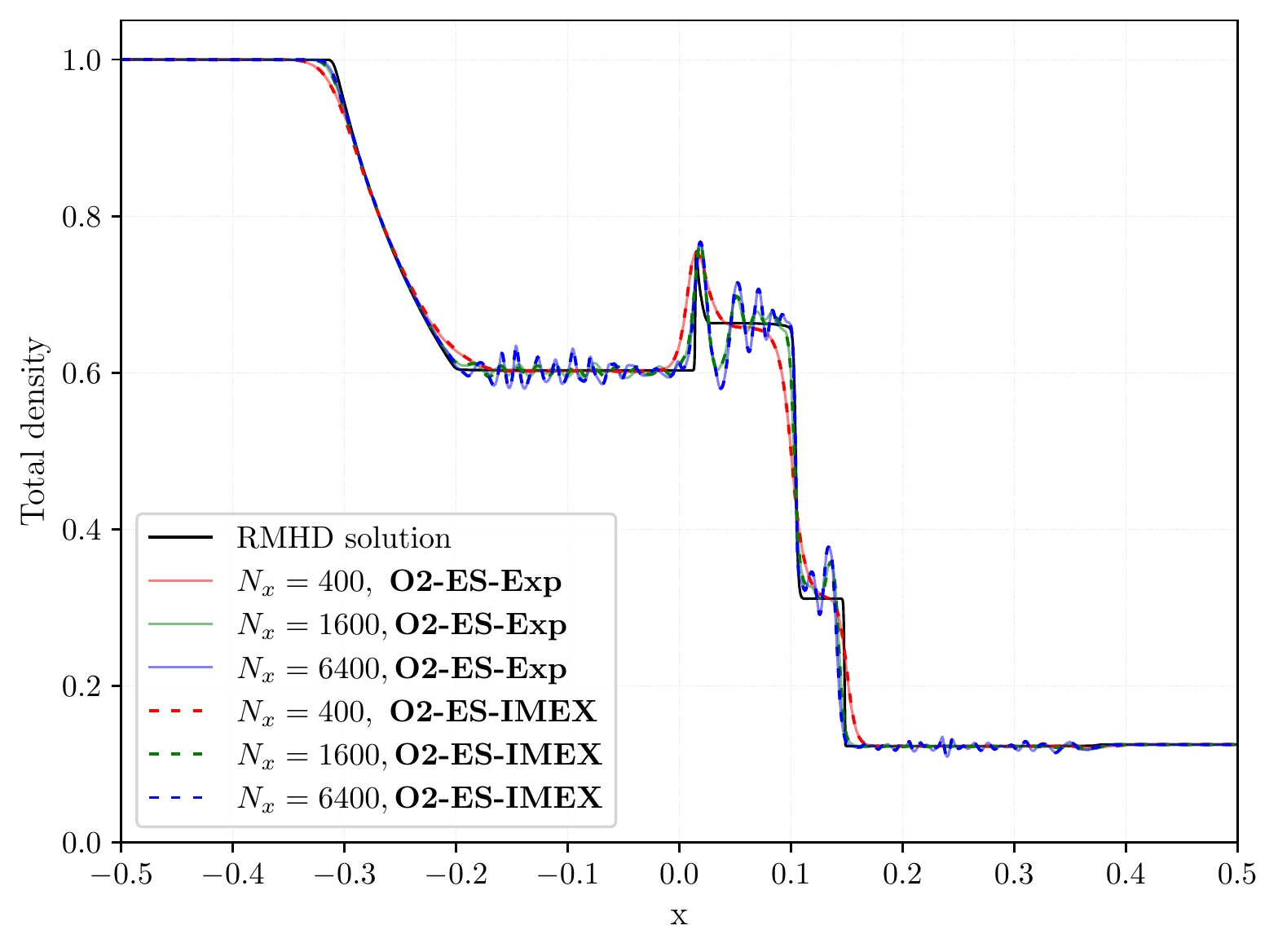}
			\label{fig:brio_rho_o2}}
		\subfigure[Plots of total density $\rho_i+\rho_e$, for second ({\bf O2-ES-Exp} and {\bf O2-ES-IMEX}), third ({\bf O3-ES-Exp} and {\bf O3-ES-IMEX}), and fourth ({\bf O4-ES-Exp} and {\bf O4-ES-IMEX}) order schemes using $400$ cells.]{
			\includegraphics[width=3.0in, height=2.3in]{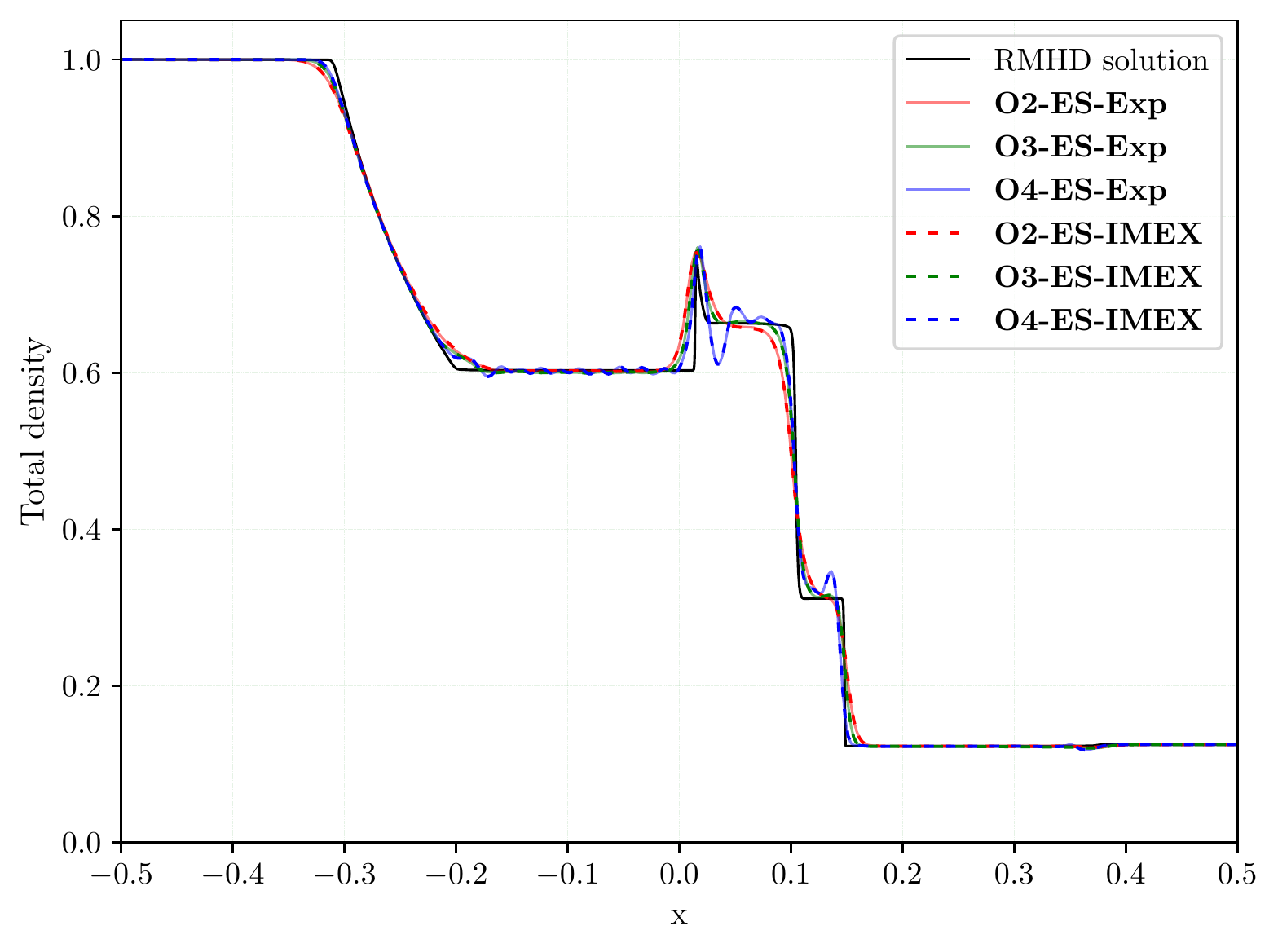}
			\label{fig:brio_rho_all}}
		\subfigure[Plots of total ion entropy for second ({\bf O2-ES-Exp} and {\bf O2-ES-IMEX}), third ({\bf O3-ES-Exp} and {\bf O3-ES-IMEX}), and fourth ({\bf O4-ES-Exp} and {\bf O4-ES-IMEX}) order schemes using $400$ cells.]{
			\includegraphics[width=2.9in, height=2.3in]{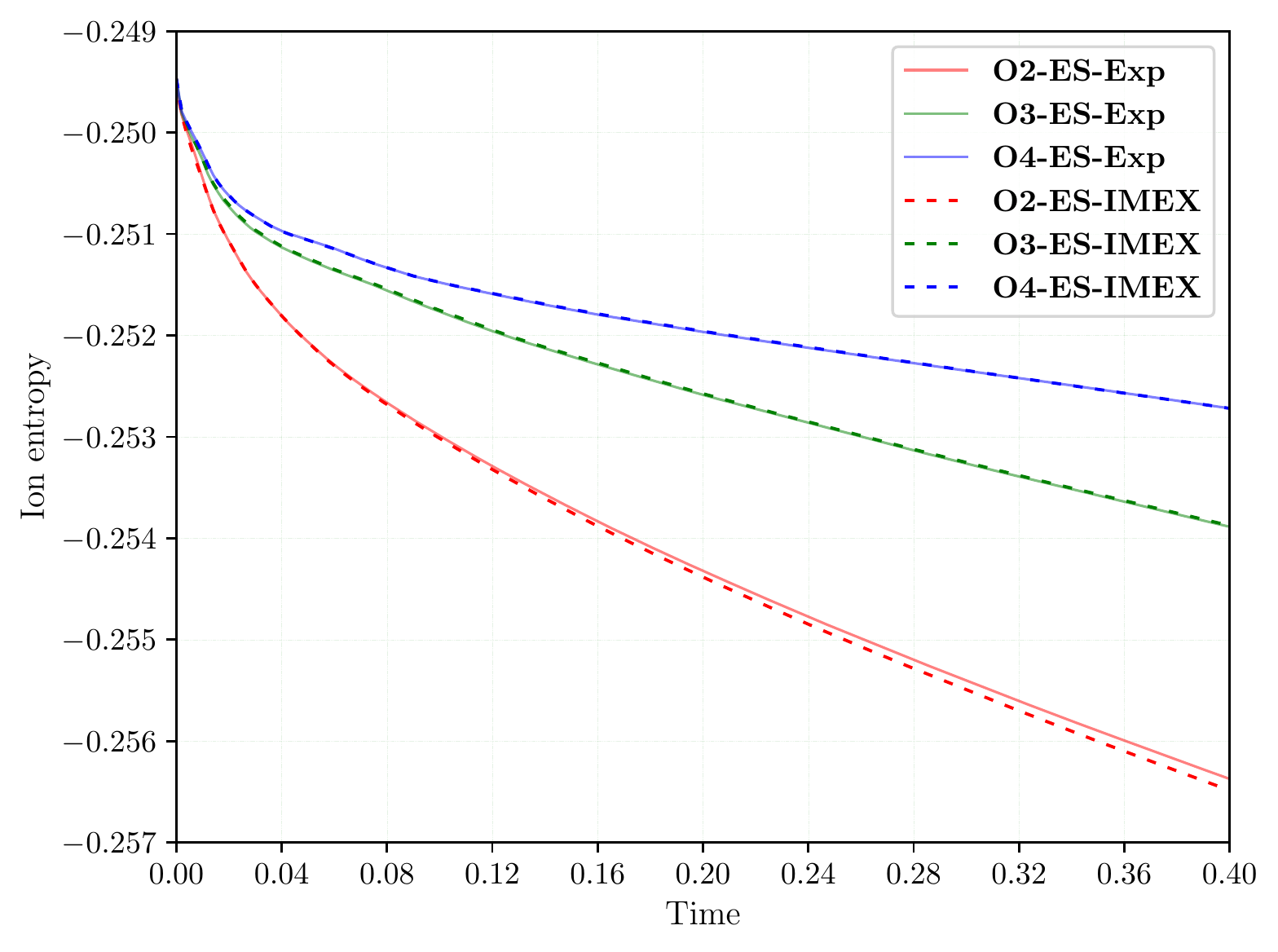}
			\label{fig:brio_ent_ion}}
		\subfigure[Plots of total electron entropy for second ({\bf O2-ES-Exp} and {\bf O2-ES-IMEX}), third ({\bf O3-ES-Exp} and {\bf O3-ES-IMEX}), and fourth ({\bf O4-ES-Exp} and {\bf O4-ES-IMEX}) order schemes using $400$ cells.]{
			\includegraphics[width=2.9in, height=2.3in]{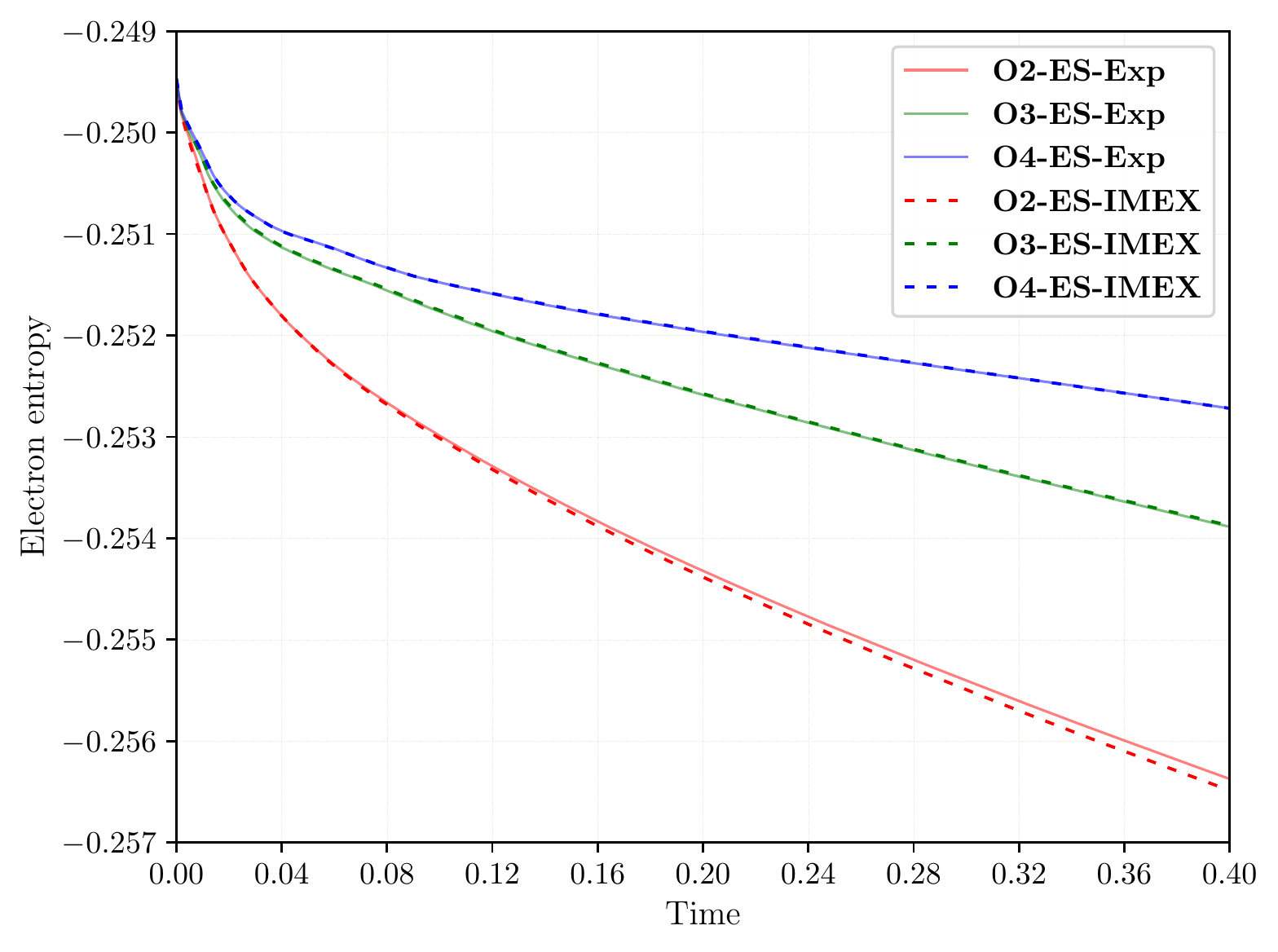}
			\label{fig:brio_ent_elec}}
		\caption{Numerical results for \nameref{test:1d_brio} using different schemes and grid resolutions.}
		\label{fig:brio}
	\end{center}
\end{figure}

\begin{figure}[!htbp]
	\begin{center}
		\subfigure[Plots of total density $\rho_i+\rho_e$, for second ({\bf O2-ES-Exp} and {\bf O2-ES-IMEX}), third ({\bf O3-ES-Exp} and {\bf O3-ES-IMEX}), and fourth ({\bf O4-ES-Exp} and {\bf O4-ES-IMEX}) order schemes with $r_i=-r_e=10^4/\sqrt{4 \pi}$ using $400$ cells.]{
			\includegraphics[width=0.48\textwidth]{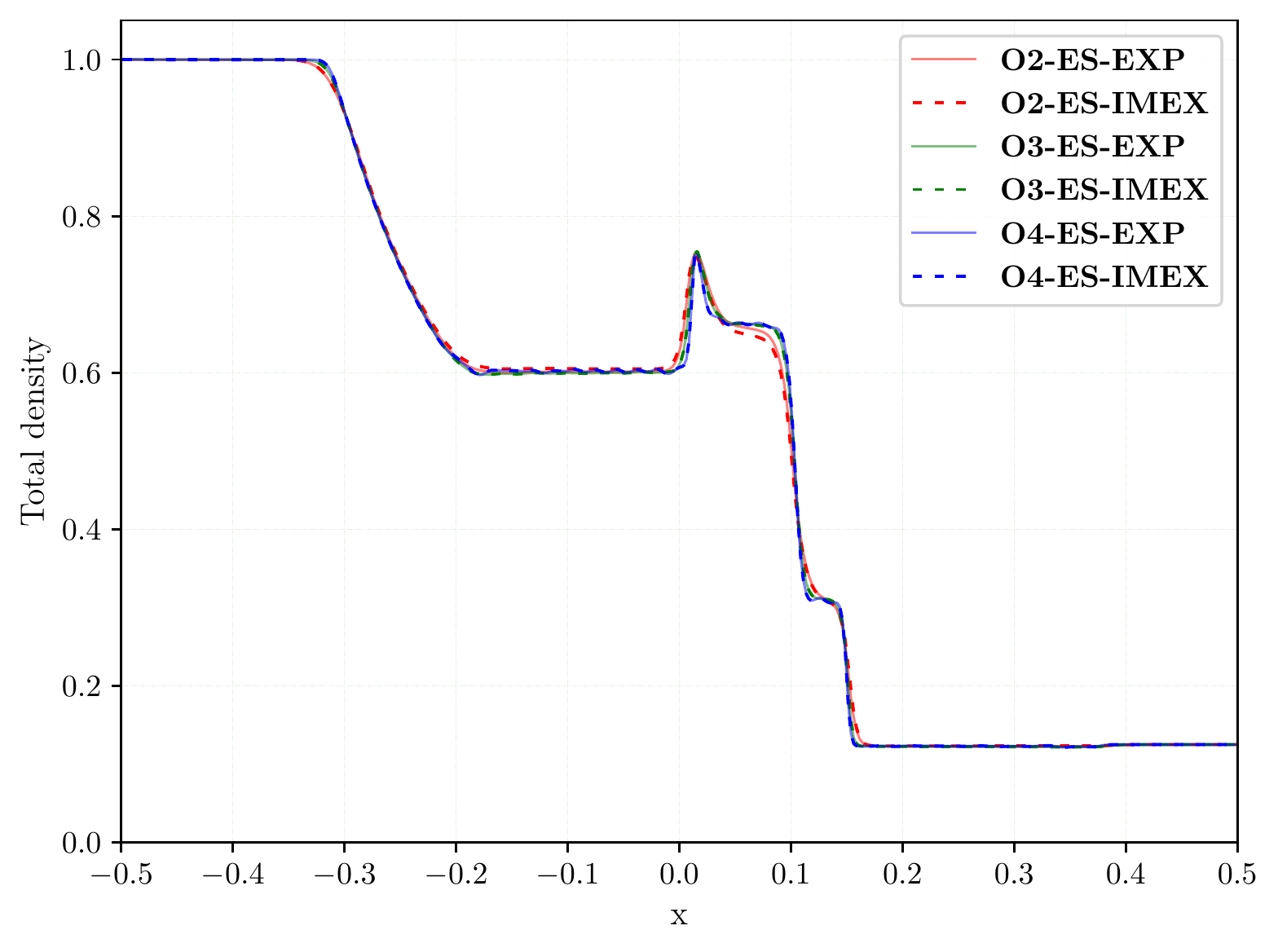}
			\label{fig:brio_ri4_all_scheme}}
		\subfigure[Plots of total density $\rho_i+\rho_e$, for second ({\bf O2-ES-Exp} and {\bf O2-ES-IMEX}), third ({\bf O3-ES-Exp} and {\bf O3-ES-IMEX}), and fourth ({\bf O4-ES-Exp} and {\bf O4-ES-IMEX}) order schemes with $r_i=-r_e=10^5/\sqrt{4 \pi}$ using $400$ cells.]{
			\includegraphics[width=0.48\textwidth]{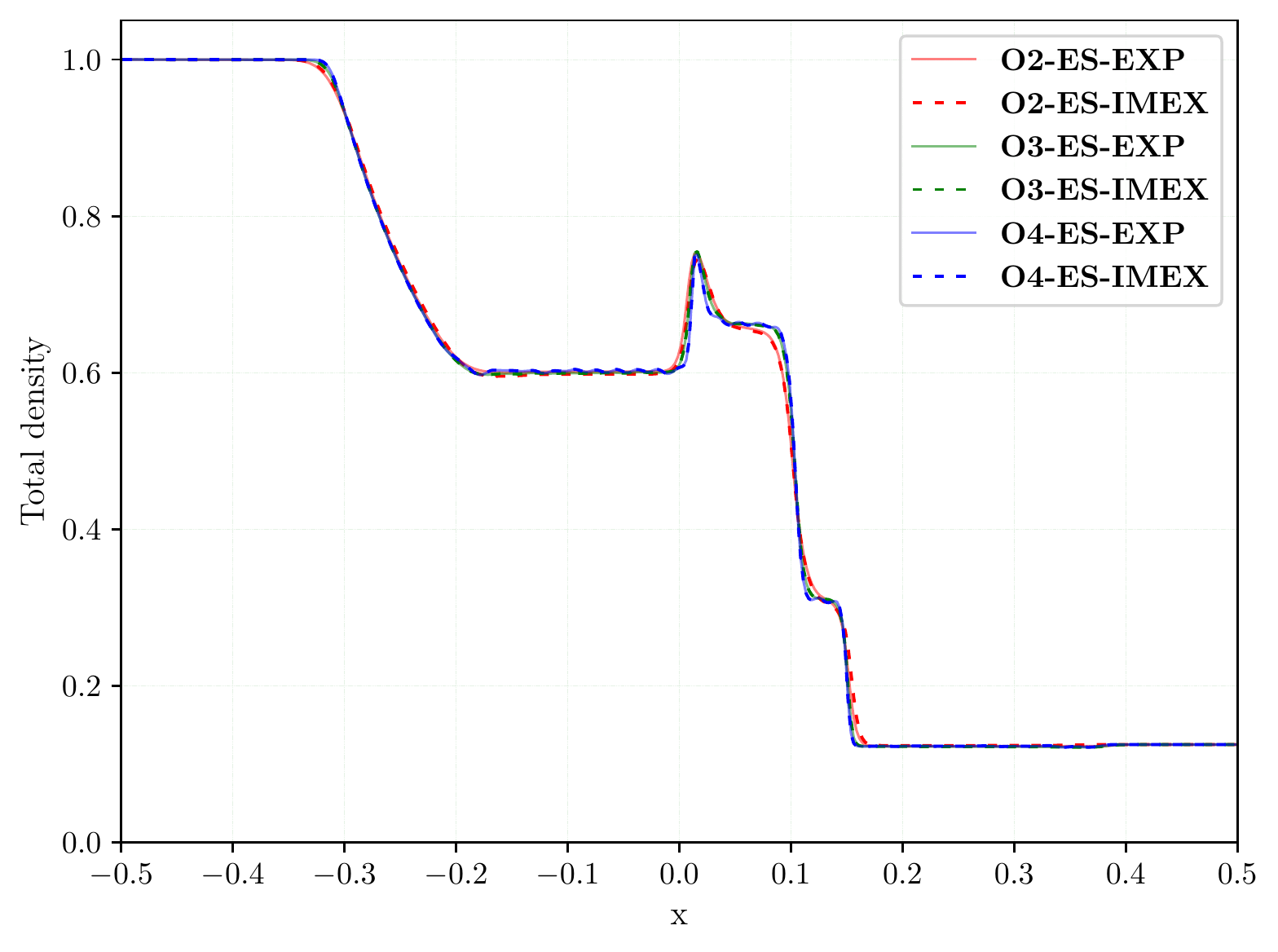}
			\label{fig:brio_ri5_all_scheme}}
		\subfigure[Plots of total density $\rho_i+\rho_e$, for {\bf O4-ES-IMEX} scheme for various stiff cases at $400$ cells.]{
			\includegraphics[width=0.48\textwidth]{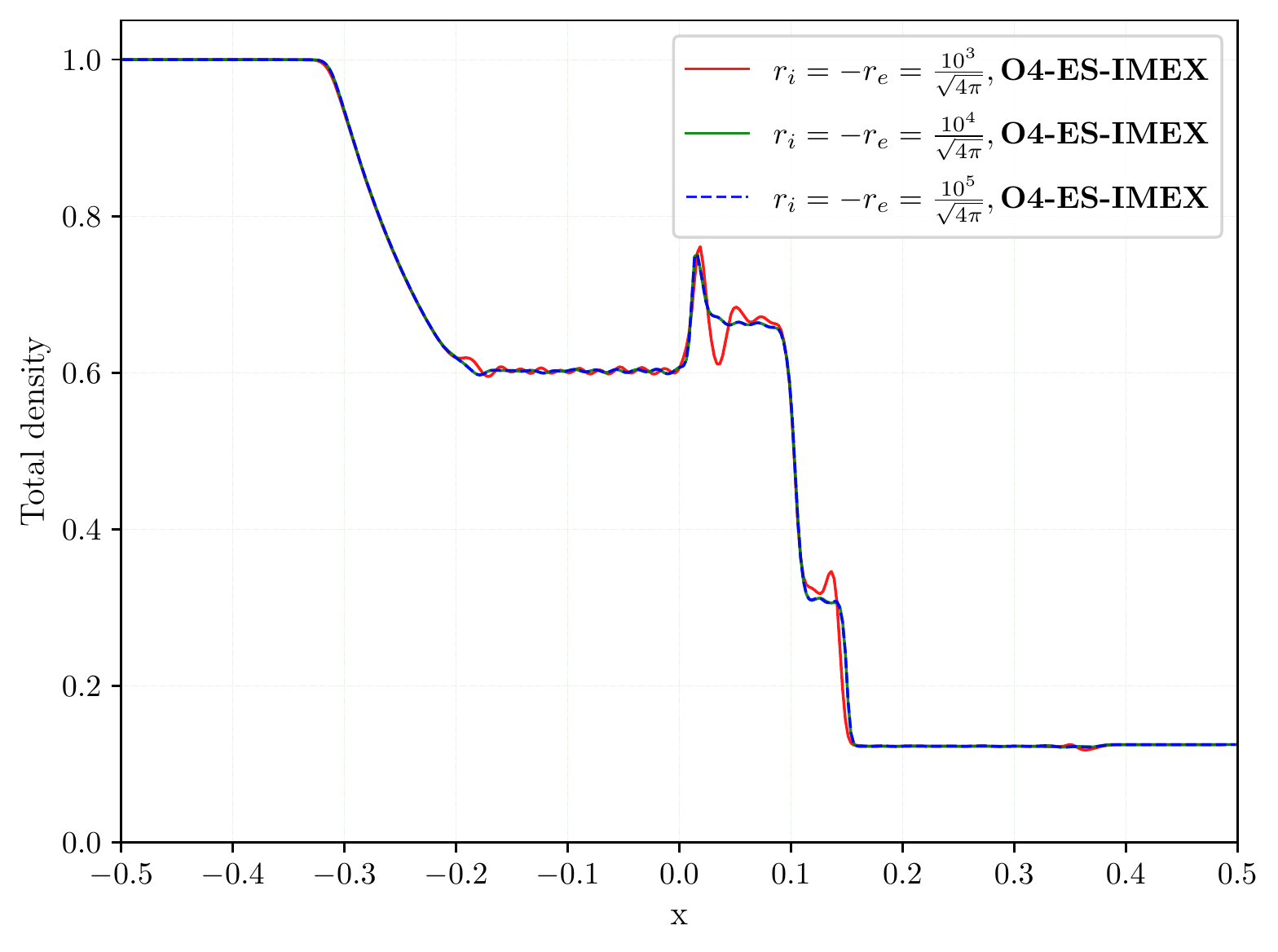}
			\label{fig:brio_o4_all_ri}}
		\caption{Numerical results for \nameref{test:1d_brio} using various stiff cases at 400 cells.}
		\label{fig:brio_stiff_analysis}
	\end{center}
\end{figure}

To compare with the results of~\cite{Balsara2016}, in Figure \eqref{fig:brio_rho_o2}, we plot the total density $\rho_i+\rho_e$ for {\bf O2-ES-Exp} and {\bf O2-ES-IMEX} schemes using $400$, $1600$ and $6400$ cells. To compare the solution with RMHD solution, we have computed ``RMHD" solution using $r_i=-r_e=10^5/\sqrt{4 \pi}$
at $6400$ cells and {\bf O2-ES-IMEX} scheme. We observe that on a resolution of $400$ cells, the cell size $\Delta x = 0.0025$ is slightly larger than the plasma skin depth in the entire domain; thus, the profile roughly matches with the RMHD results~\cite{Balsara2001}, while at the higher resolutions, $1600$ and $6400$, the cell size is much smaller than the plasma skin depth, consequently, appearance of the two-fluid effect in the form of dispersive waves is observed. We also note that both explicit ({\bf O2-ES-Exp}) and IMEX ({\bf O2-ES-IMEX}) schemes have similar performance in capturing all the waves and additional effects at the various resolutions.

To compare second-order schemes, with the third and fourth-order schemes, in Figure~\eqref{fig:brio_rho_all}, we plot total density for explicit ({\bf O2-ES-Exp}, {\bf O3-ES-Exp} and {\bf O4-ES-Exp}) and ARK-IMEX schemes
({\bf O2-ES-IMEX}, {\bf O3-ES-IMEX}, and {\bf O4-ES-IMEX}) using $400$ cells. We observe that the third and fourth-order schemes are able to capture the additional dispersive effects even on the coarser mesh. We also note that fourth-order schemes are more effective in capturing these effects compared to the third-order scheme. Furthermore, the  third-order schemes ({\bf O3-ES-Exp} and {\bf O3-ES-IMEX}) have similar performance, and similarly, the fourth-order schemes ({\bf O4-ES-Exp} and {\bf O4-ES-IMEX}) have similar performance. 

We can also observe this from the time evolution of total entropy in the domain for ion flow in Figure~\eqref{fig:brio_ent_ion} and electron flow in Figure \eqref{fig:brio_ent_elec}. We note that the same order schemes have almost the same entropy decay behaviour. Furthermore, third and fourth-order schemes have smaller entropy dissipation compared to the second-order schemes, while fourth-order schemes have less entropy dissipation compared to third-order schemes.

\begin{table}[ht]
	\centering
	\begin{tabular}{ |c|c|c|c|c|c|c|c|c|}
		\hline
		&\multicolumn{4}{|c}{$ r_i =-r_e = 10^4/ \sqrt{4 \pi}$ } &
		
		\multicolumn{4}{|c|}{$ r_i =-r_e = 10^5/ \sqrt{4 \pi}$} \\
		\hline
		\hline
		Number of cells & $100$ & $200$ & $400$ & $800$ & $100$ & $200$ & $400$ & $800$ \\
		\hline
		\hline
		{\bf O2-ES-Exp }   &  3.24s & 5.34s & 9.80s & 17.90s &    85.68s  &  133.99s &  228.22s &  413.55s  \\
		\hline
		{\bf O2-ES-IMEX}  &  0.49s & 1.30s & 4.25s & 14.28s   & 0.50s & 1.22s & 3.70s & 12.11s \\
		\hline
		\hline
		{\bf O3-ES-Exp}   &  2.23s & 3.83s & 6.94s & 13.04s & 32.88s & 53.33s & 96.31s & 182.08s \\
		\hline
		{\bf O3-ES-IMEX}  & 0.72s & 2.02s & 6.31s & 22.14s & 0.97s & 2.69s & 8.51s & 30.01s \\
		\hline
		\hline
		{\bf O4-ES-Exp}   & 6.03s & 10.72s & 20.01s & 37.98s & 61.96s & 108.51s & 198.32s & 374.54s \\
		\hline
		{\bf O4-ES-IMEX}  & 1.30s & 3.78s & 12.86s & 44.90s & 1.27s & 3.82s & 12.87s & 47.00s \\
		\hline
	\end{tabular}
	\caption[h]{\nameref{test:1d_brio}: Computational times using $ r_i =-r_e = 10^4/ \sqrt{4 \pi}$ and $ r_i =-r_e = 10^5/ \sqrt{4 \pi}$ for the second ({\bf O2-ES-Exp} and {\bf O2-ES-IMEX}), third ({\bf O3-ES-Exp} and {\bf O3-ES-IMEX}), and fourth ({\bf O4-ES-Exp} and {\bf O4-ES-IMEX}) order schemes.}
	\label{table:time_analysis}
\end{table}	

We also consider the stiff source terms with $r_i=-r_e=10^4/\sqrt{4 \pi}$ and $r_i=-r_e=10^5/\sqrt{4 \pi}$ leading to very small plasma skin depth, which will require substantially more cells to resolve the dispersive effects. In Figure \eqref{fig:brio_stiff_analysis}, we show solutions obtained from different schemes using 400 cells. For both the values of $r_i$, all the schemes produce solutions close to RMHD solution, with third and fourth-order schemes performing better than the second-order schemes. Also, there is no visible difference in Explicit and IMEX scheme of same order. To further show the effects of higher values of $r_i$, in Figure \eqref{fig:brio_o4_all_ri}, we show solutions using {\bf O4-ES-IMEX} scheme on $400$ cells, for different values of $r_i$. For $r_i=10^3/\sqrt{4 \pi}$, we can see the dispersive effects are resolved. However, for $r_i=10^4/\sqrt{4 \pi}$ and $r_i=10^5/\sqrt{4 \pi}$, the scheme was not able to resolve the small scale oscillation; nevertheless, the scheme produces a stable solution thanks to the built in entropy stability property of the scheme.

To demonstrate the effectiveness of ARK-IMEX schemes over the explicit schemes for the stiff source terms, in Table~\eqref{table:time_analysis}, we present the computational time for the cases $r_i=-r_e=10^4/\sqrt{4 \pi}$ and $r_i=-r_e=10^5/\sqrt{4 \pi}$ on $100$, $200$, $400$, and $800$ cells. We use the MPI-parallelized code on a CPU with 10 cores. At the lower resolutions, the source terms will govern the time step as time step given by the CFL restriction is large. Hence, we observe consistently that IMEX schemes outperform the explicit schemes on the coarser mesh.  In the case of $r_i=-r_e=10^4/\sqrt{4 \pi}$, we see that IMEX outperform the explicit schemes up to $400$ cells. On $800$ cells, we observe that the IMEX schemes are more expensive than the explicit schemes as at this resolution, stable time step is governed by the CFL condition, not the source terms. In the case of $r_i=-r_e=10^5/\sqrt{4 \pi}$, we observe that IMEX schemes consistently outperform explicit scheme even on $800$ cells as the stiffness in the source is very strong and hence it decides the stable time step for explicit schemes.

Given the similar accuracy of the explicit and IMEX schemes of the same order, and the ability of IMEX schemes to overcome time-step restriction due to stiff source terms, especially at the lower resolutions, in the remaining test cases, we will present numerical results only for the ARK-IMEX schemes. We have also computed the results using explicit schemes (whenever source terms are not stiff) also, and the solutions are similar to those of the IMEX schemes of the same order. 

\subsubsection{Self-similar current sheet with finite resistivity} \label{test:1d_sheet} 
The test case was first considered in~\cite{Komissarov2007}, and the two-fluid relativistic extension is presented in~\cite{Amano2016}, which is used here. We consider the modified Equations \eqref{resist_momentum} and \eqref{resist_energy} to capture the resistive effects.

For resistive RMHD solution \cite{Komissarov2007}, only the $B_y$-component has non-zero variation, whose time evolution includes the resistive effects and is governed by the diffusion equation,
\begin{equation} \label{current_diff}
	\dfrac{\partial B_y}{\partial t} - D \dfrac{\partial^2 B_y}{\partial x^2} =0.
\end{equation}
The relation between the diffusion coefficient and resistivity constant is given by the expression $D=\eta c^2$. An exact self similar solution of Eqn.~\eqref{current_diff} in the resistive RMHD regime~\cite{Komissarov2007} is given by 
\begin{equation}
	\label{eq:Byrmhd}
	B_y(x,t) = B_0 \text{  erf}\left(\dfrac{x}{2 \sqrt{Dt}} \right),
\end{equation}
where ``erf" is the error function. 

For two-fluid relativistic plasma equations, we consider the initial $B_y$ at the time $t=1.0$, using  $B_0=1.0$ and $\eta c^2 =0.01$. The domain is $[-1.5,1.5]$ with Neumann boundary conditions. The charge to mass ratios are $r_i = -r_e=10^3$, the ion and electron density are 0.5, and the ion and electron pressures are 25.0. The $z$-component of the ion and electron velocity is given by
\[
u_{z_i}=-u_{z_e}=\dfrac{B_0}{r_i \rho_i \sqrt{\pi D}} \exp\left(-\dfrac{x^2}{4D}\right).
\]
All other variables are set to zero. We evolve the solution from the initial time $t=1.0$ to final time $t=9.0$ using ARK-IMEX schemes {\bf O2-ES-IMEX}, {\bf O3-ES-IMEX} and {\bf O4-ES-IMEX}, and the solutions using 400 cells are shown in Figure~\eqref{fig:sheet}. We observe that, in the presence of friction terms as in Eqn.~\eqref{resist_momentum},~\eqref{resist_energy}, the obtained $B_y$ profile for all the schemes agrees with the resistive RMHD analytical expression~\eqref{eq:Byrmhd}.  Furthermore, all the schemes are highly accurate and there is no visible difference in the solution profiles.

\begin{figure}[!htbp]
	
	\begin{center}
		\subfigure{
			\includegraphics[width=3.5in, height=2.6in]{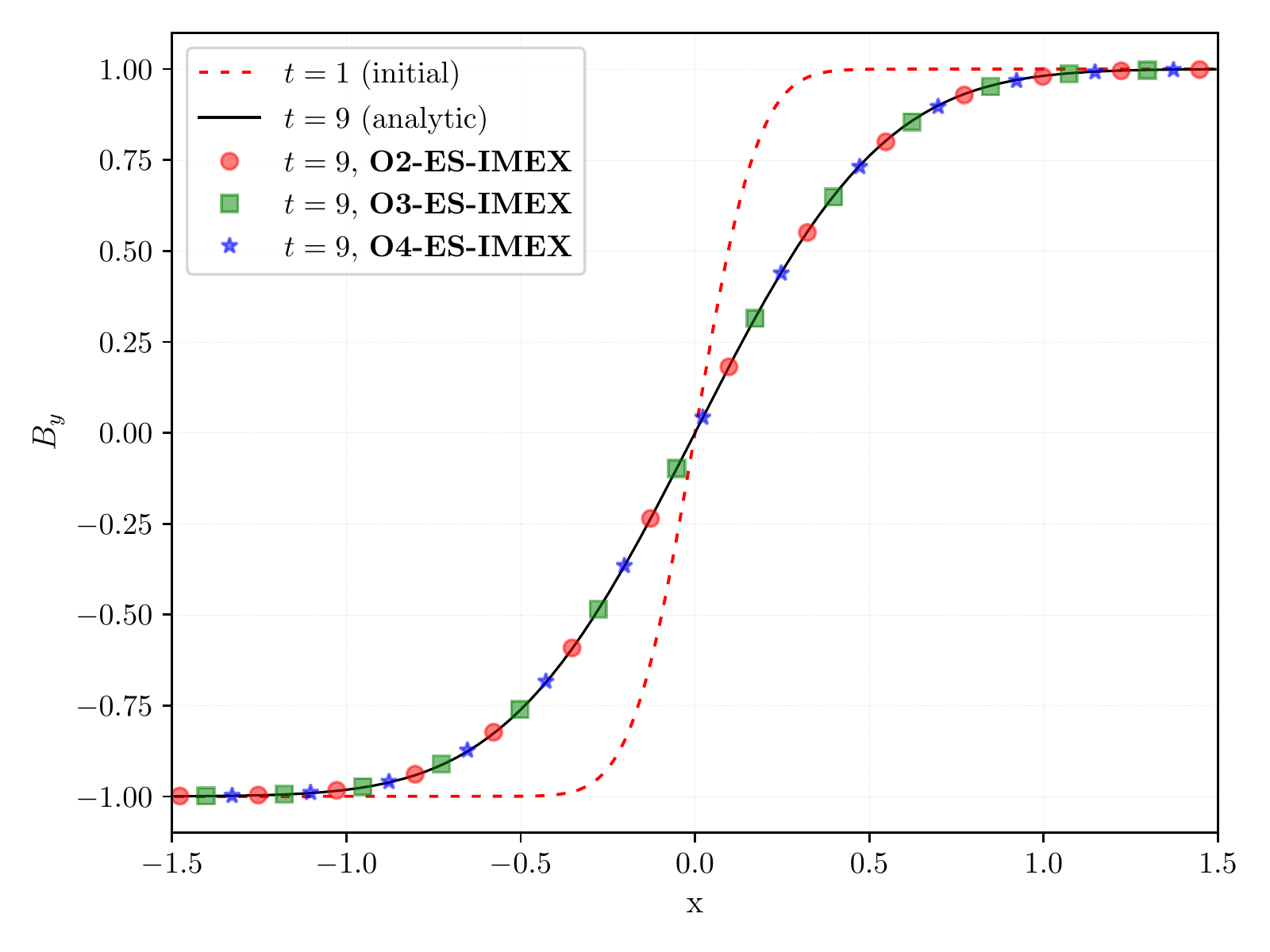}
		}
		\caption{\nameref{test:1d_sheet}: Comparison of the $B_y$ profile for ARK-IMEX schemes ({\bf O2-ES-IMEX}, {\bf O3-ES-IMEX}, and {\bf O4-ES-IMEX}) using $400$ cells.}
		\label{fig:sheet}
	\end{center}
\end{figure}

\subsection{Two-dimensional test cases}
For two-dimensional test cases, the time step is chosen using,
\[
{\Delta t} = \text{CFL} \cdot \min \left\{ \left(\dfrac{\Delta x}{\Lambda_{max}^x(\mathbf{U}_{i,j})} + \dfrac{\Delta y}{\Lambda_{max}^y(\mathbf{U}_{i,j})}\right) : 1 \le i \le N_x, \ 1 \le y \le N_y \right\}
\]
where $\Lambda_{max}^x(\mathbf{U}_{i,j}) =  \max \{|\Lambda^x_k(\mathbf{U}_{i,j})| : 1 \le k \le 18 \}$ and $\Lambda_{max}^y(\mathbf{U}_{i,j}) = \max\{|\Lambda^y_k(\mathbf{U}_{i,j})|: 1 \le k \le 18\}$. We take CFL to be $0.45$ and present the numerical results for the ARK-IMEX schemes~({\bf O3-ES-IMEX}, and {\bf O4-ES-IMEX}). We have not presented the results for second-order scheme {\bf O2-ES-IMEX} as the scheme is less accurate when compared to the third and fourth order schemes. 

\reva{To monitor the divergence of magnetic field, we also plot the evolution of the $L^1$-norm of $\nabla \cdot \mathbf{B}$, which is calculated as,
	\begin{equation}
		\Delta x \Delta y\sum_{i=1}^{N_x} \sum_{j=1}^{N_y} |  (\nabla\cdot\mathbf{B})_{i,j}|
	\end{equation}
	where $(\nabla \cdot \mathbf{B})_{i,j}$ is evaluated using the central difference approximation in each direction, i.e.
	\begin{equation}
		(\nabla \cdot \mathbf{B})_{i,j}:= \dfrac{B_{x,i+1,j} - B_{x,i-1,j}}{2 \Delta x}
		+ \dfrac{B_{y,i,j+1} - B_{y,i,j-1}}{2 \Delta y}.
		\label{eq:dis_divB}    
	\end{equation}
}

\subsubsection{Relativistic Orzag-Tang test case} \label{test:2d_ot} 
Orzag-Tang test case was first proposed for the non-relativistic MHD systems in~\cite{Orszag1979}. Here, we consider the two-fluid relativistic extension of the test case, presented in \cite{Balsara2016}. To have consistent comparison with the results in \cite{Balsara2016}, we consider the Maxwell's equations with scaled source \eqref{eq:modifed_max_source}. The computational domain is $[0.0,1.0] \times [0.0,1.0]$ with periodic boundary conditions, and the initial conditions are given by,
\begin{align*}
\renewcommand{\arraystretch}{1.5}
	\begin{pmatrix*}[c]
		\rho_i \\ u_{x_i} \\ u_{y_i} \\ p_i
	\end{pmatrix*} = 
	\begin{pmatrix*}[c]
		\frac{25}{72 \pi} \\ - \frac{\sin(2 \pi y)}{2} \\ \frac{\sin(2 \pi x)}{2} \\ \frac{5}{24 \pi}
	\end{pmatrix*}, \qquad 
	\begin{pmatrix*}[c]
		\rho_e \\u_{x_e} \\ u_{y_e} \\ p_e
	\end{pmatrix*} = 
	\begin{pmatrix*}[c]
		\frac{25}{72 \pi} \\ - \frac{\sin(2 \pi y)}{2} \\  \frac{\sin(2 \pi x)}{2} \\  \frac{5}{24 \pi}
	\end{pmatrix*}, \qquad 
	\begin{pmatrix*}[c]
		B_x \\ B_y
	\end{pmatrix*} = 
	\begin{pmatrix*}[c]
		-\sin(2 \pi y) \\  \sin (4 \pi x)
	\end{pmatrix*}
\end{align*}
The initial electric field is $-\mathbf{u}_i \times \mathbf{B}$ and all other variables are set to zero. We use the charge to mass ratios of $r_i=-r_e=10^3/\sqrt{4 \pi}$, thus the plasma skin depth approximates to $3.0 \times 10^{-3}$. We choose adiabatic indices $\gamma_i=\gamma_e = 5/3$ and simulate till time $t = 1.0$ on a mesh with $200\times 200$ cells.

\begin{figure}[!htbp]
	\begin{center}
		\subfigure[Plot of total density $\rho_i +\rho_e$.]{
			\includegraphics[width=2.9in, height=2.5in]{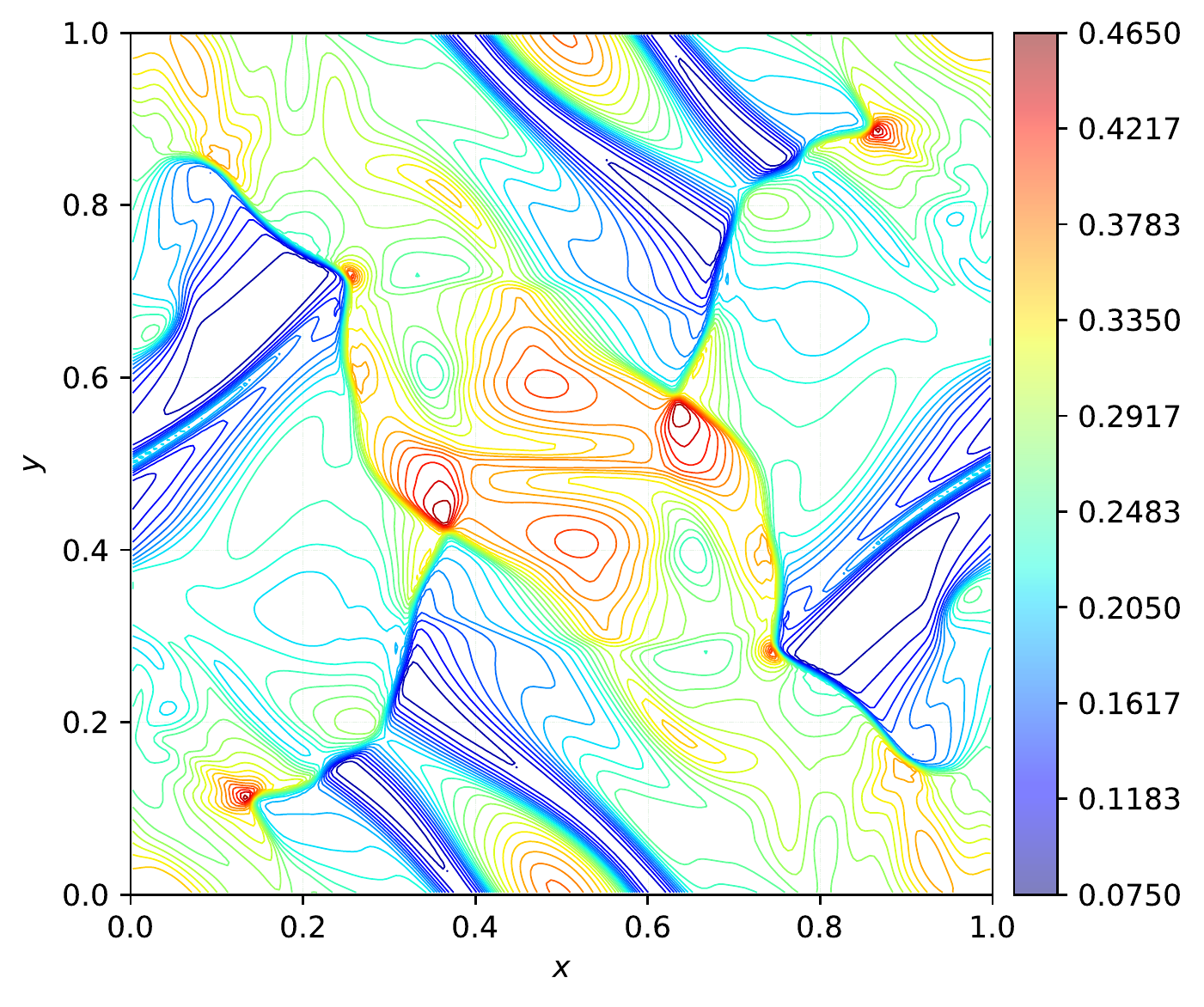}
			\label{fig:ot_o3_imp_rho}}
		\subfigure[Plot of total pressure $p_{i} + p_e$.]{
			\includegraphics[width=2.9in, height=2.5in]{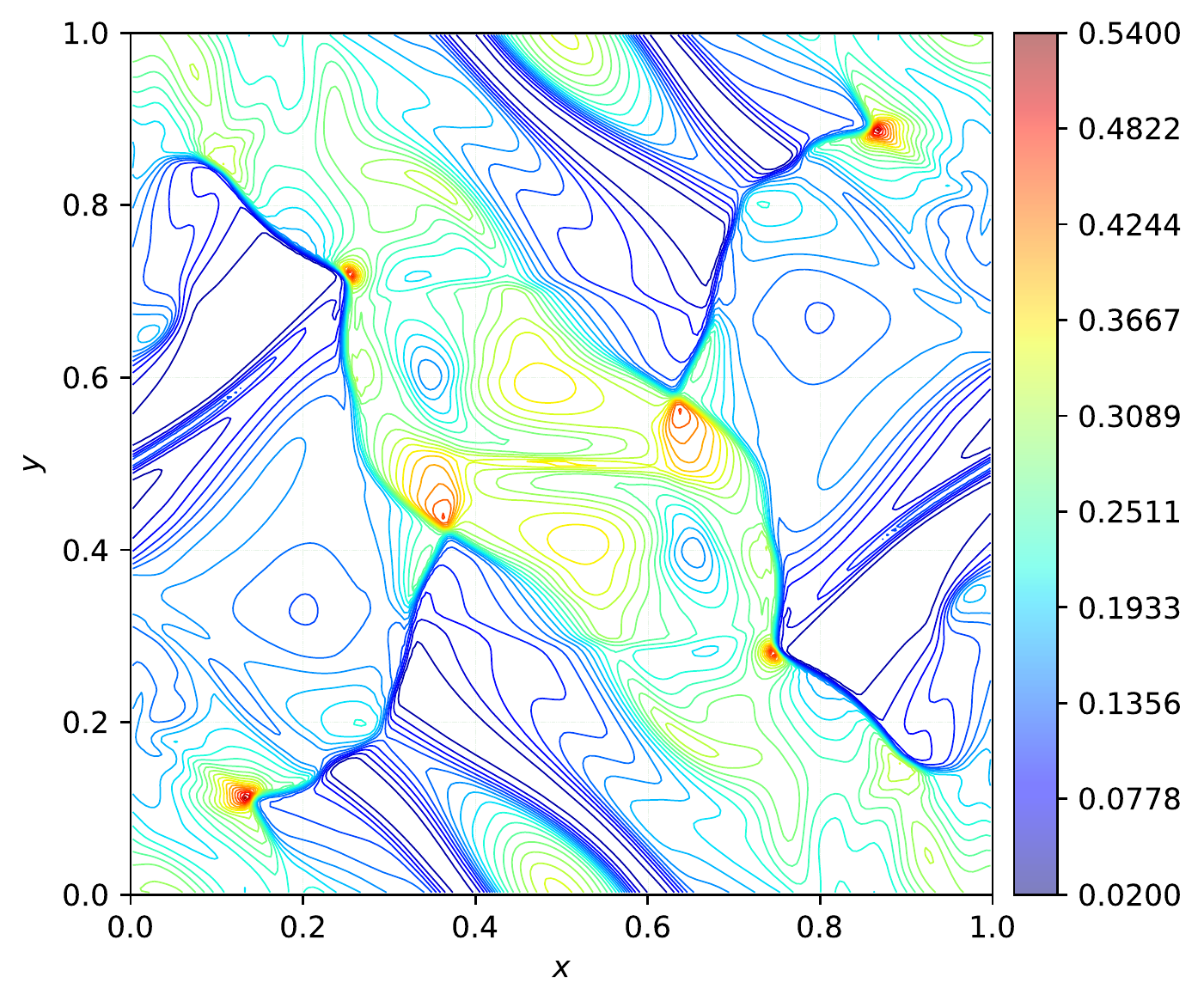}
			\label{fig:ot_o3_imp_p}}
		\subfigure[Plot of ion Lorentz factor $\Gamma_i$.]{
			\includegraphics[width=2.9in, height=2.5in]{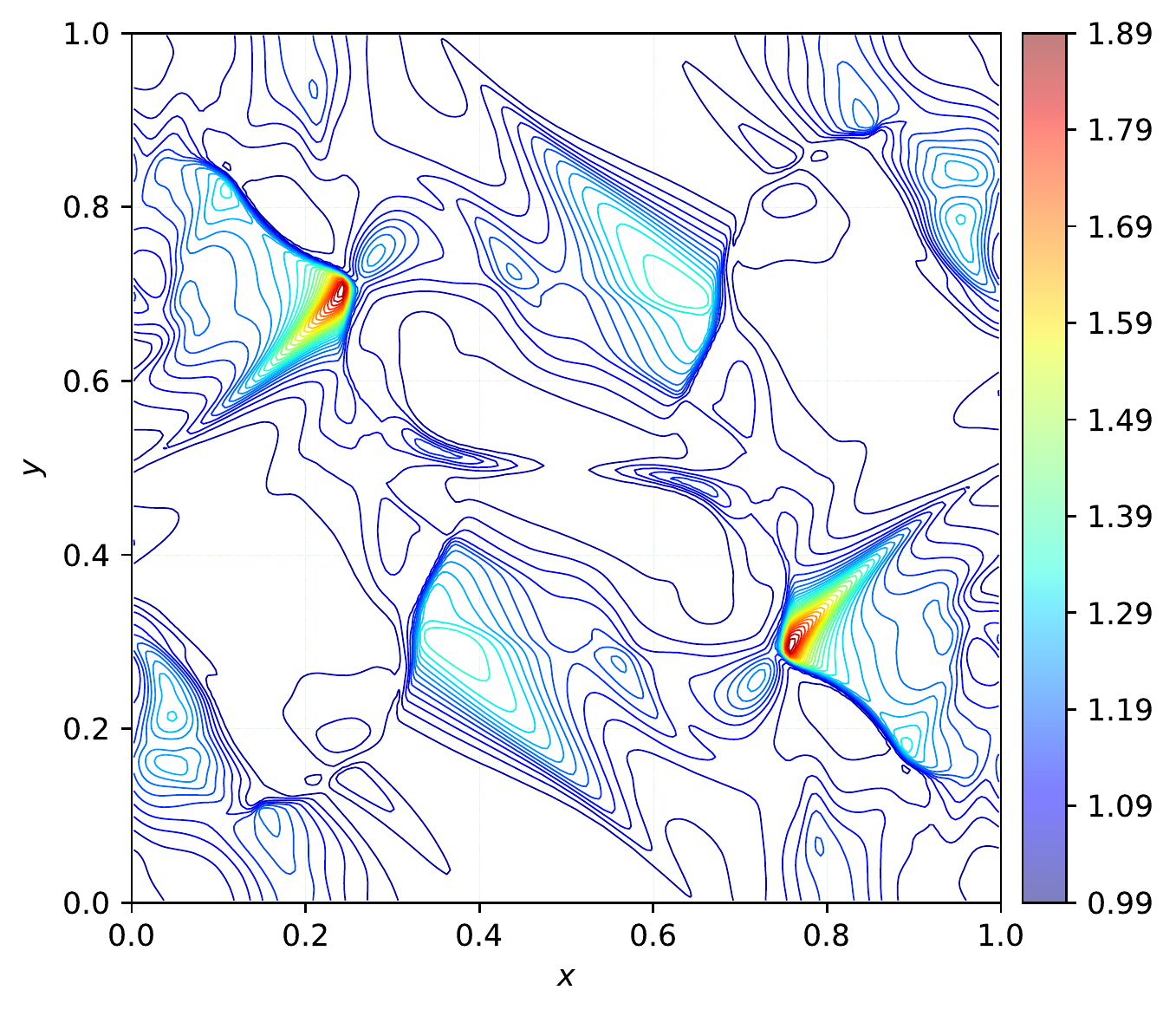}
			\label{fig:ot_o3_imp_lorentz}}
		\subfigure[Plot of magnitude of the Magnetic field, $\dfrac{|\mathbf{B}|^2}{2}$.]{
			\includegraphics[width=2.9in, height=2.5in]{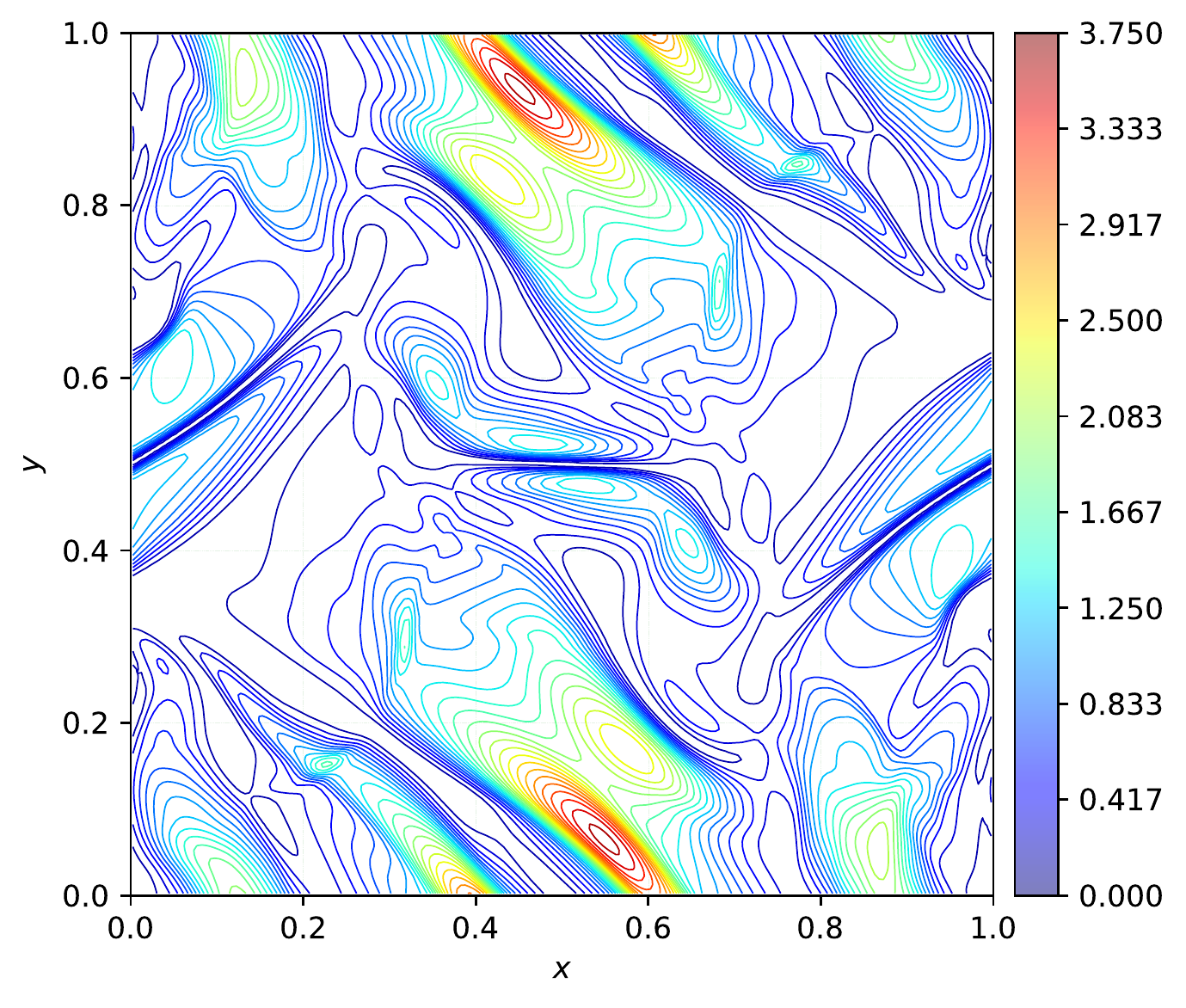}
			\label{fig:ot_o3_imp_magBby2}}
		\caption{\nameref{test:2d_ot}: Plots of total density, total pressure, Ion Lorentz factor, and magnitude of magnetic field $\dfrac{|\mathbf{B}|^2}{2}$ using {\bf O3-ES-IMEX} scheme and $200\times 200$ cells at time $t=1.0$. We have plotted 30 contours for each variable.}
		\label{fig:ot_o3}
	\end{center}
\end{figure}
\begin{figure}[!htbp]
	\begin{center}
		\subfigure[Plot of total density $\rho_i +\rho_e$.]{
			\includegraphics[width=2.9in, height=2.5in]{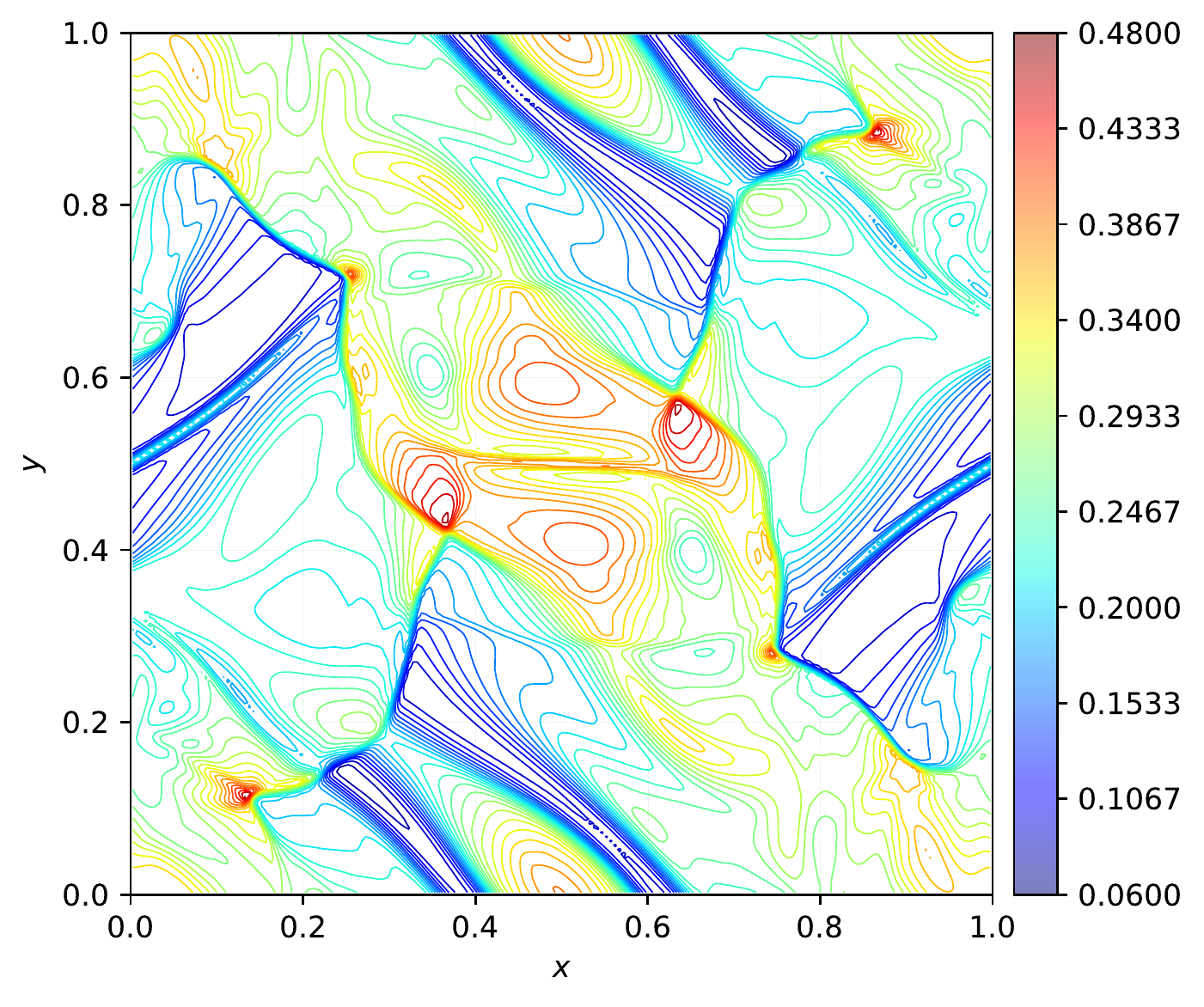}
			\label{fig:ot_o4_imp_rho}}
		\subfigure[Plot of total pressure $p_{i} + p_e$.]{
			\includegraphics[width=2.9in, height=2.5in]{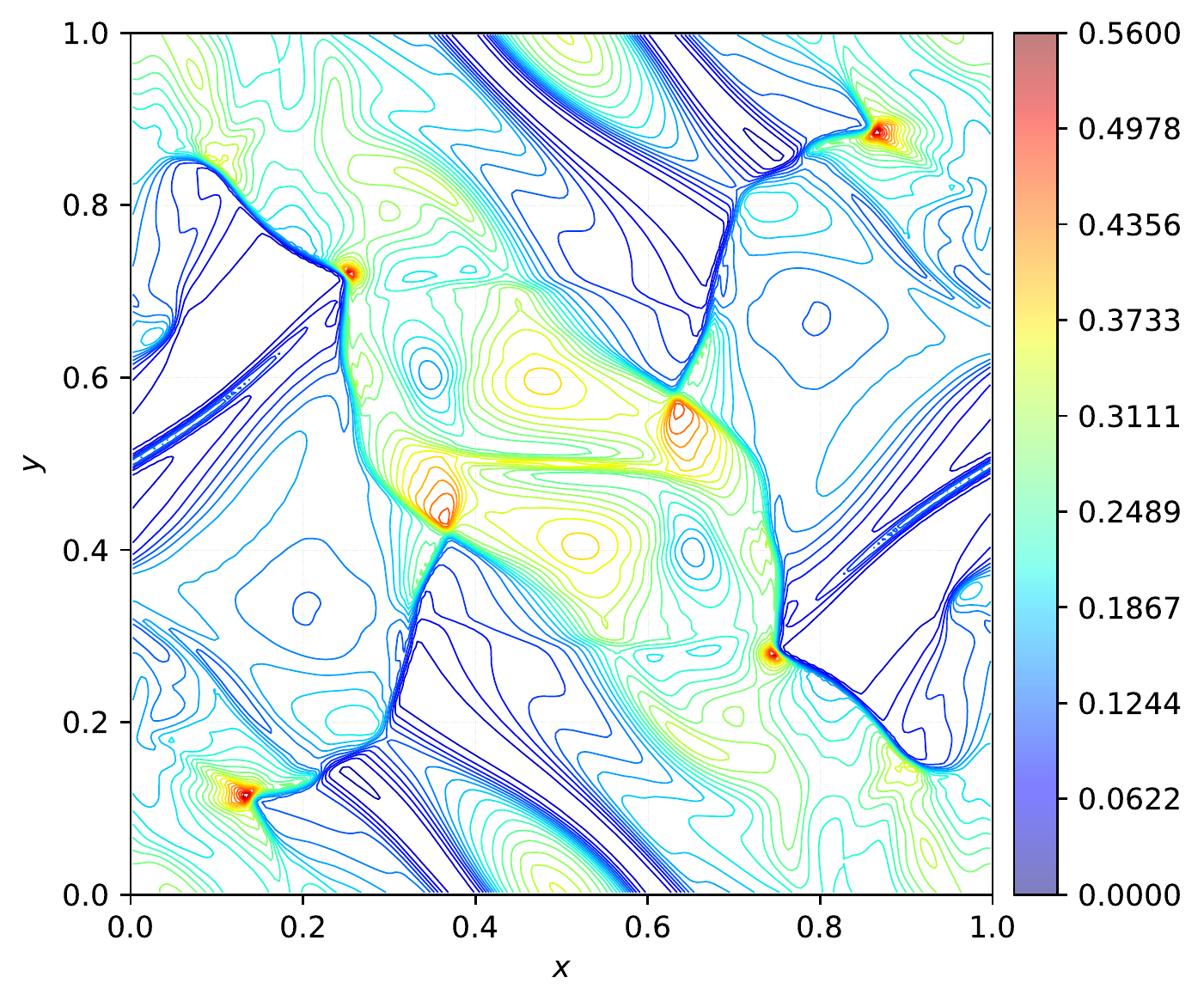}
			\label{fig:ot_o4_imp_p}}
		\subfigure[Plot of ion Lorentz factor $\Gamma_i$.]{
			\includegraphics[width=2.9in, height=2.5in]{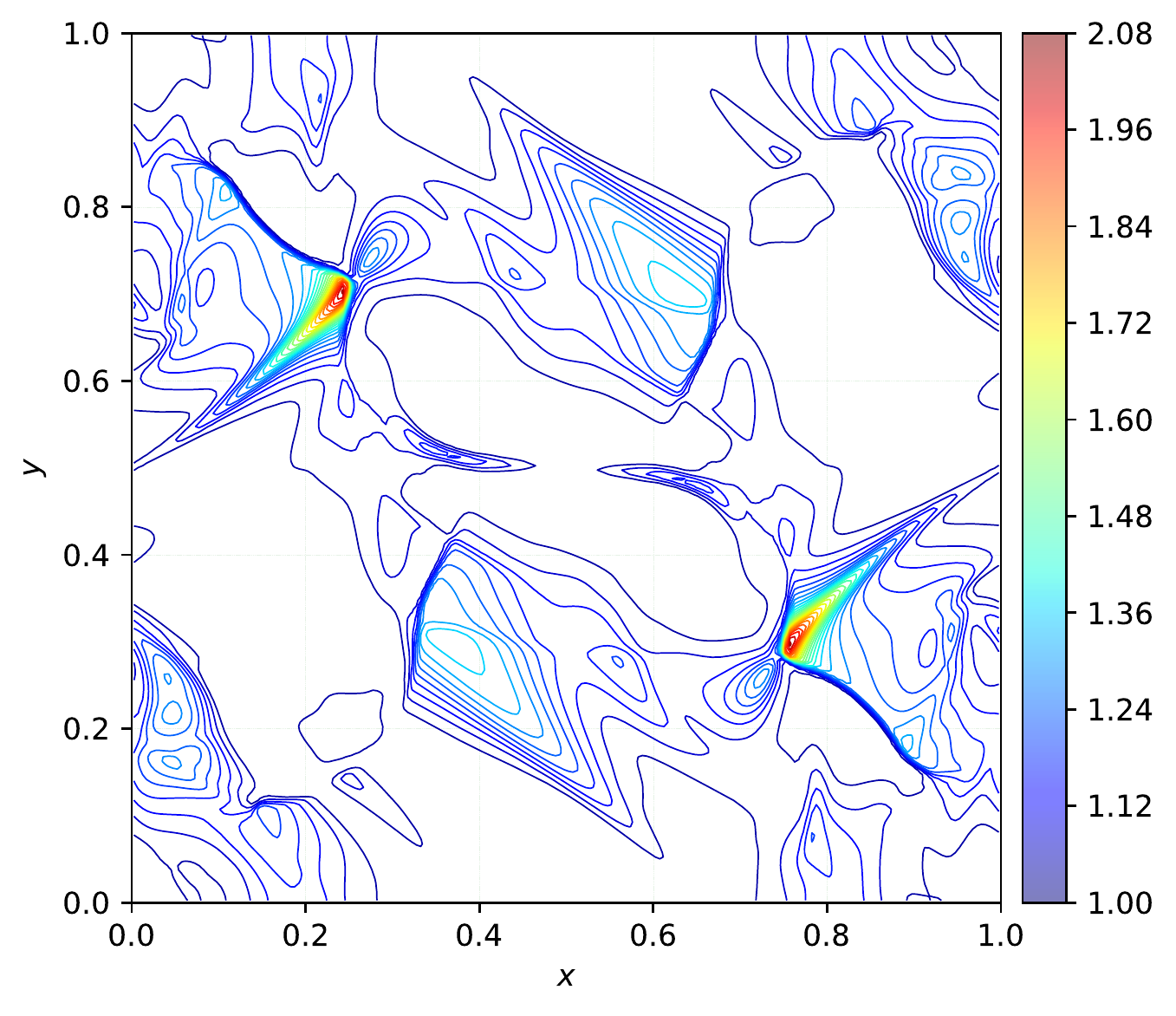}
			\label{fig:ot_o4_imp_lorentz}}
		\subfigure[Plot of magnitude of the Magnetic field, $\dfrac{|\mathbf{B}|^2}{2}$.]{
			\includegraphics[width=2.9in, height=2.5in]{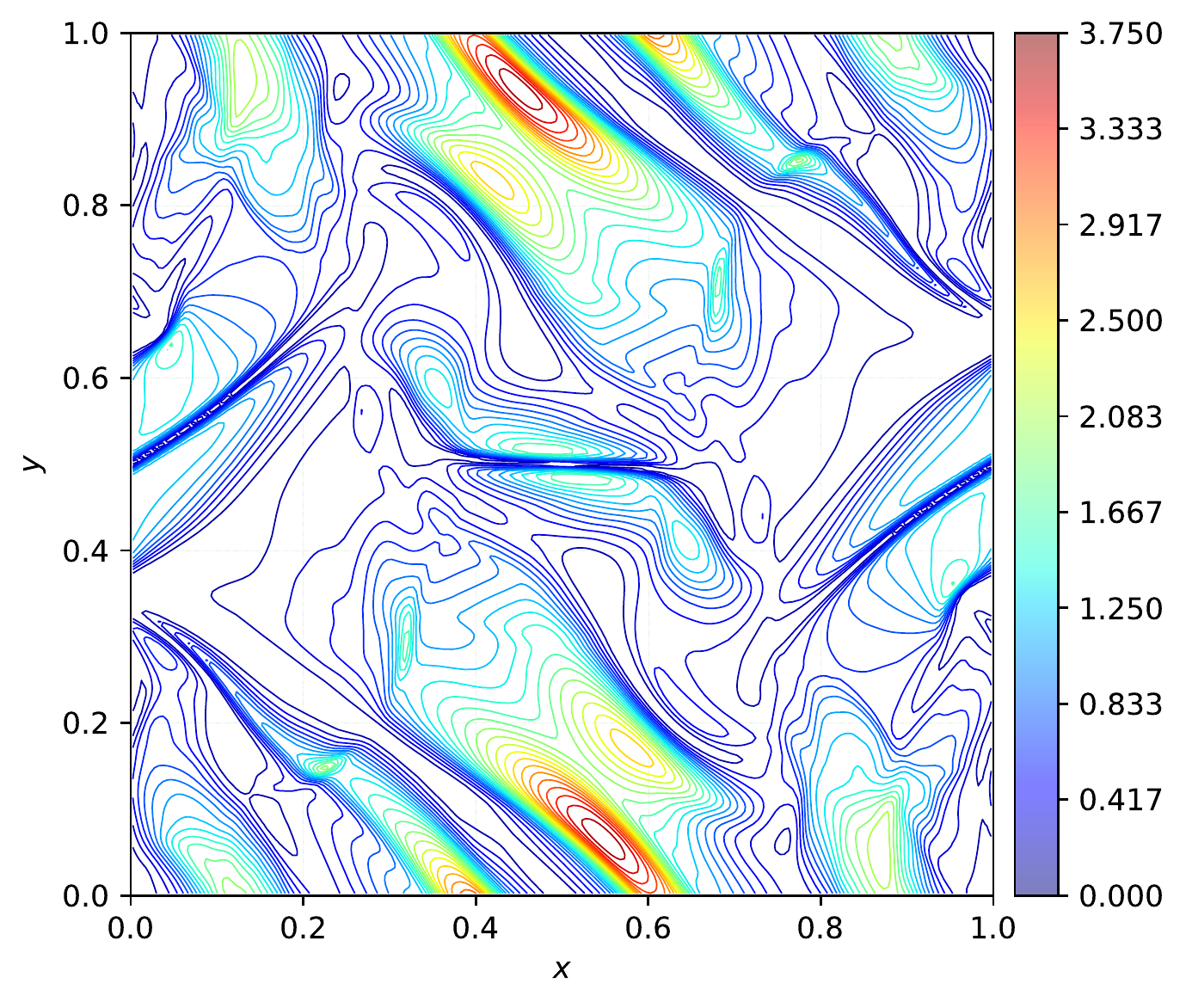}
			\label{fig:ot_o4_imp_magBby2}}
		\caption{\nameref{test:2d_ot}:  Plots of total density, total pressure, Ion Lorentz factor, and magnitude of magnetic field $\dfrac{|\mathbf{B}|^2}{2}$ using {\bf O4-ES-IMEX} scheme and $200\times 200$ cells at time $t=1.0$. We have plotted 30 contours for each variable.}
		\label{fig:ot_o4}
	\end{center}
\end{figure}

\begin{figure}[!htbp]
	\begin{center}
		\includegraphics[width=3.0in, height=2.3in]{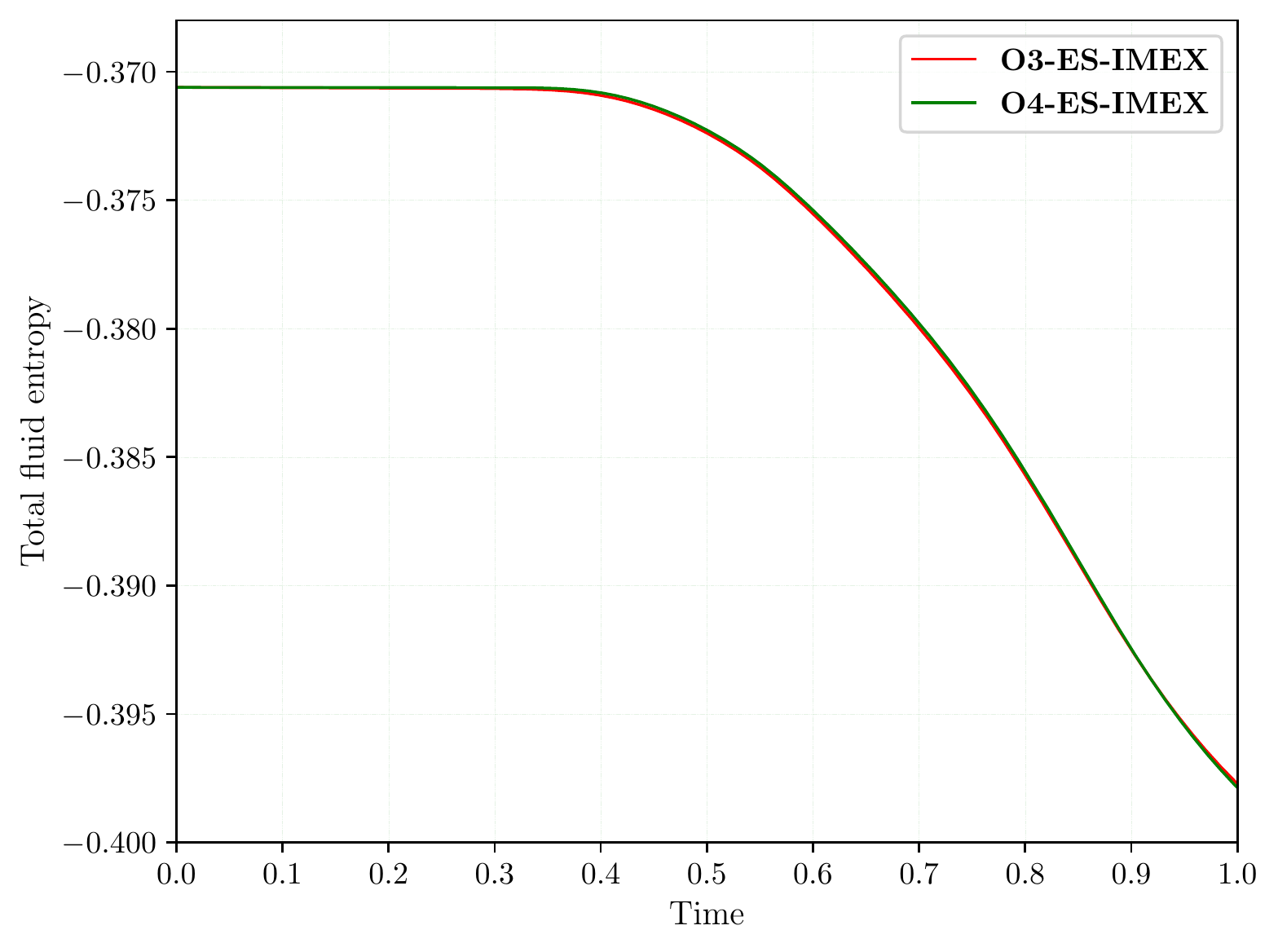}
		\caption{\nameref{test:2d_ot}: Total fluid entropy $\mathcal{U}_i+\mathcal{U}_e$ evolution for the schemes {\bf ES-O3-IMEX} and {\bf ES-O4-IMEX}.}
		\label{fig:ot_entropy}
	\end{center}
\end{figure}

Figure \eqref{fig:ot_o3} and \eqref{fig:ot_o4} show the total density, total pressure, ion Lorentz factor and magnitude of the magnetic field for {\bf O3-ES-IMEX} and {\bf O4-ES-IMEX} schemes, respectively. The results for both the schemes are comparable to those in \cite{Balsara2016}, and both the schemes have similar performance. As the cell size is higher than the plasma skin depth, the solution is comparable to those of the RMHD case.

Figure \eqref{fig:ot_entropy} shows the time evolution of the total fluid entropy $\mathcal{U}_i+\mathcal{U}_e$ for {\bf O3-ES-IMEX} and {\bf O4-ES-IMEX} schemes. Initially, we do not see significant decay as discontinuities are not present in the solution. However, at time approximately $t=0.4$, the solutions start to develop discontinuities and we see sharp decay in the entropy. We also observe that the third-order scheme decays slightly more entropy than the fourth-order scheme, however, the difference is very small.

\reva{
	\begin{figure}[!htbp]
		\begin{center}
			\subfigure[Plot of $|2\Delta x (\nabla \cdot \mathbf{B})_{i,j}|$ using the {\bf O3-ES-IMEX} scheme at time $t=1.0$.]{
				\includegraphics[width=2.9in, height=2.5in]{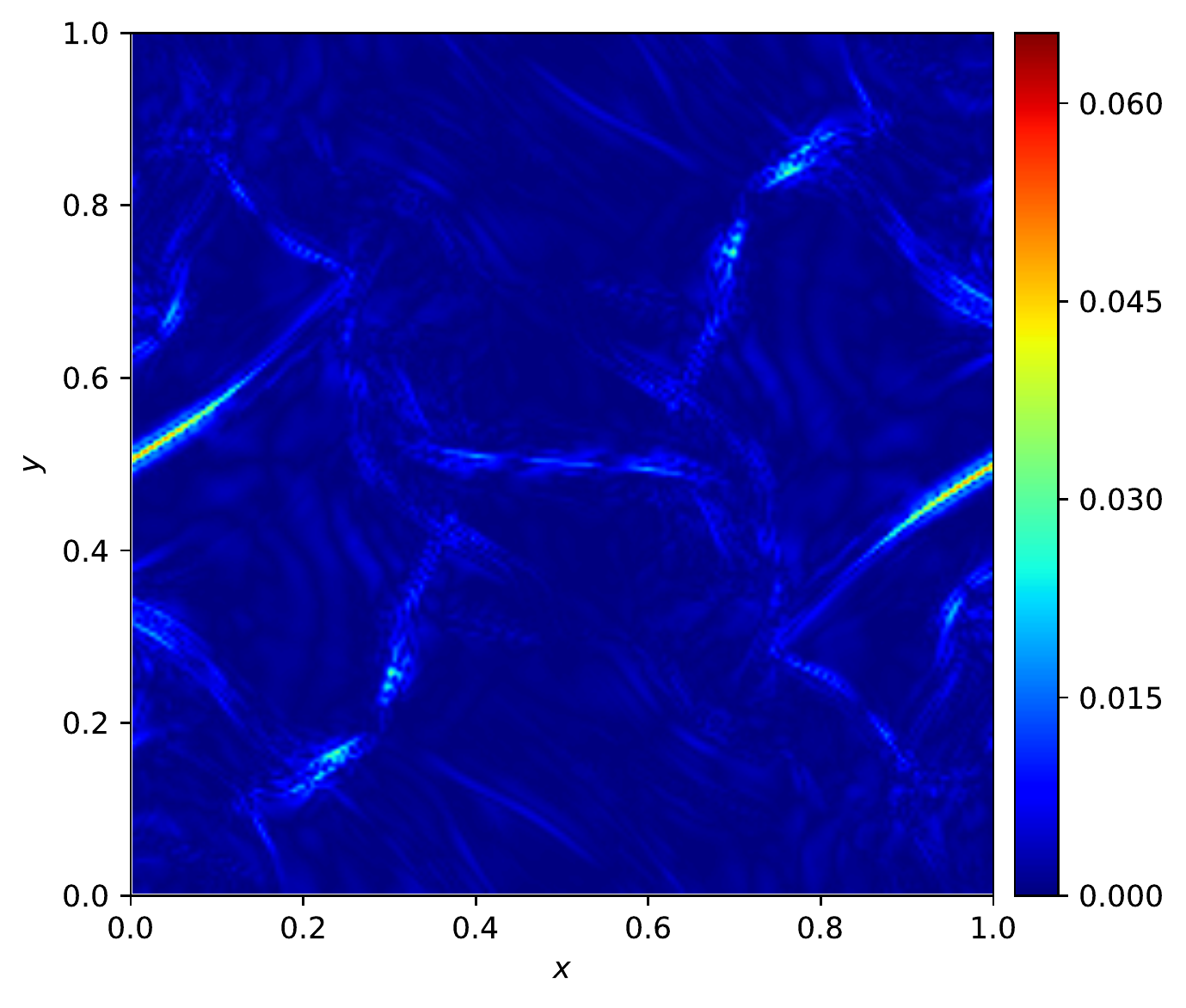}
				\label{fig:ot_2Dnorm_o3}}
			\quad
			\subfigure[Plot of $|2\Delta x (\nabla \cdot \mathbf{B})_{i,j}|$ using the {\bf O4-ES-IMEX} scheme at time $t=1.0$.]{
				\includegraphics[width=2.9in, height=2.5in]{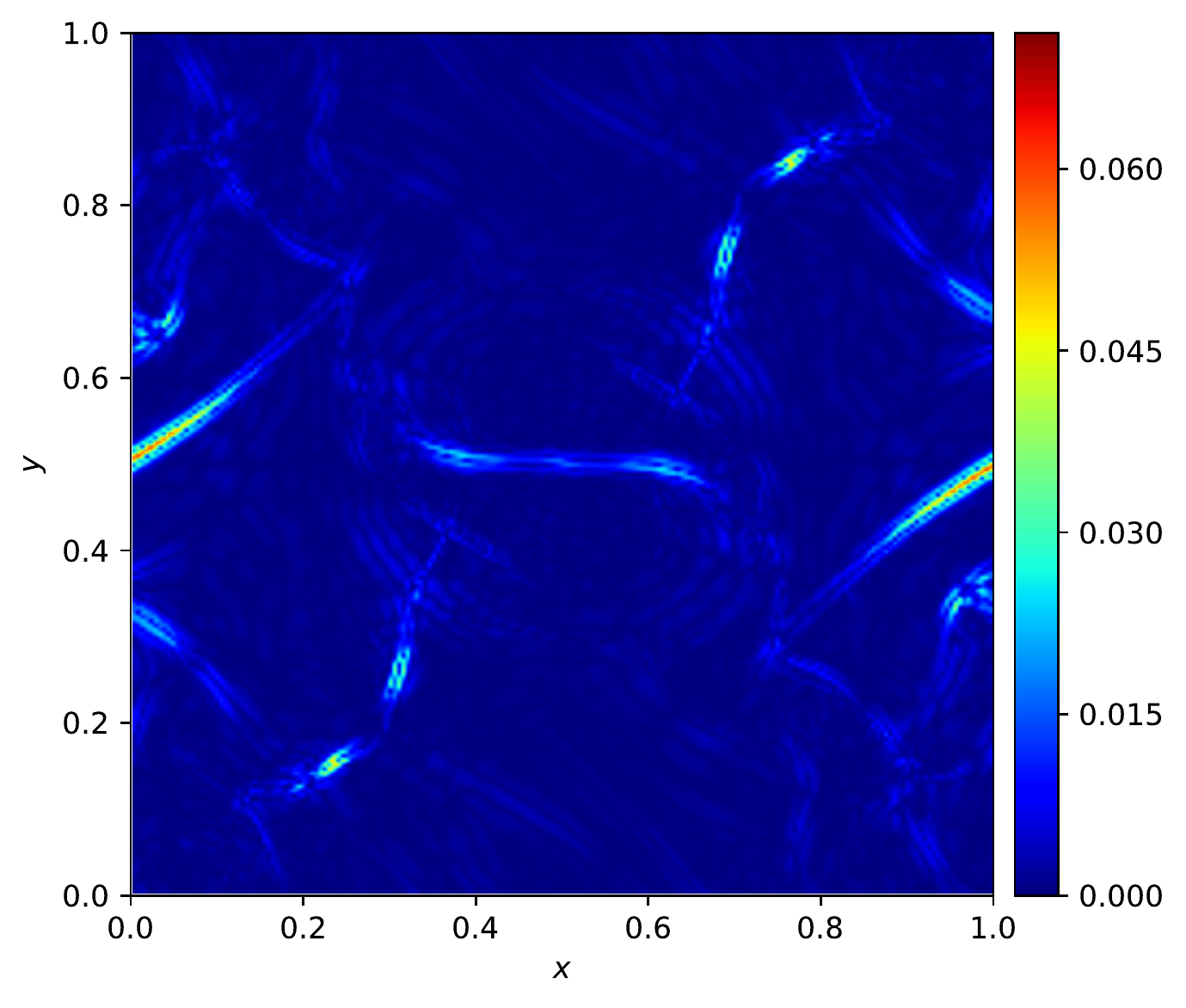}
				\label{fig:ot_2Dnorm_o4}}
			\quad
			\subfigure[Time evolution of $L^1$-norms of divergence of $\mathbf{B}$ till time $t=2.0$ for {\bf O3-ES-IMEX} and {\bf O4-ES-IMEX} schemes.]{
				\includegraphics[width=3.5in, height=2.5in]{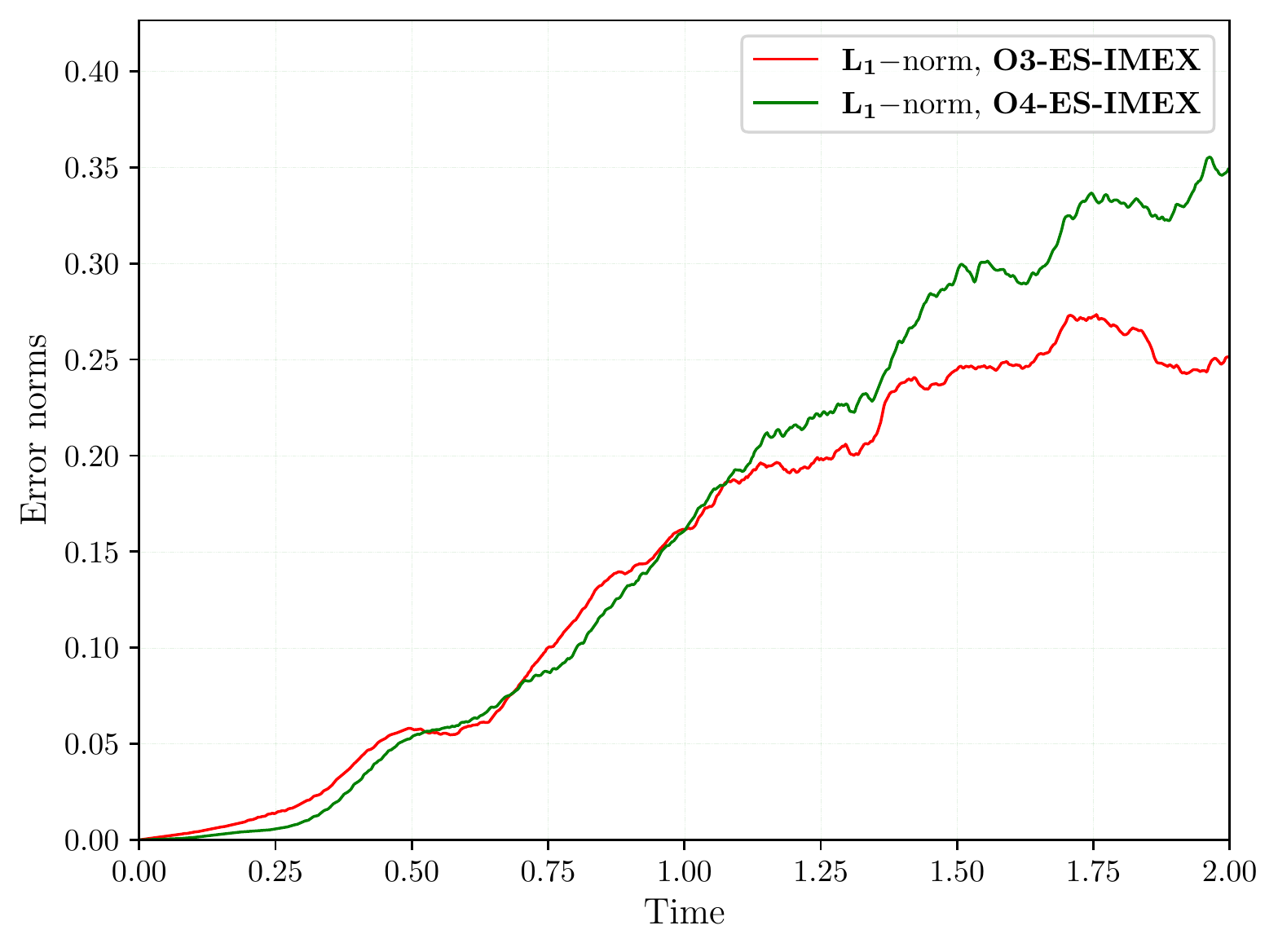}
				\label{fig:ot_l1norm}}
			\caption{\nameref{test:2d_ot}: \reva{Plots of the ($|2\Delta x (\nabla \cdot \mathbf{B})_{i,j}|$) and time evolution of $L^1$-norm of divergence of $\mathbf{B}$ for {\bf O3-ES-IMEX} and {\bf O4-ES-IMEX} schemes using $200\times 200$ cells.}}
			\label{fig:norms_ot}
		\end{center}
	\end{figure}
	
	In Figs.~\ref{fig:ot_2Dnorm_o3} and \ref{fig:ot_2Dnorm_o4}, we have plotted the absolute value of the undivided divergence of the magnetic field, i.e., $|2\Delta x (\nabla \cdot \mathbf{B})_{i,j}|$ (Note that $\Delta x=\Delta y$) for {\bf O3-ES-IMEX} and {\bf O4-ES-IMEX} schemes, respectively \cite{Balsara2016}. We observe the highest values of the divergence error correlate with the location of the discontinuities in the solutions. Furthermore, both schemes have similar absolute maximum errors. To observe the long-time behavior of the divergence errors, in Fig.~\ref{fig:ot_l1norm}, we have plotted the time evolution of $L^1$-norm of divergence of $\mathbf{B}$ for both the schemes till time $t=2$. We observe that both schemes have similar $ L^1$ errors.
}

\subsubsection{Relativistic two-fluid blast problem} \label{test:2d_blast} 
\begin{figure}[!htbp]
	\begin{center}
		\subfigure[Plot of $\log(\rho_i+\rho_e)$.]{
			\includegraphics[width=2.9in, height=2.5in]{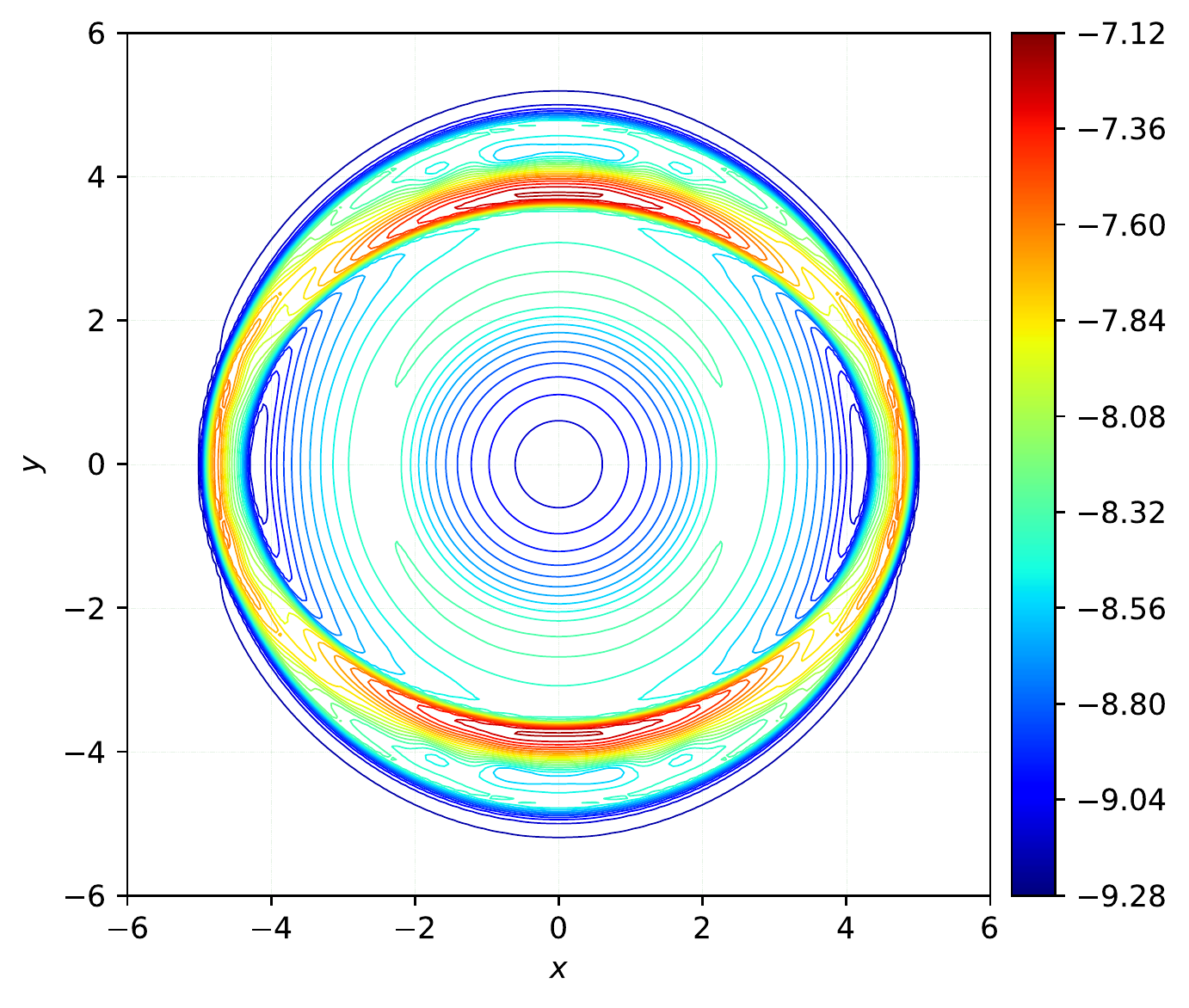}
			\label{fig:blast_o3_imp_0p1_logrho}}
		\subfigure[Plot of $\log(p_i+p_e)$.]{
			\includegraphics[width=2.9in, height=2.5in]{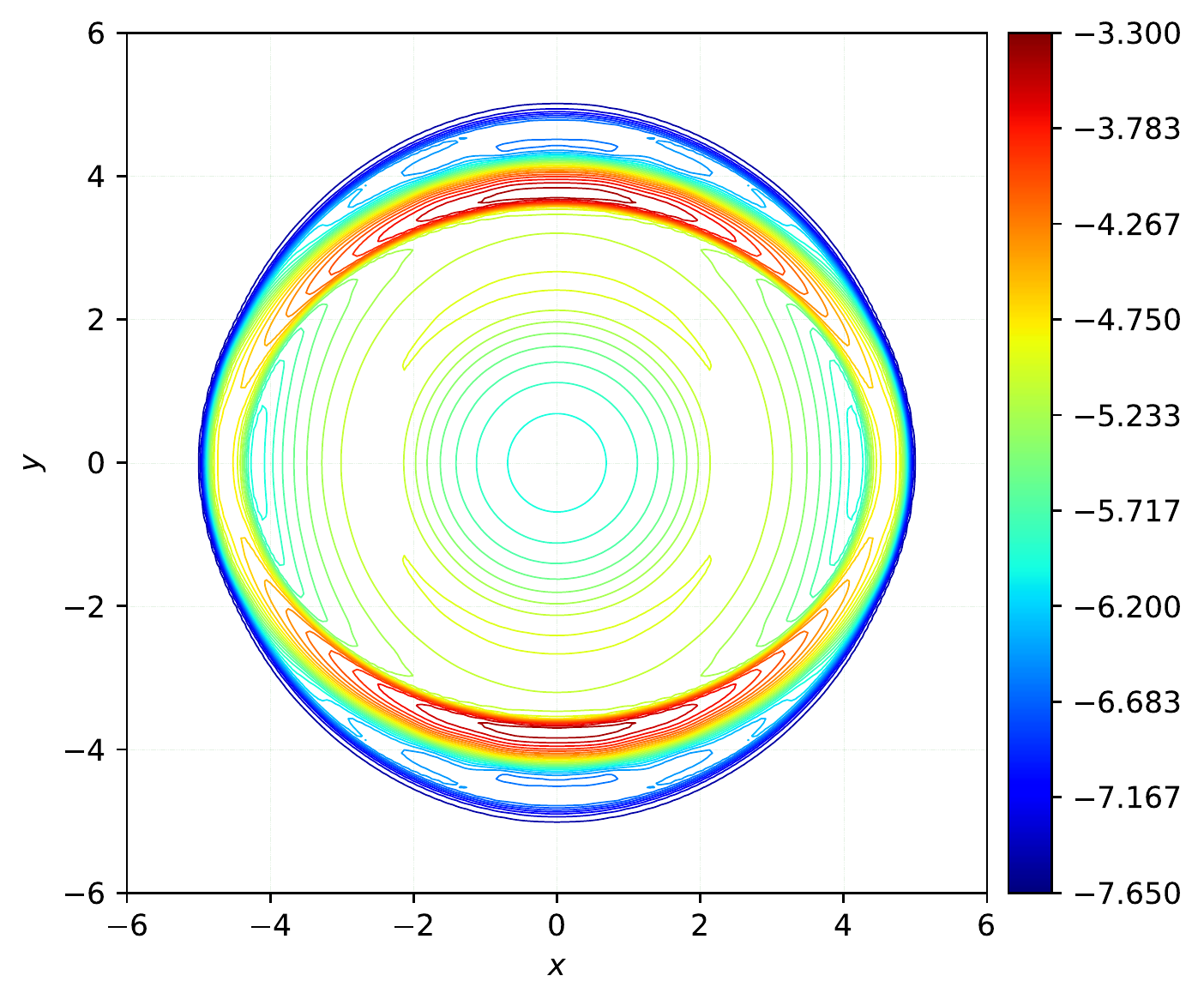}
			\label{fig:blast_o3_imp_0p1_logp}}
		\subfigure[Ion Lorentz factor $\Gamma_i$.]{
			\includegraphics[width=2.9in, height=2.5in]{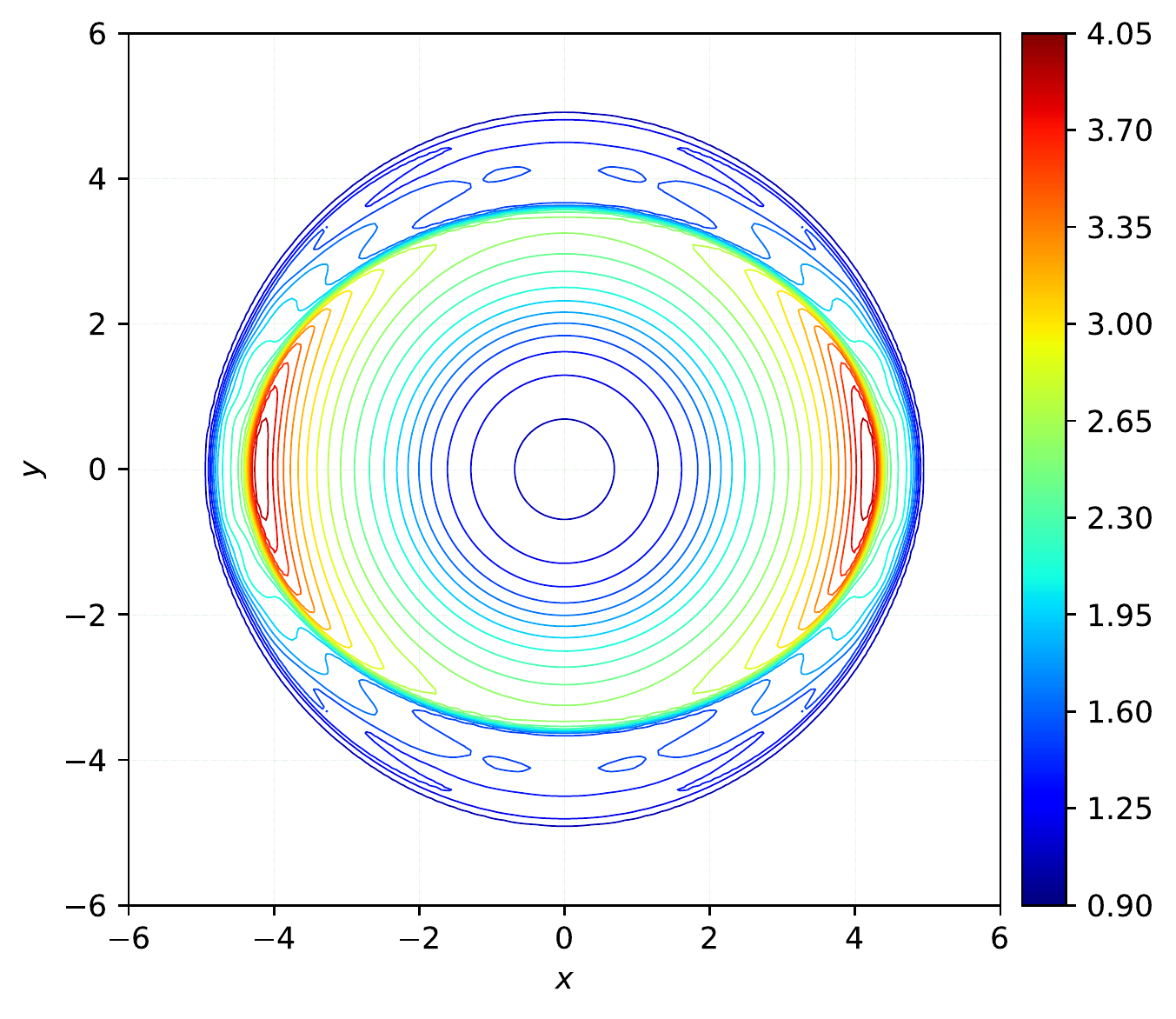}
			\label{fig:blast_o3_imp_0p1_lorentz}}
		\subfigure[Magnitude of the Magnetic field $\dfrac{|\mathbf{B}|^2}{2}$.]{
			\includegraphics[width=2.9in, height=2.5in]{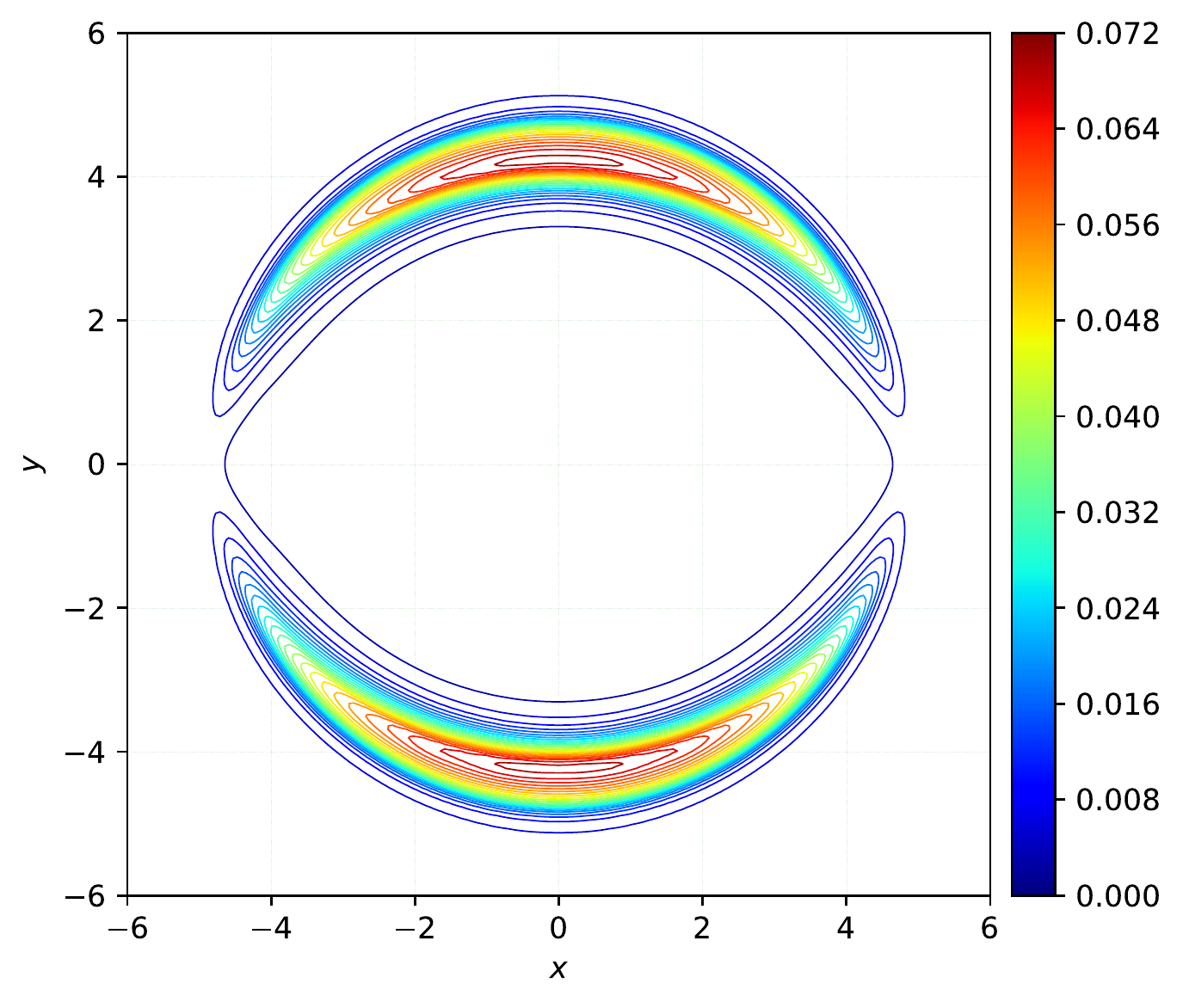}
			\label{fig:blast_o3_imp_0p1_magBby2}}
		\caption{\nameref{test:2d_blast}: Plot for the weakly magnetized medium $B_0=0.1$, using {\bf O3-ES-IMEX} scheme with $200\times 200$ cells. We have plotted 30 contours for each variable.}
		\label{fig:blast_o3_0_1}
	\end{center}
\end{figure}
\begin{figure}[!htbp]
	\begin{center}
		\subfigure[Plot of $\log(\rho_i+\rho_e)$.]{
			\includegraphics[width=2.9in, height=2.5in]{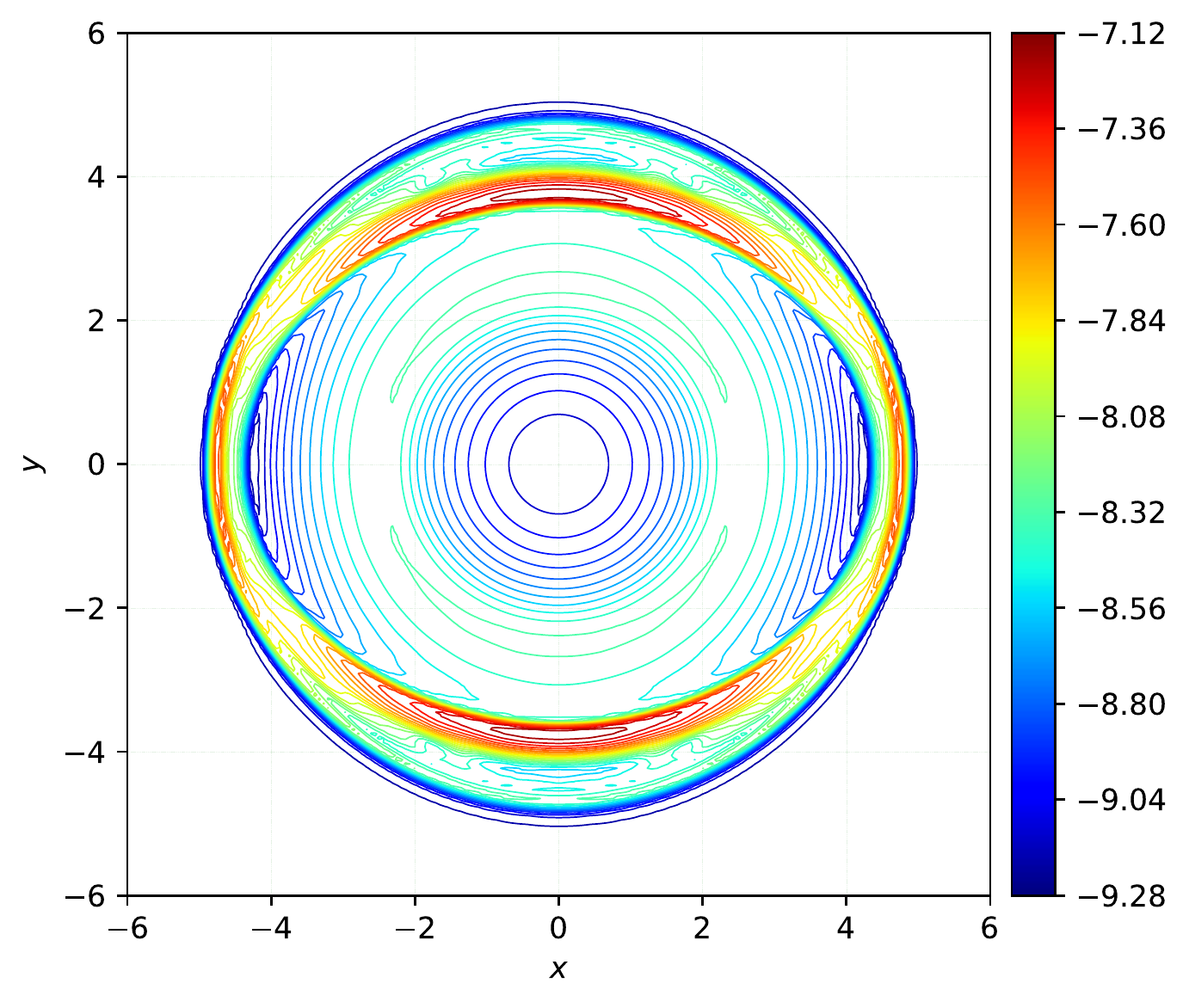}
			\label{fig:blast_o4_imp_0p1_logrho}}
		\subfigure[Plot of $\log(p_i+p_e)$.]{
			\includegraphics[width=2.9in, height=2.5in]{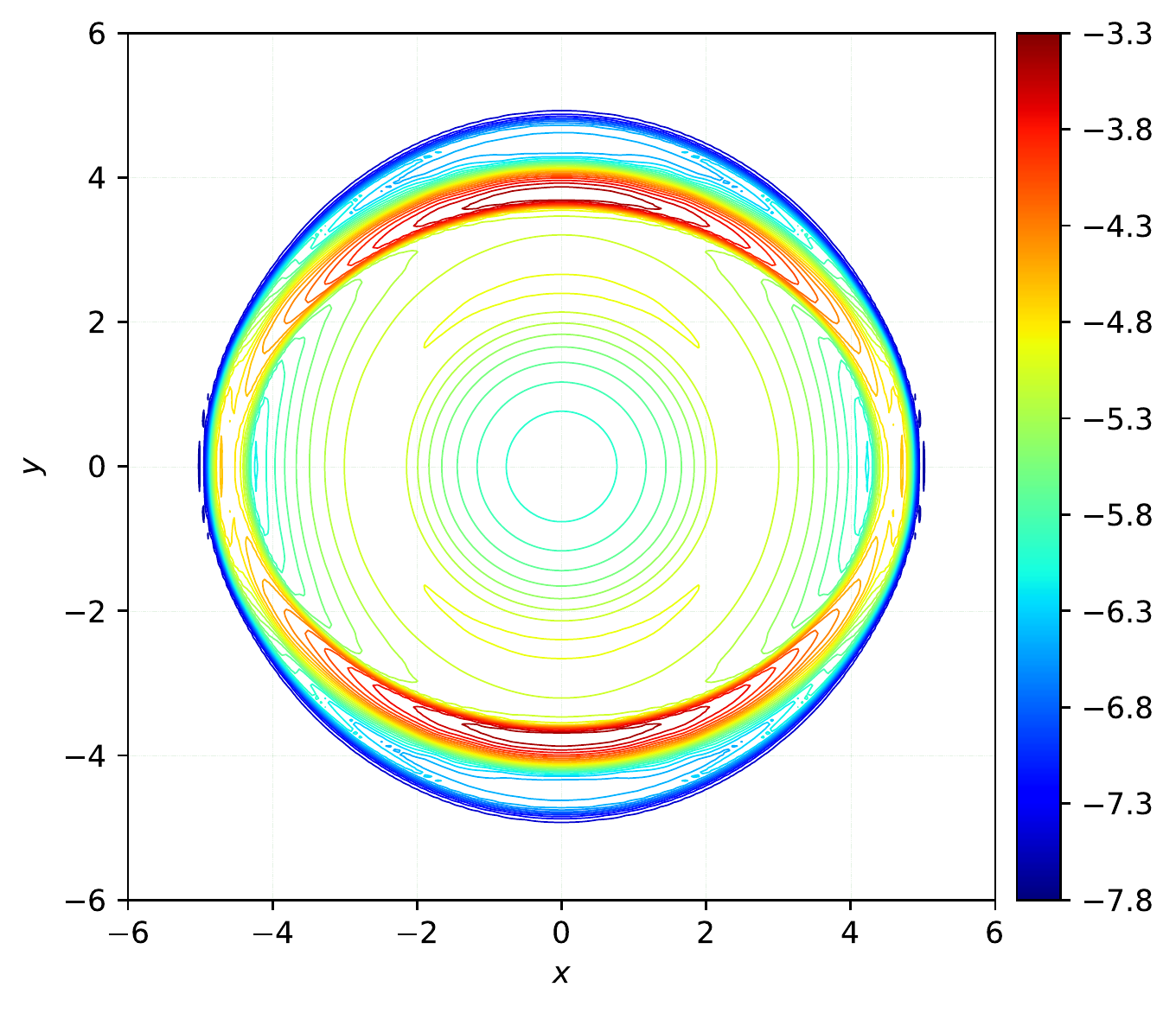}
			\label{fig:blast_o4_imp_0p1_logp}}
		\subfigure[Plot of ion Lorentz factor $\Gamma_i$.]{
			\includegraphics[width=2.9in, height=2.5in]{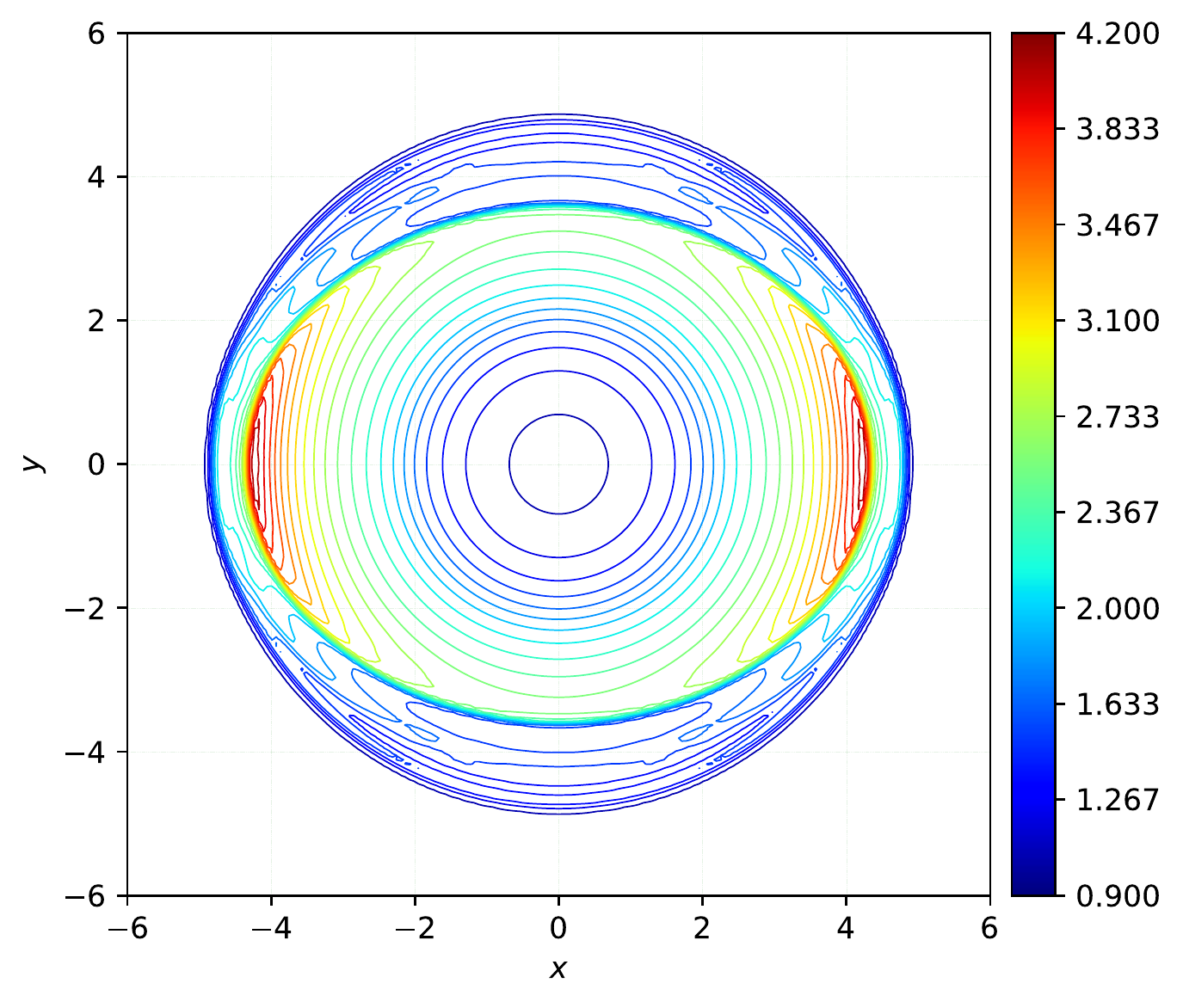}
			\label{fig:blast_o4_imp_0p1_lorentz}}
		\subfigure[Plot of magnitude of the Magnetic field $\dfrac{|\mathbf{B}|^2}{2}$.]{
			\includegraphics[width=2.9in, height=2.5in]{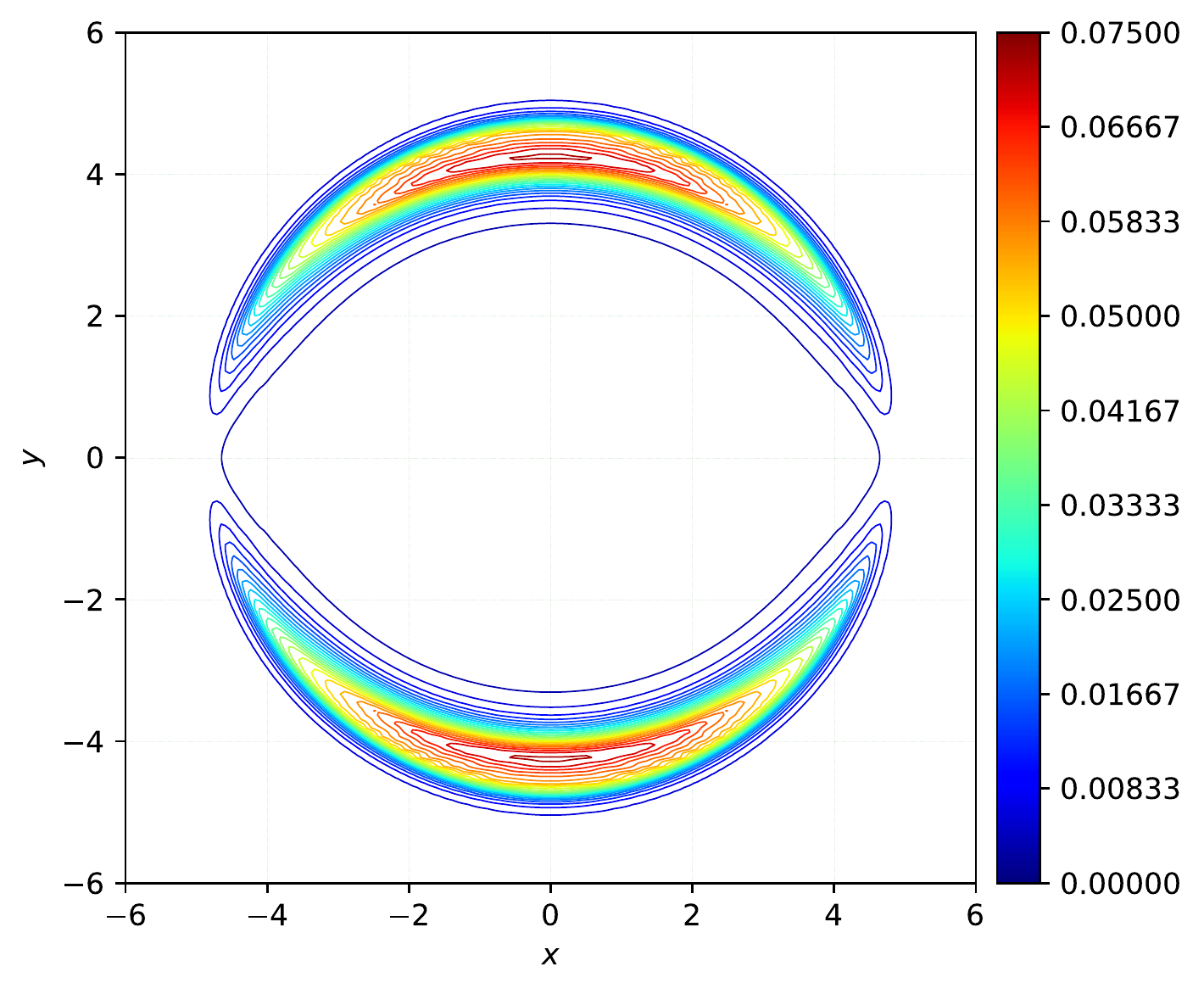}
			\label{fig:blast_o4_imp_0p1_magBby2}}
		\caption{\nameref{test:2d_blast}: Plot for the weakly magnetized medium $B_0=0.1$, using {\bf O4-ES-IMEX} scheme with $200\times 200$ cells. We have plotted 30 contours for each variable.}
		\label{fig:blast_o4_0_1}
	\end{center}
\end{figure}

\begin{figure}[!htbp]
	\begin{center}
		\subfigure[Plot of $\log(\rho_i+\rho_e)$.]{
			\includegraphics[width=2.9in, height=2.5in]{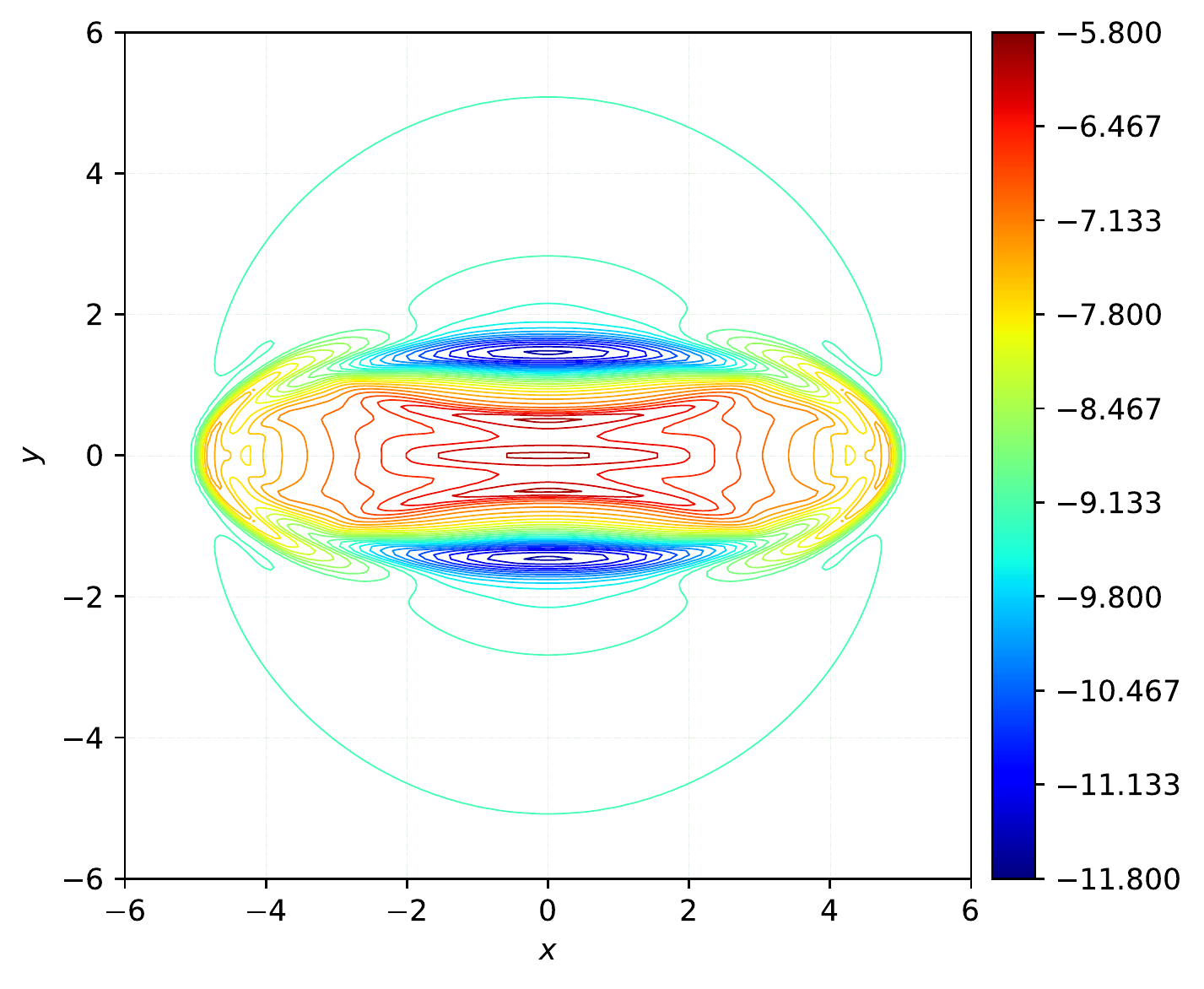}
			\label{fig:blast_o3_imp_1p0_logrho}}
		\subfigure[Plot of $\log(p_i+p_e)$.]{
			\includegraphics[width=2.9in, height=2.5in]{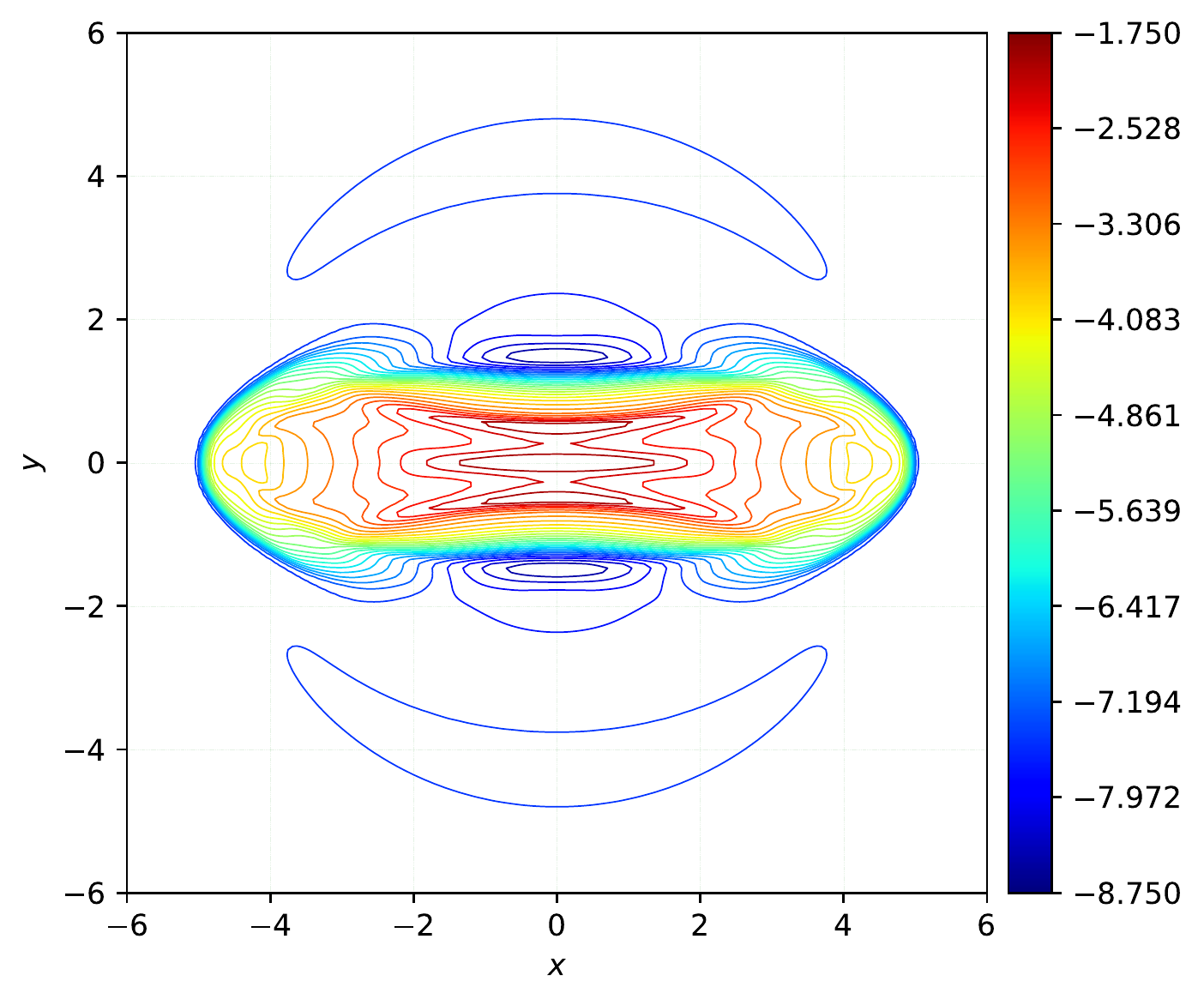}
			\label{fig:blast_o3_imp_1p0_logp}}
		\subfigure[Plot of ion Lorentz factor $\Gamma_i$.]{
			\includegraphics[width=2.9in, height=2.5in]{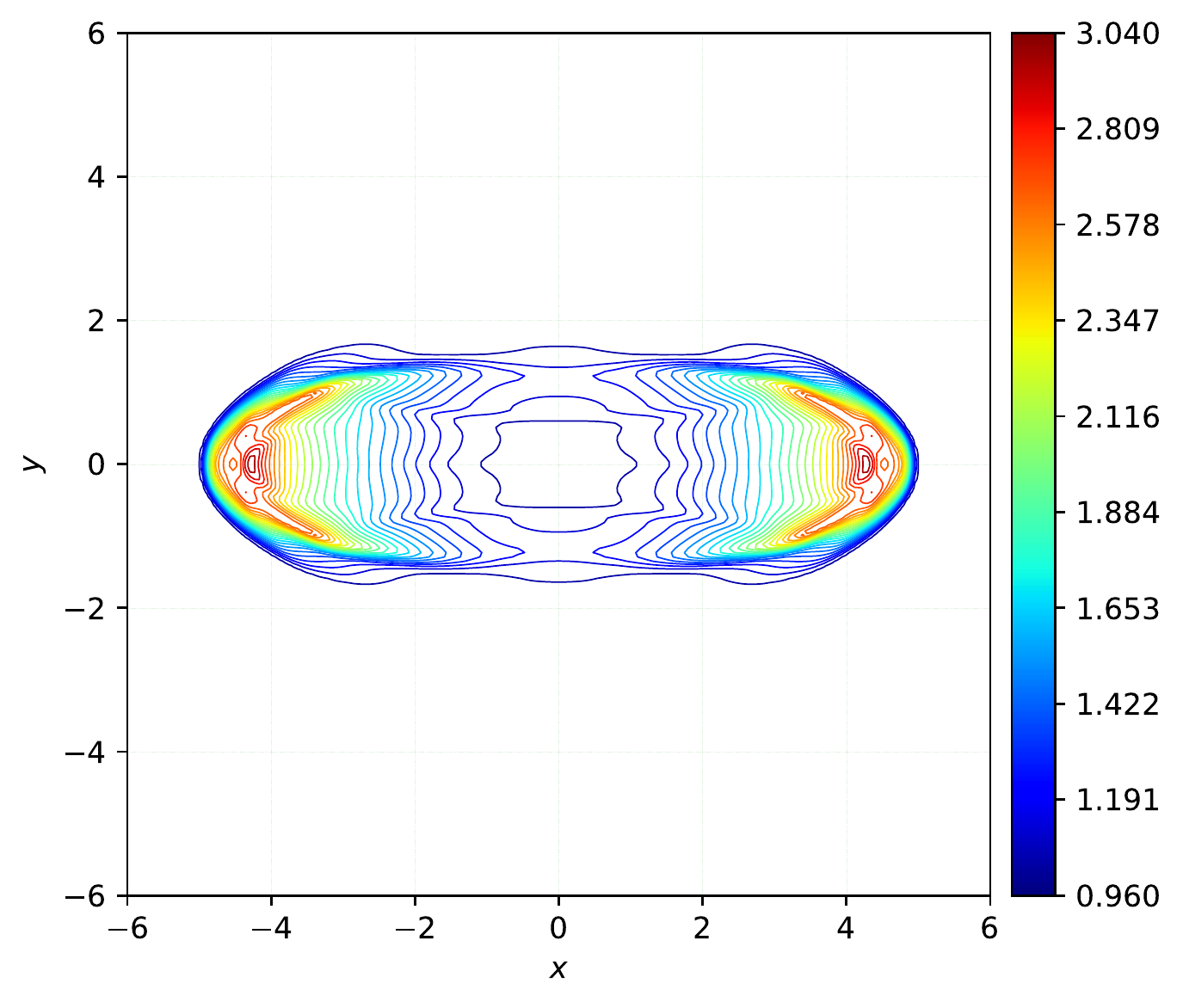}
			\label{fig:blast_o3_imp_1p0_lorentz}}
		\subfigure[Plot of magnitude of the Magnetic field $\dfrac{|\mathbf{B}|^2}{2}$.]{
			\includegraphics[width=2.9in, height=2.5in]{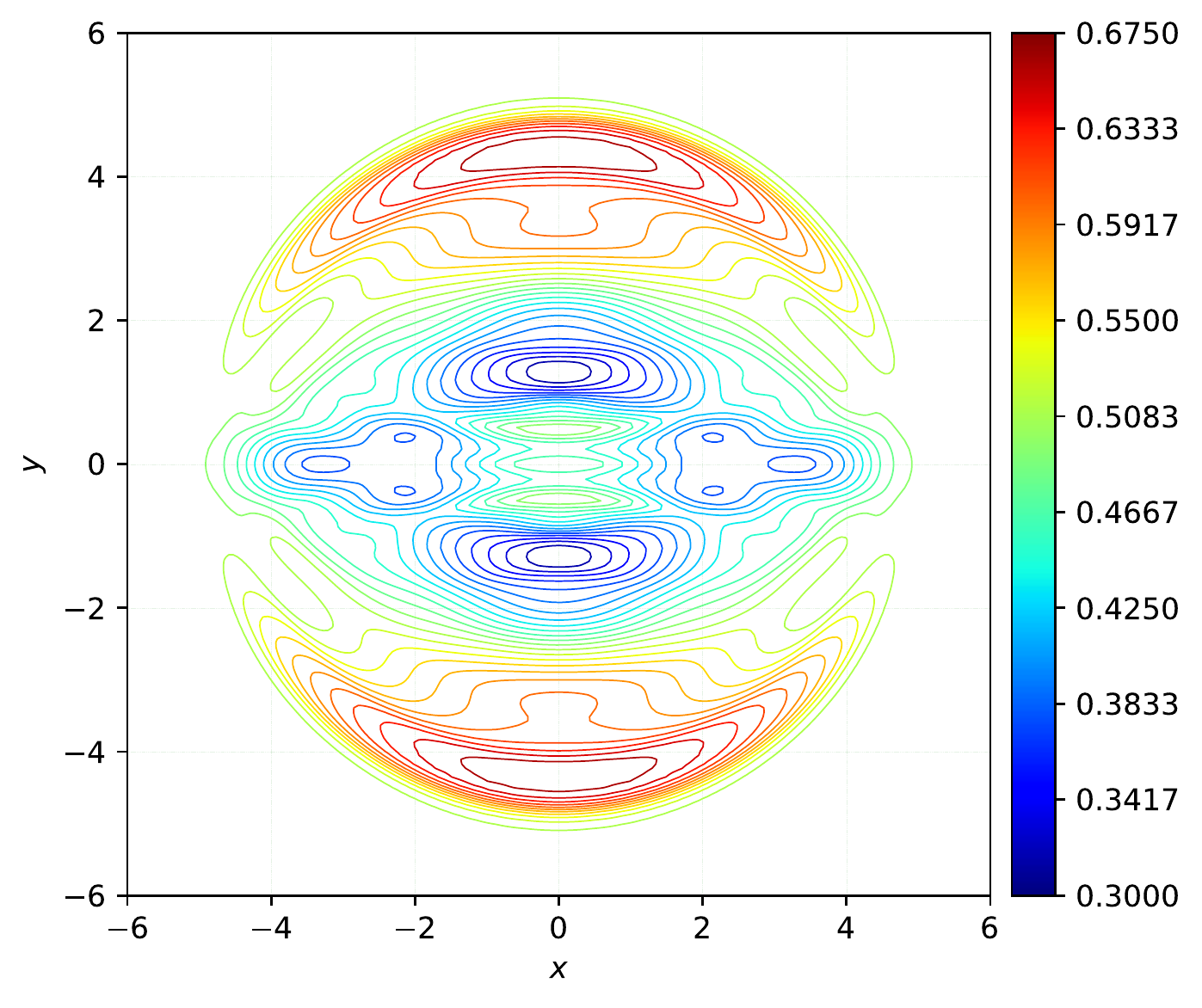}
			\label{fig:blast_o3_imp_1p0_magBby2}}
		\caption{\nameref{test:2d_blast}: Plot for the strongly magnetized medium $B_0=1.0$, using {\bf O3-ES-IMEX} scheme with $200\times 200$ cells. We have plotted 30 contours for each variable.}
		\label{fig:blast_o3_1_0}
	\end{center}
\end{figure}
\begin{figure}[!htbp]
	\begin{center}
		\subfigure[Plot of $\log(\rho_i+\rho_e)$.]{
			\includegraphics[width=2.9in, height=2.5in]{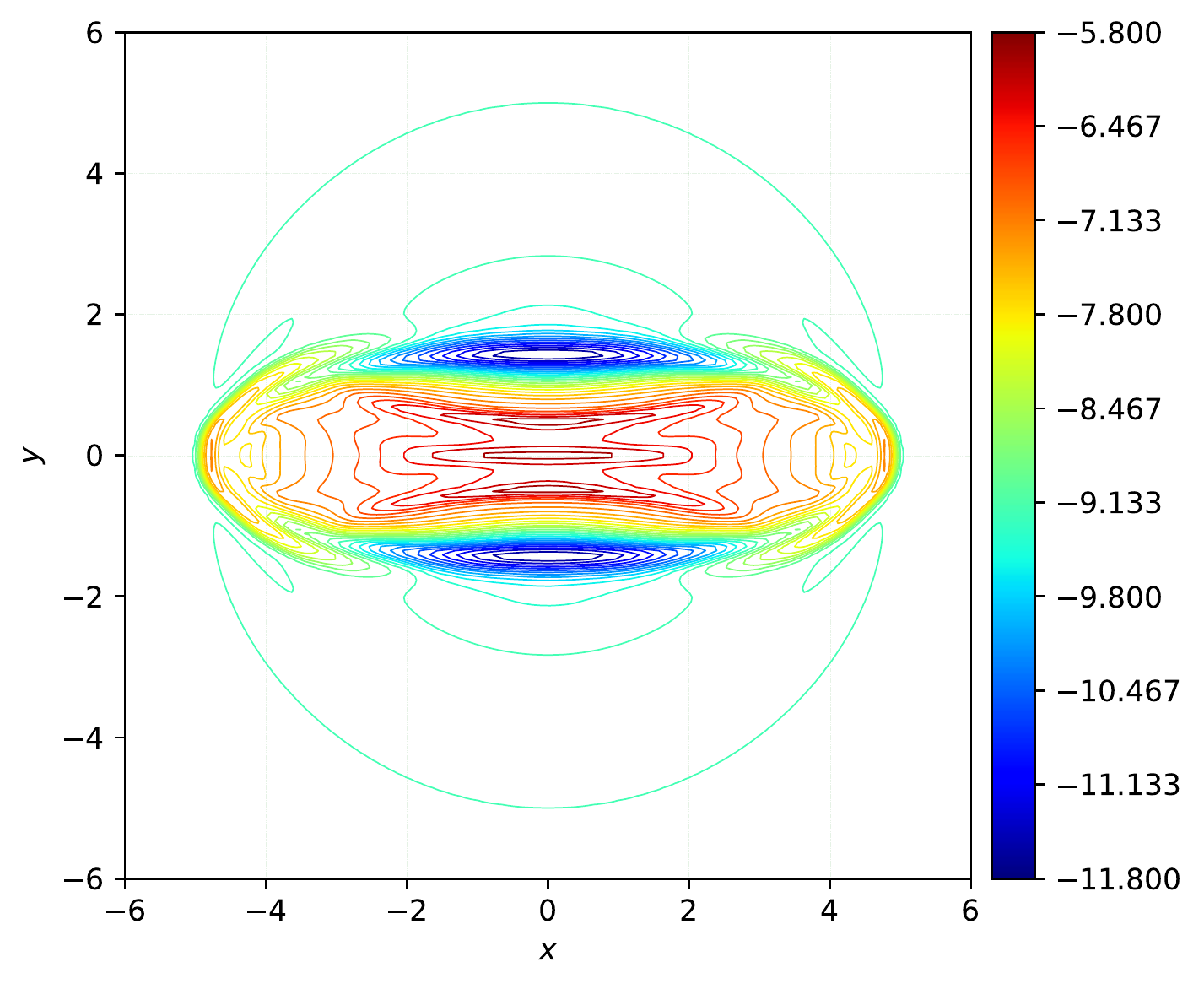}
			\label{fig:blast_o4_imp_1p0_logrho}}
		\subfigure[Plot of $\log(p_i+p_e)$.]{
			\includegraphics[width=2.9in, height=2.5in]{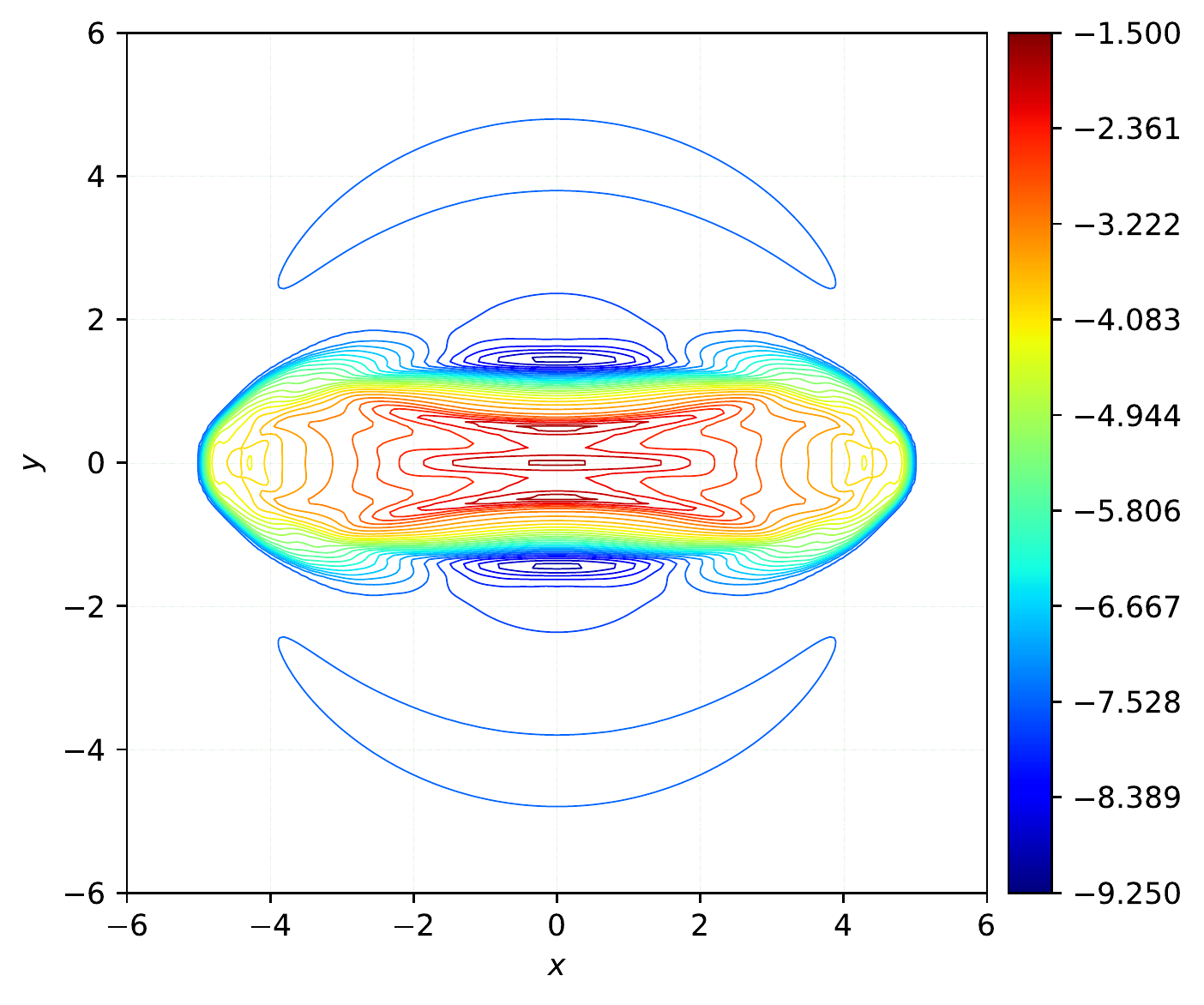}
			\label{fig:blast_o4_imp_1p0_logp}}
		\subfigure[Plot of ion Lorentz factor $\Gamma_i$.]{
			\includegraphics[width=2.9in, height=2.5in]{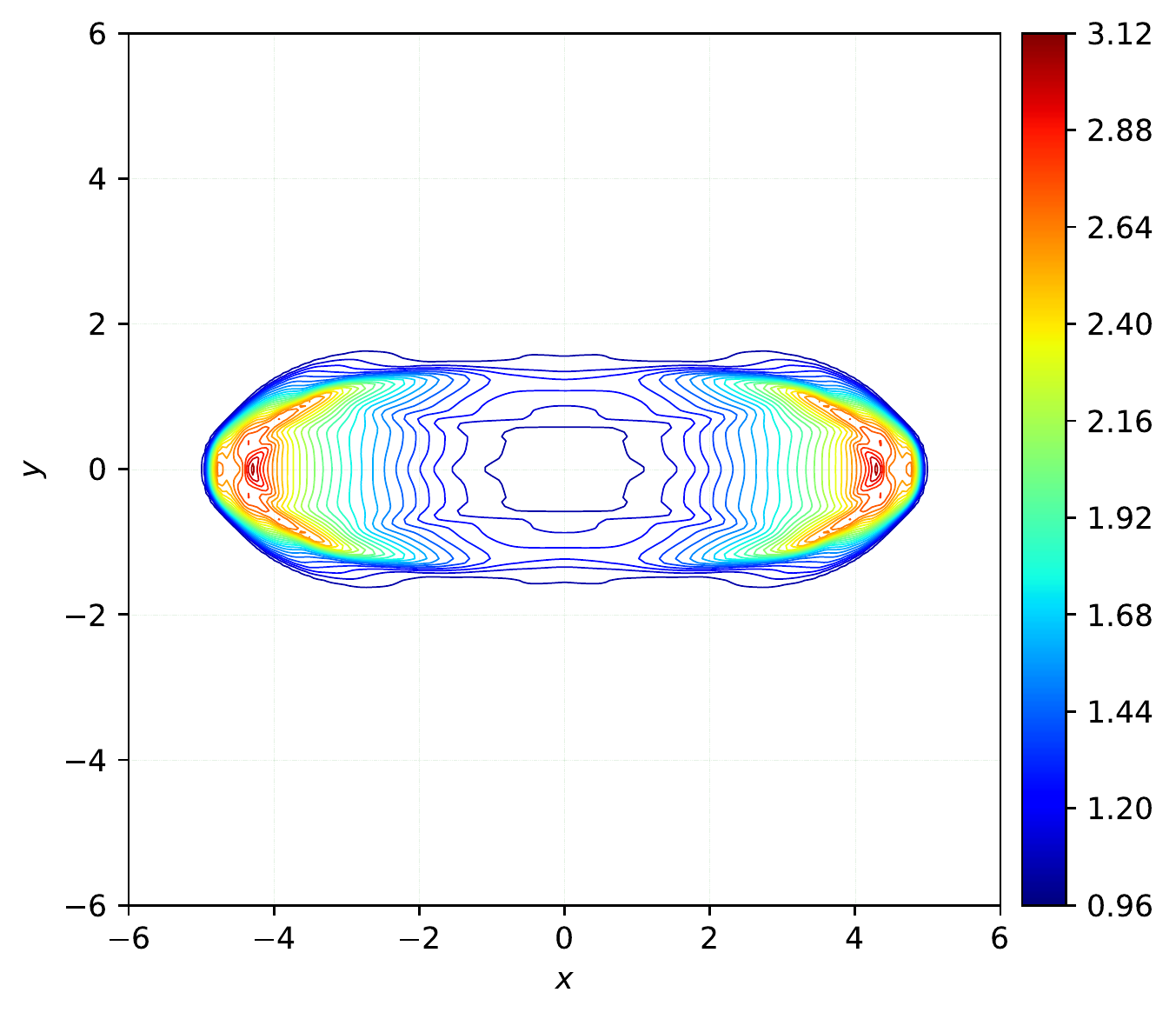}
			\label{fig:blast_o4_imp_1p0_lorentz}}
		\subfigure[Plot magnitude of the Magnetic field $\dfrac{|\mathbf{B}|^2}{2}$.]{
			\includegraphics[width=2.9in, height=2.5in]{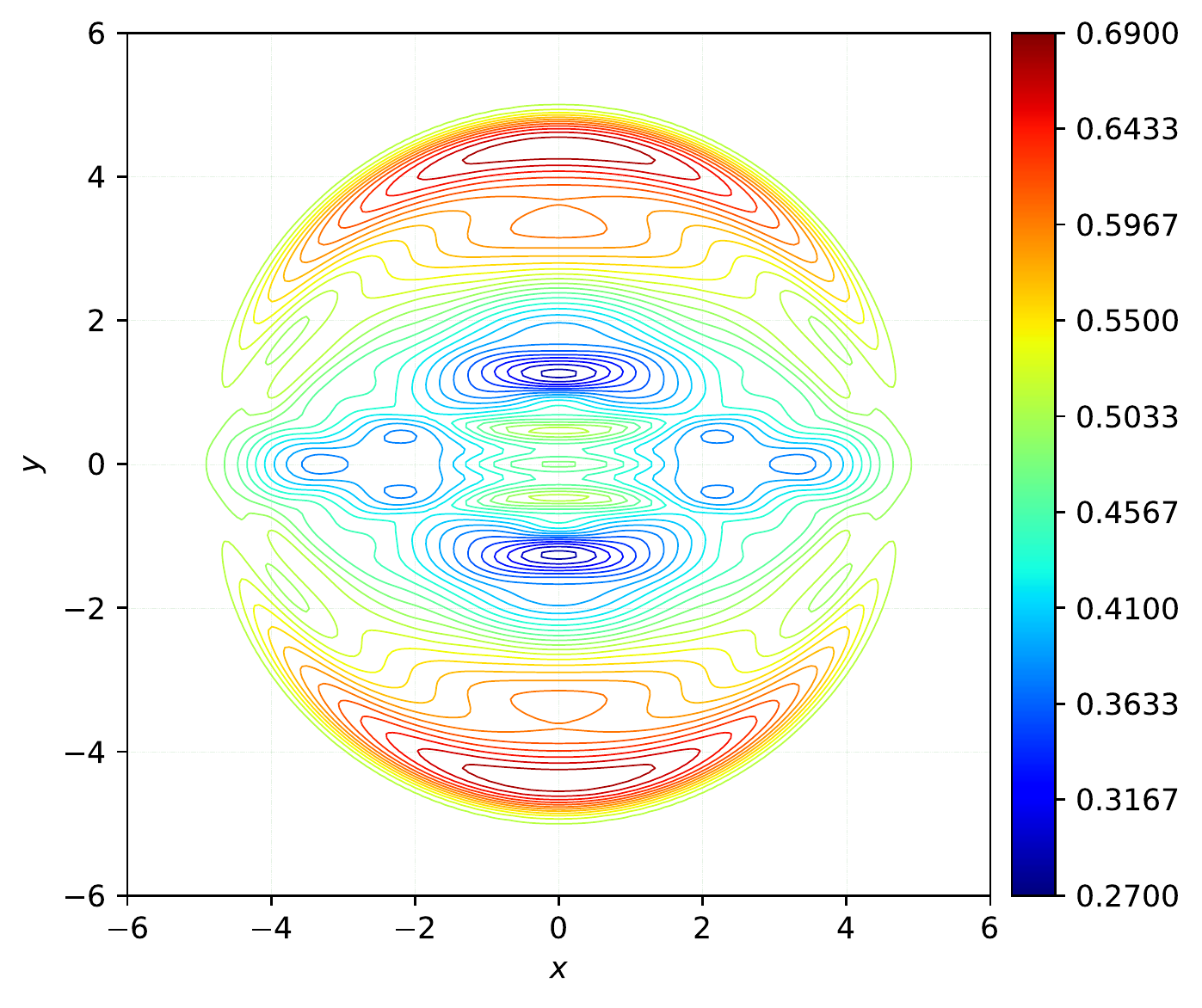}
			\label{fig:blast_o4_imp_1p0_magBby2}}
		\caption{\nameref{test:2d_blast}: Plot for the strongly magnetized medium $B_0=1.0$, using {\bf O4-ES-IMEX} scheme with $200\times 200$ cells. We have plotted 30 contours for each variable.}
		\label{fig:blast_o4_1_0}
	\end{center}
\end{figure}
\begin{figure}[!htbp]
	\begin{center}
		\subfigure[Total entropy fluid decay for weakly magnetized medium, $B_0=0.1$.]{
			\includegraphics[width=2.9in, height=2.2in]{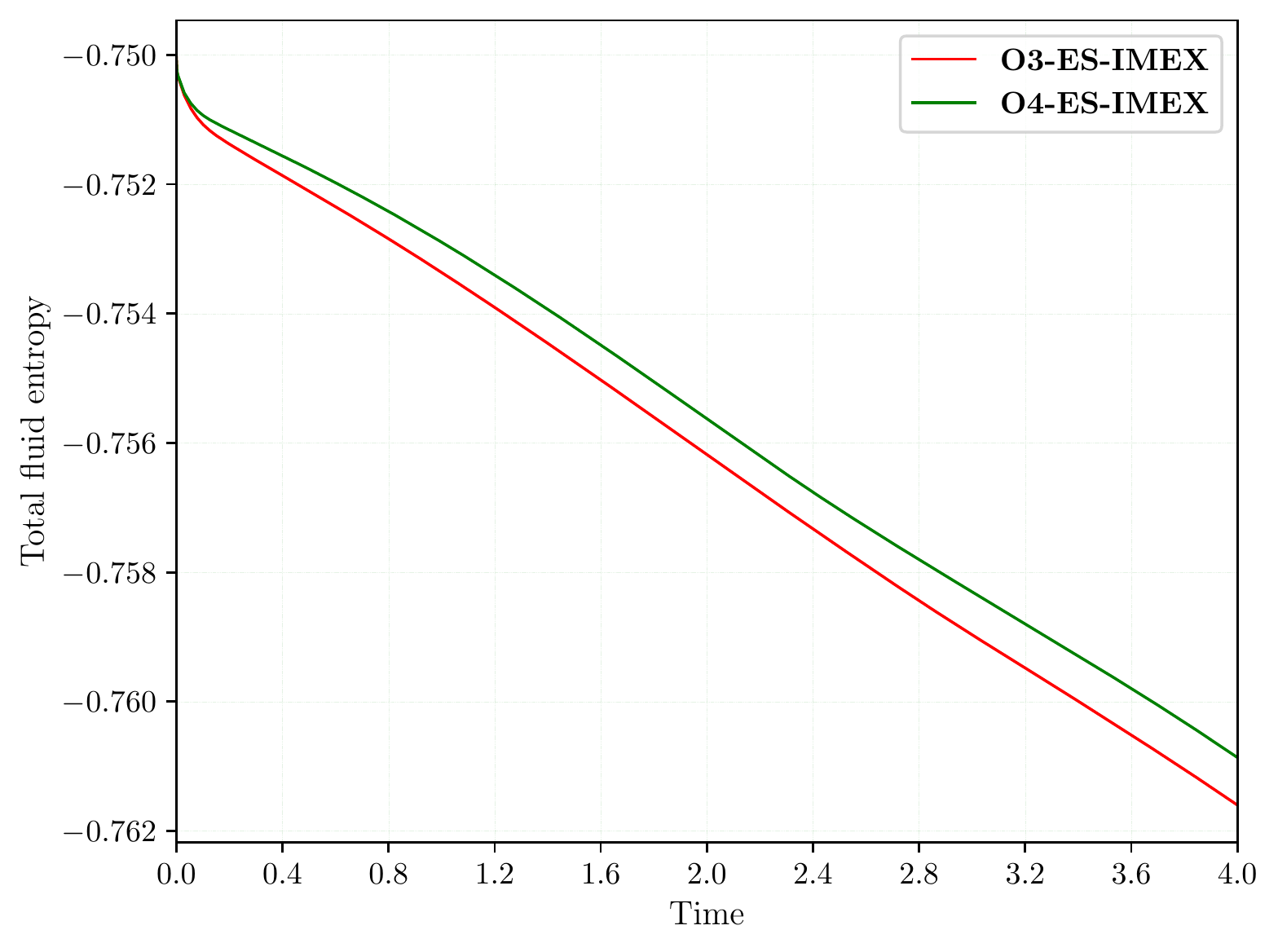}
			\label{fig:blast_0p1_entropy}}
		\subfigure[Total fluid entropy decay for strongly magnetized medium, $B_0=1.0$.]{
			\includegraphics[width=2.9in, height=2.2in]{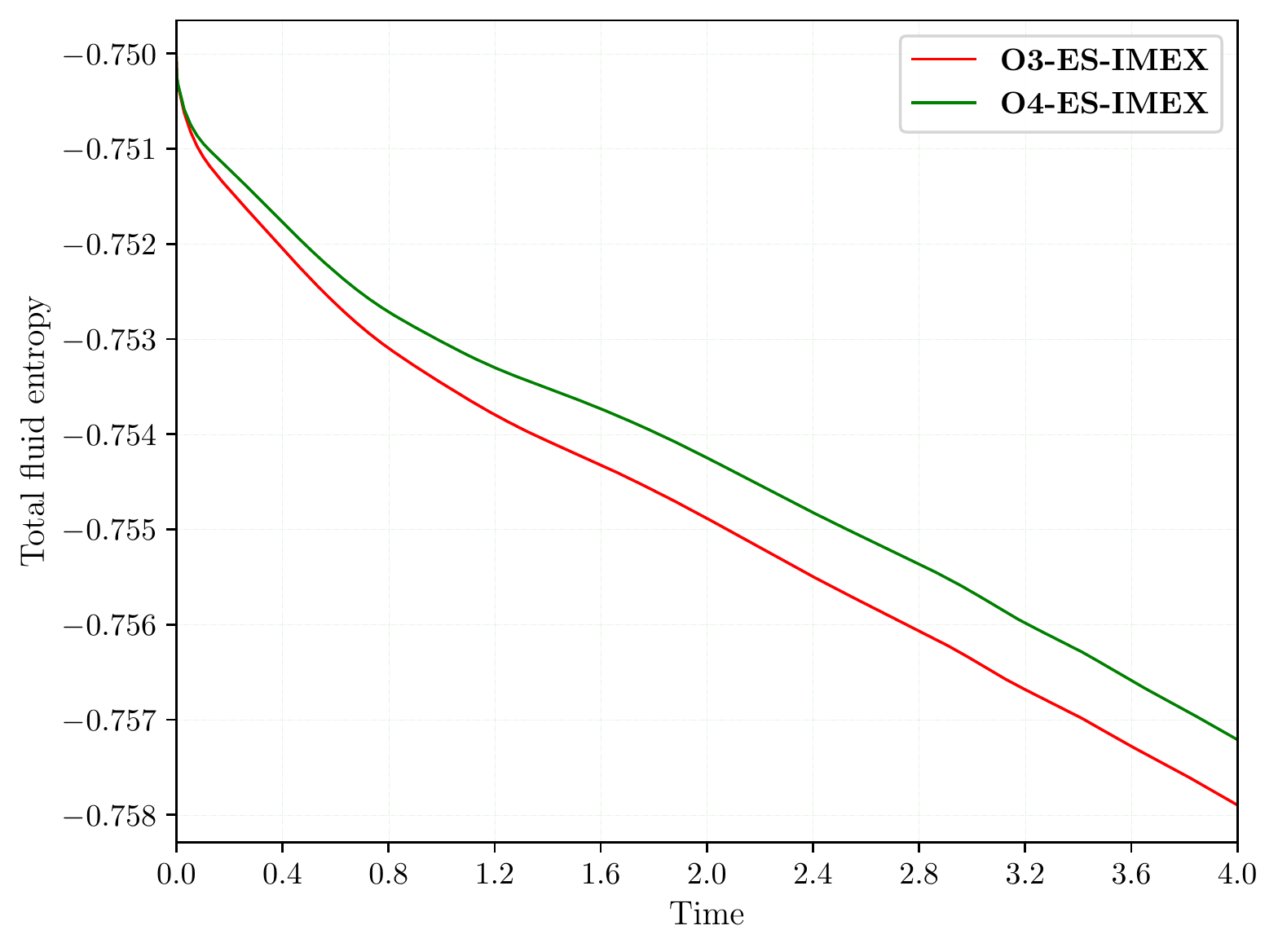}
			\label{fig:blast_1p0_entropy}}
		\caption{\nameref{test:2d_blast}: Total fluid entropy $\mathcal{U}_i+\mathcal{U}_e$ evolution of the schemes {\bf ES-O3-IMEX} and {\bf ES-O4-IMEX} for the weakly and strongly magnetized medium.}
		\label{fig:blast_entropy}
	\end{center}
\end{figure}

We consider the two-fluid relativistic extension of the strong cylindrical RMHD explosion described in~\cite{Komissarov1999}. We follow ~\cite{Amano2016,Balsara2016} to consider a computational domain of $[-6,6]\times[-6,6]$ with Neumann boundary conditions. To describe the initial fluid profile, consider $\rho_{in}=10^{-2}$, $p_{in}=1.0$, $\rho_{out}=10^{-4}$ and $p_{out}=5 \times 10^{-4}$. Also, let us denote the radial distance $r=\sqrt{x^2+y^2}$ from the center of the computational domain. Inside the disc of radius 0.8, i.e., $r<0.8$, we define the densities $\rho_i=\rho_e=0.5 \times \rho_{in}$ and the pressures $p_i=p_e=0.5 \times p_{in}$. Also, outside the disc of radius 1.0, i.e., $r>1$, we define the densities $\rho_i=\rho_e=0.5 \times \rho_{out}$ and the pressures $p_i=p_e=0.5 \times p_{out}$. In the range $0.8 \le r \le 1.0$, the densities and pressures are defined using a linear profile, such that the densities and the pressures decreases with increase in the radius $r$ and at the boundaries $r=0.8$ and $r=1.0$ they matches the prescribed constant states. The magnetic field has been initialized in the $x$-direction only, as $B_x=B_0$. We set all remaining variables to zero. The charge to mass ratios are taken to be $r_i=-r_e=10^3$ and  $\gamma_i=\gamma_e=4/3$.

The solutions are computed on a mesh of $200 \times 200$ cells using the {\bf O3-ES-IMEX} and {\bf O4-ES-IMEX} schemes till final time $t=4.0$. We consider two values of parameter $B_0$, such that $B_0=0.1$ denoting the case of the weakly magnetized medium,  and $B_0=1.0$ denoting the  case of the strongly magnetized medium. In Figures~\eqref{fig:blast_o3_0_1} and \eqref{fig:blast_o4_0_1}, we plot the results for weakly magnetized medium for  {\bf O3-ES-IMEX} and {\bf O4-ES-IMEX} schemes, respectively. The figures show the log to the base 10 of total density and total pressure; we also plot the ion Lorentz factor and magnitude of the magnetic field. It can be observed that both the schemes produce stable results and are able to capture the waves accurately. Furthermore, the results are comparable to those in \cite{Amano2016} and \cite{Balsara2016}. 

Similarly, in the case of strongly magnetized medium we plot the results for {\bf O3-ES-IMEX} and {\bf O4-ES-IMEX} schemes in Figures \eqref{fig:blast_o3_1_0} and \eqref{fig:blast_o4_1_0}, respectively. We observe a substantial change in the solution profile. Again both the schemes are able to capture the blast waves accurately.

In Figure \eqref{fig:blast_entropy}, we plot the evolution of the total fluid entropy $\mathcal{U}_i+\mathcal{U}_e$ for the {\bf O3-ES-IMEX} and {\bf O4-ES-IMEX} schemes in weakly and strongly magnetized mediums. As the solutions contain many discontinuities, we see a strong decay in the entropies. Furthermore, {\bf O4-ES-IMEX} scheme decays less entropy than the {\bf O3-ES-IMEX}  scheme.

\reva{
	\begin{figure}[!htbp]
		\begin{center}
			\subfigure[Plot of $|2\Delta x (\nabla \cdot \mathbf{B})_{i,j}|$ for $B_0=0.1$ using the {\bf O3-ES-IMEX} scheme at time $t=4.0$.]{
				\includegraphics[width=2.9in, height=2.5in]{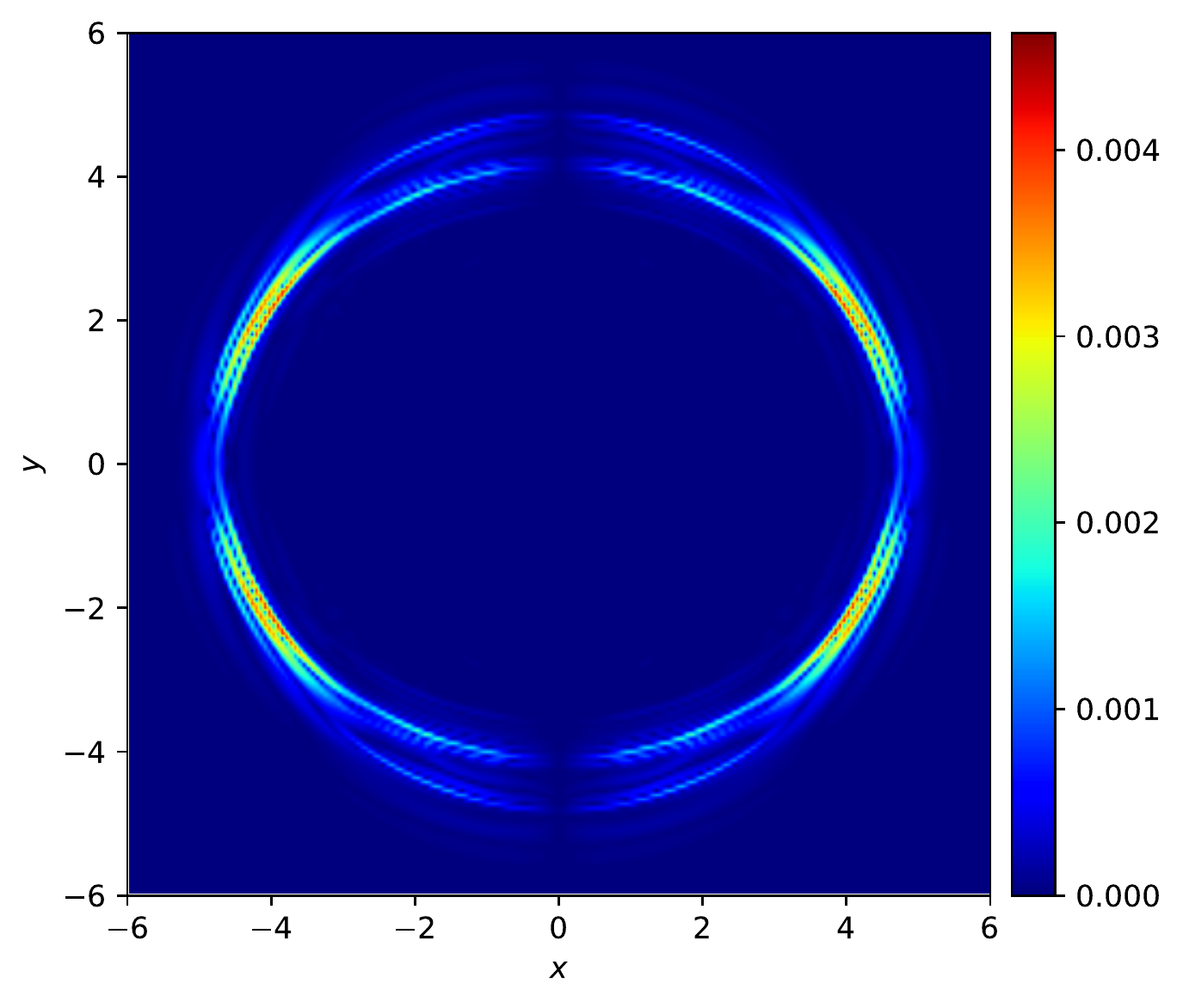}
				\label{fig:blast_0p1_2Dnorm_o3}}
			\quad
			\subfigure[Plot of $|2\Delta x (\nabla \cdot \mathbf{B})_{i,j}|$ for $B_0=0.1$ using the {\bf O4-ES-IMEX} scheme at time $t=4.0$.]{
				\includegraphics[width=2.9in, height=2.5in]{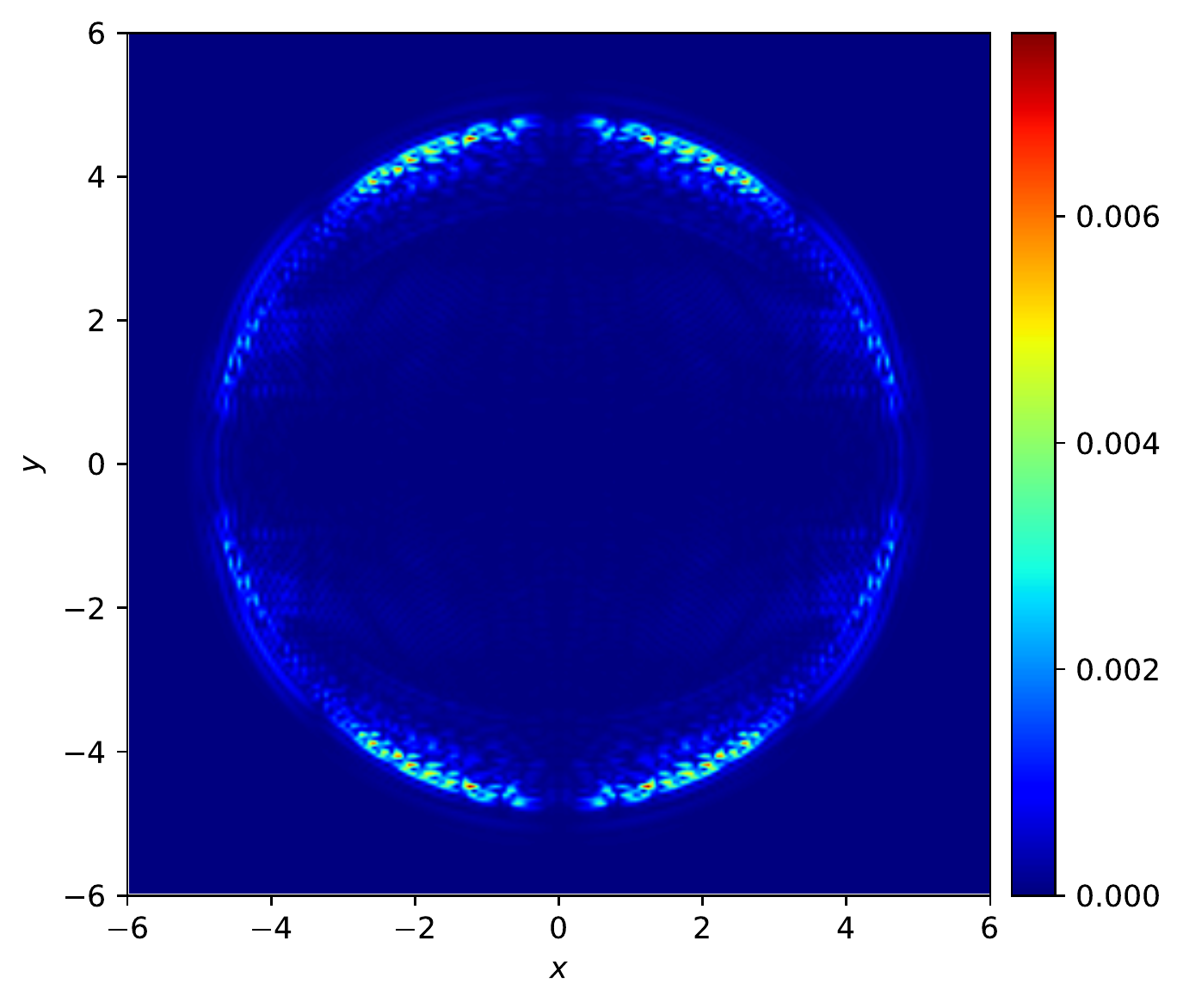}
				\label{fig:blast_0p1_2Dnorm_o4}}
			\quad
			\subfigure[Time evolution of $L^1$-norms of divergence of $\mathbf{B}$ for $B_0=0.1$ till time $t=4.0$ using {\bf O3-ES-IMEX} and {\bf O4-ES-IMEX} schemes.]{
				\includegraphics[width=2.9in, height=2.0in]{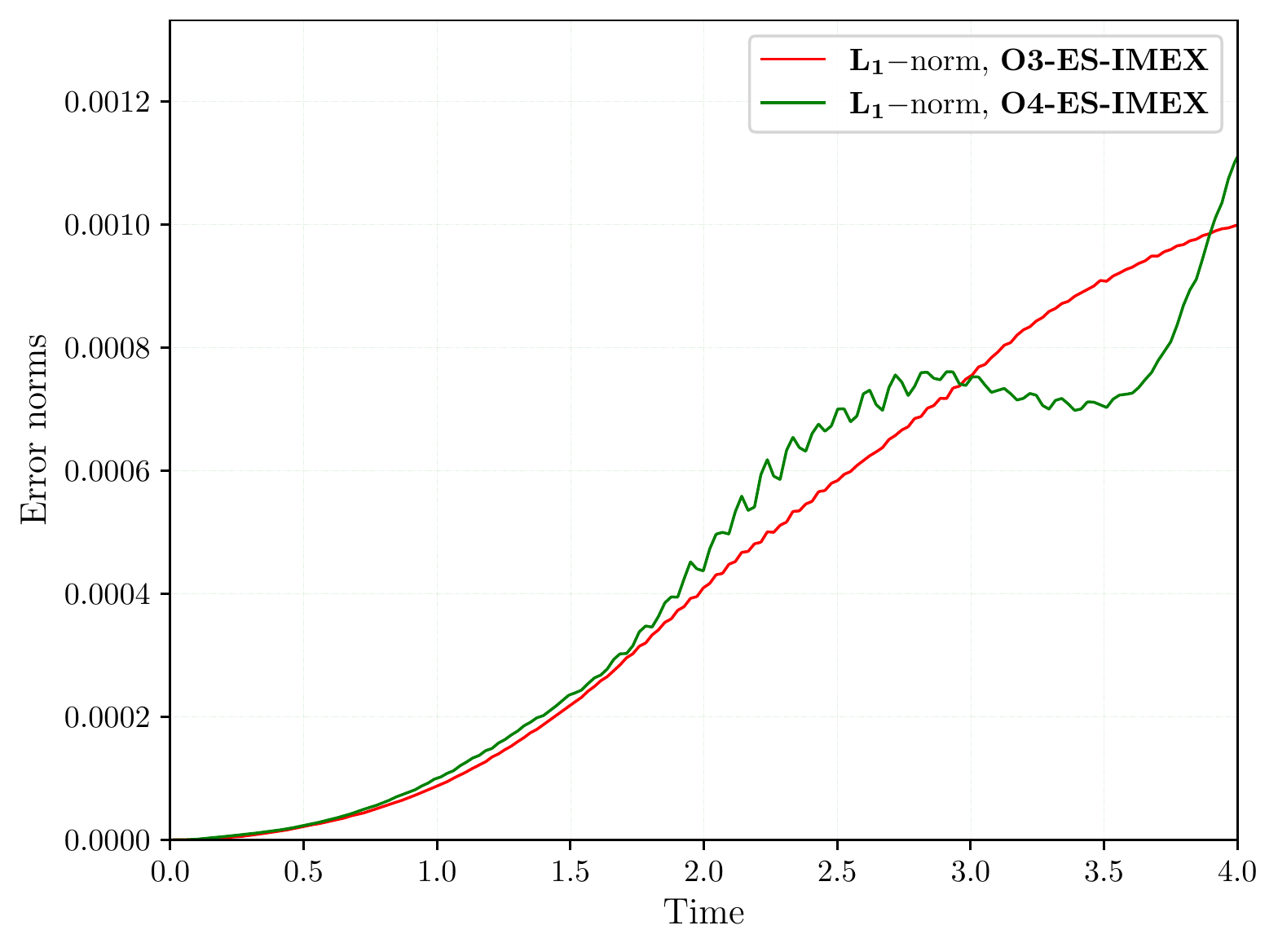}
				\label{fig:blast_0p1_l1norm}}
			\quad
			\subfigure[Plot of $|2\Delta x (\nabla \cdot \mathbf{B})_{i,j}|$ for $B_0=1.0$ using the {\bf O3-ES-IMEX} scheme at time $t=4.0$.]{
				\includegraphics[width=2.9in, height=2.5in]{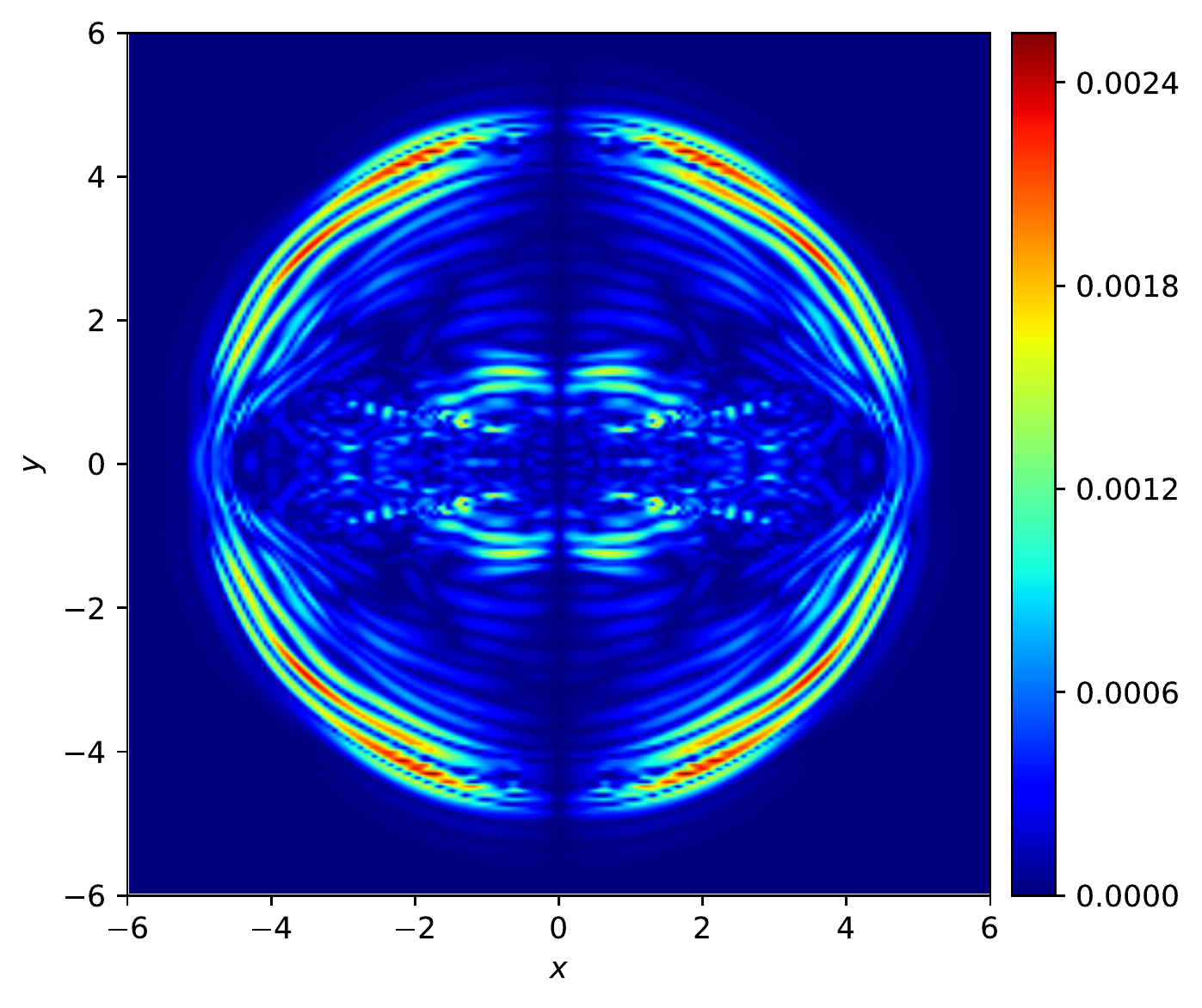}
				\label{fig:blast_1p0_2Dnorm_o3}}
			\quad
			\subfigure[Plot of $|2\Delta x (\nabla \cdot \mathbf{B})_{i,j}|$ for $B_0=1.0$ using the {\bf O4-ES-IMEX} scheme at time $t=4.0$.]{
				\includegraphics[width=2.9in, height=2.5in]{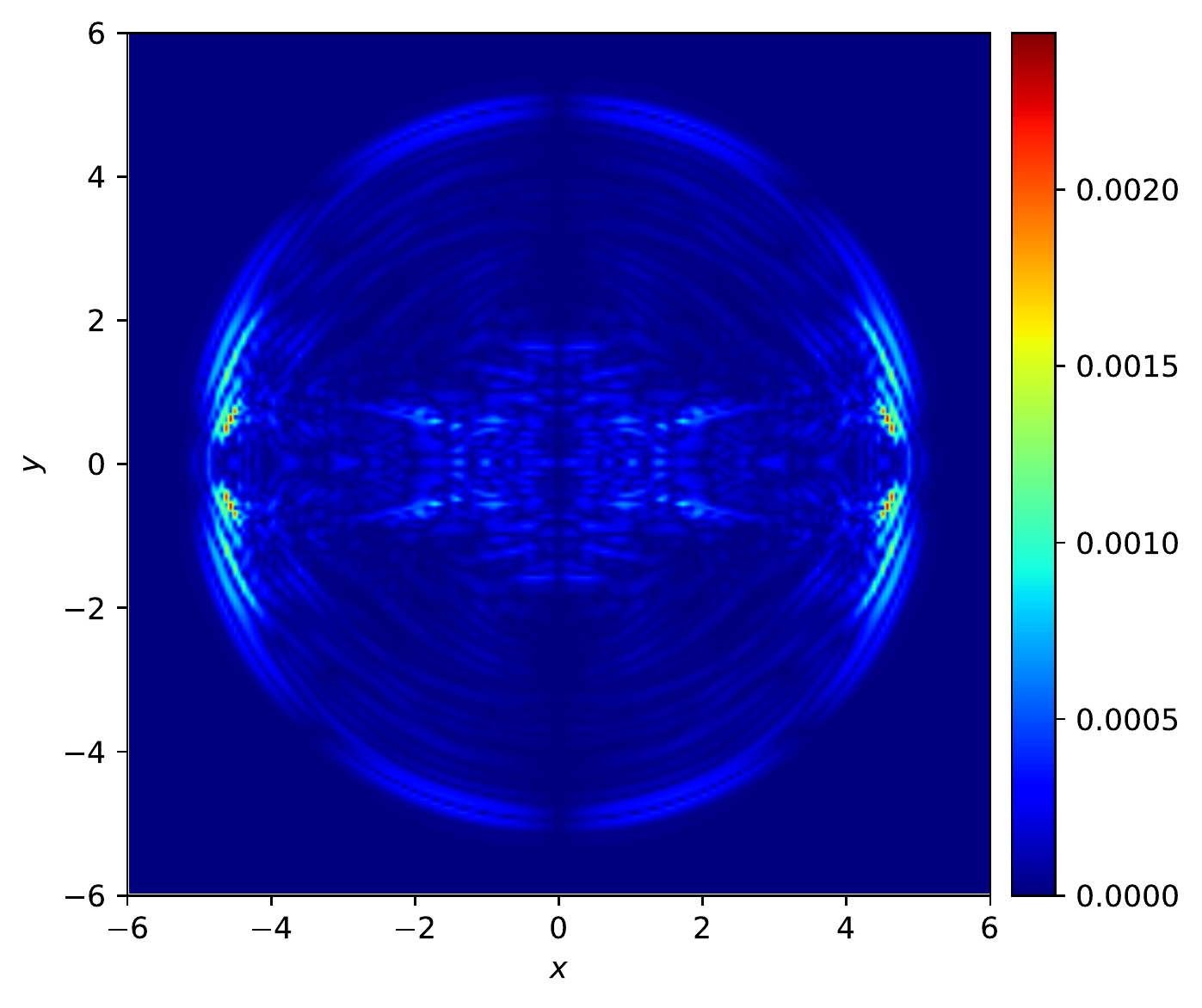}
				\label{fig:blast_1p0_2Dnorm_o4}}
			\quad
			\subfigure[Time evolution of $L^1$-norms of divergence of $\mathbf{B}$ for $B_0=1.0$ till time $t=4.0$ using {\bf O3-ES-IMEX} and {\bf O4-ES-IMEX} schemes.]{
				\includegraphics[width=2.9in, height=2.0in]{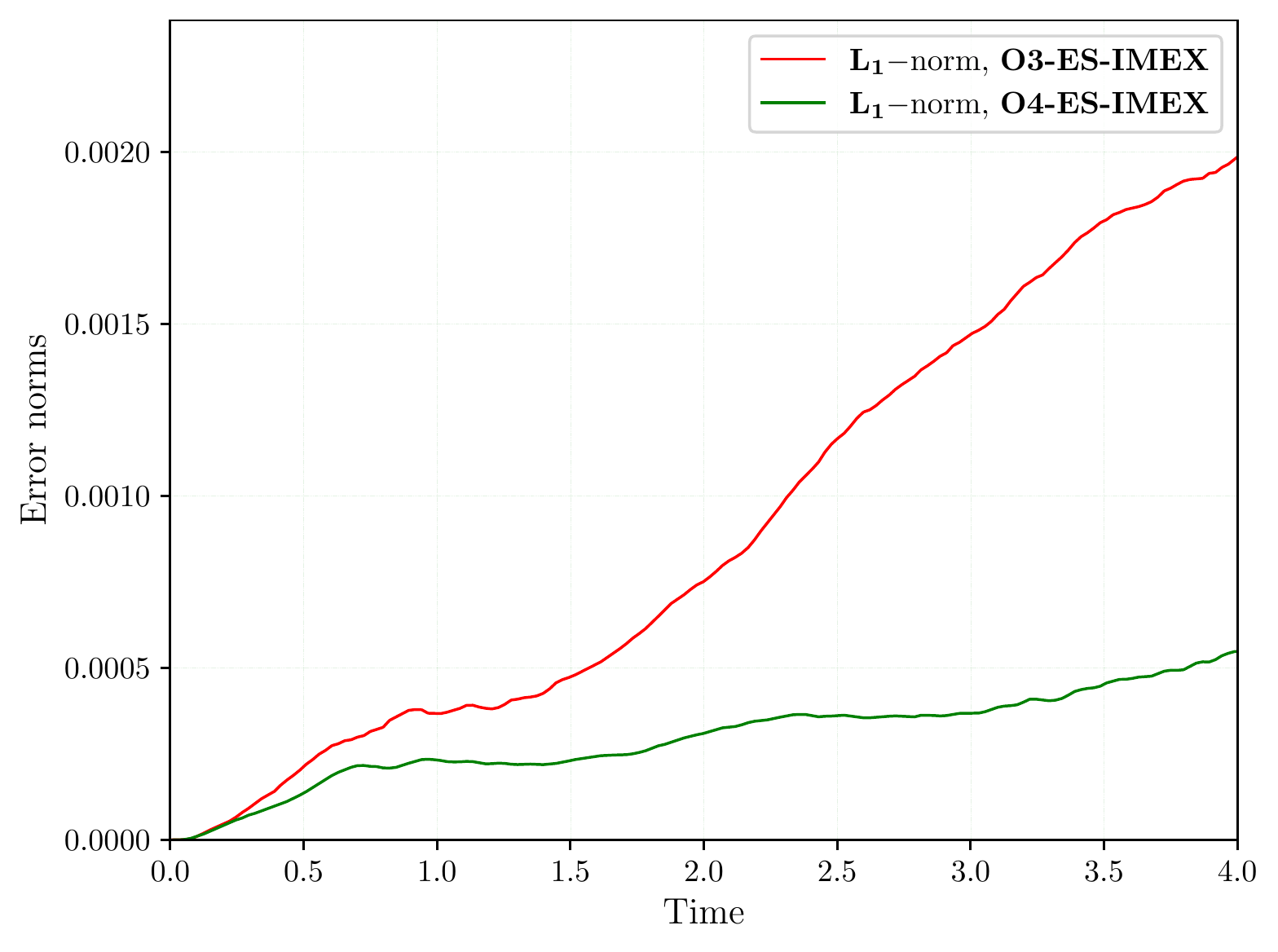}
				\label{fig:blast_1p0_l1norm}}
			\caption{\nameref{test:2d_blast}: \reva{Plots of the ($|2\Delta x (\nabla \cdot \mathbf{B})_{i,j}|$) and time evolution of $L^1$-norms of divergence of $\mathbf{B}$ for $B_0=0.1$ and $B_0=1.0$ and schemes {\bf O3-ES-IMEX} and {\bf O4-ES-IMEX} using $200\times 200$ cells.}}
			\label{fig:norms_blast}
		\end{center}
	\end{figure}
	
	In Fig.~\ref{fig:norms_blast} we have presented results for errors in the divergence of the magnetic field. In Figs.~\ref{fig:blast_0p1_2Dnorm_o3} and \ref{fig:blast_0p1_2Dnorm_o4}, we plot the absolute value of undivided divergence of the magnetic field for $B_0=0.1$ using {\bf O3-ES-IMEX} and {\bf O4-ES-IMEX} schemes, respectively \cite{Balsara2016}. For both schemes, we note that the errors are well-controlled and concentrated where we have steep gradients in the flows. Furthermore, the errors for {\bf O3-ES-IMEX} are more spread out than the {\bf O4-ES-IMEX} scheme. In Fig.~\ref{fig:blast_0p1_l1norm}, we have plotted the time evolution of the $L^1$-norm of the divergence for $B_0=0.1$ using both schemes. We note that the errors stay relatively small for both the schemes.
	
	In Figs.~\ref{fig:blast_1p0_2Dnorm_o3} and \ref{fig:blast_1p0_2Dnorm_o4}, we plot the absolute value of undivided divergence of magnetic field for $B_0=1.0$ using {\bf O3-ES-IMEX} and {\bf O4-ES-IMEX} schemes, respectively \cite{Balsara2016}. We observe that the errors for the {\bf O3-ES-IMEX} schemes are much more spread out when compared with the {\bf O4-ES-IMEX} scheme. This can also be observed from Fig.~\ref{fig:blast_1p0_l1norm}, where we plot the time evolution of the ${L^1}$-norm of divergence of the magnetic field. We note that {\bf O4-ES-IMEX} scheme outperforms {\bf O3-ES-IMEX} scheme. 
}

\subsubsection{Relativistic two-fluid GEM challenge problem} \label{test:2d_gem} 
In this test, we consider two-fluid relativistic Geospace Environment Modeling (GEM) magnetic reconnection problem from \cite{Amano2016}, which is an extension of the non-relativistic GEM magnetic reconnection problem given in \cite{Birn2001}. We consider the modified Equations \eqref{resist_momentum} and \eqref{resist_energy} to capture the resistive effects. The computational domain is $[-L_x/2,L_x/2] \times [-L_y/2,L_y/2]$ where $L_x=8\pi$ and $L_y=4\pi$, with periodic boundary conditions at $x=\pm L_x/2$ and conducting wall boundary at $y=\pm L_y/2$ boundary. The ion-electron mass ratio is taken to be $m_i/m_e = 25$ with $m_i=1$. Accordingly, we have $r_i=1.0$ and $r_e=-25.0.$ The unperturbed $x$-component of the magnetic field is taken to be $B_x(y) = B_0 \tan (y/d)$ with $B_0=1.0$, where $d=1.0$ is the thickness of the current sheet. On the other hand, unperturbed $y$- and $z$-components of the magnetic field are assumed to be zero. After the perturbation, the initial conditions become
\begin{align*}
	\begin{pmatrix}
		\rho_i  \\ 
		u_{z_i}  \\ 
		p_i  \\ \\ 
		\rho_e \\
		u_{z_e}  \\
		p_e \\ \\ 
		B_x  \\ 
		B_y
	\end{pmatrix} = 
	\begin{pmatrix}
		n  \\ 
		\frac{c}{2d}\frac{B_0 \text{sech}^2(y/d)}{n}  \\ 
		0.2 + \frac{B_0^2  \text{sech}^2(y/d)}{4} \frac{5}{24 \pi}  \\ \\ 
		\frac{m_e}{m_p} n  \\
		-(u_z)_i \\
		- p_i \\ \\ 
		B_0 \tan (y/d) - B_0 \psi_0 \frac{\pi}{L_y} \cos(\frac{\pi x}{L_x}) \sin(\frac{\pi y}{L_y}) 
		\\ 
		B_0 \psi_0 \frac{\pi}{L_x} \sin(\frac{\pi x}{L_x}) \cos(\frac{\pi y}{L_y})
	\end{pmatrix}
\end{align*}
where $n=\mathrm{sech}^2(y/d)+0.2$. All other variables are set to zero. We set resistivity constant as $\eta=0.01$, and adiabatic indices as $\gamma_i =\gamma_e= 4.0/3.0$.

Figures \eqref{fig:gem_o3_40} to \eqref{fig:gem_o4_80} show the total density, $z$-component of magnetic field, $x$-component of ion velocity and $x$-component of electron velocity on a mesh of $512\times 256$ cells. The plots also contain field line for $(B_x,B_y)$.  The solution is shown at time $t=40$ and $t=80$ using {\bf O3-ES-IMEX} and {\bf O4-ES-IMEX} schemes. We observe that solutions from both the schemes are similar and agree with the solutions presented in \cite{Amano2016}. 

Figure \eqref{fig:gem_recon} shows the time evolution of the reconnected magnetic flux, 
$$
\psi(t) = \dfrac{1}{2 B_0} \int_{-L_x/2}^{L_x/2}
| B_y(x,y=0,t) | dx.
$$
at various resolutions. We also compare it with the results presented in \cite{Amano2016}. It is observed that both schemes are able to capture the reconnected flux at various resolutions. Only for times greater than $t=80$, we see a difference in the reconnection flux for the different schemes, but at this time, the flow is already quite turbulent.

\begin{figure}[!htbp]
	\begin{center}
		\subfigure[Total Density ($\rho_i + \rho_e$)]{
			\includegraphics[width=2.9in, height=1.5in]{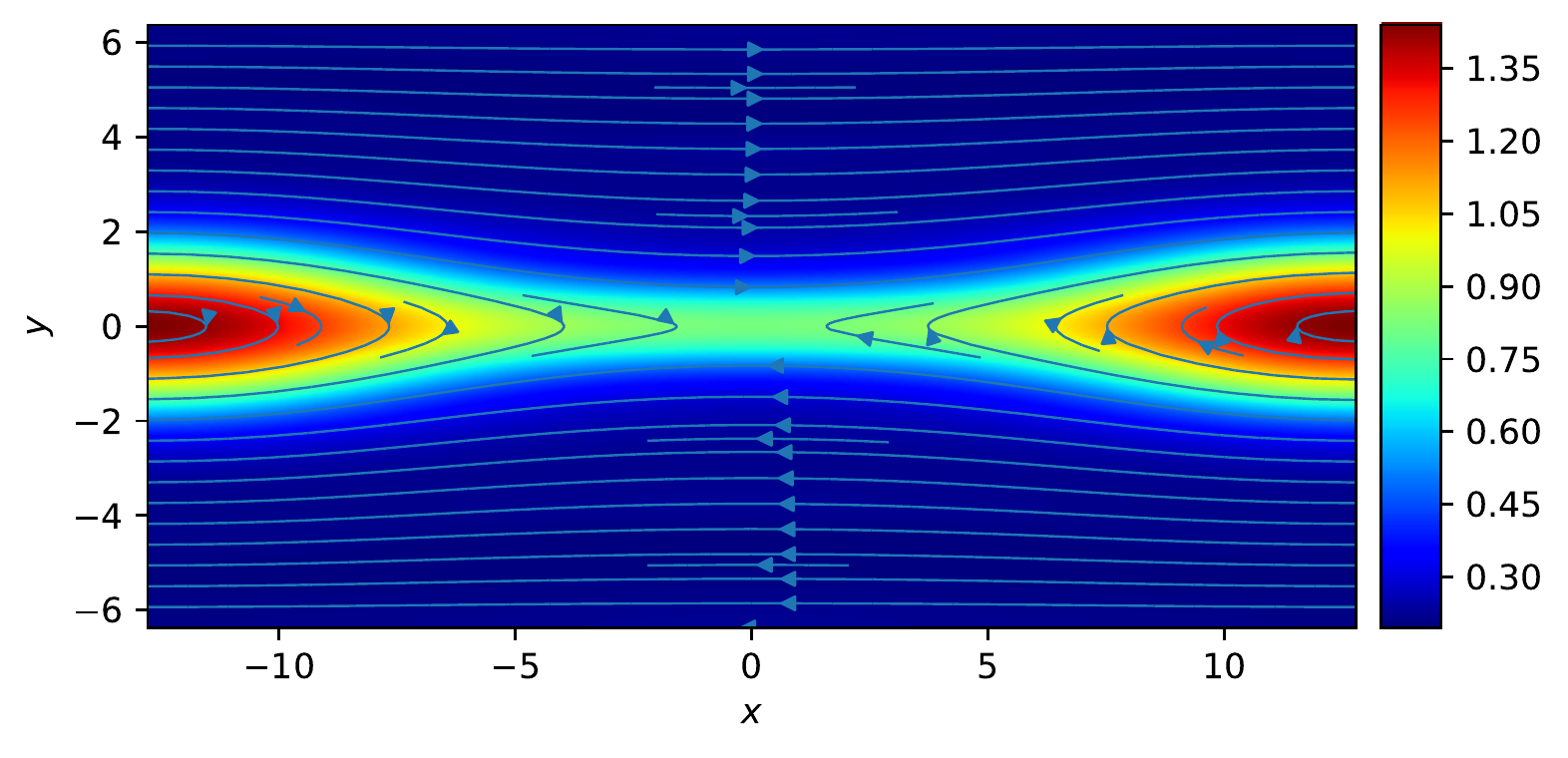}
			\label{fig:gem_o3_imp_40_rho}}
		\subfigure[$B_z$]{
			\includegraphics[width=2.9in, height=1.5in]{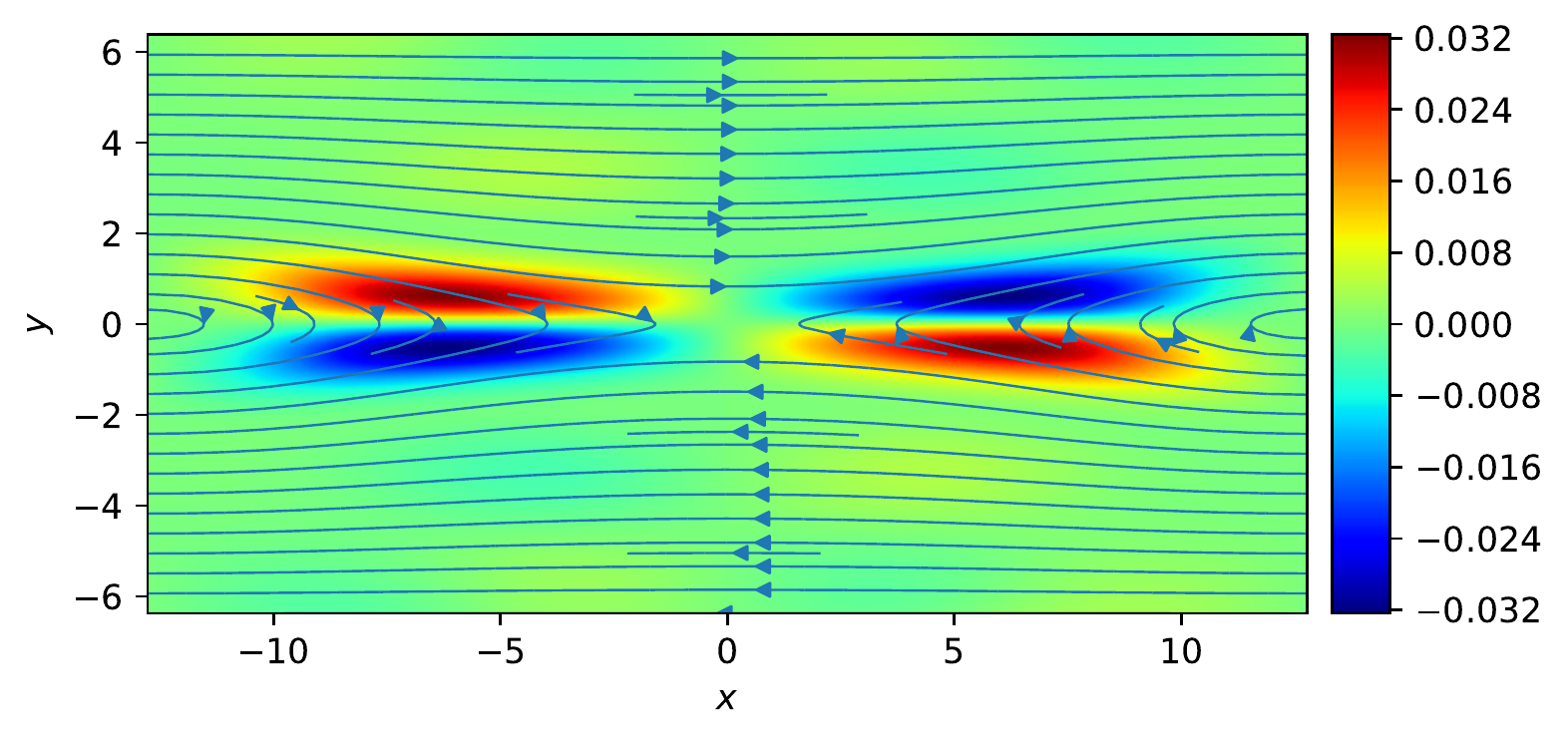}
			\label{fig:gem_o3_imp_40_Bz}}
		\subfigure[Ion x-velocity]{
			\includegraphics[width=2.9in, height=1.5in]{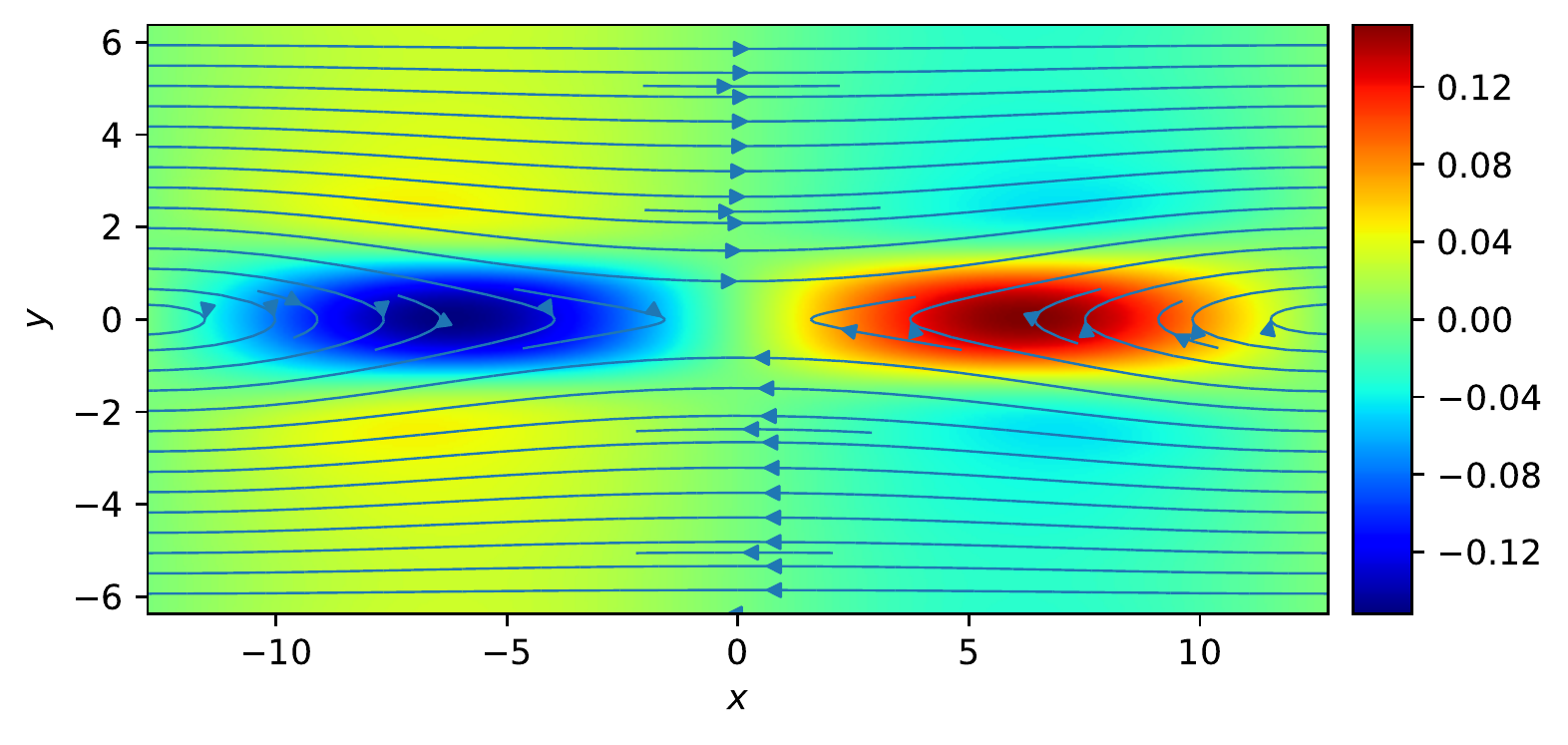}
			\label{fig:gem_o3_imp_40_uxi}}
		\subfigure[Electron x-velocity]{
			\includegraphics[width=2.9in, height=1.5in]{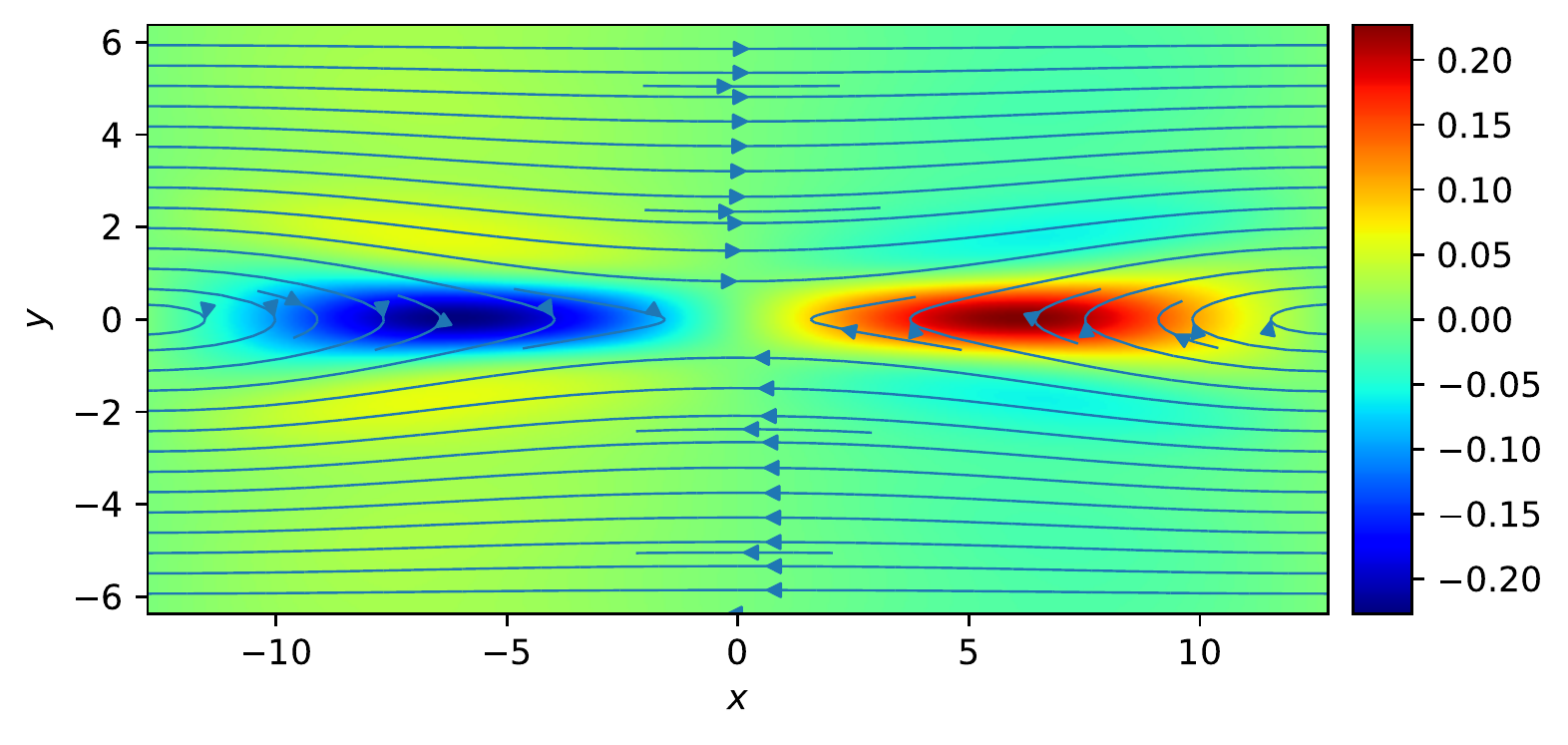}
			\label{fig:gem_o3_imp_40_uxe}}
		\caption{\nameref{test:2d_gem}: Plots for the total density, $B_z$-component, Ion $x$-velocity, and Electron $x$-velocity on the mess $512 \times 256$, using scheme {\bf O3-ES-IMEX}, at time t=40.0. }
		\label{fig:gem_o3_40}
	\end{center}
\end{figure}

\begin{figure}[!htbp]
	\begin{center}
		\subfigure[Total Density ($\rho_i + \rho_e$)]{
			\includegraphics[width=2.9in, height=1.5in]{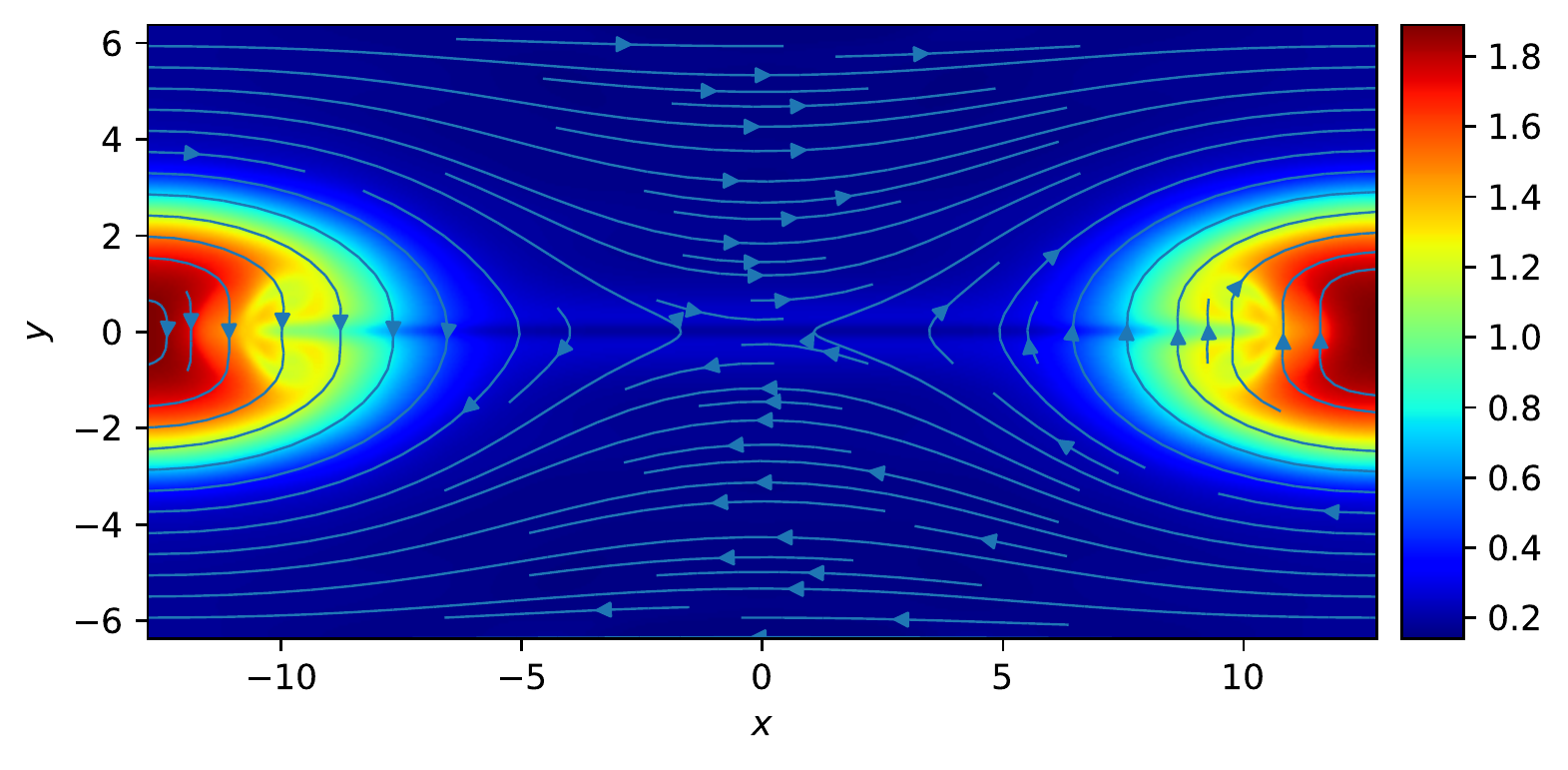}
			\label{fig:gem_o3_imp_80_rho}}
		\subfigure[$B_z$]{
			\includegraphics[width=2.9in, height=1.5in]{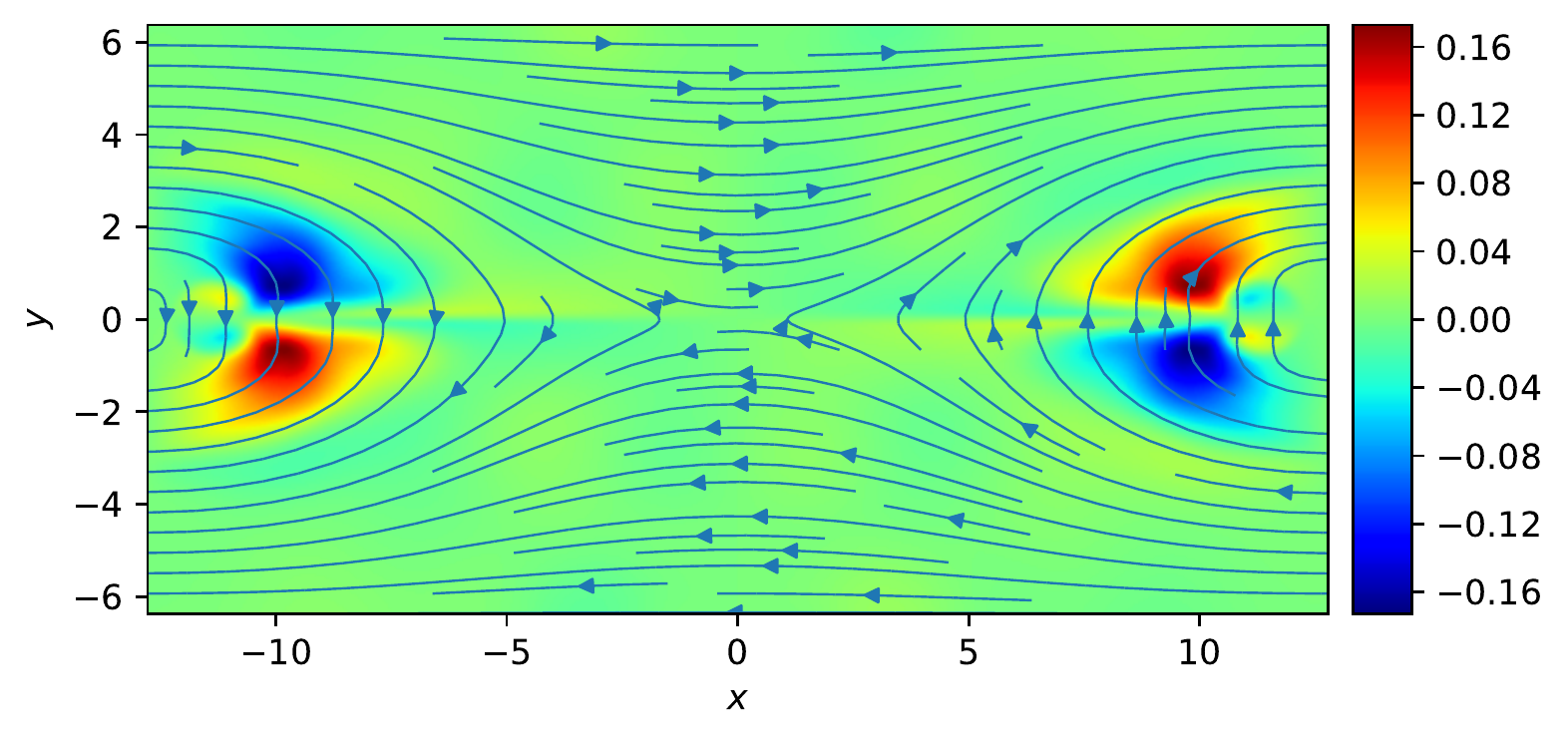}
			\label{fig:gem_o3_imp_80_Bz}}
		\subfigure[Ion x-velocity]{
			\includegraphics[width=2.9in, height=1.5in]{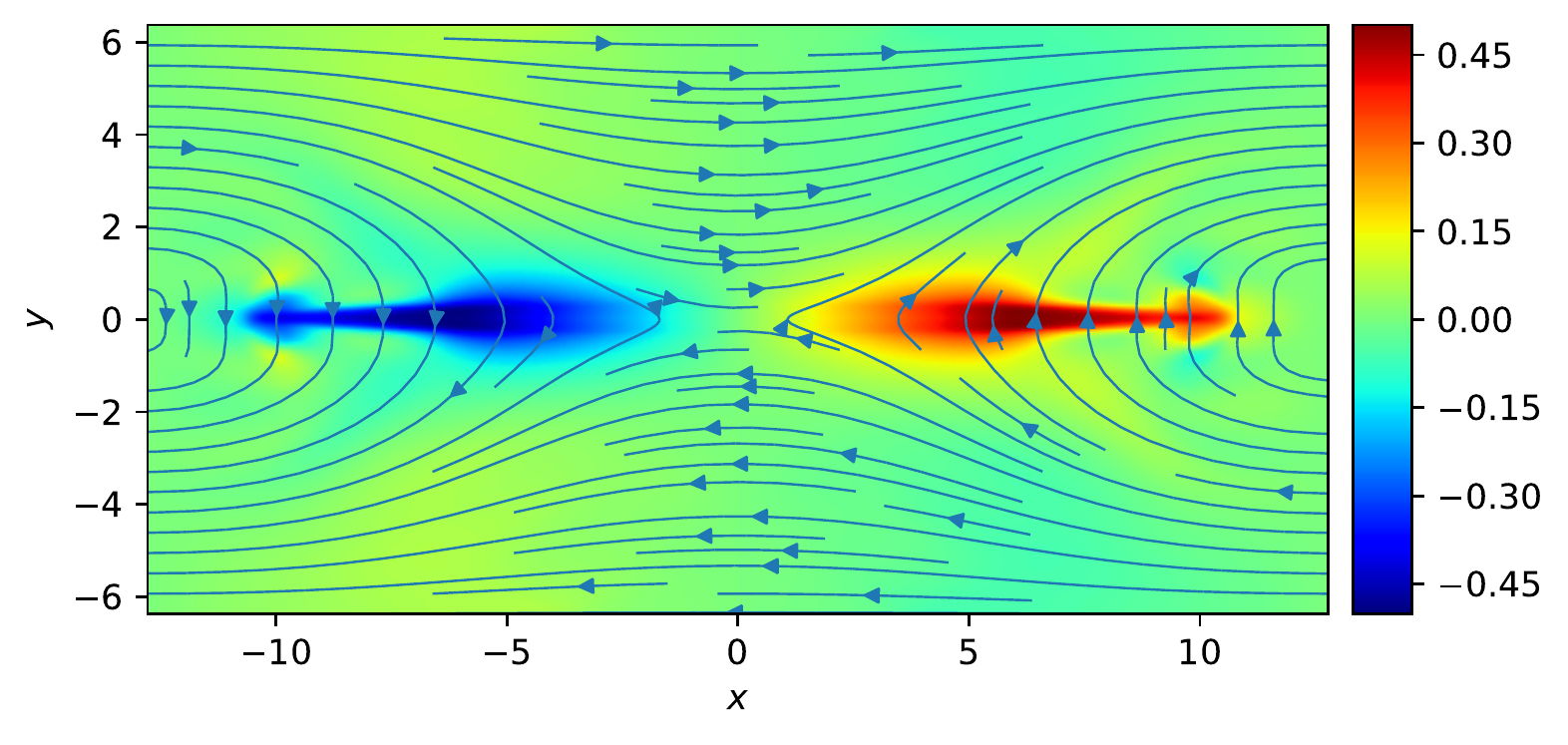}
			\label{fig:gem_o3_imp_80_uxi}}
		\subfigure[Electron x-velocity]{
			\includegraphics[width=2.9in, height=1.5in]{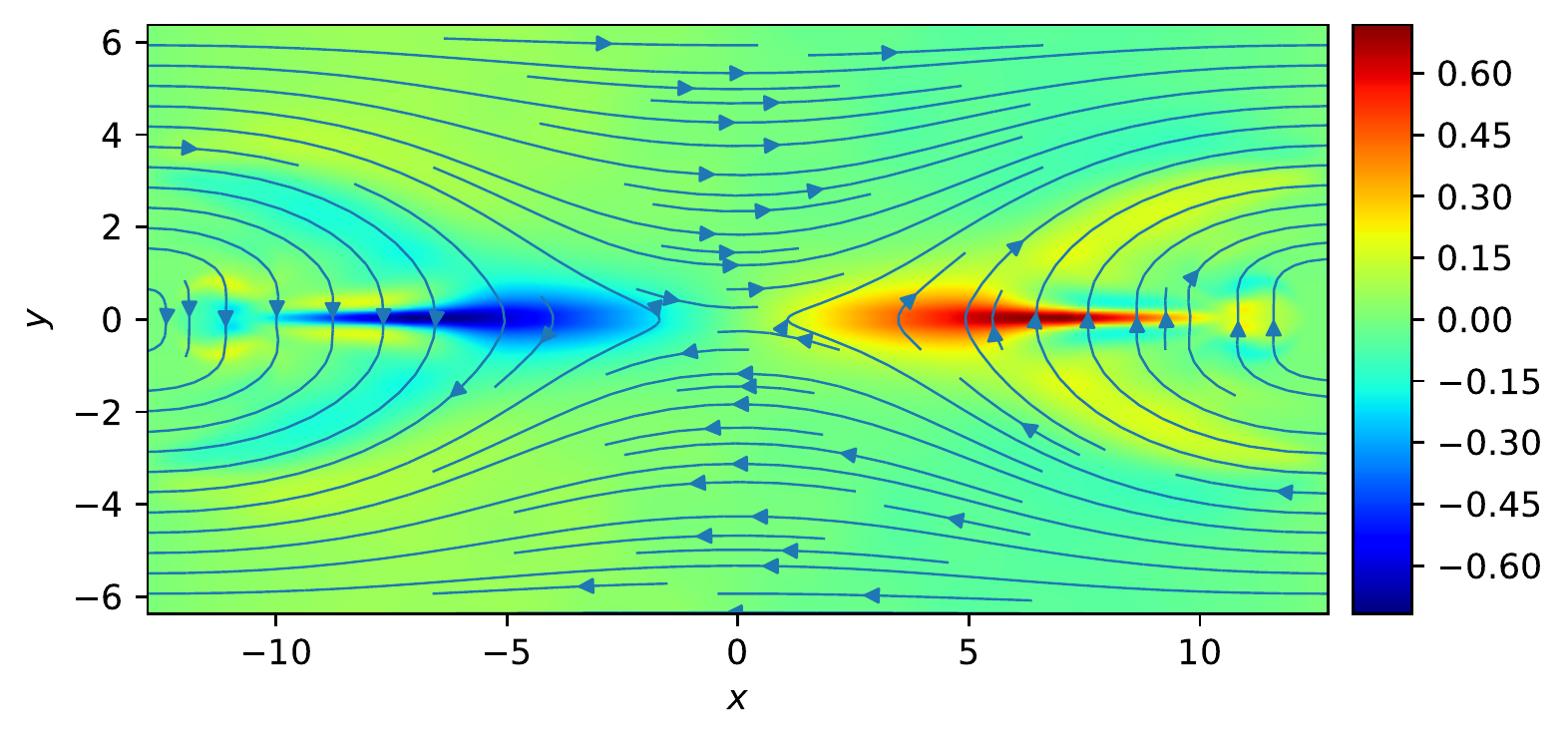}
			\label{fig:gem_o3_imp_80_uxe}}
		\caption{\nameref{test:2d_gem}: Plots for the total density, $B_z$-component, Ion $x$-velocity, and Electron $x$-velocity on the mess $512 \times 256$, using scheme {\bf O3-ES-IMEX}, at time t=80.0. }
		\label{fig:gem_o3_80}
	\end{center}
\end{figure}

\begin{figure}[!htbp]
	\begin{center}
		\subfigure[Total Density ($\rho_i + \rho_e$)]{
			\includegraphics[width=2.9in, height=1.5in]{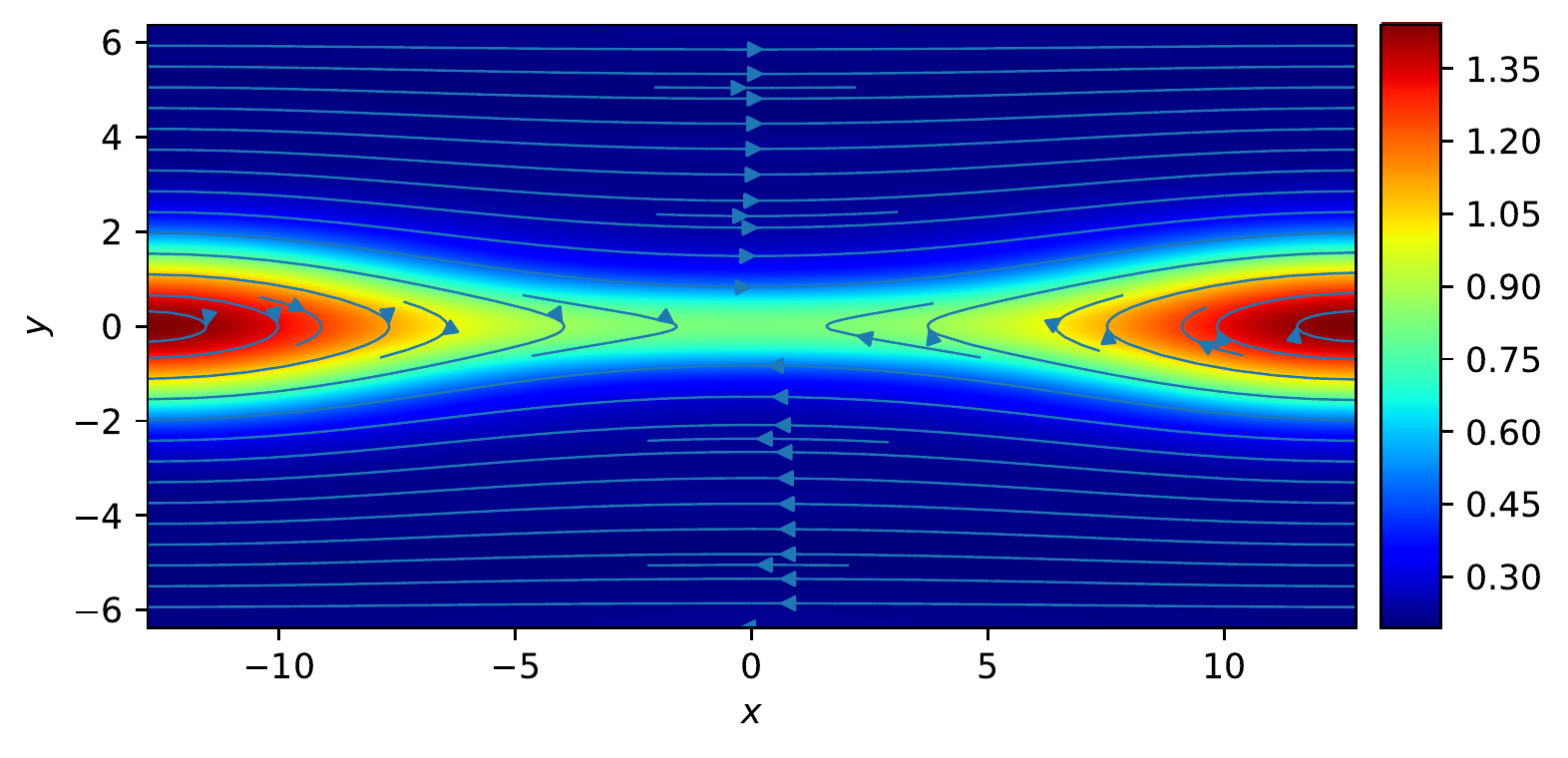}
			\label{fig:gem_o4_imp_40_rho}}
		\subfigure[$B_z$]{
			\includegraphics[width=2.9in, height=1.5in]{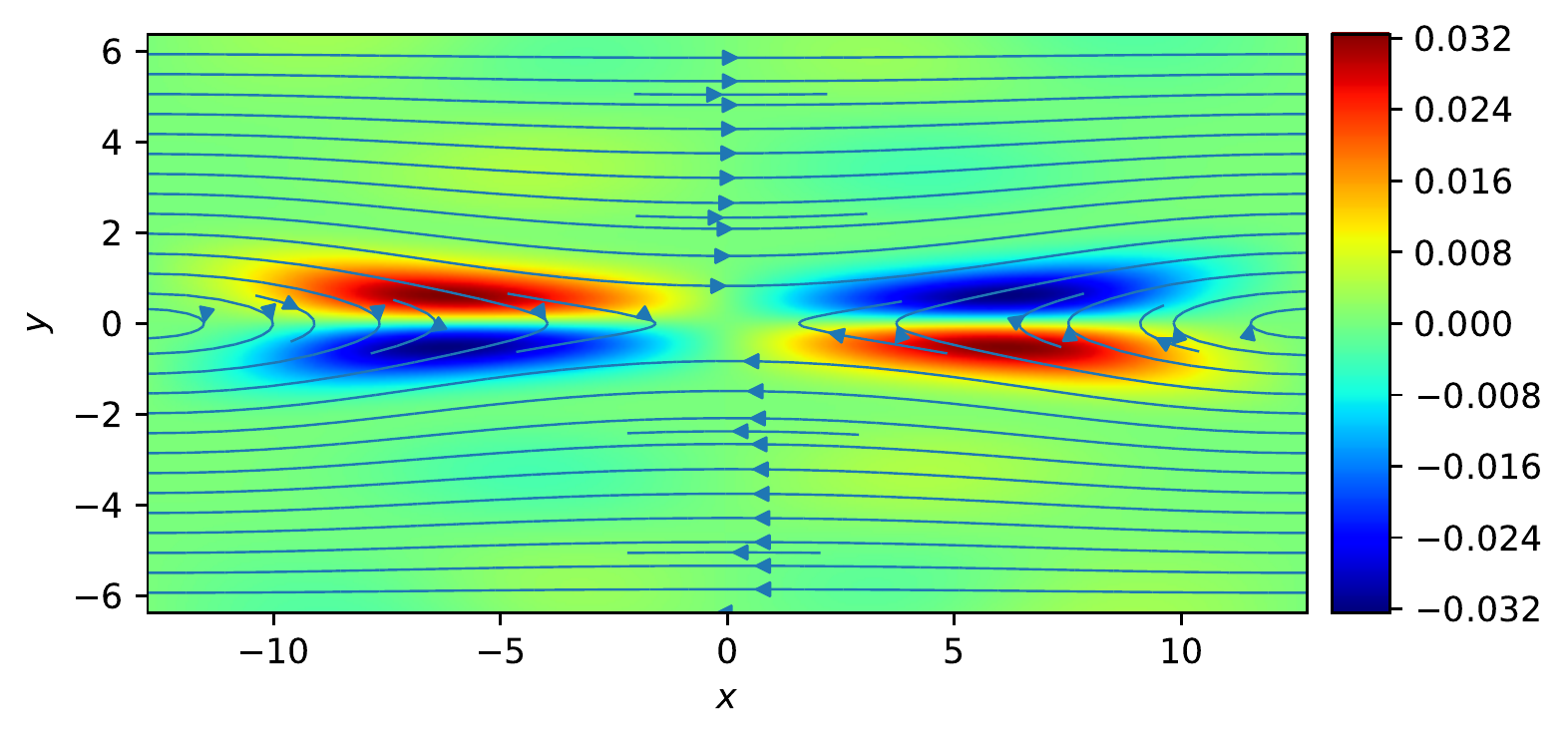}
			\label{fig:gem_o4_imp_40_Bz}}
		\subfigure[Ion x-velocity]{
			\includegraphics[width=2.9in, height=1.5in]{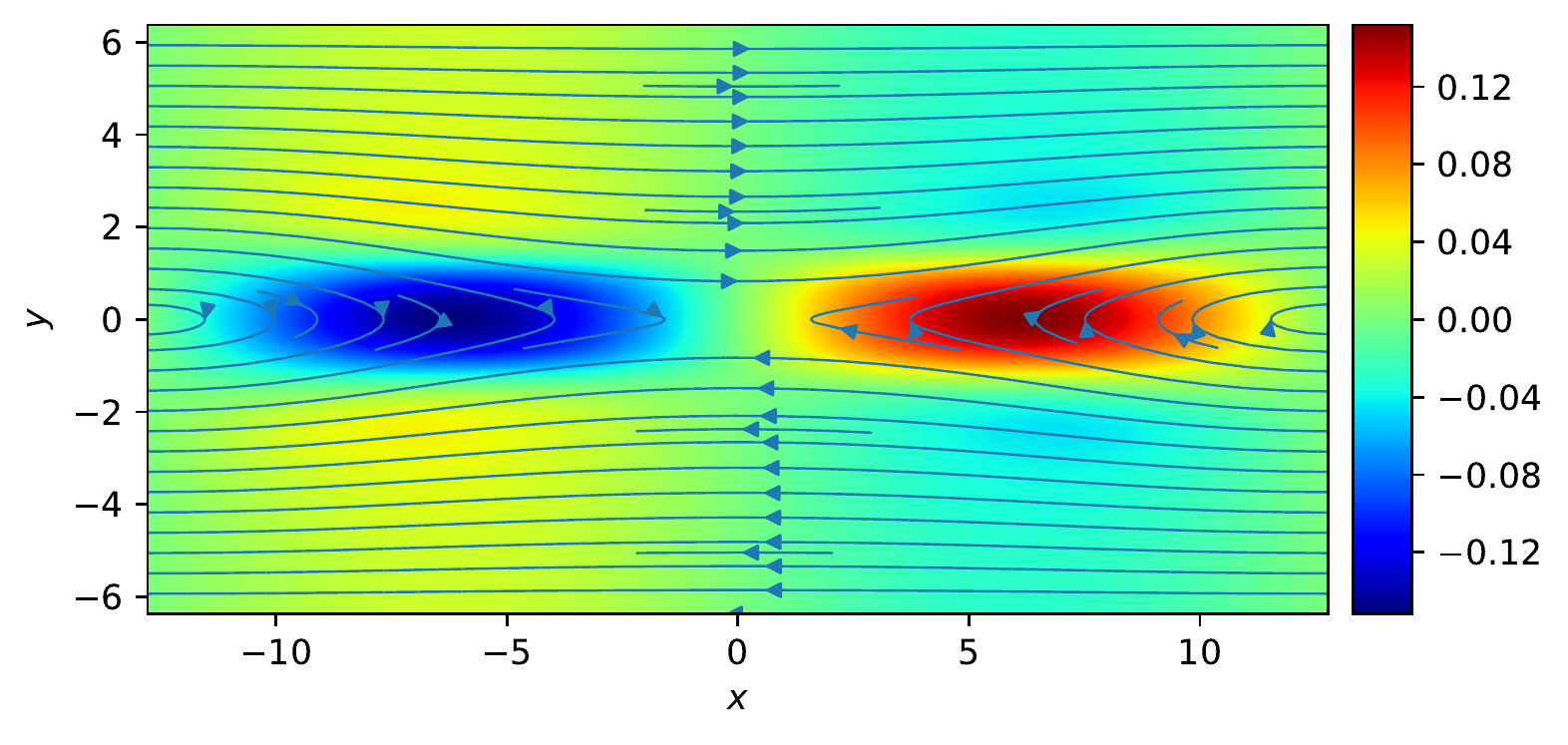}
			\label{fig:gem_o4_imp_40_uxi}}
		\subfigure[Electron x-velocity]{
			\includegraphics[width=2.9in, height=1.5in]{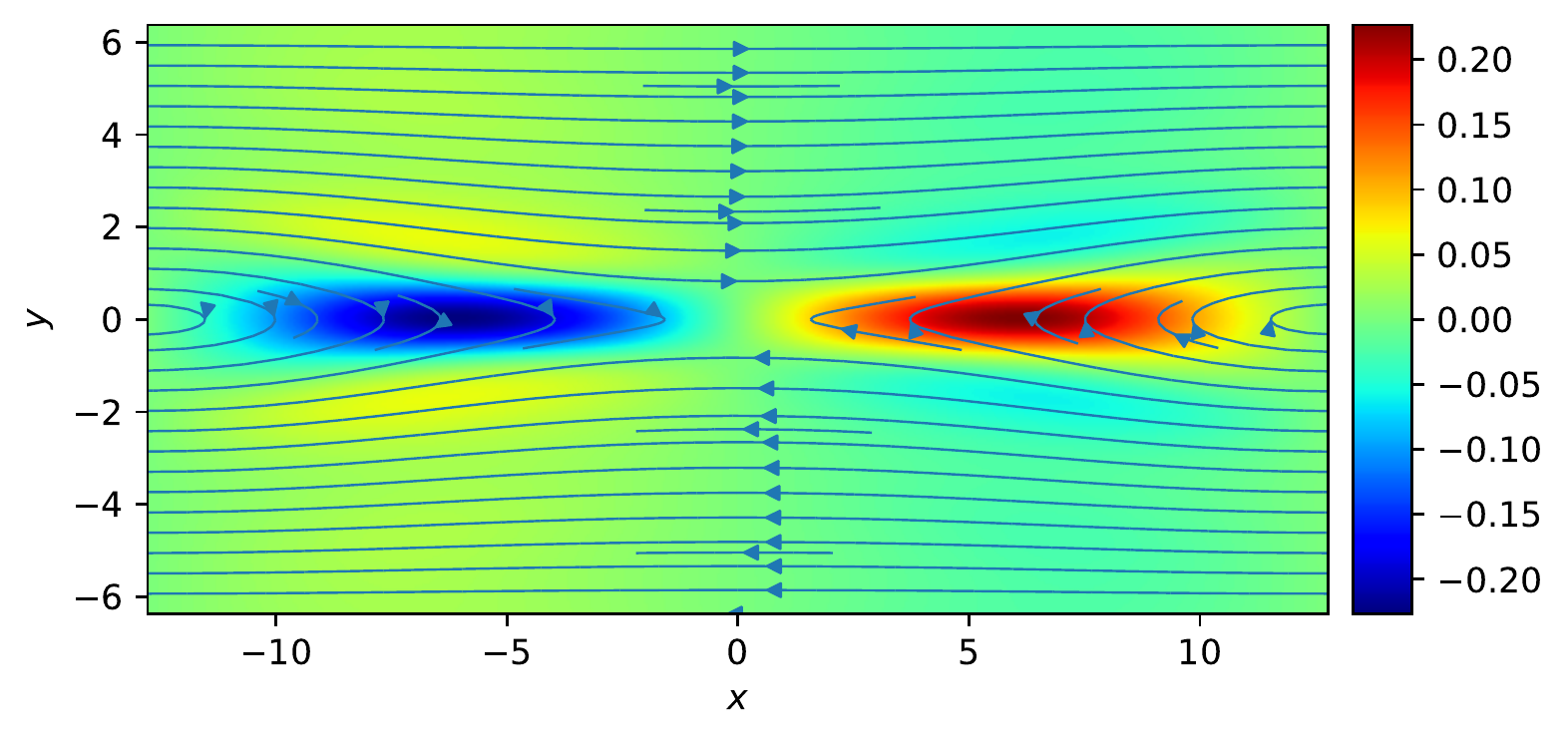}
			\label{fig:gem_o4_imp_40_uxe}}
		\caption{\nameref{test:2d_gem}: Plot for the total density, $B_z$-component, Ion $x$-velocity, and Electron $x$-velocity on the mess $512 \times 256$, using scheme {\bf O4-ES-IMEX}, at time t=40.0. }
		\label{fig:gem_o4_40}
	\end{center}
\end{figure}
\begin{figure}[!htbp]
	\begin{center}
		\subfigure[Total Density ($\rho_i + \rho_e$)]{
			\includegraphics[width=2.9in, height=1.5in]{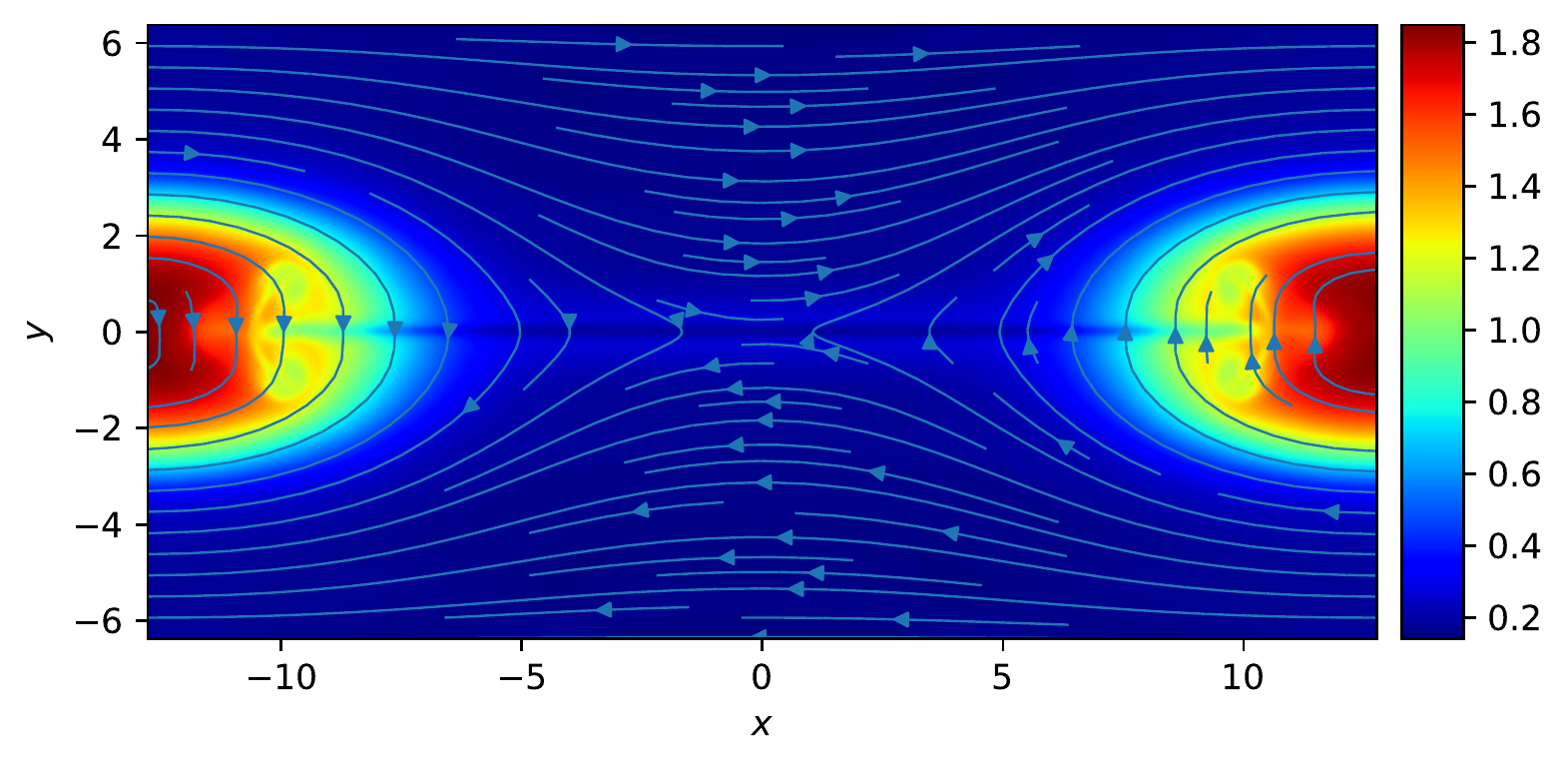}
			\label{fig:gem_o4_imp_80_rho}}
		\subfigure[$B_z$]{
			\includegraphics[width=2.9in, height=1.5in]{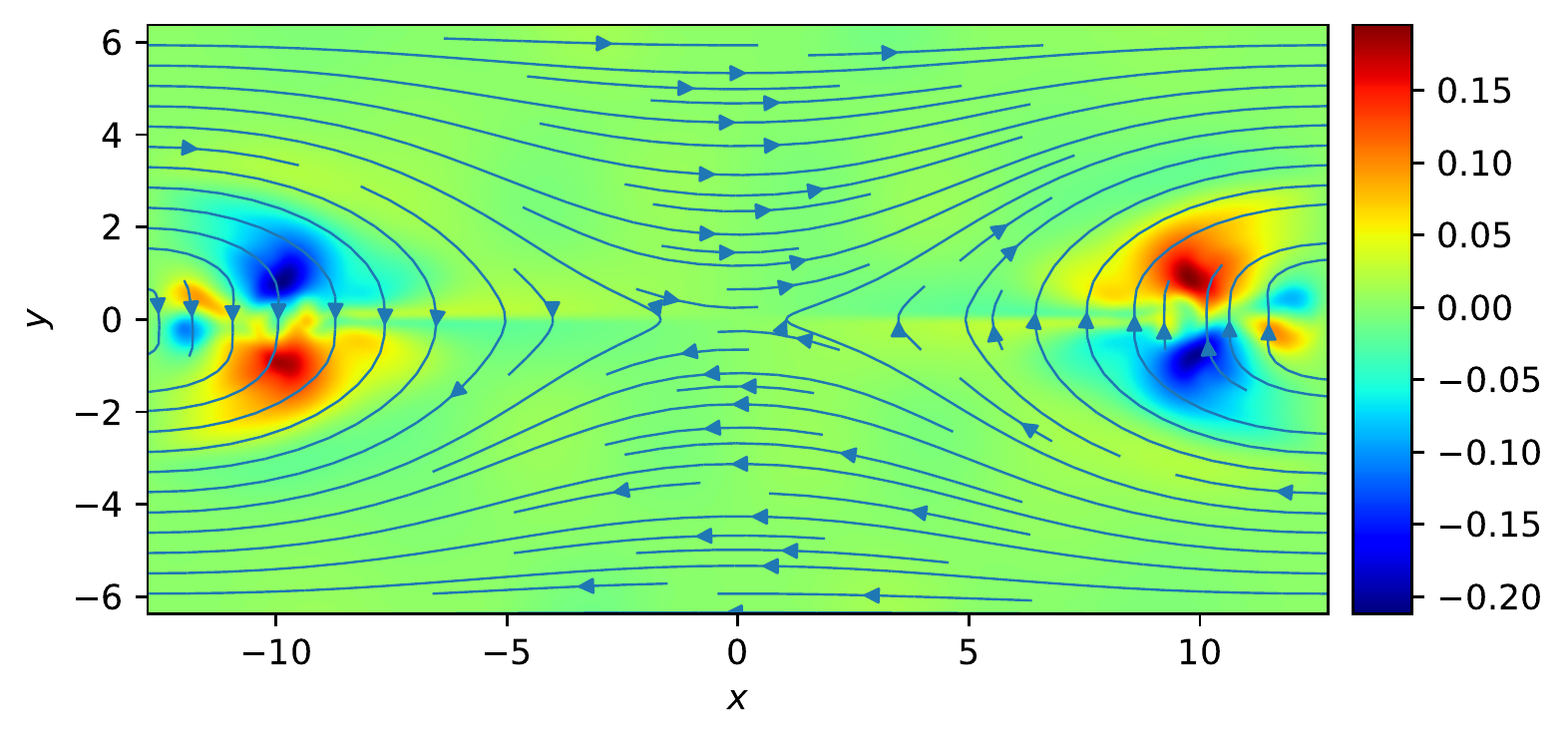}
			\label{fig:gem_o4_imp_80_Bz}}
		\subfigure[Ion x-velocity]{
			\includegraphics[width=2.9in, height=1.5in]{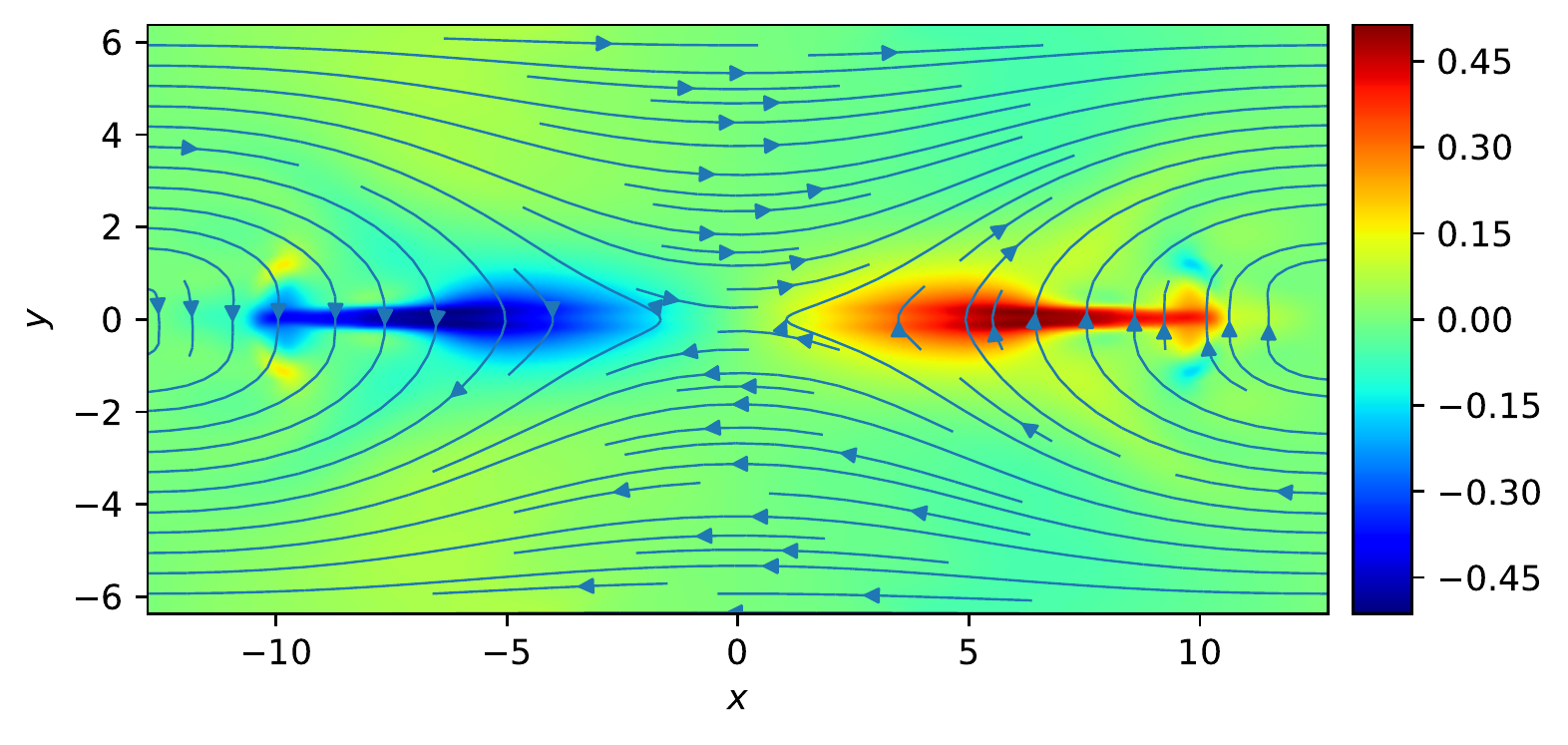}
			\label{fig:gem_o4_imp_80_uxi}}
		\subfigure[Electron x-velocity]{
			\includegraphics[width=2.9in, height=1.5in]{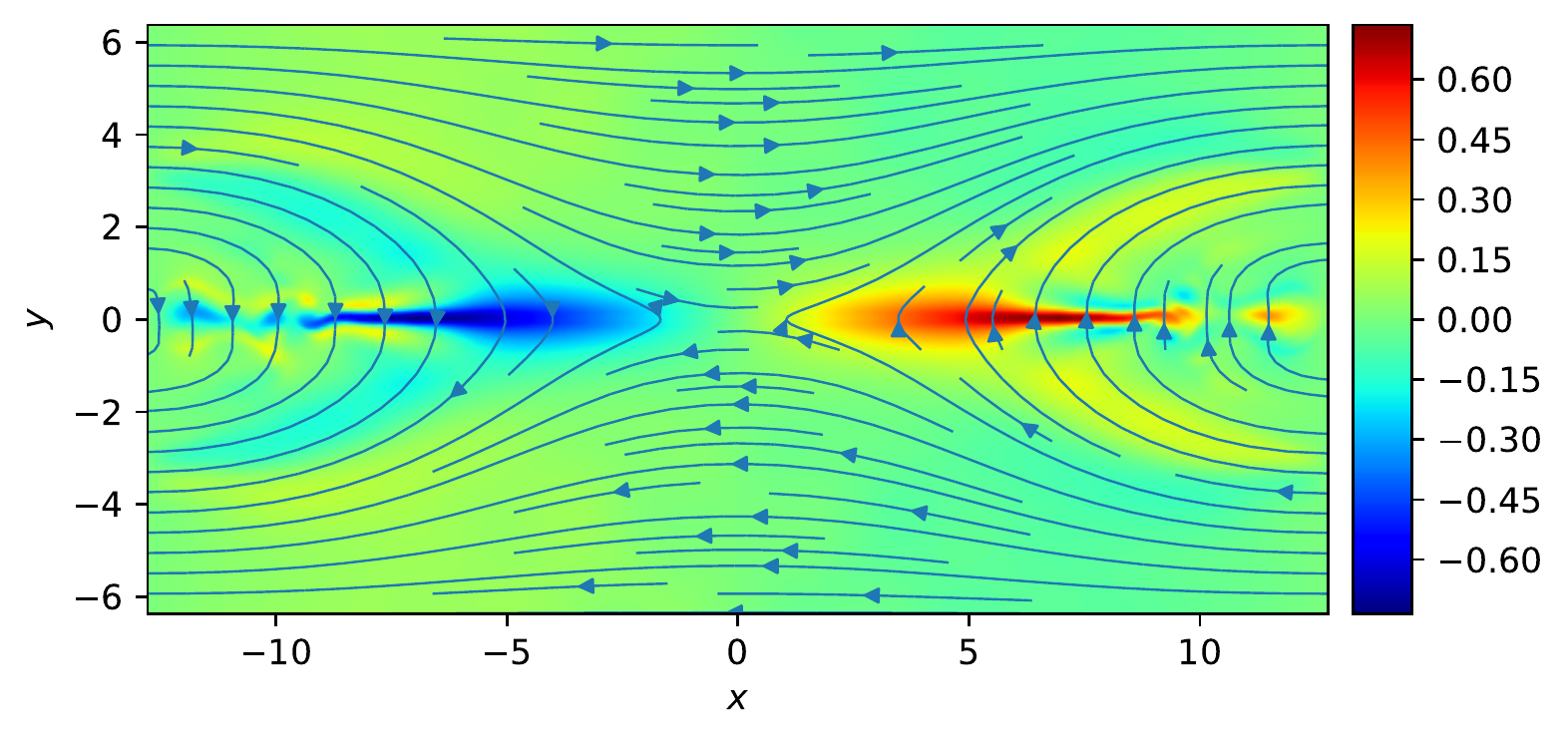}
			\label{fig:gem_o4_imp_80_uxe}}
		\caption{\nameref{test:2d_gem}: Plot for the total Density, $B_z$-component, Ion $x$-velocity, and Electron $x$-velocity on the mess $512 \times 256$, using scheme {\bf O4-ES-IMEX}, at time t=80.0. }
		\label{fig:gem_o4_80}
	\end{center}
\end{figure}

\begin{figure}[!htbp]
	\begin{center}
		\subfigure{
			\includegraphics[width=3.0in, height=2.2in]{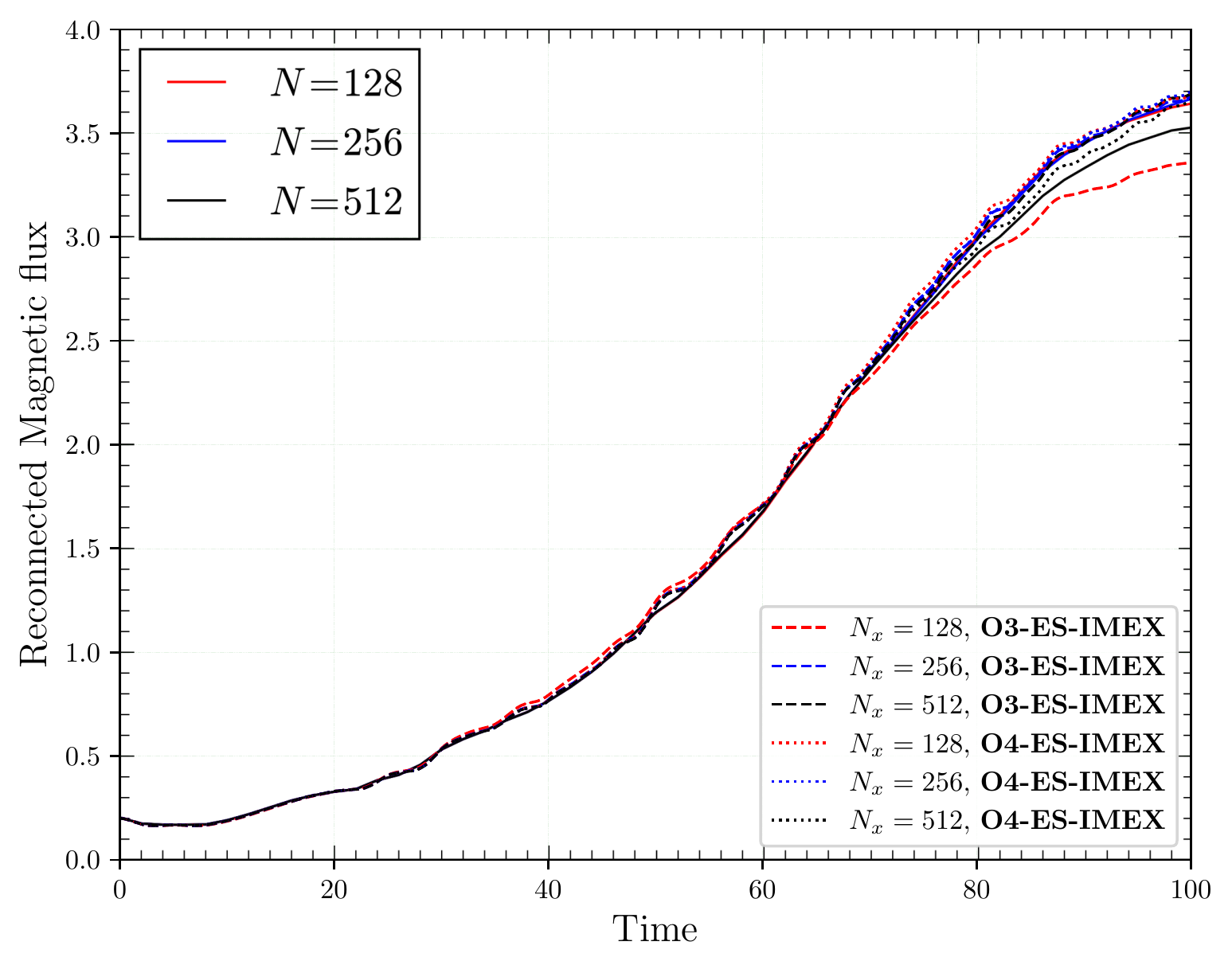}
		}
		\caption{\nameref{test:2d_gem}: Time development of the reconnected magnetic flux for $N_x=128,\ 256,\ 512$, using schemes {\bf O3-ES-IMEX} and {\bf O4-ES-IMEX}. We overlay the plot on Amano's GEM data \cite{Amano2016} (solid lines). }
		\label{fig:gem_recon}
	\end{center}
\end{figure}

\reva{
	\begin{figure}[!htbp]
		\begin{center}
			\subfigure[Plot of $|2\Delta x (\nabla \cdot \mathbf{B})_{i,j}|$ using the {\bf O3-ES-IMEX} scheme at time $t=80.0$.]{
				\includegraphics[width=2.9in, height=1.5in]{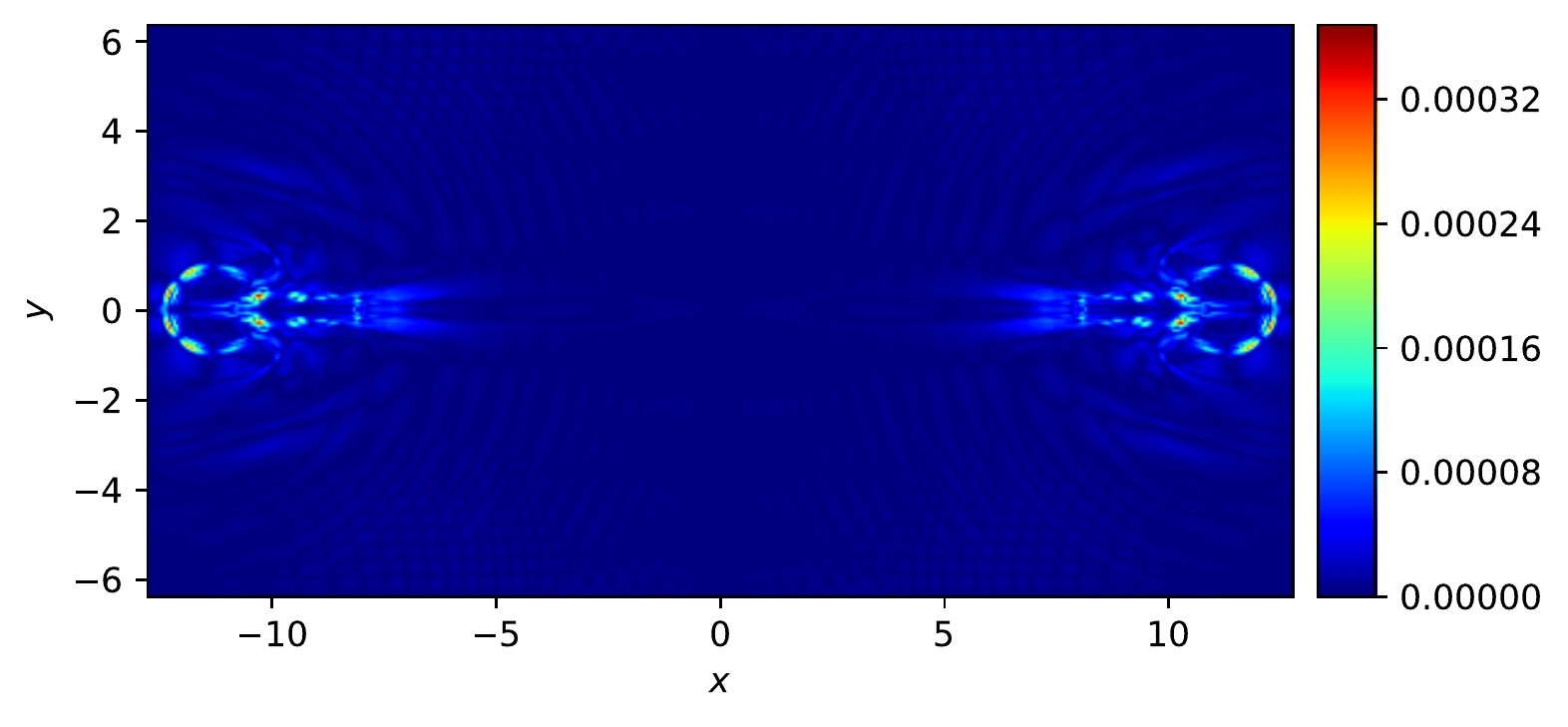}
				\label{fig:gem_2Dnorm_o3}}
			\quad
			\subfigure[Plot of $|2\Delta x (\nabla \cdot \mathbf{B})_{i,j}|$ using the {\bf O4-ES-IMEX} scheme at time $t=80.0$.]{
				\includegraphics[width=2.9in, height=1.5in]{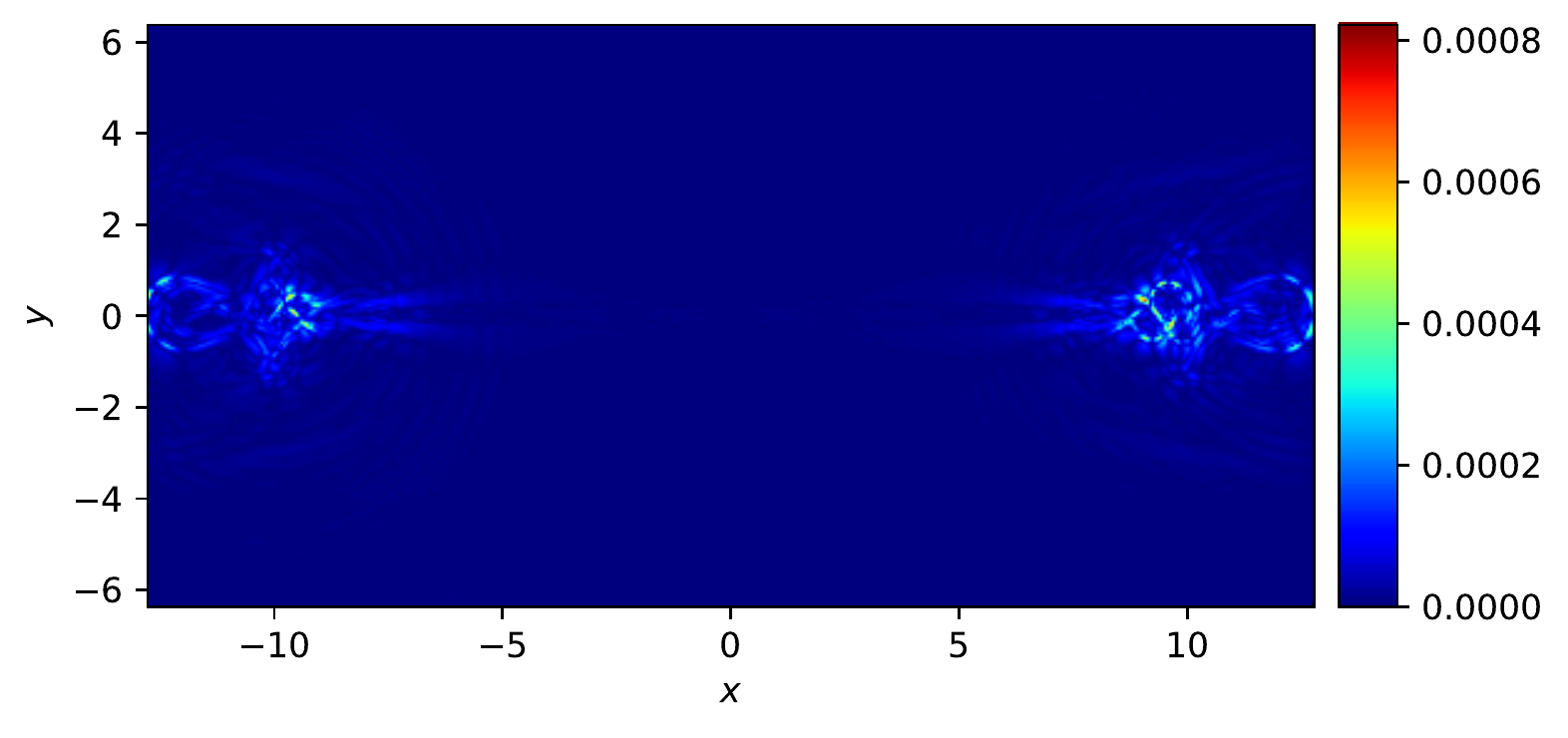}
				\label{fig:gem_2Dnorm_o4}}
			\quad
			\subfigure[Time evolution of $L^1$-norms of divergence of $\mathbf{B}$ till time $t=100.0$ for {\bf O3-ES-IMEX} and {\bf O4-ES-IMEX} schemes.]{
				\includegraphics[width=3.0in, height=2.0in]{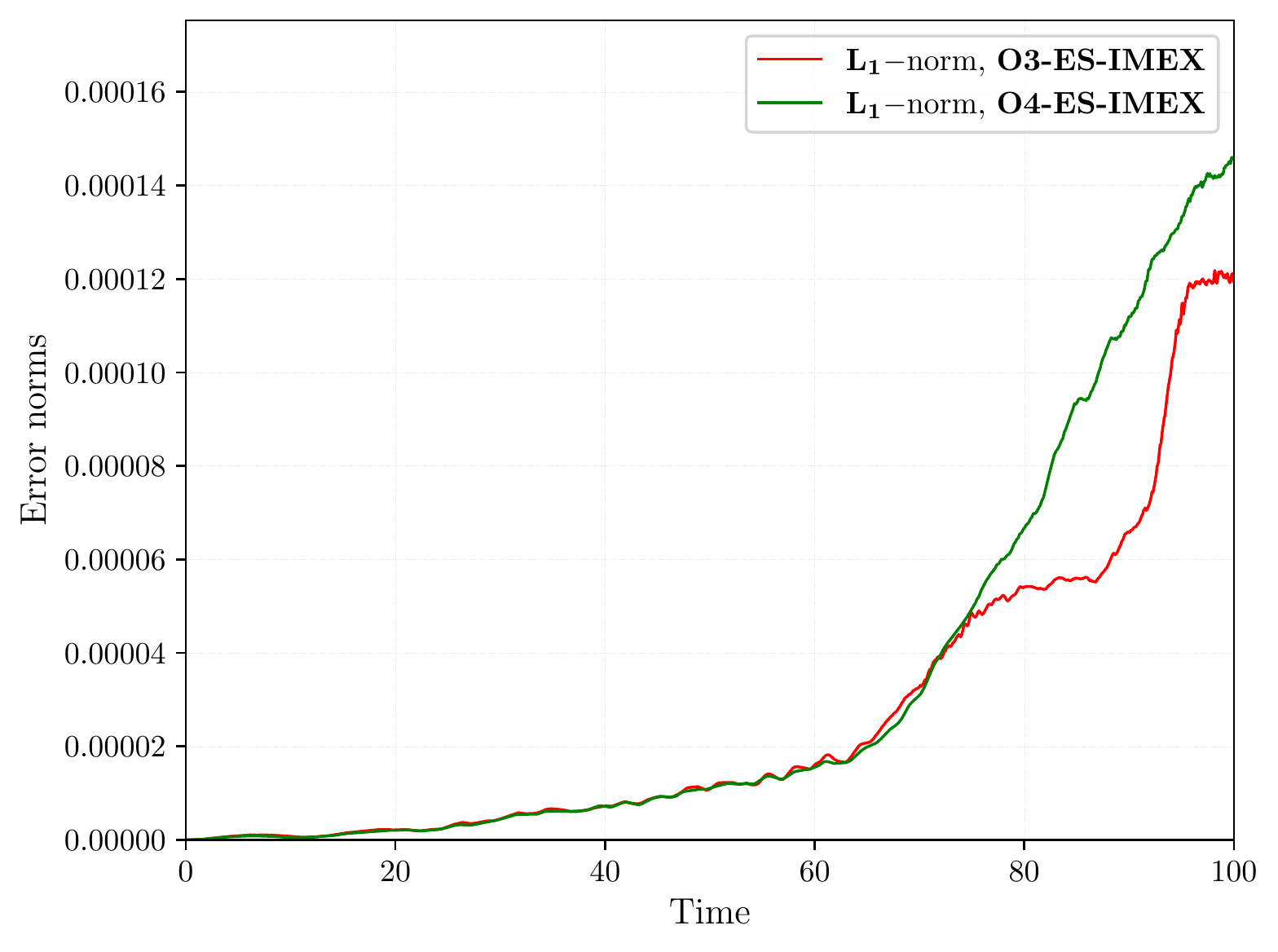}
				\label{fig:gem_l1norm}}
			\caption{\nameref{test:2d_gem}: \reva{Plots of the ($|2\Delta x (\nabla \cdot \mathbf{B})_{i,j}|$) and time evolution of the $L^1$-norms of divergence of $\mathbf{B}$ for schemes {\bf O3-ES-IMEX} and {\bf O4-ES-IMEX} schemes using $512\times 256$ cells.}}
			\label{fig:norms_gem}
		\end{center}
	\end{figure}
	
	In Fig.~\ref{fig:norms_gem}, we plot errors for the divergence of the magnetic fields. In Figs.~\ref{fig:gem_2Dnorm_o3} and~\ref{fig:gem_2Dnorm_o4} we plot the undivided divergence of magnetic field using {\bf O3-ES-IMEX} and {\bf O4-ES-IMEX} schemes, respectively. We again observe that the errors are concentrated where flows have strong gradients. Both schemes have similar errors, with {\bf O3-ES-IMEX} having more spread-out errors. In Fig.~\ref{fig:gem_l1norm}, we plot the time evolution of the ${L^1}$-norms of divergence of the magnetic field. We again observe that both schemes have similar $L^1$ errors.
}

\section{Conclusion}
The two-fluid relativistic plasma flow equations are a multi-physics model that couples relativistic hydrodynamic conservation laws with Maxwell's equations. This model is non-linear and exhibits non-smooth and multi-scale phenomena.  In this article, we have designed arbitrary high-order finite-difference entropy stable schemes for the model. This was achieved by exploiting the structure of the flux and by demonstrating that the source terms do not contribute to the fluid entropies. Furthermore, we have also presented ARK-IMEX schemes to overcome the time-step restriction imposed by the stiff source terms. The implicit step in the ARK-IMEX scheme is efficient as it involves solving only one set of nonlinear equations in each cell and there is no need for any global nonlinear solver.

The proposed schemes are applied to various test cases in one and two dimensions. First, we demonstrate that the proposed schemes have formal order of accuracy. Then we test the schemes on the Brio-Wu test problem, where we have finite plasma skin depth. We note that the proposed schemes can resolve the dispersive effects, and higher-order schemes can capture the finite skin depth effects even on coarser meshes. We also demonstrate that on the coarser meshes with the stiff source terms, IMEX schemes are more efficient. The SSP-RK and ARK-IMEX schemes have similar accuracy and entropy decay performance. We then test the scheme on a current sheet problem where resistive effects are considered. We observe that the proposed schemes accurately capture the RMHD solution. We compute the Orzag-Tang vortex problem, blast problem, and GEM challenge problem in two-dimensional test cases. We again show that the schemes are able to capture solution features very well and ensure entropy stability of the solutions. For the GEM problem, the schemes are also able to predict magnetic reconnection flux.

\section*{Acknowledgments}
The work of Harish Kumar is supported by VAJRA grant No. VJR/2018/000129 by the Dept. of Science and Technology, Govt. of India. The work of Praveen Chandrashekar was supported by SERB-DST, Govt. of India, under the MATRICS grant No. MTR/2018/000006, and by the Department of Atomic Energy,  Government of India, under project no.~12-R\&D-TFR-5.01-0520. The authors would also like to thank Prof. Dinshaw Balsara, Univ. of Notre-Dame, USA, for his several suggestions.
\bibliographystyle{elsarticle-num}

\begin{thebibliography}{10}
\expandafter\ifx\csname url\endcsname\relax
  \def\url#1{\texttt{#1}}\fi
\expandafter\ifx\csname urlprefix\endcsname\relax\def\urlprefix{URL }\fi
\expandafter\ifx\csname href\endcsname\relax
  \def\href#1#2{#2} \def\path#1{#1}\fi

\bibitem{Bhoriya2020}
D.~Bhoriya, H.~Kumar, {Entropy-stable schemes for relativistic hydrodynamics
  equations}, Zeitschrift fur Angewandte Mathematik und Physik 71~(1) (2020).
\newblock \href {https://doi.org/10.1007/s00033-020-1250-8}
  {\path{doi:10.1007/s00033-020-1250-8}}.

\bibitem{Gallant1994}
Y.~A. Gallant, J.~Arons, {Structure of relativistic shocks in pulsar winds: A
  model of the wisps in the Crab Nebula}, The Astrophysical Journal 435 (1994)
  230.
\newblock \href {https://doi.org/10.1086/174810} {\path{doi:10.1086/174810}}.

\bibitem{Mochkovitch1995}
R.~Mochkovitch, V.~Maitia, R.~Marques, {Internal shocks in a relativistic wind
  as a source for gamma-ray bursts?}, Astrophysics and Space Science 231~(1-2)
  (1995) 441--444.
\newblock \href {https://doi.org/10.1007/BF00658666}
  {\path{doi:10.1007/BF00658666}}.

\bibitem{Landau1987}
L.~D. Landau, E.~M. Lifshitz, {Relativistic Fluid Dynamics}, in: Fluid
  Mechanics, Elsevier, 1987, pp. 505--514.
\newblock \href {https://doi.org/10.1016/b978-0-08-033933-7.50023-4}
  {\path{doi:10.1016/b978-0-08-033933-7.50023-4}}.

\bibitem{Wardle1998}
J.~F. Wardle, D.~C. Homan, R.~Ojha, D.~H. Roberts, {Electron-positron jets
  associated with the quasar 3C279}, Nature 395~(6701) (1998) 457--461.
\newblock \href {https://doi.org/10.1038/26675} {\path{doi:10.1038/26675}}.

\bibitem{Komissarov1999}
S.~S. Komissarov, {A Godunov-type scheme for relativistic
  magnetohydrodynamics}, Monthly Notices of the Royal Astronomical Society
  303~(2) (1999) 343--366.
\newblock \href {https://doi.org/10.1046/j.1365-8711.1999.02244.x}
  {\path{doi:10.1046/j.1365-8711.1999.02244.x}}.

\bibitem{Balsara2001}
D.~Balsara, {Total Variation Diminishing Scheme for Relativistic
  Magnetohydrodynamics}, The Astrophysical Journal Supplement Series 132~(1)
  (2001) 83--101.
\newblock \href {https://doi.org/10.1086/318941} {\path{doi:10.1086/318941}}.

\bibitem{DelZanna2003}
L.~{Del Zanna}, N.~Bucciantini, P.~Londrillo, {An efficient shock-capturing
  central-type scheme for multidimensional relativistic flows II.
  Magnetohydrodynamics}, Astronomy and Astrophysics 400~(2) (2003) 397--413.
\newblock \href {https://doi.org/10.1051/0004-6361:20021641}
  {\path{doi:10.1051/0004-6361:20021641}}.

\bibitem{Mignone2006}
A.~Mignone, G.~Bodo, {An HLLC Riemann solver for relativistic flows – II.
  Magnetohydrodynamics}, Monthly Notices of the Royal Astronomical Society
  368~(3) (2006) 1040--1054.
\newblock \href {https://doi.org/10.1111/J.1365-2966.2006.10162.X}
  {\path{doi:10.1111/J.1365-2966.2006.10162.X}}.

\bibitem{Komissarov2007}
S.~S. Komissarov, {Multidimensional numerical scheme for resistive relativistic
  magnetohydrodynamics}, Monthly Notices of the Royal Astronomical Society
  382~(3) (2007) 995--1004.
\newblock \href {https://doi.org/10.1111/j.1365-2966.2007.12448.x}
  {\path{doi:10.1111/j.1365-2966.2007.12448.x}}.

\bibitem{Balsara2016aderweno}
D.~S. Balsara, J.~Kim, {A subluminal relativistic magnetohydrodynamics scheme
  with ADER-WENO predictor and multidimensional Riemann solver-based
  corrector}, Journal of Computational Physics 312 (2016) 357--384.
\newblock \href {https://doi.org/10.1016/j.jcp.2016.02.001}
  {\path{doi:10.1016/j.jcp.2016.02.001}}.

\bibitem{Amano2016}
T.~Amano, {A second-order divergence-constrained multidimensional numerical
  scheme for Relativistic Two-Fluid Electrodynamics}, The Astrophysical Journal
  831~(1) (2016) 100.
\newblock \href {https://doi.org/10.3847/0004-637x/831/1/100}
  {\path{doi:10.3847/0004-637x/831/1/100}}.

\bibitem{Shumlak2003}
U.~Shumlak, J.~Loverich, {Approximate Riemann solver for the two-fluid plasma
  model}, Journal of Computational Physics 187~(2) (2003) 620--638.
\newblock \href {https://doi.org/10.1016/S0021-9991(03)00151-7}
  {\path{doi:10.1016/S0021-9991(03)00151-7}}.

\bibitem{Hakim2006}
A.~Hakim, J.~Loverich, U.~Shumlak, {A high resolution wave propagation scheme
  for ideal Two-Fluid plasma equations}, Journal of Computational Physics
  219~(1) (2006) 418--442.
\newblock \href {https://doi.org/10.1016/j.jcp.2006.03.036}
  {\path{doi:10.1016/j.jcp.2006.03.036}}.

\bibitem{Kumar2012}
H.~Kumar, S.~Mishra, {Entropy stable numerical schemes for two-fluid plasma
  equations}, Journal of Scientific Computing 52~(2) (2012) 401--425.
\newblock \href {https://doi.org/10.1007/s10915-011-9554-7}
  {\path{doi:10.1007/s10915-011-9554-7}}.

\bibitem{Abgrall2014}
R.~Abgrall, H.~Kumar, {Robust Finite Volume Schemes for Two-Fluid Plasma
  Equations}, Journal of Scientific Computing 60~(3) (2014) 584--611.
\newblock \href {https://doi.org/10.1007/s10915-013-9809-6}
  {\path{doi:10.1007/s10915-013-9809-6}}.

\bibitem{bond2016plasma}
D.~M. Bond, V.~Wheatley, R.~Samtaney, {Plasma flow simulation using the
  two-fluid model}, in: Proceedings of the 20th Australasian Fluid Mechanics
  Conference, 2016.

\bibitem{Li2020}
Y.~Li, R.~Samtaney, D.~Bond, V.~Wheatley, {Richtmyer-Meshkov instability of an
  imploding flow with a two-fluid plasma model}, Physical Review Fluids 5~(11)
  (2020) 113701.
\newblock \href {https://doi.org/10.1103/PhysRevFluids.5.113701}
  {\path{doi:10.1103/PhysRevFluids.5.113701}}.

\bibitem{Meena2019}
A.~K. Meena, H.~Kumar, {Robust numerical schemes for Two-Fluid Ten-Moment
  plasma flow equations}, Zeitschrift fur Angewandte Mathematik und Physik
  70~(1) (2019) 1--30.
\newblock \href {https://doi.org/10.1007/s00033-018-1061-3}
  {\path{doi:10.1007/s00033-018-1061-3}}.

\bibitem{Zenitani2009b}
S.~Zenitani, M.~Hesse, A.~Klimas, {Relativistic two-fluid simulations of guide
  field magnetic reconnection}, Astrophysical Journal 705~(1) (2009) 907--913.
\newblock \href {https://doi.org/10.1088/0004-637X/705/1/907}
  {\path{doi:10.1088/0004-637X/705/1/907}}.

\bibitem{Zenitani2010}
S.~Zenitani, M.~Hesse, A.~Klimas, {Resistive magnetohydrodynamic simulations of
  relativistic magnetic reconnection}, Astrophysical Journal Letters 716~(2)
  (2010) L214--L218.
\newblock \href {https://doi.org/10.1088/2041-8205/716/2/L214}
  {\path{doi:10.1088/2041-8205/716/2/L214}}.

\bibitem{Amano2013}
T.~Amano, J.~G. Kirk, {The role of superluminal electromagnetic waves in pulsar
  wind termination shocks}, Astrophysical Journal 770~(1) (2013) 18.
\newblock \href {https://doi.org/10.1088/0004-637X/770/1/18}
  {\path{doi:10.1088/0004-637X/770/1/18}}.

\bibitem{Barkov2014}
M.~Barkov, S.~S. Komissarov, V.~Korolev, A.~Zankovich, {A multidimensional
  numerical scheme for two-fluid relativistic magnetohydrodynamics}, Monthly
  Notices of the Royal Astronomical Society 438~(1) (2014) 704--716.
\newblock \href {https://doi.org/10.1093/mnras/stt2247}
  {\path{doi:10.1093/mnras/stt2247}}.

\bibitem{Barkov2016}
M.~V. Barkov, S.~S. Komissarov, {Relativistic tearing and drift-kink
  instabilities in two-fluid simulations}, Monthly Notices of the Royal
  Astronomical Society 458~(2) (2016) 1939--1947.
\newblock \href {https://doi.org/10.1093/mnras/stw384}
  {\path{doi:10.1093/mnras/stw384}}.

\bibitem{LeVeque2002}
R.~J. LeVeque, {Finite Volume Methods for Hyperbolic Problems}, Cambridge
  University Press, 2002.
\newblock \href {https://doi.org/10.1017/cbo9780511791253}
  {\path{doi:10.1017/cbo9780511791253}}.

\bibitem{Chiodaroli2015}
E.~Chiodaroli, C.~{De Lellis}, O.~Kreml, {Global Ill-Posedness of the
  Isentropic System of Gas Dynamics}, Communications on Pure and Applied
  Mathematics 68~(7) (2015) 1157--1190.
\newblock \href {https://doi.org/10.1002/CPA.21537}
  {\path{doi:10.1002/CPA.21537}}.

\bibitem{Godlewski1996}
E.~Godlewski, P.-A. Raviart, {Numerical Approximation of Hyperbolic Systems of
  Conservation Laws}, Vol. 118 of Applied Mathematical Sciences, Springer New
  York, New York, NY, 1996.
\newblock \href {https://doi.org/10.1007/978-1-4612-0713-9}
  {\path{doi:10.1007/978-1-4612-0713-9}}.

\bibitem{Zenitani2009a}
S.~Zenitani, M.~Hesse, A.~Klimas, {Two-fluid magnetohydrodynamic simulations of
  relativistic magnetic reconnection}, Astrophysical Journal 696~(2) (2009)
  1385--1401.
\newblock \href {https://doi.org/10.1088/0004-637X/696/2/1385}
  {\path{doi:10.1088/0004-637X/696/2/1385}}.

\bibitem{Balsara2016}
D.~S. Balsara, T.~Amano, S.~Garain, J.~Kim, {A high-order relativistic
  two-fluid electrodynamic scheme with consistent reconstruction of
  electromagnetic fields and a multidimensional Riemann solver for
  electromagnetism}, Journal of Computational Physics 318 (2016) 169--200.
\newblock \href {https://doi.org/10.1016/j.jcp.2016.05.006}
  {\path{doi:10.1016/j.jcp.2016.05.006}}.

\bibitem{Fjordholm2012}
U.~S. Fjordholm, S.~Mishra, E.~Tadmor, {Arbitrarily high-order accurate entropy
  stable essentially nonoscillatory schemes for systems of conservation laws},
  SIAM Journal on Numerical Analysis 50~(2) (2012) 544--573.
\newblock \href {https://doi.org/10.1137/110836961}
  {\path{doi:10.1137/110836961}}.

\bibitem{Fjordholm2013}
U.~S. Fjordholm, S.~Mishra, E.~Tadmor, {ENO Reconstruction and ENO
  Interpolation Are Stable}, Foundations of Computational Mathematics 13~(2)
  (2013) 139--159.
\newblock \href {https://doi.org/10.1007/s10208-012-9117-9}
  {\path{doi:10.1007/s10208-012-9117-9}}.

\bibitem{Chandrashekar2013}
P.~Chandrashekar, {Kinetic energy preserving and entropy stable finite volume
  schemes for compressible euler and Navier-Stokes equations}, Communications
  in Computational Physics 14~(5) (2013) 1252--1286.
\newblock \href {https://doi.org/10.4208/cicp.170712.010313a}
  {\path{doi:10.4208/cicp.170712.010313a}}.

\bibitem{Sen2018}
C.~Sen, H.~Kumar, {Entropy Stable Schemes For Ten-Moment Gaussian Closure
  Equations}, Journal of Scientific Computing 75~(2) (2018) 1128--1155.
\newblock \href {https://doi.org/10.1007/S10915-017-0579-4/FIGURES/8}
  {\path{doi:10.1007/S10915-017-0579-4/FIGURES/8}}.

\bibitem{Duan2021}
J.~Duan, H.~Tang, {High-order accurate entropy stable finite difference schemes
  for the shallow water magnetohydrodynamics}, Journal of Computational Physics
  431 (2021) 110136.
\newblock \href {https://doi.org/10.1016/J.JCP.2021.110136}
  {\path{doi:10.1016/J.JCP.2021.110136}}.

\bibitem{Duan2020}
J.~Duan, H.~Tang, {High-Order Accurate Entropy Stable Finite Difference Schemes
  for One- And Two-Dimensional Special Relativistic Hydrodynamics}, Advances in
  Applied Mathematics and Mechanics 12~(1) (2020) 1--29.
\newblock \href {https://doi.org/10.4208/AAMM.OA-2019-0124}
  {\path{doi:10.4208/AAMM.OA-2019-0124}}.

\bibitem{Biswas2022}
B.~Biswas, H.~Kumar, D.~Bhoriya, {Entropy stable discontinuous Galerkin schemes
  for the special relativistic hydrodynamics equations}, Computers {\&}
  Mathematics with Applications 112~(September 2021) (2022) 55--75.
\newblock \href {https://doi.org/10.1016/j.camwa.2022.02.019}
  {\path{doi:10.1016/j.camwa.2022.02.019}}.

\bibitem{Munz2000}
C.~D. Munz, P.~Omnes, R.~Schneider, E.~Sonnendr{\"{u}}cker, U.~Vo{\ss},
  {Divergence Correction Techniques for Maxwell Solvers Based on a Hyperbolic
  Model}, Journal of Computational Physics 161~(2) (2000) 484--511.
\newblock \href {https://doi.org/10.1006/jcph.2000.6507}
  {\path{doi:10.1006/jcph.2000.6507}}.

\bibitem{Schneider1993}
V.~Schneider, U.~Katscher, D.~H. Rischke, B.~Waldhauser, J.~A. Maruhn, C.~D.
  Munz, {New Algorithms for Ultra-relativistic Numerical Hydrodynamics},
  Journal of Computational Physics 105~(1) (1993) 92--107.
\newblock \href {https://doi.org/10.1006/jcph.1993.1056}
  {\path{doi:10.1006/jcph.1993.1056}}.

\bibitem{Tadmor1987}
E.~Tadmor, {The Numerical Viscosity of Entropy Stable Schemes for Systems of
  Conservation Laws. I}, Mathematics of Computation 49~(179) (1987) 91.
\newblock \href {https://doi.org/10.2307/2008251} {\path{doi:10.2307/2008251}}.

\bibitem{Ismail2009}
F.~Ismail, P.~L. Roe, {Affordable, entropy-consistent Euler flux functions II:
  Entropy production at shocks}, Journal of Computational Physics 228~(15)
  (2009) 5410--5436.
\newblock \href {https://doi.org/10.1016/j.jcp.2009.04.021}
  {\path{doi:10.1016/j.jcp.2009.04.021}}.

\bibitem{Lefloch2002}
P.~G. Lefloch, J.~M. Mercier, C.~Rohde, {Fully discrete, entropy conservative
  schemes of arbitrary order}, SIAM Journal on Numerical Analysis 40~(5) (2002)
  1968--1992.
\newblock \href {https://doi.org/10.1137/S003614290240069X}
  {\path{doi:10.1137/S003614290240069X}}.

\bibitem{Barth1999}
T.~J. Barth, {Numerical methods for gas-dynamics systems on unstructured
  meshes, An introduction to recent developments in theory and numerics of
  conservation Laws, Lecture notes in computational science and engineering
  volume {\{}5{\}}, Springer, Berlin. Eds: D. Kroner, M.}, Springer, Berlin,
  Heidelberg, 1999.

\bibitem{Jiang1996}
G.~S. Jiang, C.~W. Shu, {Efficient implementation of weighted ENO schemes},
  Journal of Computational Physics 126~(1) (1996) 202--228.
\newblock \href {https://doi.org/10.1006/jcph.1996.0130}
  {\path{doi:10.1006/jcph.1996.0130}}.

\bibitem{Gottlieb2001}
S.~Gottlieb, C.~W. Shu, E.~Tadmor, {Strong stability-preserving high-order time
  discretization methods}, SIAM Review 43~(1) (2001) 89--112.
\newblock \href {https://doi.org/10.1137/S003614450036757X}
  {\path{doi:10.1137/S003614450036757X}}.

\bibitem{Pareschi2005}
L.~Pareschi, G.~Russo, {Implicit-explicit Runge-Kutta schemes and applications
  to hyperbolic systems with relaxation}, Journal of Scientific Computing
  25~(1) (2005) 129--155.
\newblock \href {https://doi.org/10.1007/s10915-004-4636-4}
  {\path{doi:10.1007/s10915-004-4636-4}}.

\bibitem{Dennis1996}
J.~E. Dennis, R.~B. Schnabel, {Numerical Methods for Unconstrained Optimization
  and Nonlinear Equations}, Society for Industrial and Applied Mathematics,
  1996.
\newblock \href {https://doi.org/10.1137/1.9781611971200}
  {\path{doi:10.1137/1.9781611971200}}.

\bibitem{Kennedy2003}
C.~A. Kennedy, M.~H. Carpenter, {Additive Runge-Kutta schemes for
  convection-diffusion-reaction equations}, Applied Numerical Mathematics
  44~(1-2) (2003) 139--181.
\newblock \href {https://doi.org/10.1016/S0168-9274(02)00138-1}
  {\path{doi:10.1016/S0168-9274(02)00138-1}}.

\bibitem{Orszag1979}
S.~A. Orszag, C.~M. Tang, {Small-scale structure of two-dimensional
  magnetohydrodynamic turbulence}, Journal of Fluid Mechanics 90~(1) (1979)
  129--143.
\newblock \href {https://doi.org/10.1017/S002211207900210X}
  {\path{doi:10.1017/S002211207900210X}}.

\bibitem{Birn2001}
J.~Birn, J.~F. Drake, M.~A. Shay, B.~N. Rogers, R.~E. Denton, M.~Hesse,
  M.~Kuznetsova, Z.~W. Ma, A.~Bhattacharjee, A.~Otto, P.~L. Pritchett,
  {Geospace Environmental Modeling (GEM) Magnetic Reconnection Challenge},
  Journal of Geophysical Research: Space Physics 106~(A3) (2001) 3715--3719.
\newblock \href {https://doi.org/10.1029/1999ja900449}
  {\path{doi:10.1029/1999ja900449}}.

\end{thebibliography}

\appendix

\section{Proof of Proposition~\ref{prop:entropy}}
\label{ap:entropy}
We will give the proof of Proposition~\eqref{prop:entropy} in two steps. Firstly, we will prove two lemmas, and then we give the proof of proposition using these two lemmas. The proof is similar to those in \cite{Bhoriya2020} for relativistic hydrodynamics except for the additional source terms and three-dimensional velocity field.
\begin{lemma} \label{lemma1} The smooth solutions of the system~\eqref{conservedform_1d} satisfies the following identity
	\begin{gather}
		\frac{1}{\Gamma_\alpha^2} \partial_t p_\alpha= \partial_t p_\alpha +u_{x\alpha} \partial_x p_\alpha +\left( \frac{\rho_\alpha h_\alpha}{\Gamma_\alpha} \right) (\partial_t \Gamma_\alpha 	+\partial_x {(\Gamma_\alpha u_{x\alpha})}) - \rho_\alpha h_\alpha \partial_x u_{x,\alpha} - \mathbf{u}_\alpha\cdot(\mathbf{s}_{\mathbf{M}_\alpha} - \mathbf{u}_\alpha \mathbf{s}_{\mathcal{E}_\alpha}),
		\label{lemma1_eq}
	\end{gather}
	where 
	$$
	\mathbf{s}_{\mathbf{M}_\alpha} =r_\alpha \Gamma_\alpha \rho_\alpha (\mathbf{E}+\mathbf{u}_\alpha \times \mathbf{B}), \qquad
	\mathbf{s}_{\mathcal{E}_\alpha}=r_\alpha \Gamma_\alpha \rho_\alpha (\mathbf{u}_\alpha \cdot \mathbf{E}).
	$$
	are the source terms in the momentum and energy equations.
\end{lemma}  
\begin{proof}
	Let $\mathbf{U}$ be the smooth solution of the system \eqref{conservedform_1d}. Then we can write,
	\begin{align}
		\partial_t(\rho_\alpha \Gamma_\alpha) + \partial_x (\rho_\alpha \Gamma_\alpha u_{x_\alpha}) = 0, \label{mass} 
		\\ 
		\partial _t(\rho_\alpha h_\alpha \Gamma_\alpha ^2 u_{x_\alpha}) + \partial _x (\rho_\alpha h_\alpha \Gamma_\alpha ^2 u_{x_\alpha}^2) + \partial_x p_\alpha = \mathbf{s}_{M_{x_\alpha}}, \label{momentum1}
		\\ 
		\partial _t(\rho_\alpha h_\alpha \Gamma_\alpha ^2 u_{y_\alpha}) + \partial _x (\rho_\alpha h_\alpha \Gamma_\alpha ^2 u_{x_\alpha} u_{y_\alpha})  = \mathbf{s}_{M_{y_\alpha}}, \label{momentum2}
		\\ 
		\partial _t(\rho_\alpha h_\alpha \Gamma_\alpha ^2 u_{z_\alpha}) + \partial _x (\rho_\alpha h_\alpha \Gamma_\alpha ^2 u_{x_\alpha} u_{z_\alpha})  = \mathbf{s}_{M_{z_\alpha}}, \label{momentum3}
		\\ 
		\partial _t(\rho_\alpha h_\alpha \Gamma_\alpha ^2 ) -\partial_t p_\alpha  + \partial _x (\rho_\alpha h_\alpha \Gamma_\alpha ^2 u_{x_\alpha})  = \mathbf{s}_{\mathcal{E}_\alpha}, \label{ener}
	\end{align}
	where, we have used $\mathbf{s}_{\mathbf{M}_\alpha} = (\mathbf{s}_{M_{x_\alpha}}, \mathbf{s}_{M_{y_\alpha}}, \mathbf{s}_{M_{z_\alpha}})$,  as the momentum source vector, i.e., source terms of the Eqns.~\eqref{ion_momentum} and \eqref{elec_momentum}. Also, $\mathbf{s}_{\mathcal{E}_\alpha}$ is the source term of the energy equations~\eqref{ion_energy}, and \eqref{elec_energy}. We apply the product rule on Equations~\eqref{momentum1}-\eqref{momentum3} and substitute the value of $\partial_x(\rho_\alpha h_\alpha \Gamma^2_\alpha u_{x_\alpha})$ from Eqn.~\eqref{ener} in each of the three equations to obtain the following set of three identities:
	\begin{subequations}
		\begin{align}
			(\rho_\alpha h_\alpha \Gamma^2_{\alpha} u_{x_\alpha})(\partial_t u_{x_\alpha} + u_{x_\alpha} \partial_x u_{x_\alpha})+ u_{x_\alpha}^2 \partial_t p_{\alpha} + u_{x_\alpha} \partial_x p_{\alpha} = u_{x_\alpha}(\mathbf{s}_{M_{x_\alpha}} - u_{x_\alpha} \mathbf{s}_{\mathcal{E}_\alpha}) \label{lemma1a} 
			\\
			(\rho_{\alpha} h_{\alpha} \Gamma^2_{\alpha} u_{y_\alpha}) ( \partial_t u_{y_\alpha} + u_{y_\alpha} \partial_x u_{y_\alpha}) + u_{y_\alpha}^2 \partial_t p_{\alpha}  = u_{y_\alpha}(\mathbf{s}_{M_{y_\alpha}} - u_{y_\alpha} \mathbf{s}_{\mathcal{E}_\alpha}) \label{lemma1b}
			\\
			(\rho_{\alpha} h_{\alpha} \Gamma^2_{\alpha} u_{z_\alpha}) ( \partial_t u_{z_\alpha} + u_{z_\alpha} \partial_x u_{z_\alpha}) + u_{z_\alpha}^2 \partial_t p_{\alpha}  = u_{z_\alpha}(\mathbf{s}_{M_{z_\alpha}} - u_{z_\alpha} \mathbf{s}_{\mathcal{E}_{\alpha}}) \label{lemma1c}	
		\end{align}
	\end{subequations}
	We observe that 
	$ \partial_t\Gamma_\alpha = \Gamma^3_\alpha \boldsymbol{u}_\alpha \cdot \partial_t \boldsymbol{u}_\alpha$, and 
	$ \partial_x\Gamma_\alpha = \Gamma^3_\alpha \boldsymbol{u}_\alpha \cdot \partial_x \boldsymbol{u}_\alpha.$
	Consequently, Eqn.~\eqref{lemma1a} simplifies to
	\begin{gather}
		\left( \frac{\rho_\alpha h_\alpha}{\Gamma_\alpha} \right)(\partial_t \Gamma_\alpha + u_{x_\alpha} \partial_x \Gamma_\alpha)
		- \rho_\alpha h_\alpha \Gamma_\alpha^2 u_{y_\alpha} ( \partial_t u_{y_\alpha} + u_{x_\alpha} \partial_x u_{y_\alpha})  
		- \rho_\alpha h_\alpha \Gamma_\alpha^2 u_{z_\alpha} ( \partial_t u_{z_\alpha} + u_{x_\alpha} \partial_x u_{z_\alpha}) \nonumber
		\\
		+ u_{x_\alpha}^2 \partial_t p_\alpha + u_{x_\alpha} \partial_x p_\alpha 
		= 
		u_{x_\alpha}(\mathbf{s}_{M_{x_\alpha}} - u_{x_\alpha} \mathbf{s}_{\mathcal{E}_\alpha}), \label{temp2}
	\end{gather} 
	From the definition of Lorentz factor $\Gamma_\alpha$, we have $u_{x_\alpha}^2= 1- \dfrac{1}{\Gamma_\alpha^2 } - u_{y_\alpha}^2 - u_{z_\alpha}^2$. We substitute this value of $u_{x_\alpha}^2$ in Eqn.~\eqref{temp2}, followed by the substitution of value of $u_{y_\alpha}^2 \partial_t p_\alpha$ and $u_{z_\alpha}^2 \partial_t p_\alpha$ from the Eqn.~\eqref{lemma1b}-\eqref{lemma1c}, to obtain the required identity~\eqref{lemma1_eq}.
\end{proof}
\begin{lemma} \label{lemma2}
	For smooth solutions of the system \eqref{conservedform_1d}, we have the following identity
	\begin{equation}
		(\partial_t p_\alpha +u_{x_\alpha} \partial_x p_\alpha )+\frac{p_\alpha \gamma_\alpha}{\Gamma_\alpha}(\partial_t \Gamma_\alpha +\partial_x {(\Gamma_\alpha u_{x_\alpha})}) = 0.
		\label{lemma2_eq}
	\end{equation}
\end{lemma}
\begin{proof}
	We expand Eqn.~\eqref{ener} using the product rule and simplify further using the substitution of variables from Eqn.~\eqref{mass} to obtain,
	\begin{gather}
		\partial_t p_\alpha+ u_{x_\alpha} \partial_x p_\alpha = \left( \dfrac{\gamma_\alpha-1}{\Gamma_\alpha^2 \gamma_\alpha}\right) \mathbf{s}_{\mathcal{E}_\alpha} + \left(\frac{\gamma_\alpha-1}{\Gamma_\alpha^2 \gamma_\alpha} \right)\partial_t p_\alpha 
		-\left(\frac{\rho_\alpha}{\Gamma_\alpha}\frac{\gamma_\alpha-1}{\gamma_\alpha}\right)(\partial_t \Gamma_\alpha +\partial_x {(\Gamma_\alpha u_{x_\alpha})})\nonumber \\
		- 2(\partial_t \Gamma_\alpha +\partial_x {(\Gamma_\alpha u_{x_\alpha})})\frac{p_\alpha}{\Gamma_\alpha} 
		+ \rho_\alpha \left(\frac{\gamma_\alpha-1}{\gamma_\alpha} \right)\partial_x u_{x_\alpha}  + p_\alpha\partial_x u_{x_\alpha}. 
	\end{gather}
	Substituting the value of $\dfrac{1}{\Gamma_\alpha^2} \partial_t p_\alpha$ from the identity of Lemma~\eqref{lemma1}, and simplifying further we get
	\begin{align}
		\partial_t p_\alpha+u_{x_\alpha} \partial_x p_\alpha
		&=
		\left(\frac{\gamma_\alpha-1}{\gamma_\alpha}\right) \left(- \mathbf{u}_\alpha\cdot\mathbf{s}_{\mathbf{M}_\alpha}  +  \left(\frac{1+\mathbf{u}_\alpha^2}{\Gamma_\alpha^2}\right) \mathbf{s}_{\mathcal{E}_\alpha}\right)\nonumber
		\\
		&+ 
		\left(\frac{\gamma_\alpha-1}{\gamma_\alpha}\right)(\partial_t p_\alpha+u_{x_\alpha} \partial_x p_\alpha) \nonumber
		\\ 
		&+
		\left(\frac{\gamma_\alpha-1}{\gamma_\alpha}\right) \left[ \left( \frac{\rho_\alpha h_\alpha}{\Gamma_\alpha}\right) (\partial_t \Gamma_\alpha +\partial_x (\Gamma_\alpha u_{x_\alpha})) -\rho_\alpha h_\alpha \partial_x u_{x_\alpha} \right] \nonumber
		\\ 
		&-
		\left(\frac{\rho_\alpha}{\Gamma_\alpha}\frac{\gamma_\alpha-1}{\gamma_\alpha} \right)(\partial_t \Gamma_\alpha +\partial_x {(\Gamma_\alpha u_{x_\alpha})})- \frac{2p_\alpha}{\Gamma_\alpha}(\partial_t \Gamma_\alpha +\partial_x {(\Gamma_\alpha u_{x_\alpha})})  \nonumber
		\\
		&+ \left(\frac{\gamma_\alpha-1}{\gamma_\alpha}\right)\rho_\alpha \partial_x u_{x_\alpha}   +p_\alpha \partial_x u_{x_\alpha}, \nonumber
	\end{align}
	We observe that $\mathbf{u}_\alpha\cdot \mathbf{s_M}_\alpha  -  \mathbf{s}_{\mathcal{E}_\alpha} = 0$, thus, after further simplifications and cancellations \cite{Bhoriya2020}, we obtain~\eqref{lemma2_eq}.
\end{proof}

\begin{proof}[\textbf{Proof of Proposition~\eqref{prop:entropy}}]
	We find the partial derivatives of $s_\alpha = \ln(p_\alpha \rho_\alpha^{-\gamma_\alpha})$ with respect to $t$ and $x$, to obtain 
	\begin{equation}
		\partial_t s_\alpha= \frac{1}{p_\alpha} \partial_t p_\alpha -  \frac{\gamma_\alpha}{\rho_\alpha} \partial_t \rho_\alpha 
		\quad \text{ and } \quad
		\partial_x s_\alpha=  \frac{1}{p_\alpha} \partial_x p_\alpha -  \frac{\gamma_\alpha}{\rho_\alpha} \partial_x \rho_\alpha.   	\label{partial_s} 
	\end{equation} 
	Applying the product rule on the mass density Eqn.~\eqref{mass} we obtain an identity in terms of the partial derivatives of $\rho_\alpha$ 
	\begin{equation}
		\frac{1}{\rho_\alpha} (\partial_t \rho_\alpha +u_{x_\alpha} \partial_x \rho_\alpha) = -\frac{1}{\Gamma_\alpha } (\partial_t \Gamma_\alpha +\partial_x {(\Gamma_\alpha u_{x_\alpha})}). \label{masscomb}
	\end{equation}
	Combining Eqn.~\eqref{lemma2_eq} and Eqn.~\eqref{masscomb}, we obtain the expression
	\begin{equation*}
		\frac{1}{p_\alpha} (\partial_t p_\alpha + u_{x_\alpha} \partial_x p_\alpha) -\frac{\gamma_\alpha}{\rho_\alpha} (\partial_t \rho_\alpha + u_{x_\alpha} \partial_x \rho_\alpha)=0, 
	\end{equation*}
	which, on using Eqn.~\eqref{partial_s}, is equivalent to $\partial_t s_\alpha+u_{x_\alpha} \partial_x s_\alpha=0$. The identity~\eqref{h_s} is a direct consequence of the product rule. To obtain the final entropy equality~\eqref{entropy_pair_equality}, we simplify the identity~\eqref{h_s} by choosing particular value of $H(s_\alpha)$ as $H(s_\alpha)=\dfrac{-s_\alpha}{\gamma_\alpha-1}$.
\end{proof}
\
\section{Barth scaling of right eigenvectors}
Here, we present expressions for the right eigenvectors and scaled right eigenvectors for the conservative system~\eqref{conservedform}. The right eigenvectors are described in Section~\ref{ap:right_eig}, while the expressions and procedure to obtain the scaled eigenvectors is presented in Section~\ref{ap:scaled_eig} 
\subsection{Right eigenvectors} \label{ap:right_eig}

The set of right eigenvectors of the matrix $\mathbf{A}^x = \frac{\partial \mathbf{f}^x}{\partial \mathbf{U}}$ corresponding to the eigenvalues $\Lambda^x$ of the system~\eqref{conservedform} is obtained by taking $d=x$ in the ordered set $\mathbf{R}^d_{\Lambda^d} = \left\{\left(\mathbf{R}_{\Lambda^d}^d\right)_n: \ n = 1, \ 2 ,\ 3,\dots, \ 18\right\}$ where the vectors $\left(\mathbf{R}_{\Lambda^d}^d\right)_n$ are defined as 
\begin{align} \label{x_eigvec}
	\left(\mathbf{R}_{\Lambda^d}^d\right)_n =
	\begin{cases}
		\Big((\mathbf{R}_{i,k}^{d})_{1 \times 5}, 
		\mathbf{0}_{1 \times 5},
		\mathbf{0}_{1 \times 8}\Big)^\top, 
		&
		1 \leq n \leq 5, \ k = n
		\\
		\Big(\mathbf{0}_{1 \times 5},
		(\mathbf{R}_{e,k}^{d})_{1 \times 5}, 
		\mathbf{0}_{1 \times 8}\Big)^\top, 
		&
		6 \leq n \leq 10, \ k = n-5
		\\
		\Big(\mathbf{0}_{1 \times 5},
		\mathbf{0}_{1 \times 5}, 
		(\mathbf{R}_{m,k}^{d})_{1 \times 8}\Big)^\top, 
		&
		11 \leq n \leq 18, \ k = n-10,
	\end{cases}
\end{align}		
where $\mathbf{R}^{d}_{\alpha,k}$, $\alpha \in \{i,e\}$, is the $k^{th}$ column vector of the $5 \times 5$ right eigenvector matrices $\mathbf{R}^d_\alpha$ of the flux jacobians $\dfrac{\partial\mathbf{f}^d_\alpha}{\partial \mathbf{U}_\alpha}$, and $\mathbf{R}^{d}_{m,k}$ is the $k^{th}$ column vector of the right eigenvector matrix $\mathbf{R}^d_m$ of the flux jacobian matrix $\dfrac{\partial\mathbf{f}^d_m}{\partial \mathbf{U}_m}$. The matrices $\mathbf{R}^{k}_{\alpha}$ and $\mathbf{R}^{k}_{m}$ have the following expressions.
\begin{itemize}
	\item For $d =x,y$, and $\alpha \in \{i,e\}$, the right eigenvector matrix $\mathbf{R}^d_\alpha$ is given by the relation, 
	\begin{equation} \label{right_eig_con}
		\mathbf{R}_{\alpha}^{d}= \left(\dfrac{\partial \mathbf{U}_\alpha}{\partial \mathbf{W}_\alpha} \right) \mathbf{R}_{\alpha,\mathbf{W}}^d,
	\end{equation}
	where the matrix $\mathbf{R}_{\alpha,\mathbf{W}}^d$ for $d=x$ is given by, 
	%
	\begin{gather}
		\mathbf{R}_{\alpha,\mathbf{W}}^x = 
		\begin{pmatrix}
			\frac{1}{c_\alpha^2 h_\alpha} & 1 & 0 & 0 & \frac{1}{c_\alpha^2 h_\alpha}  
			\\
			\frac{-\sqrt{Q^x_\alpha}}{c_\alpha h_\alpha \Gamma_\alpha \rho_\alpha} & 0 & 0 & 0 &  \frac{+\sqrt{Q^x_\alpha}}{c_\alpha h_\alpha \Gamma_\alpha \rho_\alpha}
			\\
			\frac{\left(c_\alpha-\Gamma_\alpha \sqrt{Q^x_\alpha} {u_{x_\alpha}}\right) {u_{y_\alpha}}}{c_\alpha h_\alpha \Gamma^2_\alpha \rho_\alpha \left(u_{x_\alpha}^2-1\right)} & 0 & 1 & 0 & \frac{\left(c_\alpha+\Gamma_\alpha \sqrt{Q^x_\alpha} {u_{x_\alpha}}\right) {u_{y_\alpha}}}{c_\alpha h_\alpha \Gamma^2_\alpha \rho_\alpha \left(u_{x_\alpha}^2-1\right)}  
			\\
			\frac{\left(c_\alpha-\Gamma_\alpha \sqrt{Q^x_\alpha} {u_{x_\alpha}}\right) {u_{z_\alpha}}}{c_\alpha h_\alpha \Gamma^2_\alpha \rho_\alpha \left(u_{x_\alpha}^2-1\right)} & 0 & 0 & 1 & \frac{\left(c_\alpha+\Gamma_\alpha \sqrt{Q^x_\alpha} {u_{x_\alpha}}\right) {u_{z_\alpha}}}{c_\alpha h_\alpha \Gamma^2_\alpha \rho_\alpha \left(u_{x_\alpha}^2-1\right)}  
			\\
			1 & 0 & 0 & 0 & 1
		\end{pmatrix}. \label{x_eigvec:fluid}
	\end{gather}
	%
	\item The eigenvector matrix $\mathbf{R}^{d}_{m}$ for $d=x$ is given by,
	\begin{gather*}
		\mathbf{R}_{m}^{x}=
			\begin{pmatrix}
			0 &-1 & 0 & 0 & 0 & 0 & 1 & 0 \\
			0 & 0 & 0 & 1 & 0 &-1 & 0 & 0 \\
			0 & 0 &-1 & 0 & 1 & 0 & 0 & 0 \\
			-1 & 0 & 0 & 0 & 0 & 0 & 0 & 1 \\
			0 & 0 & 1 & 0 & 1 & 0 & 0 & 0 \\
			0 & 0 & 0 & 1 & 0 & 1 & 0 & 0 \\
			1 & 0 & 0 & 0 & 0 & 0 & 0 & 1 \\
			0 & 1 & 0 & 0 & 0 & 0 & 1 & 0 
		\end{pmatrix}.  
	\end{gather*}
\end{itemize}

\begin{remark}
	For the y-directional flux, $\mathbf{f}^y$, we proceed similarly to get the set of eigenvalues $\Lambda^y$ of the jacobian matrix $\dfrac{\partial \mathbf{f}^y}{\partial \mathbf{U}}$ as,
	\begin{align*} 
		\Lambda^y 
		= 
		\biggl\{
		& 	\frac{(1-c_i^2)u_{y_i}-(c_i/\Gamma_i) \sqrt{Q_i^y}}{1-c_i^2 |\mathbf{u}_i|^2},\
		u_{y_i}, \ u_{y_i}, \ u_{y_i},
		\frac{(1-c_i^2)u_{y_i}+(c_i/\Gamma_i) \sqrt{Q_i^y}}{1-c_i^2 |\mathbf{u}_i|^2}, 
		\nonumber
		\\  
		& 	\frac{(1-c_e^2)u_{y_e}-(c_e/\Gamma_e) \sqrt{Q_e^y}}{1-c_e^2 |\mathbf{u}_e|^2},\
		u_{y_e}, \ u_{y_e}, \ u_{y_e},
		\frac{(1-c_e^2)u_{y_e}+(c_e/\Gamma_e) \sqrt{Q_e^y}}{1-c_e^2 |\mathbf{u}_e|^2}, 
		\nonumber
		\\
		&	-\chi, \
		-\kappa, \
		-1, \
		-1, \
		1, \
		1, \
		\kappa, \
		\chi \
		\biggr\},
	\end{align*}
	where, $Q_\alpha^y=1-u_{y_\alpha}^2-c_\alpha^2 (u_{x_\alpha}^2+u_{z_\alpha}^2), \ \alpha \in \{ i, \ e \}$. The corresponding right eigenvectors are given by fixing $d=y$ in the ordered set $\mathbf{R}^d_{\Lambda^d} = \left\{\left(\mathbf{R}_{\Lambda^d}^d\right)_n: \ n = 1, \ 2 ,\ 3,\dots, \ 18\right\}$ where $\left(\mathbf{R}_{\Lambda^d}^d\right)_n$ is defined by Eqn.~\eqref{x_eigvec} and the updated $y$-directional matrices $\mathbf{R}_{\alpha,\mathbf{W}}^y$ and $\mathbf{R}^{y}_{m}$ are given by
	\begin{gather}
		\mathbf{R}_{\alpha,\mathbf{W}}^y = 
		\begin{pmatrix}
			\frac{1}{c_\alpha^2 h_\alpha} & 1 & 0 & 0 & \frac{1}{c_\alpha^2 h_\alpha}  
			\\
			\frac{\left(c_\alpha-\Gamma_\alpha \sqrt{Q^y_\alpha} {u_{y_\alpha}}\right) {u_{x_\alpha}}}{c_\alpha h_\alpha \Gamma^2_\alpha \rho_\alpha \left(u_{y_\alpha}^2-1\right)} & 0 & 1 & 0 & \frac{\left(c_\alpha+\Gamma_\alpha \sqrt{Q^y_\alpha} {u_{y_\alpha}}\right) {u_{x_\alpha}}}{c_\alpha h_\alpha \Gamma^2_\alpha \rho_\alpha \left(u_{y_\alpha}^2-1\right)}  
			\\
			\frac{-\sqrt{Q^y_\alpha}}{c_\alpha h_\alpha \Gamma_\alpha \rho_\alpha} & 0 & 0 & 0 &  \frac{+\sqrt{Q^y_\alpha}}{c_\alpha h_\alpha \Gamma_\alpha \rho_\alpha}
			\\
			\frac{\left(c_\alpha-\Gamma_\alpha \sqrt{Q^y_\alpha} {u_{y_\alpha}}\right) {u_{z_\alpha}}}{c_\alpha h_\alpha \Gamma^2_\alpha \rho_\alpha \left(u_{y_\alpha}^2-1\right)} & 0 & 0 & 1 & \frac{\left(c_\alpha+\Gamma_\alpha \sqrt{Q^y_\alpha} {u_{y_\alpha}}\right) {u_{z_\alpha}}}{c_\alpha h_\alpha \Gamma^2_\alpha \rho_\alpha \left(u_{y_\alpha}^2-1\right)}  
			\\
			1 & 0 & 0 & 0 & 1
		\end{pmatrix}, \label{y_eigvec:fluid}
		\\
		\mathbf{R}_{m}^{y}=
			\begin{pmatrix*}
			0 & 0 & 0 &-1 & 0 & 1 & 0 & 0 \\
			0 &-1 & 0 & 0 & 0 & 0 & 1 & 0 \\
			0 & 0 & 1 & 0 &-1 & 0 & 0 & 0 \\
			0 & 0 & 1 & 0 & 1 & 0 & 0 & 0 \\
			-1 & 0 & 0 & 0 & 0 & 0 & 0 & 1 \\
			0 & 0 & 0 & 1 & 0 & 1 & 0 & 0 \\
			1 & 0 & 0 & 0 & 0 & 0 & 0 & 1 \\
			0 & 1 & 0 & 0 & 0 & 0 & 1 & 0
		\end{pmatrix*}.  
	\end{gather}
\end{remark}

\subsection{Barth scaling and entropy scaled right eigenvectors}
\label{ap:scaled_eig}
We ignore the superscript $x$ to derive the entropy scaled right eigenvectors for the $x$-directional part; for the $y$-directional flux, we proceed similarly. We have the entropy variable vector $\mathbf{V}_\alpha$ as

\begin{equation*}
	\mathbf{V}_\alpha=\begin{pmatrix}
		\dfrac{\gamma_\alpha- s_\alpha}{\gamma_\alpha -1} +{\beta_\alpha} \\
		{u_{x_\alpha} \Gamma_\alpha \beta_\alpha } \\ 
		{u_{y_\alpha} \Gamma_\alpha \beta_\alpha } \\
		{u_{z_\alpha} \Gamma_\alpha \beta_\alpha } \\
		-{\Gamma_\alpha \beta_\alpha}
	\end{pmatrix}, \qquad \text{ where $\beta_\alpha=\frac{\rho_\alpha}{p_\alpha}$.} 
\end{equation*}
We want to find scaling matrices $\mathbf{T}^d_{\alpha}$ ($d\in\{x,y\} \text{ and } \alpha\in\{i,e\}$) such that the scaled right eigenvector matrices ${\mathbf{\tilde{R}}}^d_{\alpha}= \mathbf{R}^d_{\alpha} \mathbf{T}^d_{\alpha}$, with $\mathbf{R}^d_\alpha$ as in Eqn.~\eqref{right_eig_con}, satisfy
\begin{equation*}
	\frac{\partial \mathbf{U}_\alpha}{\partial \mathbf{V}_\alpha} =
	{\mathbf{\tilde{R}}^d_\alpha} ({\mathbf{\tilde{R}}^d_{\alpha}})^\top. 
\end{equation*}
The matrices $\mathbf{T}^d_\alpha$ are known as Barth scaling matrices.
Following the procedure of \cite{Barth1999}, the scaling matrices $\mathbf{T}^d_\alpha$ are given by the square root of the matrices $\mathbf{Y}^d_\alpha$ where the expression for matrices $\mathbf{Y}^d_\alpha$ is given by the formulae 
\begin{equation*}
	\mathbf{Y}^d_\alpha= ({\mathbf{R}^d_{\alpha,\mathbf{W}}})^{-1} 
	\frac{\partial \mathbf{W}_\alpha}{\partial \mathbf{V}_\alpha}
	\frac{\partial \mathbf{U}_\alpha}{\partial \mathbf{W}_\alpha}^{- \top}
	{\mathbf{R}^d_{\alpha,\mathbf{W}}}^{- \top},
\end{equation*}
where the matrices $\mathbf{R}_{\alpha,\mathbf{W}}^d$ are given by Eqn.~\eqref{x_eigvec:fluid} for $d=x$, and for the case $d=y$ the expressions for the matrices follows from Eqn.~\eqref{y_eigvec:fluid}. 
Assuming $d=x$ for the clarity. A long simplification leads to the expressions
\begin{equation*}
	\mathbf{Y}^x_\alpha=
	\begin{pmatrix}
		\frac{c_\alpha^2 h_\alpha p_\alpha \left(1+\frac{c_\alpha {u_{x_\alpha}}}{\Gamma_\alpha \sqrt{Q^x_\alpha}}\right)}{2 \Gamma_\alpha \left(1-c_\alpha^2 |\mathbf{u}|_\alpha^2\right)} & 0 & 0 & 0 & 0
		\\ 
		0 & \frac{\rho_\alpha}{\Gamma_\alpha} \frac{\gamma_\alpha-1}{\gamma_\alpha}& 0 & 0 & 0 
		\\ 
		0 & 
		0 & 
		Y^x_{\alpha_{33}} &
		Y^x_{\alpha_{34}} &
		0 
		\\ 
		0 & 
		0 & 
		Y^x_{\alpha_{43}} &
		Y^x_{\alpha_{44}} &
		0 
		\\ 
		0 & 0 & 0 & 0 &
		\frac{c_\alpha^2 h_\alpha p_\alpha \left(1-\frac{c_\alpha {u_{x_\alpha}}}{\Gamma_\alpha \sqrt{Q^x_\alpha}}\right)}{2 \Gamma_\alpha \left(1-c_\alpha^2 |\mathbf{u}|_\alpha^2\right)} 
	\end{pmatrix}
\end{equation*}
where, 
$Q_\alpha^x=
1-u_{x_\alpha}^2-c_\alpha^2 (u_{y_\alpha}^2+u_{z_\alpha}^2)$, 
$Y^x_{\alpha_{33}}=
C^x_\alpha {(1-u_{x_\alpha}^2-u_{y_\alpha}^2)} $, 
$Y^x_{\alpha_{44}}=
C^x_\alpha {(1-u_{x_\alpha}^2-u_{z_\alpha}^2)} $, 
$Y^x_{\alpha_{34}} = Y^x_{\alpha_{43}}= -C^x_\alpha u_{y_\alpha} u_{z_\alpha}$ and
\begin{equation*}
	C^x_\alpha = 
	\frac{p_\alpha}{h_\alpha \Gamma_\alpha^3 \rho_\alpha^2 \left(1-u_{x_\alpha}^2\right)} . 
\end{equation*} 
Consequently, the matrix $\mathbf{T}^x_\alpha = \sqrt{ \mathbf{Y}^x_\alpha}$ is given by
\begin{equation*}
	\mathbf{T}^x_\alpha=
	\begin{pmatrix}
		\sqrt{
			\frac{c_\alpha^2 h_\alpha p_\alpha \left(1+\frac{c_\alpha {u_{x_\alpha}}}{\Gamma_\alpha \sqrt{Q^x_\alpha}}\right)}{2 \Gamma_\alpha \left(1-c_\alpha^2 |\mathbf{u}|_\alpha^2\right)} 
		} 
		& 0 & 0 & 0 & 0
		\\ 
		0 & \sqrt{\frac{\rho_\alpha}{\Gamma_\alpha} \frac{\gamma_\alpha-1}{\gamma_\alpha} }& 0 & 0 & 0 
		\\ 
		0 & 
		0 & 
		T^x_{\alpha_{33}} &
		T^x_{\alpha_{34}} &
		0 
		\\ 
		0 & 
		0 & 
		T^x_{\alpha_{43}} &
		T^x_{\alpha_{44}} &
		0 
		\\ 
		0 & 0 & 0 & 0 &
		\sqrt{\frac{c_\alpha^2 h_\alpha p_\alpha \left(1-\frac{c_\alpha {u_{x_\alpha}}}{\Gamma_\alpha \sqrt{Q^x_\alpha}}\right)}{2 \Gamma_\alpha \left(1-c_\alpha^2 |\mathbf{u}|_\alpha^2\right)} } 
	\end{pmatrix}
\end{equation*}
where, 
\begin{equation*}
	\scriptstyle
	\begin{pmatrix}
		T^x_{\alpha_{33}} & T^x_{\alpha_{34}} \\
		T^x_{\alpha_{43}} & T^x_{\alpha_{44}} 
	\end{pmatrix}
	=
	\scriptstyle
	\begin{cases}
		\begin{pmatrix}
			\sqrt{Y^x_{\alpha_{33}}} & 0 \\
			0 & \sqrt{Y^x_{\alpha_{44}}}
		\end{pmatrix} & \text{ if } u_{y_\alpha} = u_{z_\alpha} = 0,
		\\ \\
		\sqrt{C^x_\alpha}
		\scriptstyle
		\begin{pmatrix}
			\frac{u_{z_\alpha}^2 \sqrt{1-u_{x_\alpha}^2}+\frac{u_{y_\alpha}^2}{\Gamma_\alpha}}{u_{y_\alpha}^2+u_{z_\alpha}^2}
			& 
			\frac{u_{y_\alpha} u_{z_\alpha} }{u_{y_\alpha}^2+u_{z_\alpha}^2} 
			\left(
			\frac{1}{\Gamma_\alpha} - \sqrt{1-u_{x_\alpha}^2}
			\right)
			\\ 
			\frac{u_{y_\alpha} u_{z_\alpha} }{u_{y_\alpha}^2+u_{z_\alpha}^2} 
			\left(
			\frac{1}{\Gamma_\alpha} - \sqrt{1-u_{x_\alpha}^2}
			\right)
			&
			\frac{u_{y_\alpha}^2 \sqrt{1-u_{x_\alpha}^2}+\frac{u_{z_\alpha}^2}{\Gamma_\alpha}}{u_{y_\alpha}^2+u_{z_\alpha}^2}
		\end{pmatrix} & \text{ otherwise.}
	\end{cases}
\end{equation*}

\begin{remark}
	Proceeding similarly in the y-direction with the matrix $\mathbf{R}_{\alpha,\mathbf{W}}^y$ as in Eqn.~\eqref{y_eigvec:fluid}, the scaling matrix is given by
	\begin{equation*}
		\mathbf{T}^y_\alpha=
		\begin{pmatrix}
			\sqrt{
				\frac{c_\alpha^2 h_\alpha p_\alpha \left(1+\frac{c_\alpha {u_{y_\alpha}}}{\Gamma_\alpha \sqrt{Q^y_\alpha}}\right)}{2 \Gamma_\alpha \left(1-c_\alpha^2 |\mathbf{u}|_\alpha^2\right)} 
			} 
			& 0 & 0 & 0 & 0
			\\ 
			0 & \sqrt{\frac{\rho_\alpha}{\Gamma_\alpha} \frac{\gamma_\alpha-1}{\gamma_\alpha} }& 0 & 0 & 0 
			\\ 
			0 & 
			0 & 
			T^y_{\alpha_{33}} &
			T^y_{\alpha_{34}} &
			0 
			\\ 
			0 & 
			0 & 
			T^y_{\alpha_{43}} &
			T^y_{\alpha_{44}} &
			0 
			\\ 
			0 & 0 & 0 & 0 &
			\sqrt{\frac{c_\alpha^2 h_\alpha p_\alpha \left(1-\frac{c_\alpha {u_{y_\alpha}}}{\Gamma_\alpha \sqrt{Q^y_\alpha}}\right)}{2 \Gamma_\alpha \left(1-c_\alpha^2 |\mathbf{u}|_\alpha^2\right)} } 
		\end{pmatrix}
	\end{equation*}
	where, 
	$Q_\alpha^y=
	1-u_{y_\alpha}^2-c_\alpha^2 (u_{x_\alpha}^2+u_{z_\alpha}^2)$, and the block entries are given by the matrices 
	\begin{equation*}
		\scriptstyle
		\begin{pmatrix}
			T^y_{\alpha_{33}} & T^y_{\alpha_{34}} \\
			T^y_{\alpha_{43}} & T^y_{\alpha_{44}} 
		\end{pmatrix}
		=
		\scriptstyle
		\begin{cases}
			\begin{pmatrix}
				\sqrt{Y^y_{\alpha_{33}}} & 0 \\
				0 & \sqrt{Y^y_{\alpha_{44}}}
			\end{pmatrix} & \text{ if } u_{x_\alpha} = u_{z_\alpha} = 0,
			\\ \\
			\sqrt{C^y_\alpha}
			\scriptstyle
			\begin{pmatrix}
				\frac{u_{z_\alpha}^2 \sqrt{1-u_{y_\alpha}^2}+\frac{u_{x_\alpha}^2}{\Gamma_\alpha}}{u_{x_\alpha}^2+u_{z_\alpha}^2}
				& 
				\frac{u_{x_\alpha} u_{z_\alpha} }{u_{x_\alpha}^2+u_{z_\alpha}^2} 
				\left(
				\frac{1}{\Gamma_\alpha} - \sqrt{1-u_{y_\alpha}^2}
				\right)
				\\ 
				\frac{u_{x_\alpha} u_{z_\alpha} }{u_{x_\alpha}^2+u_{z_\alpha}^2} 
				\left(
				\frac{1}{\Gamma_\alpha} - \sqrt{1-u_{y_\alpha}^2}
				\right)
				&
				\frac{u_{x_\alpha}^2 \sqrt{1-u_{y_\alpha}^2}+\frac{u_{z_\alpha}^2}{\Gamma_\alpha}}{u_{x_\alpha}^2+u_{z_\alpha}^2}
			\end{pmatrix} & \text{ otherwise,}
		\end{cases}
	\end{equation*}
	with 
	$Y^y_{\alpha_{33}}=
	C^y_\alpha {(1-u_{y_\alpha}^2-u_{x_\alpha}^2)} $, 
	$Y^y_{\alpha_{44}}=
	C^y_\alpha {(1-u_{y_\alpha}^2-u_{z_\alpha}^2)} $  and 
	\begin{equation*}
		C^y_\alpha = 
		\frac{p_\alpha}{h_\alpha \Gamma_\alpha^3 \rho_\alpha^2 \left(1-u_{y_\alpha}^2\right)} . 
	\end{equation*} 
	
\end{remark}

\pagebreak
\section{ARK3-IMEX coefficients table} \label{ap:ark_coeff}
The coefficients for the third order ARK-IMEX time update are given as follows,
\begin{table}[ht]
	\centering
	\begin{tabular}{ ll|cccc }
		&&&&& \\
		&  & $l=0$ &  $l=1$ & $l=2$ & $l=3$ \\
		\hline
		&&&&& \\
		& $m=1$ & 
		$\frac{1767732205903}{2027836641118}$ &
		0 &
		0 &
		0
		\\
		&&&&& \\
		$a_{ml}^{[NS]}$ &
		$m=2$ & 
		$\frac{5535828885825}{10492691773637}$ &
		$\frac{788022342437}{10882634858940}$ &
		0 &
		0
		\\
		&&&&& \\
		& $m=3$ & 
		$\frac{6485989280629}{16251701735622}$ &
		$-\frac{4246266847089}{9704473918619}$ &
		$\frac{10755448449292}{10357097424841}$ &
		0
		\\
		&&&&& \\
		\hline
		&&&&& \\
		& $m=1$ & 
		$\frac{1767732205903}{4055673282236}$ &
		$\frac{1767732205903}{4055673282236}$ &
		0 &
		0
		\\
		&&&&& \\
		$a_{ml}^{[S]}$ &
		$m=2$ & 
		$\frac{2746238789719}{10658868560708}$ &
		$-\frac{640167445237}{6845629431997}$ &
		$\frac{1767732205903}{4055673282236}$ &
		0
		\\
		&&&&& \\
		& $m=3$ & 
		$\frac{1471266399579}{7840856788654}$ &
		$-\frac{4482444167858}{7529755066697}$ &
		$\frac{11266239266428}{11593286722821}$ &
		$\frac{1767732205903}{4055673282236}$ 
		\\
		&&&&& \\
		\hline
		&&&&& \\
		& $b_{l}^{[NS]}$   & 
		$\frac{1471266399579}{7840856788654}$ &
		$-\frac{4482444167858}{7529755066697}$ &
		$\frac{11266239266428}{11593286722821}$ &
		$\frac{1767732205903}{4055673282236}$ 
		\\
		&&&&& \\
		\hline
		&&&&& \\
		& $b_{l}^{[S]}$   & 
		$\frac{2756255671327}{12835298489170}$ &
		$-\frac{10771552573575}{22201958757719}$ &
		$\frac{9247589265047}{10645013368117}$ &
		$\frac{2193209047091}{5459859503100}$ 
		\\
		&&&&& \\
		\hline
		&&&&& \\
		& $c_{l}$   &
		0 &
		$\frac{1767732205903}{2027836641118}$ &
		$\frac{3}{5}$ &
		1	
		\\
		&&&&& \\
		\hline
	\end{tabular}
	\caption[h]{The coefficients for ARK3-IMEX time update scheme.} 
	\label{table:carpenter}
\end{table}
\end{document}